\documentclass{amsart}
\usepackage{quiver}
\usepackage[lofdepth, lotdepth]{subfig}
\usepackage{adjustbox}
\usepackage{enumerate}
\usepackage{etex}
\usepackage{pstricks,amssymb}
\usepackage{pict2e}
\usepackage[utf8]{inputenc}
\usepackage{tikz}
\usepackage{tikz-cd}
\usetikzlibrary{arrows}
\usepackage{geometry}
\usepackage{rotating}
\usepackage{graphicx}
\usepackage{pst-node}
\usepackage{amsmath}
\usepackage{kbordermatrix}
\usepackage{hyperref}
\hypersetup{colorlinks,
	  citecolor=black,
	  filecolor=black,
	  linkcolor=black,
	  urlcolor=black}
\usepackage[capitalize,nameinlink,noabbrev,nosort]{cleveref}
\usepackage{amsmath,amscd}
\usepackage{youngtab}
\usepackage[boxsize=.3 em]{ytableau}
\usepackage{verbatim}

\psset{unit=1pt, arrowsize=4pt, linewidth=.7pt}
\psset{linecolor=blue}
\newgray{grayish}{.90}
\newrgbcolor{embgreen}{0 .5 0}
\def\Le{\hbox{\rotatebox[origin=c]{180}{$\Gamma$}}}
\def\vblack(#1, #2)#3{\cnode*[linecolor=black](#1, #2){3}{#3}}
\def\vwhite(#1,#2)#3{\cnode[linecolor=black,fillcolor=white,fillstyle=solid](#1,#2){3}{#3}}
\countdef\x=23
\countdef\y=24
\countdef\z=25
\countdef\t=26

\font\co=lcircle10

\def\jr{\rotatedown{\smash{\raise2pt\hbox{\co \rlap{\rlap{\char'005} \char'007}}
               \raise6pt\hbox{\rlap{\vrule height6.5pt}}
                \raise2pt\hbox{\rlap{\hskip4pt \vrule
          height0.4pt depth0pt
                width7.7pt}}}}}

\def\textcross{\ \smash{\lower4pt\hbox{\rlap{\hskip4.15pt\vrule height14pt}}
                \raise2.8pt\hbox{\rlap{\hskip-3pt \vrule height.4pt depth0pt
                                width14.7pt}}}\hskip12.7pt}

\def\textelbow{\ \hskip.1pt\smash{\raise2.75pt%
                \hbox{\co \hskip 4.15pt\rlap{\rlap{\char'004} \char'006}
                \lower6.8pt\rlap{\vrule height3.5pt}
                \raise3.6pt\rlap{\vrule height3.5pt}}
                \raise2.8pt\hbox{%
                  \rlap{\hskip-7.15pt \vrule height.4pt depth0pt
width3.5pt}%
                  \rlap{\hskip4.05pt \vrule height.4pt depth0pt
width3.5pt}}}
                \hskip8.7pt}

\def\tbox(#1,#2)#3{
\x=#1 \y=#2
\multiply\x by 12
\multiply\y by 12
\z=\x \t=\y
\advance\z by 12
\advance\t by 12
\psline(\x,\y)(\x,\t)(\z,\t)(\z,\y)(\x,\y)
\advance\x by 6
\advance\y by 6
\rput(\x,\y){{\bf #3}}}

\newtheorem{mainthm}{Theorem}

\newtheorem{thm}{Theorem}[section]
\newtheorem{theorem}[thm]{Theorem}
\newtheorem{cor}[thm]{Corollary}
\newtheorem{corollary}[thm]{Corollary}

\newtheorem{lem}[thm]{Lemma}
\newtheorem{lemma}[thm]{Lemma}

\newtheorem{prop}[thm]{Proposition}
\newtheorem{proposition}[thm]{Proposition}
\theoremstyle{definition}
\newtheorem{defn}[thm]{Definition}
\newtheorem{notation}[thm]{Notation}
\newtheorem{example}[thm]{Example}
\newtheorem{definition}[thm]{Definition}
\newtheorem{rem}[thm]{Remark}  
\newtheorem{remark}[thm]{Remark}

\numberwithin{equation}{section}

\newcommand{\nV}{\mathcal{V}_{\bullet}}
\newcommand{\vv}{\upsilon}
\newcommand{\V}{\mathcal{V}}
\newcommand{\rk}{\operatorname{rk}}
\newcommand{\Comment}[1]{{\color{blue} \sf ($\clubsuit$ #1 $\clubsuit$)}}

\newcommand{\Pic}{\operatorname{Pic}}

\newcommand{\aff}{\operatorname{Aff}}
\newcommand{\add}{\operatorname{add}}
\newcommand{\shift}{\operatorname{shift}}

\newcommand{\muminus}{\mu_{-}}
\newcommand{\mutilde}{\tilde{\mu}}
\newcommand{\muright}{\mu^{\square}}
\newcommand{\mudown}{_{\square}\mu}
\newcommand{\hh}{\mathbf h}
\newcommand{\rr}{\mathbf r}
\newcommand{\nn}{\mathbf n}
\newcommand{\uu}{\mathbf u}

\newcommand{\Mat}{\operatorname{Mat}}
\newcommand{\Gl}{\operatorname{GL}}
\newcommand{\Fl}{\operatorname{Fl}}
\newcommand{\pathne}[1]{L^\nearrow_{#1}} %For lattice paths, going northeast. #1 either a partition or the vertical steps
\newcommand{\pathsw}[1]{L^\swarrow_{#1}}
\newcommand{\vertne}[1]{V^\nearrow(#1)} %For vertical steps of paths, going northeast. #1 either a lattice path or partition
\newcommand{\vertsw}[1]{V^\swarrow(#1)}
\newcommand{\partne}[1]{\lambda^\nearrow(#1)} % For partitions associated to subsets of vertical steps (or lattice paths, I guess), going northeast
\newcommand{\partsw}[1]{\lambda^\swarrow(#1)}
\newcommand{\conv}{\operatorname{Conv}}
\newcommand{\Arr}{\operatorname{Arr}}
\newcommand{\Conv}{\operatorname{Conv}}
\newcommand{\Mut}{\operatorname{Mut}}
\newcommand{\MutVar}{\operatorname{MutVar}}
\newcommand{\NO}{\Delta}

\newcommand{\R}{\mathbb R}

\newcommand{\Ppoly}{\mathbf P}

\newcommand{\RG}{r_G}
\newcommand{\Vol}{\operatorname{Volume}}

\newcommand{\Q}{\Gamma}

\newcommand{\pyo}{p_{\ydiagram{1}}}
\newcommand{\pyz}{p_{\ydiagram{2}}}
\newcommand{\pyoo}{p_{\ydiagram{1,1}}}
\newcommand{\pyzz}{p_{\ydiagram{2,2}}}
\newcommand{\pytz}{p_{\ydiagram{4}}}
\newcommand{\pyt}{p_{\ydiagram{3}}}
\newcommand{\pytt}{p_{\ydiagram{3,3}}}
\newcommand{\pttt}{p_{\ydiagram{2,2,2}}}
\newcommand{\pooo}{p_{\ydiagram{1,1,1}}}

\newcommand{\FF}{\mathcal{F}}
\newcommand{\NN}{\mathcal{N}}
\newcommand{\CC}{\mathcal{N}}

\newcommand{\Shkn}{{\mathcal P_{k,n}}}
\newcommand{\rect}{\mathrm{rec}}

\newcommand{\wt}{\mathrm{wt}}

\newcommand{\Spec}{\mathrm{Spec}}

\newcommand{\Cl}{\operatorname{Cl}}

\newcommand{\OO}{\mathbb{O}}

\newcommand{\N}{\mathcal{N}}
\newcommand{\NP}{\operatorname{NP}}
\newcommand{\Root}{\operatorname{Root}}
\newcommand{\rW}{\overline{W}}
\newcommand{\rGamma}{\overline{\Gamma}}
\newcommand{\Poly}{\mathbf{P}}

\def\source{\operatorname{source}}
\def\target{\operatorname{target}}
\def\Rect{\operatorname{Rect}}
\def\sh{\operatorname{sh}}
\newcommand{\Rout}{\mathcal R_{\rm{out}}}
\newcommand{\Rin}{\mathcal R_{\rm{in}}}

\def\NW{\operatorname{NW}}
\def\SE{\operatorname{SE}}
\def\out{\operatorname{out}}
\def\O{\mathcal{O}}

\newcommand{\C}{\mathbb{C}}
\newcommand{\Z}{\mathbb{Z}}
\newcommand{\inv}{^{-1}}

\newcommand{\Dac}{D_{\operatorname{ac}}}

\newcommand{\XX}{{\mathcal X}}

\newcommand{\BB}{{\mathbb B}}

\newcommand{\val}{{\operatorname{val}}}

\renewcommand{\min}{{\operatorname{min}}}

\newcommand{\maxdiag}{{\operatorname{MaxDiag}}}

\newcommand{\A}{\mathcal{A}}

\newcommand{\Fr}{\operatorname{Fr}}

\newcommand{\minimal}{\operatorname{min}}
\newcommand{\PCG}{{\mathcal A\hskip -.05cm\operatorname{Coord}(G)}}

\newcommand{\frozen}{\operatorname{fr}}

\newcommand{\Trop}{\operatorname{Trop}}

\newcommand{\TBG}{{\mathcal X\hskip -.05cm\operatorname{Coord}(G)}}
\newcommand{\TB}{{\mathcal X\hskip -.05cm\operatorname{Coord}}}

\newcommand{\mathbbX}{\mathbb X}
\newcommand{\checkOmega}{\check{\Omega}}
\newcommand{\checkX}{\check{X}}
\newcommand{\X}{\mathbb X}
\newcommand{\opencheckX}{X_G^\circ}
\newcommand{\openX}{{X}^\circ}

\newcommand{\dualSchub}{\check{X}_{\lambda}}
\newcommand{\openSchub}{{X_{\lambda}^{\circ}}}
\newcommand{\opencheckSchub}{{\check{X}_{\lambda}^{\circ}}}
\newcommand{\checkD}{\check{D}}

\newcommand{\ac}{\operatorname{ac}}
\newcommand{\p}{{p}}

\makeatletter
\newcommand*\rel@kern[1]{\kern#1\dimexpr\macc@kerna}
\newcommand*\widebar[1]{%
  \begingroup
  \def\mathaccent##1##2{%
    \rel@kern{0.8}%
    \overline{\rel@kern{-0.8}\macc@nucleus\rel@kern{0.2}}%
    \rel@kern{-0.2}%
  }%
  \macc@depth\@ne
  \let\math@bgroup\@empty \let\math@egroup\macc@set@skewchar
  \mathsurround\z@ \frozen@everymath{\mathgroup\macc@group\relax}%
  \macc@set@skewchar\relax
  \let\mathaccentV\macc@nested@a
  \macc@nested@a\relax111{#1}%
  \endgroup
}
\makeatletter

\begin{document}

\title{A superpotential for Grassmannian Schubert varieties}

\author{K. Rietsch}
\address{Department of Mathematics,
            King's College London,
            Strand, London
            WC2R 2LS
            UK
}
\email{konstanze.rietsch@kcl.ac.uk}%
\author{L. Williams}%
\address{Department of Mathematics,
            Harvard University,
            Cambridge, MA
	    USA
}
\email{williams@math.harvard.edu}


\thanks{K.R. was supported by EPSRC grant EP/V002546/1.
L.W. was supported by 
 the NSF award
 DMS-2152991.}

\begin{abstract} 
While mirror symmetry for flag varieties and Grassmannians 
has been extensively studied, Schubert varieties in the Grassmannian
are singular, and hence standard mirror symmetry statements are not 
well-defined.  Nevertheless, in this article we introduce a 
``superpotential'' $W^{\lambda}$
for each Grassmannian Schubert variety $X_{\lambda}$, generalizing the Marsh-Rietsch superpotential for Grassmannians,
and we show that $W^{\lambda}$ governs many toric degenerations
of $X_{\lambda}$.  We also generalize the ``polytopal mirror theorem''
for Grassmannians from our previous work: namely,  
for any cluster seed $G$ for $X_{\lambda}$, we construct a  corresponding  
Newton-Okounkov convex body $\Delta_G^{\lambda}$, and show that it coincides
with the superpotential polytope $\Gamma_G^{\lambda}$, that is,
it is cut out by the inequalities obtained by tropicalizing 
an associated Laurent expansion of $W^{\lambda}$.
This gives us a toric degeneration
of the Schubert variety $X_{\lambda}$ to the (singular)
toric variety $Y(\NN_{\lambda})$
 of the Newton-Okounkov body.  Finally, for a particular cluster seed  $G=G^\lambda_\rect$ we show that the toric variety
$Y(\NN_{\lambda})$
has a small toric desingularisation, and we describe an intermediate partial desingularisation $Y(\FF_\lambda)$  that is Gorenstein Fano. 
Many of our results extend to more general varieties in the Grassmannian.
 \end{abstract}

\maketitle
\setcounter{tocdepth}{1}
\tableofcontents

\section{Introduction}

A Landau-Ginzburg mirror $(\check X^\circ, W)$ of a smooth Fano variety $X$ can be thought 
of as giving a dual description of a decomposition of $X=X^\circ\sqcup D$ where $D$ is an anti-canonical divisor of $X$. 
In this paper we initiate the study of mirror symmetry for general Schubert varieties $X_\lambda$ in the Grassmannian, and some generalisations thereof (e.g. skew shaped positroid varieties), in terms of a remarkable function $W^\lambda$. 
%In this paper we introduce a Landau-Ginzburg mirror for general Schubert varieties in the Grassmannian, and some generalisations thereof (e.g. skew Schubert varieties). 
While the Grassmannian is a smooth Fano variety, note that its Schubert varieties $X_\lambda$ are 
never smooth, apart from the trivial cases where  $X_\lambda$ is isomorphic 
to a (possibly lower-dimensional) Grassmannian. 
Most Schubert varieties $X_\lambda$ are not even Gorenstein, and therefore cannot be considered Fano.  
Nevertheless, in this paper we introduce a conjectural 
``Landau-Ginzburg mirror" for a Grassmannian Schubert
variety $X_{\lambda}$,
 associated to a Young diagram  $\lambda$.  
Our Landau-Ginzburg mirror for $X_\lambda$ is an affine subvariety of a Langlands dual Schubert variety, with a function  
$W^\lambda:\check X_{\lambda}^\circ\to\C$ on it called the \emph{superpotential}. 
We think of it as associated to the pair $(X_{\lambda},\Dac^{\lambda})$, where 
$\Dac^{\lambda}$ is a distinguished anticanonical divisor. Let us suppose a minimal Grassmannian containing  $X_\lambda$ is a
Grassmannian of subspaces of $\C^n$, and suppose that $d$ denotes the number of removable boxes in the Young diagram $\lambda$. (Note that $d$ is also  the dimension of the homology $d=\dim(H_{2|\lambda|-2}(X_\lambda,\C))$). Then our special anticanonical divisor
$\Dac^\lambda$ consists of $d+n-1$ irreducible components. 
Our associated LG-model $(\check X_\lambda^\circ, W^\lambda)$ defined in \cref{s:superpotential} has superpotential $W^\lambda=W^\lambda_q$ that is given as a sum of $d+n-1$ (Laurent) monomials in Pl\"ucker coordinates, and additionally depends on $d$ `quantum parameters' $q_1,\dotsc, q_d$.

Recall that in the original framework of  mirror symmetry for smooth Fano varieties  $X$ 
going back to \cite{Batyrev0,  Givental:ICM, Givental:toricCI, Givental:equivariant, OT, HoriVafa2000}, a mirror dual LG model $(\check X^\circ, W)$ consisting of an affine Calabi-Yau $\check X^\circ$ and a regular function $W$ on it, encodes Gromov-Witten invariants of $X$ in a variety of ways, e.g. via period integrals or the Jacobi ring of $W$.  
Our LG model $(\check X_{\lambda}^\circ,W^\lambda)$ introduced in this paper is formally of this type, with a base $\check X_{\lambda}^\circ$ which is a cluster variety and has a `standard' holomorphic volume form\footnote{The volume form is analogous to the form introduced for $G/P$ in \cite{rietsch}, see for example \cite{LamSp:ClusterCoh}.}, so that  one can in principle construct all the analogous generating functions. 
However, the Schubert variety $X_\lambda$ is singular and has no  Gromov-Witten theory. 
Thus this conjectural picture where, roughly speaking, 
 on the LG model side we are dealing with functions $\check X^\circ\to\C$ such as the 
superpotential  $W$ and its derivatives, and we are trying to reinterpret them on the $X$ side via Gromov-Witten theory  (that is, via moduli of maps 
$\mathbb P^1\to X$
in the other direction), is not applicable on this level.\footnote{However, in \cite{Miura:minuscule} conjectural mirror partners of this type are constructed for $3$-dimensional smooth complete intersection Calabi-Yau submanifolds in Gorenstein Schubert varieties, relating to our work via a particular coordinate chart, see \cref{s:toric}.} 

Our approach is instead to switch the roles of the two sides and study a different variant of mirror duality. Namely, let us now consider maps {\it from} $X=X_\lambda$ (or more generally sections of line bundles) on the compact side, and maps {\it into}  $\check X^\circ=\check X^\circ_\lambda$, (and their composition with the superpotential) on  
the LG model side.
Then we can study another form of mirror symmetry which relates these two, and where the above problem does not arise.  Namely on the compact side we can construct Newton-Okounkov bodies associated to ample divisors of $X_\lambda$ supported on an anticanonical `boundary divisor'  $D_{\ac}$ of $X_\lambda$. Meanwhile on the mirror LG model side we construct `superpotential polytopes' whose lattice points are in effect (equivalence classes of) maps $\phi:\Spec\, \C[t,t\inv]\to \check X^\circ$ such that $W^\lambda\circ\phi$ extends across $0$. On both sides it is necessary to pick an open torus $T_G\subset X_\lambda$ and $T^\vee_G\subset \check X_\lambda^\circ$ with a basis of characters, i.e. `coordinates', in order to set up the comparison. These tori are precisely what are given to us by an $\mathcal A$-cluster structure on $\C[\check X^\circ_\lambda]$ on the one hand, and its dual $\mathcal X$-cluster structure on the homogeneous coordinate ring  $\C[\widehat{X_\lambda}]$ on the other. Our first
main result is a `polytopal mirror duality'
statement, which says that the Newton-Okounkov convex bodies of $X_\lambda$ associated to $T_G$ coincide precisely with the superpotential polytopes associated to the restriction of $(\check X_{\lambda}^\circ,W_q^\lambda)$ to the torus chart $T^\vee_G\subset \check X_\lambda^\circ$.

This polytopal mirror theorem  generalizes our previous result for Grassmannians in \cite{RW}, 
and is related to the duality of cluster varieties of Fock and Goncharov~\cite{FG1} and 
Gross, Hacking, Keel and Kontsevich~\cite{GHKK}.  
See also 
the related work on cluster duality by Shen and Weng \cite{ShenWeng}, 
Genz, Koshevoy and Schumann \cite{Schumann}, 
 Bossinger, Cheung, Magee, and N\'ajera Ch\'avez
\cite{bossinger2023newtonokounkov}, and Spacek and Wang \cite{SW1, Wang}.

Our explicit description of the Newton-Okounkov bodies of $X_{\lambda}$
gives rise to many toric degenerations of $X_{\lambda}$, all governed
by the superpotential $W^{\lambda}$.  
In the special case where our cluster seed $G$ is the `rectangles seed,'
we get the well-known toric degeneration of $X_{\lambda}$ to the projective
toric variety $Y(\N_{\lambda})$ 
of the Newton-Okounkov body
$\Delta_{\rect}^{\lambda}$, which is a \emph{Gelfand-Tsetlin polytope},
and is unimodularly equivalent to 
the \emph{order polytope} $\OO(\lambda)$ associated to the poset
of rectangles contained in $\lambda$.\footnote{
This toric degeneration, to the \emph{Hibi toric variety} of 
$\OO(\lambda)$,
was first studied by Gonciulea and Lakshmibai
\cite{GL}; see also 
 \cite{BiswalFourier} and references therein.}
The toric variety $Y(\N_{\lambda})$ 
is singular,
but we show that it admits a small partial desingularization
to a Gorenstein toric  Fano variety $Y(\FF_{\lambda})$ with at most terminal singularities,
the toric variety of the face fan of the Newton polytope of 
$W_{\rect}^{\lambda}$, see \cref{fig:degeneration}.
Moreover, via $Y(\FF_{\lambda})$,  
we have a
small desingularisation of the  toric variety $Y(\CC_{\lambda})$, 
$$Y(\widehat{\FF}_{\lambda}) \to 
        Y(\FF_{\lambda}) \to Y(\CC_{\lambda}).$$

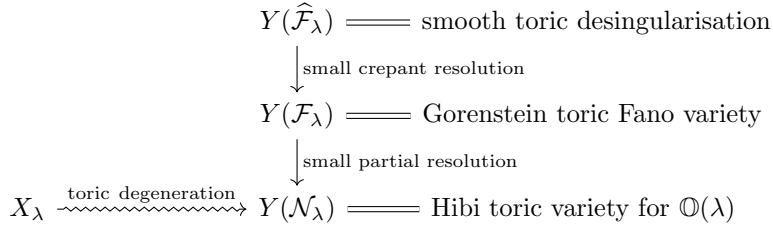
\begin{figure}
\[\begin{tikzcd}
	&&& {Y(\widehat{\mathcal F}_\lambda)} & {\text{smooth toric desingularisation}} \\
	&&& {Y({\mathcal F}_{\lambda})} & {\text{Gorenstein toric Fano variety\ }} \\
	{X_\lambda} &&& {Y(\mathcal N_{\lambda})} & {\text{Hibi toric  variety for } \mathbb O(\lambda)}\quad
	\arrow[Rightarrow, no head, from=1-4, to=1-5]
	\arrow["{\text{small crepant resolution}}", from=1-4, to=2-4]
	\arrow[Rightarrow, no head, from=2-4, to=2-5]
	\arrow["{\text{small partial resolution} }", from=2-4, to=3-4]
	\arrow["{\rm{toric\ degeneration}}", squiggly, from=3-1, to=3-4]
	\arrow[Rightarrow, no head, from=3-4, to=3-5]
\end{tikzcd}\]       
	\caption{ 
\label{fig:degeneration} 
	The Schubert variety $X_{\lambda}$
	degenerates to the projective toric variety 
	$Y(\mathcal{N}_{\lambda})$
	of the Newton-Okounkov body $\Delta^{\lambda}_{\rect}$, which in turn
has a small partial desingularization given by the Fano Gorenstein
toric variety 
	$Y(\mathcal{F}_{\lambda})$. Moreover $Y(\mathcal{F}_{\lambda})$ has a small
	crepant resolution.
	Via a coordinate change,
	$Y(\mathcal{N}_{\lambda})$ and 
	$Y(\mathcal{F}_{\lambda})$ 
	are isomorphic to the Hibi toric variety
	associated to the
	order polytope $\OO(\lambda)$ and to the toric variety
	for the face fan of the
	root polytope $\Root(Q_{\lambda})$, respectively.
	}
\end{figure}

While Schubert varieties are not in general smooth (unless $d=1$),
they are Cohen-Macaulay \cite{Hochster:Grass,Laksov:Grass,Musili:Grass} and normal \cite{RamRam:Schubert,DeCLak:Schubert}. Since $X_\lambda$ is normal, an anti-canonical divisor for $X_\lambda$ is any divisor $D$  whose restriction to the smooth part $U$ of $X_\lambda$ is anticanonical for $U$.
We moreover have a natural choice of an anti-canonical divisor for $X_{\lambda}$, namely the `boundary' anti-canonical divisor, described 
explicitly in \cref{c:ac}.
 This distinguished anticanonical divisor is 
  made up of $d+n-1$ irreducible components, 
which
   are precisely the codimension $1$ positroid strata in $X_\lambda$ (consisting of $d$ Schubert divisors and $n-1$ other positroid divisors). We denote this divisor by $\Dac^\lambda = D_1 + \dots + D_d+D'_1+\dotsc+ D'_{n-1}$, and we denote its complement in $X_\lambda$
 by $\openSchub$.

 We now introduce a conjectural 
``mirror Landau-Ginzburg model" 
$(\check{X_{\lambda}^{\circ}}, 
W^{\lambda}_{\mathbf q})$, where  
 $\dualSchub^\circ $ is the analogue of $X_{\lambda}^\circ=X_{\lambda} \setminus \Dac^{\lambda}$, 
but inside a Langlands dual Schubert variety $\check{X}_{\lambda}$, and 
$W^{\lambda}_{\mathbf q}:\dualSchub^\circ \to\C$ is a regular function that we call the {\it superpotential}. The superpotential is given by an explicit formula in terms of Pl\"ucker coordinates as a sum of $d+n-1$ terms, 
and it depends on several parameters $\mathbf{q} = (q_1,\dots, q_d)$.  

For example, if $\lambda = (4,4,2)$, then $X_{\lambda} \subset Gr_3(\C^7)$ 
and $\dim H^2(X_\lambda,\C) = 2$, the number of removable boxes of $\lambda$.  
The superpotential on $\dualSchub^\circ$ has $8$ summands, with the first two below associated to the two removable boxes, and it is explicitly given by the formula
\begin{equation}\label{e:IntroExample}
	W^{\lambda}=
	W^{\lambda}_{\mathbf q}=
	q_1 \frac{p_{\ydiagram{3}}}{p_{\ydiagram{4,4}}} +
	q_2 \frac{p_{\ydiagram{1,1}}}{p_{\ydiagram{2,2,2}}} +
	\frac{p_{\ydiagram{1}}}{p_{\emptyset}}+
	\frac{p_{\ydiagram{4,1}}}{\p_{\ydiagram{4}}} +
	\frac{p_{\ydiagram{4,3}}}{\p_{\ydiagram{3,3}}} +
	\frac{\left(p_{\ydiagram{2,2,1}}+p_{\ydiagram{3,2}}\right)}{\p_{\ydiagram{2,2}}}+
	\frac{ p_{\ydiagram{2,1,1}}}{p_{\ydiagram{1,1,1}}}.
\end{equation}
Here the $p_\lambda$ are Pl\"ucker coordinates for $Gr_4(\C^7)$; see \cref{s:notation}
for an explanation of the notation.

The $d+n-1$ summands of the superpotential individually give rise to functions, the first $d$ of which 
correspond to Schubert divisors and which  we denote by
\begin{equation*}\label{eq:Wi}
	 W_1 = 
	  \frac{p_{\ydiagram{3}}}{p_{\ydiagram{4,4}}}, 
	\quad 
	 W_2 =   \frac{ p_{\ydiagram{1,1}}}{p_{\ydiagram{2,2,2}}},
	\end{equation*}
so that summands of $W_{\mathbf q}^\lambda$ involving the $q_i$ are $q_1W_1$
and $q_2W_2$. 
The remaining  $n-1=6$ summands are denoted by
\begin{equation*}\label{eq:W'i}
	W'_1 = \frac{p_{\ydiagram{2,2,1}}}{\p_{\ydiagram{2,2}}},\quad 
	W'_2 = 
	\frac{p_{\ydiagram{4,1}}}{\p_{\ydiagram{4}}}, \quad 
	W'_3 = \frac{p_{\ydiagram{1}}}{p_{\emptyset}}, \quad
	W'_4 = \frac{ p_{\ydiagram{2,1,1}}}{p_{\ydiagram{1,1,1}}}, \quad
		W'_5=\frac {p_{\ydiagram{3,2}}}{\p_{\ydiagram{2,2}}}, \quad
				W'_6=
	\frac{p_{\ydiagram{4,3}}}{\p_{\ydiagram{3,3}}}. 
\end{equation*}
Note that each of the nonempty Young diagrams $\mu$ appearing in the denominator of some $W_i$ or $W_i'$ is a rectangle
contained in $\lambda$ whose lower-right box lies on the southeast \emph{rim} of $\lambda$. Note that the rim is made up of $n-1$ boxes. The corresponding Pl\"ucker coordinates together with $p_\emptyset$, that is the $n$ Pl\"ucker coordinates appearing in the denominator of $W^\lambda$, are precisely the frozen variables for an $\mathcal A$-cluster structure on $\C[\check X^{\circ}_\lambda]$, see \cref{sec:Acluster}.  We remark that 
this `canonical' expression for the superpotential
makes it clear that it is a regular function on $\check X^{\circ}_{\lambda}$,
and hence we can express it as a Laurent polynomial on any cluster torus.

If $\mu$ is such a rectangular Young diagram, and its lower-right box is the $j$-th  removable box of $\lambda$ (counting from northeast to southwest, see $W_1,W_2$  above), then there is only one term in $W^\lambda$ with denominator $p_\mu$, and the associated numerator is the product $q_j p_{\muminus}$, where $\muminus$ is obtained from $\mu$ by removing a rim hook. If on the other hand the lower-right box of $\mu$ is not a removable box of $\lambda$, then there are 
one or two terms in the 
numerator above $p_\mu$, each  
 obtained by adding a box to $\mu$ while remaining inside $\lambda$ 
(see $W'_1,\dotsc, W'_6$). The precise rules of the  construction of these summands for general $\lambda$ are given in \cref{s:superpotential}.

We will often use the normalisation $p_{\emptyset}=1$ so that the Pl\"ucker coordinates are
 actual coordinates on~$\dualSchub$.

\subsection{Main results}
We now give an overview of the main results of this paper.

Our first main theorem is the polytopal mirror theorem for Schubert varieties.
Given an
$\mathcal{X}$-cluster seed 
 $\Sigma_G^{\mathcal{X}}$ for
the open Schubert variety $X^{\circ}_{\lambda}$, we define 
an associated valuation $\val_G$ and use this to define a \emph{Newton-Okounkov body} 
$\Delta_G^{\lambda}$, see \cref{sec:NO}.
On the other side, we use the dual $\mathcal{A}$-cluster seed
 $\Sigma_G^{\mathcal{A}}$ for
the dual Schubert variety ${\check X}^{\circ}_{\lambda}$, and express the superpotential
$W^{\lambda}$ as a Laurent polynomial $W^{\lambda}_G$ in the variables of 
 $\Sigma_G^{\mathcal{A}}$.  By tropicalizing this Laurent polynomial we obtain 
 a set of inequalities which define the \emph{superpotential polytope}
$\Gamma^{\lambda}_{G},$ see \cref{sec:trop}.
The following theorem says that these two polytopes coincide.

\begin{mainthm}\label{t:mainIntro}
Let $\Sigma_G^{\mathcal{X}}$ be an arbitrary
$\mathcal{X}$-cluster seed for 
the open Schubert variety $X^{\circ}_{\lambda}$.
Then
 the Newton-Okounkov body
$\Delta^{\lambda}_{G}$
is a rational polytope
with lattice points
$\{\val_G(P_{\mu}) \ \vert \  \mu \subseteq \lambda\}$,
       and it coincides with
the superpotential polytope
$\Gamma^{\lambda}_{G}.$
We get a flat degeneration of $X_{\lambda}$ to the toric variety associated to the normal fan of the superpotential polytope.
\end{mainthm}
The constructions underlying the above theorem  implicitly use the divisor 
$D=D_1 + \dots + D_d$ associated to the Pl\"ucker embedding of $X_\lambda$. 
However, the above result can be generalized (see
\cref{t:maingen} and \cref{cor:degenerationgen}) to 
Newton-Okounkov bodies and superpotential polytopes
defined using arbitrary ample  divisors supported on the boundary.  (We 
 describe explicitly which divisors $r_1D_1+\dotsc+ r_{d} D_{d}+ r'_{1} D'_{1}+\dotsc + r'_{n-1} D'_{n-1}$ are Cartier and ample in \cref{s:GeometrySchubert}.) 
The most important example for us is however the one which we highlight in 
\cref{t:mainIntro}.

Our next main result, which appears as \cref{p:latticepoints-comp}, gives
an explicit ``maximal diagonal'' formula for the lattice points in the Newton-Okounkov body
$\Delta^{\lambda}_G$, when the cluster seed
 $\Sigma_G^{\mathcal{X}}$ comes from a \emph{reduced plabic graph}  $G$
 (see \cref{sec:plabic}).
\begin{mainthm}\label{t:mainmaxdiag}
Let $G$ be any reduced plabic graph for $X_{\lambda}^{\circ}$.
Then the lattice points 
$\{\val_G(P_{\mu}) \ \vert \  \mu \subseteq \lambda\}$ of 
the Newton-Okounkov body
$\Delta^{\lambda}_{G}$ have coordinates 
	$\val_G(P_\mu)_\eta$ (where 
	$\eta\in\mathcal P_G(\lambda)$)
	given by 
\begin{equation}\label{e2:GenMaxDiag}
\val_G(P_\mu)_\eta=\maxdiag(\eta\setminus\mu),
\end{equation}
	where $\maxdiag(\eta\setminus\mu)$ is the maximum number of boxes of 
	$\eta \setminus \mu$ that lie along any diagonal of slope $-1$ of the rectangle, 
	see \cref{d:MaxDiagVectors}.
\end{mainthm} 
We note that $\maxdiag(\eta\setminus\mu)$ is an interesting quantity that
has appeared in a variety of other contexts.  
By work of Fulton and Woodward \cite{FW}, 
$\maxdiag (\eta \setminus \mu)$ is equal to the
smallest degree $d$ such that $q^d$ appears in the Schubert expansion of 
the quantum product of two Schubert
classes $\sigma^{\eta}\star \sigma^{\mu^c}
$ in the quantum cohomology ring  $QH^*(Gr_k(\C^n))$, 
where $\sigma^{\mu^c}$ is the Poincar\'e dual Schubert class to $\sigma^\mu$.
See also \cite{Yong, PostnikovDuke}. 
Moreover, in the quantum cluster algebra $\C_q[Gr_{k,n}]$, if $I,J\in {[n] \choose k}$ are
\emph{non-crossing}, so that the quantum minors $\Delta_I$ and $\Delta_J$ quasi-commute,
then by a result of Jenson, King and Su \cite[Lemma 7.1 and Theorem 6.5]{JKS1},
\begin{equation}
	q^{\maxdiag(\lambda_I\setminus \lambda_J)} \Delta_I \Delta_J = 
	q^{\maxdiag(\lambda_J\setminus \lambda_I)} \Delta_J \Delta_I,
\end{equation}
where $\lambda_I$ is the partition associated to the subset $I$.

Our third set of results concerns the special case where $G$ is the \emph{rectangles seed}
(see \cref{def:rectseed}).  
The various statements in the following theorem appear as 
\cref{prop:super2}, \cref{p:unimodular},
\cref{prop:refine}, and \cref{cor:small}.

\begin{mainthm}\label{t:mainrectangles}
Suppose that  $G=G_{\rect}$ is the rectangles seed.  Then the following statements hold.
\begin{enumerate}
\item	 We have an explicit ``head over tails''
expression for the Laurent expansion $W^{\lambda}_{\rect}$ of the superpotential $W^{\lambda}$
in terms of $G$, which 
can be read off of an associated quiver $Q_\lambda$,
see \cref{fig:superpotential}. Moreover, in this case the Newton-Okounkov body $\Delta_G^\lambda=\Gamma^\lambda_G$  
is unimodularly equivalent to the \emph{order polytope $\OO(\lambda)$} of the poset $P(\lambda)$
of rectangles contained in $\lambda$.
\item 
Let $\CC_\lambda$ denote the normal fan of $\Delta^\lambda_G$.  We have a toric degeneration of $X_{\lambda}$ to the associated toric variety $Y(\CC_\lambda)$, which by (1) has an interpretation as the \emph{Hibi toric variety} associated to the poset $P(\lambda)$. 
This recovers Gonciulea and Lakshmibai's result that a minuscule Schubert variety degenerates to the Hibi toric variety for the order polytope of the associated minuscule poset \cite[Theorem 7.34]{GL}, in our setting.
\item 
The Newton polytope of $W_{\rect}^{\lambda}$ is unimodularly equivalent to the \emph{root polytope} of the quiver $Q_\lambda$, and is reflexive and terminal. The face fan $\FF_{\lambda}$ of the Newton polytope of $W_{\rect}^{\lambda}$
refines the normal fan $\N_{\lambda}$
of $\Delta_G^{\lambda}$, and both fans have the same set of rays.
\item Hence 
we obtain a small partial desingularization of
the Hibi toric variety $Y(\N_{\lambda})$ 
to the toric variety $Y(\FF_{\lambda})$, and  
$Y(\FF_{\lambda})$ 
is  Gorenstein Fano  
with at most
terminal singularities. The Picard group of $Y(\FF_\lambda)$ and the ample cone are combinatorially derived from the poset $P(\lambda)$.
\item By combining above results with \cite[Theorem E]{RWReflexive}, 
we have a 
small desingularisation of the Hibi toric variety
$Y(\widehat{\FF}_{\lambda}) 
		\to Y(\CC_{\lambda})$,
see \cref{fig:degeneration}. 
\end{enumerate}
\end{mainthm}

One key idea of this work is that one can 
use the paradigm ``Newton-Okounkov body equals  superpotential polytope'' to 
arrive at a conjectural definition of  superpotential, by computing the facet inequalities
of the Newton-Okounkov body then ``detropicalizing''.
Indeed, this is how we 
arrived at our notion of superpotential for Schubert varieties.  We plan to extend
our results to more general varieties such as positroid varieties
in a subsequent
work \cite{RW3}.

\vskip .2cm

The structure of this paper is as follows.
In \cref{s:notation} we start by setting up our notation 
for Grassmannians and Schubert varieties.  We then 
give a quick overview of the $\mathcal{A}$
and $\mathcal{X}$-cluster structures for (open) Schubert varieties 
in \cref{sec:A}.
%\cref{sec:Xcluster}
\cref{s:superpotential} gives the key definition of this paper, 
the definition of the superpotential $W^{\lambda}$ 
for Schubert varieties; 
we give several expressions for the superpotential,
including a ``canonical formula'' and a ``head over tails'' Laurent
polynomial from a quiver.
Then in \cref{s:GeometrySchubert} we describe the geometry
of the Schubert variety $X_{\lambda}$ and its boundary divisor;
we observe that the summands of the superpotential 
$W^{\lambda}$ are in natural bijection with the positroid divisors in 
$X_{\lambda}$.
In \cref{sec:NO} we define the \emph{Newton-Okounkov body}
$\Delta_G^{\lambda}(D)$ associated to an ample divisor $D$ in $X_{\lambda}$
and  a transcendence basis coming from a choice of 
$\mathcal{X}$-cluster for $X^{\circ}_{\lambda}$.
In \cref{sec:trop} we define the \emph{superpotential polytope}
$\Gamma_G^{\lambda}(D)$, which is also associated to a divisor
and a choice of $\mathcal{A}$-cluster for $\check{X}^{\circ}_{\lambda}$.
There is a particularly nice \emph{rectangles seed} 
for the cluster structure on a Schubert variety, 
and in \cref{sec:rectangles}, we explain how 
when we use this seed, the superpotential polytope becomes an \emph{order
polytope} (up to a unimodular change of variables).
This observation is the starting point for 
our proof, given in \cref{sec:polytopescoincide}, that 
the Newton-Okounkov body $\Delta_G^{\lambda}$ coincides with 
the superpotential polytope $\Gamma_G^{\lambda}$ when the divisor $D$
corresponds to the Pl\"ucker embedding.  (Our proof also uses
ingredients such as the \emph{theta function basis}, and the fact
that every frozen variable for a Schubert variety has an 
\emph{optimized seed}.)
In \cref{s:MaxDiag} we present our ``max-diagonal'' formula for 
valuations of Pl\"ucker coordinates in Schubert varieties, i.e. 
for the lattice points of our Newton-Okounkov bodies.
We then generalize our ``Newton-Okounkov body equals superpotential
polytope'' theorem to arbitrary Cartier boundary divisors in \cref{s:GenD}.
In \cref{s:toric} we give further context for referring to our 
function $W^{\lambda}$ as a superpotential: in particular, it governs
many toric degenerations of $X_{\lambda}$, including
a degeneration to the Hibi toric variety of an order polytope, which in turn
admits a small toric resolution.
\cref{sec:skew} explains how to generalize our superpotential
and our results to the 
setting of 
\emph{skew shaped positroid varieties}.
The paper ends with two appendices:
\cref{sec:appendix} gives a quick overview of positroids cells
and positroid varieties, while 
\cref{a:positroidhomology} proves an expression for the homology
class of a positroid divisor in terms of the Schubert classes.

\vskip .2cm

\noindent{\bf Acknowledgements:~}
The authors would like to thank Lara Bossinger, Eugene Gorsky,
Mark Gross, Elana Kalashnikov, Soyeon Kim, Alastair King, 
Timothy Magee, 
 Matthew Pressland, and Richard Stanley for helpful discussions.
K.R. is supported by EPSRC grant EP/V002546/1.
L.W. was supported by the National Science Foundation under Award No.
 DMS-2152991 until May 12, 2025, when this grant was terminated.

\section{Notation for Grassmannians and Schubert varieties}\label{s:notation}

Our conventions regarding Schubert varieties generalise those of Grassmannians used in \cite{RW} and \cite{MR-Adv}, which we begin by recalling. We use the shorthand notation $[n]:=\{1,\dots,n\}$, and let $\binom{[n]}{m}$ denote the set of all $m$-element subsets of $[n]$.

\subsection{Dual Grassmannians 
and their Pl\"ucker coordinates}

Let $\mathbbX = Gr_{n-k}(\C^n)$ be the Grassmannian of $(n-k)$-planes in $\C^n$ and let 
$\mathbb\checkX=Gr_{k}((\C^n)^*)$ be the Grassmannian of $k$-planes in the vector space $(\C^n)^*$ of row vectors.  We think of $\mathbbX$ as a homogeneous space for the group $GL_n(\C)$, acting on the left, and $\mathbb\checkX$ as a homogeneous space for the Langlands dual general linear group $GL_n^\vee(\C)$ acting on the right.

An element of $\mathbb X=Gr_{n-k}(\C^n)$
can be represented as the column-span of a full-rank $n\times (n-k)$ matrix $A$.
For any $J\in \binom{[n]}{n-k}$ let $P_J(A)$
denote the maximal minor of the $n\times (n-k)$ matrix~$A$ with row set $J$.
The map $A\mapsto (P_J(A))$, where $J$ ranges over $\binom{[n]}{n-k}$,
induces the {\it Pl\"ucker embedding\/} $\mathbbX\hookrightarrow \mathbb{P}^{\binom{n}{n-k}-1}$, and the $P_J$, interpreted as homogeneous coordinates on $\X$, are called the \emph{Pl\"ucker coordinates}.

For $\mathbb\checkX$ on the other hand we represent an element as row span of a $k\times n$ matrix $M$, and the Pl\"ucker coordinates  are naturally parameterized by $\binom{[n]}{k}$; for every $k$-subset $I$ in $[n]$ the Pl\"ucker coordinate $p_I$ is associated to the $k\times k$ minor of $M$ with column set given by~$I$.

\color{black}
\subsection{Young diagrams}\label{s:Young}
 It is convenient to index Pl\"ucker coordinates of both 
$\mathbbX$ and $\mathbb\checkX$ using Young diagrams. 
Let $\Shkn$ denote the set of Young diagrams fitting in an 
$(n-k)\times k $ rectangle. We identify a Young diagram with its corresponding partition, so that $\mu=(\mu_1\ge \dotsc\ge\mu_m)$ lies in $\Shkn$ if and only if $\mu_1\le k$ and $m\le n-k$.  There is a natural bijection between 
$\Shkn$ and 
${[n] \choose n-k}$, defined as follows.  Let $\mu$ be an element of 
$\Shkn$, justified so that its top-left corner coincides with 
the top-left corner of the 
$(n-k) \times k$ rectangle.  The south-east border of $\mu$ is then
cut out by a path 
 $\pathsw{\mu}$.
from the 
northeast to southwest corner of the rectangle, which consists of $k$
west steps and $(n-k)$ south steps.  After labeling the $n$ steps by 
the numbers $\{1,\dots,n\}$, we map $\mu$ to the labels of the 
south steps.  This gives a bijection 
from $\Shkn$ to
${[n] \choose n-k}$.  If we use the labels of the west steps instead,
we get a bijection
from $\Shkn$ to
${[n] \choose k}$.

\subsection{Schubert varieties and open Schubert varieties}\label{s:SchubandOpenSchub}

Let us consider a Young diagram $\lambda\in\Shkn$ with corresponding partition denoted $(\lambda_1\ge\dotsc\ge\lambda_m)$. 

\begin{definition}\label{d:SchubertCell}
	The \emph{Schubert cell} $\Omega_{\lambda}$  is defined to be the subvariety in $ \X = Gr_{n-k}(\C^n)$ given by
	$$\Omega_{\lambda} := \{A \in Gr_{n-k}(\C^n) \ \vert \ 
	P_{\lambda}(A) \neq 0 \text{ and }P_{\mu}(A)=0 \text{ unless }
	\mu \subseteq \lambda \}.$$
	The \emph{Schubert variety} $X_{\lambda}$ is defined to be the closure
	$\overline{\Omega}_\lambda$ of $\Omega_\lambda$.

	If $J$ is the $(n-k)$-element subset of $[n]$ 
	corresponding to the south
	steps of $\lambda$, as in \cref{s:Young}, then 
	we also denote the above Schubert cell and Schubert variety by 
	$\Omega_{J}$
	and $X_{J}$, respectively. Note that 
	$$\Omega_{J} = \{A \in Gr_{n-k}(\C^n) \ \vert \ 
	\text{the lexicographically minimal nonvanishing
Pl\"ucker coordinate of }A\text{ is }P_{J}(A) \}.$$
\end{definition}

We also have Schubert cells and varieties in the Langlands dual Grassmannian $\mathbb\checkX= Gr_{k}((\C^n)^*)$.
\begin{definition}
The \emph{dual Schubert cell} 
	$\check{\Omega}_{\lambda} \subset \mathbb\checkX $ is defined  as
	$$\check{\Omega}_{\lambda} = \{M \in Gr_{k}((\C^n)^*) \ \vert \ 
	p_{\lambda}(M) \neq 0 \text{ and }p_{\mu}(M)=0 \text{ unless }
	\mu \subseteq \lambda \}.$$
	The \emph{dual Schubert variety} 
	$\check{X}_{\lambda}$ is defined to be the closure
	of $\check{\Omega}_\lambda$.

	If $I$ is the $k$-element subset of $[n]$ 
	corresponding to the horizontal
	steps of $\lambda$, then 
	we also denote the above dual Schubert cell and variety by 
	$\check{\Omega}_{I}$
	and $\check{X}_{I}$, respectively. Note that 
	$$\checkOmega_{I} = \{M \in Gr_{k}((\C^n)^*) \ \vert \ 
	\text{the lexicographically maximal nonvanishing
Pl\"ucker coordinate of }M\text{ is }p_{I}(M) \}.$$
\end{definition}
	
The dimensions of $\Omega_{\lambda}$, $X_{\lambda}$,
$\checkOmega_{\lambda}$, and $\check{X}_{\lambda}$ are all $|\lambda|$, the number of boxes of 
$\lambda$.

We now fix a partition $\lambda=(\lambda_1\ge \dotsc \ge\lambda_m)$ as our starting point, and focus on the Schubert variety $X_\lambda$. We choose the ambient Grassmannian to be minimal for $\lambda$ and adopt the following conventions.

\begin{notation}\label{n:knd} 
Associated to our fixed partition $(\lambda_1\ge\lambda_2\ge\cdots\ge\lambda_{m})$ we set $k:=\lambda_1$ and $n:=\lambda_1+m$, so that the $(n-k)\times k$ rectangular Young diagram is the minimal rectangle containing the Young diagram $\lambda$; we call it the `bounding rectangle' of $\lambda$. We also let $d$ denote the number of removable boxes in $\lambda$. Clearly, $\lambda\in\Shkn$.
\end{notation}

\begin{defn}
We let $\mathcal{P}_{\lambda}\subseteq\mathcal P_{k,n}$ denote the set of all Young diagrams contained in $\lambda$.
Therefore the elements of $\mathcal{P}_{\lambda}$
index those Pl\"ucker coordinates $P_\nu$ of $\mathbb X$ whose
	restriction to the Schubert variety $X_{\lambda}$ is not constant equal to zero,
and simultaneously those Pl\"ucker coordinates  $p_\nu$ of $\mathbb\checkX$  whose restriction to  the dual Schubert variety $\check{X}_{\lambda}$
 is not constant equal to zero. The Schubert variety is the disjoint union of Schubert cells $X_\lambda=\bigsqcup_{\nu\in\mathcal {P}_\lambda}\Omega_\nu$,  and similarly for $\check{X}_\lambda$.
\end{defn}

\begin{definition}\label{def:BB} 
Let $\lambda$ be a Young diagram in a $(n-k) \times k$
bounding rectangle, so $\lambda\in\Shkn$.

We let 
$\BB^{\SE}(\lambda)$ be the set of boxes of $\lambda$
on the {\it southeast} border of $\lambda$; in other words, these
are the boxes which touch the southeast border of $\lambda$
either along a side or sides, or just at their  southeast corner. We also 
refer to this set of boxes as the \emph{rim} of $\lambda$.
The rim of $\lambda$  consists of $n-1$ boxes which we number
from northeast to southwest, writing 
$b_i$ for the $i$-th box in the rim of $\lambda$. 

We also let 
$\BB^{\NW}(\lambda)$ be the set of boxes of $\lambda$ on the {\it northwest}
border of $\lambda$; in other words, these are the boxes in the leftmost column or topmost row of $\lambda$.  There are $n-1$ boxes in $\BB^{\NW}(\lambda)$, 
which we label $b'_1,\dots,b'_{n-1}$,
starting from the bottom left and counting up and then to the right. 
\end{definition}
\begin{definition}[Frozen rectangles for $\lambda$]\label{d:frozens}
Consider our fixed partition $\lambda$ with its bounding $(n-k)\times k$ rectangle. 
Given any box $b$ contained in 
 $\lambda$ 
we write $\Rect(b)$ for the maximal rectangle contained in $\lambda$ 
whose lower right hand corner is the box $b$.  We also refer to $\Rect(b)$ 
as the \emph{shape} $\sh(b)$ associated to $b$.
We define 
	\[\mu_i:=\sh(b_i) \text{ for }1 \leq i \leq n-1, \text{ and }
\mu_n:=\emptyset
	.\]
Thus for $1\leq i \leq n-1$, $\mu_i$ is the 
maximal rectangle contained in $\lambda$ whose lower right corner is the box $b_i$, see \cref{Rim}.
	We let $\Fr(\lambda) := \{\mu_1,\dots, \mu_{n-1}, \mu_n\} \subseteq \mathcal{P_\lambda}$, and call 
	the elements of $\Fr(\lambda)$ the \emph{frozen rectangles} for $\lambda$. We treat the indices modulo $n$ so that $\mu_{n+1}=\mu_1$.
\end{definition}

\begin{figure}
	\begin{tikzpicture}[scale=0.5]
\draw (0,1.6) rectangle (1,2.6);
\draw (1,1.6) rectangle (2,2.6);
\draw (2,1.6) rectangle (3,2.6); 
\draw (3,1.6) rectangle (4,2.6); \node at (3.5,2.1) {\Large 1};
\draw (0,0.6) rectangle (1,1.6);
\draw (1,0.6) rectangle (2,1.6); \node at (1.5,1.1) {\Large 4};
\draw (2,0.6) rectangle (3,1.6); \node at (2.5,1.1) {\Large 3};
\draw (3,0.6) rectangle (4,1.6); \node at (3.5,1.1) {\Large 2};
\draw (0,-0.4) rectangle (1,0.6); \node at (0.5,0.1) {\Large 6};
\draw (1,-0.4) rectangle (2,0.6); \node at (1.5,0.1) {\Large 5};

\node at (6.5,2.5) {\Large $\mu_{1} = $};
\draw (7.5,2.3) rectangle (7.9,2.7);
\draw (7.9,2.3) rectangle (8.3,2.7);
\draw (8.3,2.3) rectangle (8.7,2.7);
\draw (8.7,2.3) rectangle (9.1,2.7);

\node at (11.5,2.5) {\Large $\mu_{2} = $};
\draw (12.6,2.7) rectangle (13,3.1);
\draw (13, 2.7) rectangle (13.4,3.1);
\draw (13.4,2.7) rectangle (13.8,3.1);
\draw (13.8,2.7) rectangle (14.2,3.1);
\draw (12.6,2.3) rectangle (13,2.7);
\draw (13, 2.3) rectangle (13.4,2.7);
\draw (13.4,2.3) rectangle (13.8,2.7);
\draw (13.8,2.3) rectangle (14.2,2.7);

\node at (16.5,2.5) {\Large $\mu_{7} = \varnothing$};

\node at (6.5,1.5) {\Large $\mu_{3} = $};
\draw (7.5,1.5) rectangle (7.9,1.9);
\draw (7.9,1.5) rectangle (8.3,1.9);
\draw (8.3,1.5) rectangle (8.7,1.9);
\draw (7.5,1.1) rectangle (7.9,1.5);
\draw (7.9,1.1) rectangle (8.3,1.5);
\draw (8.3,1.1) rectangle (8.7,1.5);

\node at (11.5,1.5) {\Large $\mu_{4} = $};
\draw (12.6,1.5) rectangle (13,1.9);
\draw (13,1.5) rectangle (13.4,1.9);
\draw (12.6,1.1) rectangle (13,1.5);
\draw (13,1.1) rectangle (13.4,1.5);

\node at (6.5,0) {\Large $\mu_{5} = $};
\draw (7.5,-0.5) rectangle (7.9,-0.1);
\draw (7.9,-0.5) rectangle (8.3,-0.1);
\draw (7.5,-0.1) rectangle (7.9,0.3);
\draw (7.9,-0.1) rectangle (8.3,0.3);
\draw (7.5,0.3) rectangle (7.9,0.7);
\draw (7.9,0.3) rectangle (8.3,0.7);

\node at (11.5,0) {\Large $\mu_{6} = $};
\draw (12.6,-0.5) rectangle (13,-0.1);
\draw (12.6,-0.1) rectangle (13,0.3);
\draw (12.6,0.3) rectangle (13,0.7);
\end{tikzpicture}
\caption{The rim of a Young diagram $\lambda$, together with the rectangles $\mu_1,\dots, \mu_7$.}  
\label{Rim}
\end{figure}
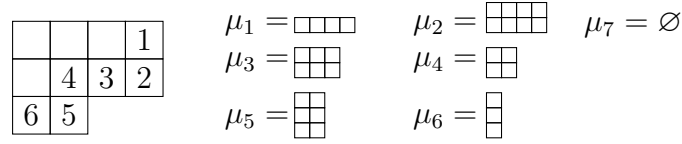

We use these special Young diagrams $\mu_i$ 
from \cref{d:frozens} to  define a distinguished 
 divisor 
in $X_{\lambda}$, 
 \begin{equation}\label{e:boundarydivisor}
 D^{\lambda}_{\ac}:=\bigcup_{i=1}^n \{P_{\mu_i}=0\}.
 \end{equation}
 \begin{rem} Unlike in the case of the full Grassmannian, in $X_\lambda$ the individual divisors $\{P_{\mu_i}=0\}$ will not necessarily be irreducible. We will describe $D^{\lambda}_{\ac}$ in terms of its irreducible components in \cref{c:ac} using the positroid stratification of \cite{KLS}. This description then implies that $D_{\ac}^{\lambda}$ is an anti-canonical divisor in $X_{\lambda}$, see~\cref{c:ac}.
\end{rem}

\begin{definition}[The open Schubert variety]\label{d:openX}
Let $\lambda, k, n$ be as in \cref{n:knd}.
We define $\openSchub$ the \textit{open Schubert variety} to be the complement of the 
divisor 
	$D_{\ac}^{\lambda}=\bigcup_{i=1}^n \{P_{\mu_i}=0\}$, 
\begin{equation*}
	\openSchub:= X_{\lambda} \setminus  D_{\ac}^{\lambda} =\{x\in X_{\lambda}  \ | \ P_{\mu_i}(x)\ne 0 \ \forall i\in[n] \}.
\end{equation*}
It is not hard to see that we have the inclusions $\openSchub\subset\Omega_\lambda\subset X_\lambda$.  
Similarly, we define $\opencheckSchub$ to be the complement of the analogous divisor, namely
	$\checkD_{\ac}^{\lambda}=\bigcup_{i=1}^n \{p_{\mu_i}=0\}$, in $\check{X}_{\lambda}$,
\begin{equation*}
	\opencheckSchub:= \check{X}_{\lambda} \setminus  \checkD^{\lambda}_{\ac} =\{x\in \check{X}_{\lambda}  \ | \ p_{\mu_i}(x)\ne 0 \ \forall i\in[n] \}.
\end{equation*}
  \end{definition}
\begin{rem}\label{r:openSchub}
The open Schubert variety $\openSchub$ is an \emph{open positroid variety} 
as defined in \cref{def:posvariety}, 
	and can be described as the projection of an \emph{open Richardson variety},
	see \cref{d:projRich}.
The subsets $I_{\mu_i}$ corresponding to the Pl\"ucker coordinates $P_{\mu_i}=P_{I_{\mu_i}}$ are the components of the 
\emph{reverse Grassmann necklace} of the positroid, which we can obtain from a corresponding \emph{plabic graph} for $\lambda$ if we use the \emph{source labeling} for faces. 
See \cref{sec:appendix} for background on these objects;
in particular, \cref{fig:plabic} shows a plabic graph for $\lambda=(4,4,2)$, whose 
reverse Grassmann necklace is $(567, 167, 127, 237, 347, 345, 456)$ and corresponds
to the rectangles $\mu_7, \mu_1,\dots, \mu_6$ from 
\cref{Rim}.
\end{rem}

\begin{remark}
Throughout this paper we will primarily be working with open Schubert 
varieties.  The reader should be cautioned that 
we will mostly drop the adjective ``open'' from now on but will consistently use the notation 
$X_{\lambda}^\circ$ 
or $\opencheckSchub$ 
	for clarity.
\end{remark}

\begin{remark}\label{r:R(lambda)} We may fix $n$ and consider all of the Schubert varieties $X_\lambda$ such that the minimal Grassmannian containing $\lambda$ is a $Gr_{\ell}(\C^n)$ for some dimension $\ell$. This is equivalent to $n$ being the length of the bounding path of the Young diagram $\lambda$. We observe that such Young diagrams $\lambda$ are in bijection with  subsets of $[n-1]$ of odd cardinality.  Namely
associate to $\lambda$ the set
\[
\mathcal R(\lambda):=\{i\in [n-1]\mid \text{the rim box $b_i$ in $\lambda$ is an outer or an inner corner}\}.
\] 
	Then $\mathcal R(\lambda)$ is the union, $\mathcal R(\lambda)=\Rout(\lambda)\sqcup \Rin(\lambda)$, of sets of indices labelling outer corner boxes $b_i$ and inner corner boxes $b_i$, respectively. Moreover 
	if 
	$\mathcal R(\lambda)=\{\rho_1<\dotsc < \rho_{2d-1}\}$ then $\Rout(\lambda)=\{\rho_{\ell}\mid \text {$\ell$ is odd}\}$ and $\Rin(\lambda)=\{\rho_{\ell}\mid \text {$\ell$ is even}\}$, since outer and inner corners alternate, see \cref{fig:Rho}.
 
Note that $\mathcal R(\lambda)$ is indeed an odd cardinality subset of  $[n-1]$, since the first and last elements must label outer corners. Conversely any subset of $[n-1]$ of odd cardinality, together with the fixed $n$, determines a Young diagram $\lambda$ with the given inner and outer corners.  
\end{remark}

\begin{figure}[h]
	\begin{tikzpicture}[scale=0.4]
\draw (0,5) rectangle (1,6);
\draw (1,5) rectangle (2,6);
\draw (2,5) rectangle (3,6);
\draw (3,5) rectangle (4,6);
\draw (4,5) rectangle (5,6);
\draw (5,5) rectangle (6,6); \node at (5.5,5.5) {1};
\draw (0,4) rectangle (1,5);
\draw (1,4) rectangle (2,5);
\draw (2,4) rectangle (3,5);
\draw (3,4) rectangle (4,5);
\draw (4,4) rectangle (5,5);
\draw (5,4) rectangle (6,5); \node at (5.5,4.5) {2};
\draw (0,3) rectangle (1,4);
\draw (1,3) rectangle (2,4);
\draw (2,3) rectangle (3,4);
\draw (3,3) rectangle (4,4); \filldraw[fill=yellow] (3,3) rectangle (4,4); \node at (3.5,3.5) {5};
\draw (4,3) rectangle (5,4); \node at (4.5,3.5) { 4};
\draw (5,3) rectangle (6,4); \filldraw[fill=yellow] (5,3) rectangle (6,4); \node at (5.5,3.5) { 3};
\draw (0,2) rectangle (1,3);
\draw (1,2) rectangle (2,3);
\draw (2,2) rectangle (3,3);
\draw (3,2) rectangle (4,3); \node at (3.5,2.5) {6};
\draw (0,1) rectangle (1,2);
\draw (1,1) rectangle (2,2); \filldraw[fill=yellow] (1,1) rectangle (2,2); \node at (1.5,1.5) {9};
\draw (2,1) rectangle (3,2); \node at (2.5,1.5) {8};
\draw (3,1) rectangle (4,2); \filldraw[fill=yellow] (3,1) rectangle (4,2); \node at (3.5,1.5) {7};
\draw (0,0) rectangle (1,1); \node at (0.5,0.5) {11};
\draw (1,0) rectangle (2,1); \filldraw[fill=yellow] (1,0) rectangle (2,1); \node at (1.5,0.5) {10};
\end{tikzpicture}
	\caption{When $\lambda=(6,6,6,4,4,2)$, we have
	$\mathcal{R}(\lambda) =\{\rho_1< \dots < \rho_5\} =  \{3,5,7,9,10\}$, $\Rout(\lambda) = \{3,7,10\}$,
	and $\Rin(\lambda) = \{5,9\}$.}
\label{fig:Rho}
\end{figure}
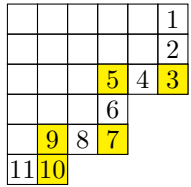

\section{Cluster structures for Schubert varieties}\label{sec:A}

In this section we explain how (open) Schubert varieties naturally 
have a cluster structure; in particular, we will concretely describe the 
\emph{rectangles seed} for each Schubert variety.  (A larger class of seeds associated to 
\emph{plabic graphs} is described in 
\cref{sec:cluster}.)
We will also explain the notion of \emph{restricted seed},
and observe that the rectangles seed for a Schubert variety
is a restricted seed obtained from the rectangles seed for the Grassmannian.

\subsection{The $\A$-cluster structure for a Schubert variety $\dualSchub$}
\label{sec:Acluster}
Fix a Young diagram $\lambda\in \Shkn$.
 Without loss of generality,
we assume again that $k$ and $n$ are minimal for $\lambda$, i.e. the first row of $\lambda$ has length $k$, and $\lambda$ has $(n-k)$ rows.
We let $\Rect(\lambda)$ denote the set of all rectangular Young diagrams in $\mathcal P_\lambda$, that is all rectangles (including $\emptyset$) which fit inside 
$\lambda$.  Recall the rectangle
  $\Rect(b)$ associated to a box $b$ of $\lambda$ in \cref{d:frozens} whose lower right-hand corner is $b$.

\begin{definition}[\emph{The rectangles seed $G_{\rect}^{\lambda}$}]\label{def:rectseed}
Fix $\lambda$ as above.
	We obtain a quiver $Q_{\lambda}$ as follows: place one vertex in each box of $\lambda$, plus one
more vertex labeled $\emptyset$. A vertex is mutable whenever it lies in a box $b$ of the Young diagram
and the box immediately southeast of $b$ is also in $\lambda$. 
We add one arrow from the vertex in the northwest corner of $\lambda$ to the vertex labeled $\emptyset$.
We also add arrows between vertices in adjacent boxes, with all arrows pointing either up or to the left. Finally, in every $2 \times 2$ rectangle in $\lambda$, we add an arrow from the upper left box to the lower right box. Equivalently, we add an arrow from the vertex in box $a$ to the vertex in box $b$ if  
\begin{itemize}
\item $\Rect(b)$ is obtained from $\Rect(a)$ by removing a row or column.
\item $\Rect(b)$ is obtained from $\Rect(a)$ by adding a hook shape.
\end{itemize}
(We then remove any arrow between two frozen vertices.)

For each box $b$ of the Young diagram, we label the 
corresponding vertex by $\Rect(b)$, which we identify with the 
corresponding Pl\"ucker coordinate for $Gr_k(\C^n)$.  We denote
	the resulting seed by $G_{\rect}^{\lambda}$.
\end{definition}  

\begin{remark}
Note that the frozen variables are labeled precisely
	by the rectangles from $\Fr(\lambda)$.
\end{remark}

The following result was shown in \cite{SSW}.

\begin{theorem}\cite{SSW}\label{thm:rectangles}
The seed $G^{\lambda}_{\rect}$ is
a seed for a cluster algebra which equals the coordinate ring of the 
(affine cone over the) open Schubert variety $\check \openSchub$, i.e. 
$\C[\widehat{\check \openSchub}]= 
\mathcal{A}(G^{\lambda}_{\rect}).$ 
\end{theorem}

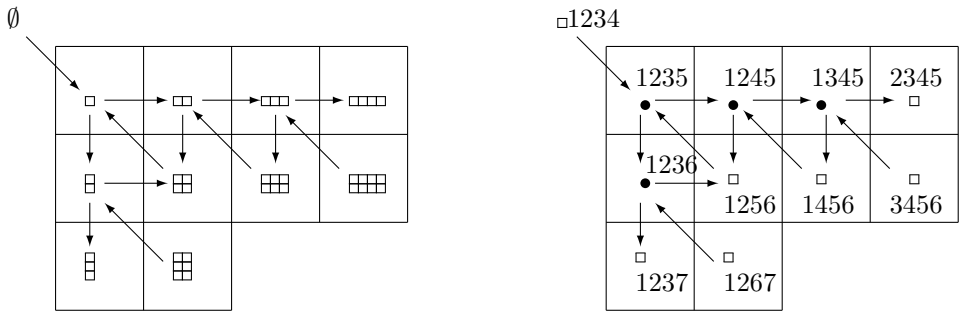
\begin{figure}[h]
\setlength{\unitlength}{1.3mm}
\begin{center}
 \begin{picture}(50,35)
% horizontal lines
  \put(5,32){\line(1,0){36}}
  \put(5,23){\line(1,0){36}}
  \put(5,14){\line(1,0){36}}
  \put(5,5){\line(1,0){18}}
% vertical lines
  \put(5,5){\line(0,1){27}}
  \put(14,5){\line(0,1){27}}
  \put(23,5){\line(0,1){27}}
  \put(32,14){\line(0,1){18}}
  \put(41,14){\line(0,1){18}}
% extra node
	 \put(0,34){$\emptyset$}
% first row
	 \put(8,26){$\ydiagram{1}$}
	 \put(17,26){$\ydiagram{2}$}
	 \put(26,26){$\ydiagram{3}$}
	 \put(35,26){$\ydiagram{4}$}
% second row
	 \put(8,18){$\ydiagram{1,1}$}
	 \put(17,18){$\ydiagram{2,2}$}
	 \put(26,18){$\ydiagram{3,3}$}
	 \put(35,18){$\ydiagram{4,4}$}
% third row
	 \put(8,10){$\ydiagram{1,1,1}$}
	 \put(17,10){$\ydiagram{2,2,2}$}
 % arrows
 \put(2,33){{\vector(1,-1){5.5}}}
         \put(10,26.5){{\vector(1,0){6.5}}}
         \put(20, 26.5){{\vector(1,0){6}}}
         \put(10,18){{\vector(1,0){6.5}}}
         \put(8.5,16){{\vector(0,-1){4.5}}}
         \put(8.5,25){{\vector(0,-1){5}}}
         \put(18,25){{\vector(0,-1){5}}}
         \put(16,19.5){{\vector(-1,1){6}}}
         \put(25,19.5){{\vector(-1,1){6}}}
         \put(16,10){{\vector(-1,1){6}}}
         %\put(19.5,18){{\vector(11,0){6}}}
         \put(29.5, 26.5){{\vector(1,0){5}}}
         %\put(29.5,18){{\vector(1,0){5}}}
         \put(27.5,25){{\vector(0,-1){5}}}
         \put(34,19.5){{\vector(-1,1){5.5}}}

 \end{picture}
    \qquad
        \begin{picture}(50,35)
% Young Diagram
% horizontal lines
  \put(5,32){\line(1,0){36}}
  \put(5,23){\line(1,0){36}}
  %\put(5,14){\line(1,0){27}}
  \put(5,14){\line(1,0){36}}
  \put(5,5){\line(1,0){18}}
% vertical lines
  \put(5,5){\line(0,1){27}}
  \put(14,5){\line(0,1){27}}
  \put(23,5){\line(0,1){27}}
  \put(32,14){\line(0,1){18}}
  %\put(41,23){\line(0,1){9}}
  \put(41,14){\line(0,1){18}}
% extra	
	    \put(0,34){$\ydiagram{1}$}
		\put(1,34){$1234$}
% first row
		\put(9,26){\circle*{1}}
		\put(8,28){$1235$}
		\put(18,26){\circle*{1}}
		\put(17,28){$1245$}
		\put(27,26){\circle*{1}}
		\put(26,28){$1345$}	
	 \put(36,26){$\ydiagram{1}$}
		\put(34,28){$2345$}
% second row
		\put(9,18){\circle*{1}}
		\put(9,19){$1236$}
	 \put(17.5,18){$\ydiagram{1}$}
		\put(17,15){$1256$}
	 \put(26.5,18){$\ydiagram{1}$}
		\put(25,15){$1456$}
	 \put(36,18){$\ydiagram{1}$}
		\put(34,15){$3456$}
% third row
	 \put(8,10){$\ydiagram{1}$}
		\put(8,7){$1237$}
	 \put(17,10){$\ydiagram{1}$}
		\put(17,7){$1267$}
 % arrows
 \put(2,33){{\vector(1,-1){5.5}}}
         \put(10,26.5){{\vector(1,0){6.5}}}
         \put(20, 26.5){{\vector(1,0){6}}}
         \put(10,18){{\vector(1,0){6.5}}}
         \put(8.5,16){{\vector(0,-1){4.5}}}
         \put(8.5,25){{\vector(0,-1){5}}}
         \put(18,25){{\vector(0,-1){5}}}
         \put(16,19.5){{\vector(-1,1){6}}}
         \put(25,19.5){{\vector(-1,1){6}}}
         \put(16,10){{\vector(-1,1){6}}}
         \put(29.5, 26.5){{\vector(1,0){5}}}
         \put(27.5,25){{\vector(0,-1){5}}}
         \put(34,19.5){{\vector(-1,1){5.5}}}

	\end{picture}
\end{center}
\caption{	\label{fig:combconstruct}  
	An example of $G^{\lambda}_{\rect}$ for
	$k=4$, $n=7$, and $\lambda = (4,4,2)$.
	At the left, the frozen rectangles are 
	$\ydiagram{4}$, $\ydiagram{4,4}$, $\ydiagram{3,3}$, $\ydiagram{2,2}$,
	$\ydiagram{2,2,2}$, $\ydiagram{1,1,1}$.  On the right, the same quiver is shown but rectangles have been replaced by the corresponding 
	$4$-element subsets of $[7]$, which should be interpreted as  Pl\"ucker coordinates.} 
\end{figure}

\subsection{The rectangles seed for a Schubert variety as a restricted
seed for the Grassmannian}

\begin{figure}[h]
\setlength{\unitlength}{1.3mm}
\begin{center}
 \begin{picture}(50,35)
% Young Diagram
% horizontal lines
  \put(5,32){\line(1,0){36}}
  \put(5,23){\line(1,0){36}}
  %\put(5,14){\line(1,0){27}}
  \put(5,14){\line(1,0){36}}
  \put(5,5){\line(1,0){18}}
% vertical lines
  \put(5,5){\line(0,1){27}}
  \put(14,5){\line(0,1){27}}
  \put(23,5){\line(0,1){27}}
  \put(32,14){\line(0,1){18}}
  %\put(41,23){\line(0,1){9}}
  \put(41,14){\line(0,1){18}}
% extra node
	 \put(0,34){$\color{blue}\emptyset$}
% first row
	 \put(8,26){$\ydiagram{1}$}
	 \put(17,26){$\ydiagram{2}$}
	 \put(26,26){$\ydiagram{3}$}
	 \put(35,26){$\color{blue}\ydiagram{4}$}
% second row
	 \put(8,18){$\ydiagram{1,1}$}
	 \put(17,18){$\color{blue}\ydiagram{2,2}$}
	 \put(26,18){$\color{blue}\ydiagram{3,3}$}
	 \put(35,18){$\color{blue}\ydiagram{4,4}$}
% third row
	 \put(8,10){$\color{blue}\ydiagram{1,1,1}$}
	 \put(17,10){$\color{blue}\ydiagram{2,2,2}$}
 % arrows
 \put(2,33){{\vector(1,-1){5.5}}}
         \put(10,26.5){{\vector(1,0){6.5}}}
         \put(20, 26.5){{\vector(1,0){6}}}
         \put(10,18){{\vector(1,0){6.5}}}
         \put(8.5,16){{\vector(0,-1){4.5}}}
         \put(8.5,25){{\vector(0,-1){5}}}
         \put(18,25){{\vector(0,-1){5}}}
         \put(16,19.5){{\vector(-1,1){6}}}
         \put(25,19.5){{\vector(-1,1){6}}}
         \put(16,10){{\vector(-1,1){6}}}
         \put(29.5, 26.5){{\vector(1,0){5}}}
         \put(27.5,25){{\vector(0,-1){5}}}
         \put(34,19.5){{\vector(-1,1){5.5}}}

 \end{picture}
 \begin{picture}(50,35)
% Young Diagram
% horizontal lines
  \put(5,32){\line(1,0){36}}
  \put(5,23){\line(1,0){36}}
  \put(5,14){\line(1,0){36}}
  \put(5,5){\line(1,0){36}}
% vertical lines
  \put(5,5){\line(0,1){27}}
  \put(14,5){\line(0,1){27}}
  \put(23,5){\line(0,1){27}}
  \put(32,32){\line(0,-1){27}}
  \put(41,32){\line(0,-1){27}}
% extra node
	    \put(0,34){$\color{blue}\emptyset$}
% first row
	 \put(8,26){$\ydiagram{1}$}
	 \put(17,26){$\ydiagram{2}$}
	 \put(26,26){$\color{black}\ydiagram{3}$}
	 \put(35,26){$\color{blue}\ydiagram{4}$}
% second row
	 \put(8,18){$\ydiagram{1,1}$}
	 \put(17,18){$\color{blue}\ydiagram{2,2}$}
	 \put(26,18){$\color{blue}\ydiagram{3,3}$}
	 \put(35,18){$\color{blue}\ydiagram{4,4}$}
% third row
	 \put(8,10){$\color{blue}\ydiagram{1,1,1}$}
	 \put(17,10){$\color{blue}\ydiagram{2,2,2}$}
	 \put(26,10){$\color{red}\ydiagram{3,3,3}$}
	 \put(35,10){$\color{red}\ydiagram{4,4,4}$}
 % arrows
 \put(2,33){{\vector(1,-1){5.5}}}
         \put(10,26.5){{\vector(1,0){6.5}}}
         \put(20, 26.5){{\vector(1,0){6}}}
         \put(10,18){{\vector(1,0){6.5}}}
         \put(8.5,16){{\vector(0,-1){4.5}}}
         \put(8.5,25){{\vector(0,-1){5}}}
         \put(18,25){{\vector(0,-1){5}}}
         \put(18,16.5){{\vector(0,-1){5}}}
         \put(27.5,16.5){{\vector(0,-1){5}}}
         \put(16,19.5){{\vector(-1,1){6}}}
         \put(25,19.5){{\vector(-1,1){6}}}
         \put(16,10){{\vector(-1,1){6}}}
         \put(25,10){{\vector(-1,1){6}}}
         \put(34,10){{\vector(-1,1){6}}}
         \put(19.5,18){{\vector(11,0){6}}}
         \put(29.5, 26.5){{\vector(1,0){5}}}
         \put(29.5,18){{\vector(1,0){5}}}
         \put(27.5,25){{\vector(0,-1){5}}}
         \put(34,19.5){{\vector(-1,1){5.5}}}

 \end{picture}
 \end{center}
 \caption{The rectangles seed for a Schubert variety is a restricted seed obtained from 
 the rectangles seed for the Grassmannian.}
 \label{fig:restrictedseed}
  \end{figure}
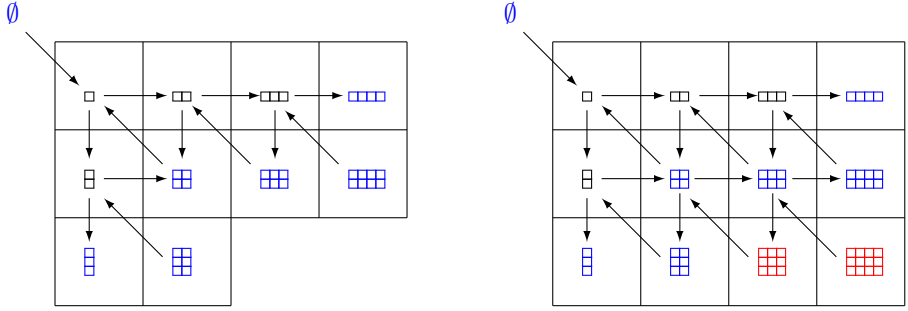  

\begin{definition}[Restricted seeds]\cite[Definition 4.2.6]{FWZ2}\label{def:restricted}
Let $G$ be a seed whose quiver has its
vertices labeled by 
$[m]=\{1,\dots,m\}$.
Choose a subset $I \subset [m]$; some elements of $I$ may be frozen,
in which case they will remain frozen, but we now
freeze some (possibly empty) subset of the mutable
vertices in $I$, so as to ensure that there are no arrows between
unfrozen vertices in $I$ and vertices in $[m]\setminus I$.
We define the \emph{restricted seed}
$G|_I$ to be the seed obtained from $G$ by restricting
to the induced quiver on $I$ (and removing any arrow
between two frozen vertices).
\end{definition}

The following lemma shows that the above operations are well-behaved.
\begin{lemma}\cite[Lemma 4.2.2 and Lemma 4.2.5]{FWZ2} \label{l:freezingcommutesmutation}
Freezing commutes with mutation.  
Passing to a restricted seed commutes with seed mutation.
\end{lemma}

\begin{lemma}\label{lem:restrictedseed}
Let $\nu\subseteq \lambda$.
Then the rectangles seed $G^\nu_\rect$ of $X_\nu$ is a restricted seed obtained 
from the rectangles seed $G^{\lambda}_{\rect}$ of $X_\lambda$.
It follows that any seed $G$ for $X_\nu$ is a restricted seed  obtained from a seed
$G'$ for $X_\lambda$; to get to the seed $G'$ for $X_{\lambda}$, we perform
on $G^\lambda_\rect$ the same sequence of mutations that were used on $G^\nu_\rect$ to obtain the seed $G$ for $X_{\nu}$.  
\end{lemma}
\begin{proof}
The first statement is clear from the definition of the rectangles seed; see \cref{fig:restrictedseed}.
The second statement follows from \cref{l:freezingcommutesmutation}.
\end{proof}

\subsection{The $\XX$-cluster structure on a Schubert variety $\openSchub$}\label{sec:Xcluster}

In this section we give a concrete description of the
$\XX$-cluster structure on a Schubert variety.
(More details on the network charts associated to plabic graphs can be found in
\cref{sec:poschart}.)
To do so, 
we associate to each Schubert variety a corresponding directed network 
$G_{\lambda}^{\rect}$ as in 
\cref{SchubertNetwork}.
This will give rise to a map of a torus into $\openSchub$, as in 
\cref{1network_param}.
We will then obtain the other $\XX$-charts
from this one by mutation.

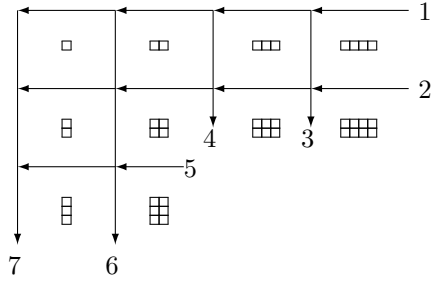
\begin{figure}[h]
\setlength{\unitlength}{1.3mm}
\begin{center}
        \begin{picture}(50,35)
 \put(43,29){$1$}
 \put(43,21){$2$}
 \put(31,16){$3$}
 \put(21,16){$4$}
 \put(19,13){$5$}
 \put(11,3){$6$}
 \put(1,3){$7$}
% first row
         \put(6.5,26){$\ydiagram{1}$}
         \put(15.5,26){$\ydiagram{2}$}
         \put(26,26){$\ydiagram{3}$}
         \put(35,26){$\ydiagram{4}$}
% second row
         \put(6.5,18){$\ydiagram{1,1}$}
         \put(15.5,18){$\ydiagram{2,2}$}
         \put(26,18){$\ydiagram{3,3}$}
         \put(35,18){$\ydiagram{4,4}$}
% third row
         \put(6.5,10){$\ydiagram{1,1,1}$}
         \put(15.5,10){$\ydiagram{2,2,2}$}
 % arrows
         \put(42, 30){{\vector(-1,0){10}}}
         \put(32, 30){{\vector(-1,0){10}}}
         \put(22, 30){{\vector(-1,0){10}}}
         \put(12, 30){{\vector(-1,0){10}}}
         \put(42, 22){{\vector(-1,0){10}}}
         \put(32, 22){{\vector(-1,0){10}}}
         \put(22, 22){{\vector(-1,0){10}}}
         \put(12, 22){{\vector(-1,0){10}}}
         \put(19, 14){{\vector(-1,0){7}}}
         \put(12, 14){{\vector(-1,0){10}}}

         \put(2,30){{\vector(0,-1){24}}}
         \put(12,30){{\vector(0,-1){24}}}
         \put(22,30){{\vector(0,-1){12}}}
         \put(32,30){{\vector(0,-1){12}}}
 \end{picture}
\end{center}
\caption{The ``rectangles'' network 
$G^{\rect}_{\lambda}$
	for the
Schubert variety $X_{\lambda}^{\circ}$ with $\lambda=(4,4,2)$.}
\label{SchubertNetwork}
\end{figure}

\begin{definition}
Let $I$ denote the boundary vertices which are sources; in 
\cref{SchubertNetwork}, $I = \{1,2,5\}$.
A \emph{flow} $F$ from $I$ to a set $J$ of boundary vertices
with
$|J|=|I|$
is a collection of paths
in the network, all pairwise vertex-disjoint,
such that the sources of these paths are $I - (I \cap J)$
and the destinations are $J - (I \cap J)$.

Note that each path %
$w$ in the network partitions the faces of the network into those
which are on the left and those which are on the right of the
walk. %
We define the \emph{weight} $\wt(w)$ of
each such path %
to be the product of
parameters $x_{\mu}$, where $\mu$ ranges over all face labels
to the left of the path.  And we define the
\emph{weight} $\wt(F)$ of a flow $F$ to be the product of the weights of
all paths in the flow.

Given $J \in {[n] \choose n-k}$,
we define the {\it flow polynomial}
\begin{equation}\label{eq:Plucker}
P_J^\rect = \sum_F \wt(F),
\end{equation}
where $F$ ranges over all flows from $I_{\O}$ to $J$.
\end{definition}

\begin{example}
Consider the network 
from \cref{SchubertNetwork}.
     There are two flows $F$
      from $I=\{1,2,5\}$ to $J=\{1,5,7\}$ (corresponding to the two
     paths from vertex $2$ to vertex $7$) and
$P^\rect_{\{1,5,7\}} = 
        x_{\ydiagram{2,2}} x_{\ydiagram{3,3}} x_{\ydiagram{4,4}} x_{\ydiagram{1,1,1}}
        x_{\ydiagram{2,2,2}} (1+x_{\ydiagram{1,1}})$.

Using the terminology of \cref{sec:appendix},
this network is actually the network $N(D)$
associated to the \emph{$\Le$-diagram}
in the middle of \cref{fig:Le}.
If we label the faces of the plabic graph $G(D)$
by source labels, then map the source labels to partitions,
then each face is labeled by a rectangular partition
as  in \cref{SchubertNetwork}.
\end{example}

We now describe the network chart for $\openSchub$ 
associated to the network 
$G^{\rect}_{\lambda}$.
Initially \cref{1network_param} was proved for the 
totally nonnegative part of $\openSchub$
 (see \cite[Section 6]{Postnikov} and \cite{Talaska}),
 while the extension to $\openSchub$ comes from
\cite{TalaskaWilliams} (see also \cite{MullerSpeyer}).

\begin{theorem}\label{1network_param}
Consider the map $\Phi^\rect_\lambda$ sending
$(x_{\mu})_{\mu \in \Rect(\lambda)}
\in (\C^*)^{|\lambda|}$ to  projective space
of dimension ${n \choose n-k}-1$ with nonvanishing Pl\"ucker coordinates 
given by the flow polynomials $P_J^\rect$.
Then this map is well-defined, and is 
	an injective map onto a dense open subset of $\openSchub$.    
We call the map $\Phi_\lambda^{\rect}$ 
a \emph{network chart} for $\openSchub$.
\end{theorem}

\begin{definition}[Network torus $\mathbb T_\lambda^{\rect}$]  \label{d:networktorus}
	Define the open dense torus $\mathbb T_\lambda^{\rect} $ in $\openSchub$ to be the image  of the network chart $\Phi_\lambda^{\rect}$, namely
	$\mathbb T_\lambda^{\rect}
	:=\Phi_\lambda^{\rect}((\C^*)^{|\lambda|})$.
	We call $\mathbb T_\lambda^{\rect}$ 
	the {\it network torus} associated to the rectangles cluster 
	for $\lambda$.
\end{definition}

While this paper will mostly be concerned with 
the network chart coming from $G^{\rect}_{\lambda}$, 
one can get many other $\mathcal{X}$-cluster charts 
coming from 
cluster $\mathcal{X}$-mutation, see \cref{Xseed} for 
more details.

\section{The definition of the superpotential for Schubert varieties}
\label{s:superpotential}

In this section we define our 
conjectural 
``mirror Landau-Ginzburg model" 
$(\check{X_{\lambda}^{\circ}}, 
W^{\lambda}_{\mathbf q})$  for the Schubert variety $X_\lambda$, where  
$W^{\lambda}_{\mathbf q}:
\dualSchub^{\circ} 
\to\C$ is a regular function 
we call the {\it superpotential} of $X_\lambda$.  
This superpotential generalizes the 
Marsh-Rietsch superpotential for Grassmannians from \cite{MR-Adv}. 
In \cref{d:LG1}, we give the \emph{canonical formula} for the 
superpotential, defining it using cluster variables such that 
only frozen variables appear in the denominator; thus, the 
superpotential is manifestly
 a regular function on $\dualSchub^{\circ}$.  In particular, 
if $\mathcal{A}(\Sigma_{\lambda}^{\rect})$ 
denotes the 
cluster algebra associated to the open Schubert variety
$\dualSchub^{\circ}$ (see \cref{thm:rectangles}), then 
\cref{d:LG1}
 expresses the superpotential as an element
of $\mathcal{A}(\Sigma_{\lambda}^{\rect})[q_1,\dots,q_d]$.
In \cref{e:Wq2} and 
 \cref{eq:3rdpotential} we give two equivalent ways to 
 express $W^{\lambda}$, using different combinatorial ways to index
 the summands (boxes in the rim of $\lambda$ versus boxes
 in the northwest border of $\lambda$).
Finally in \cref{prop:super2} we express the superpotential
as a Laurent polynomial in the cluster variables of the rectangles 
cluster.

Let $\lambda$ be a Young diagram 
corresponding to a partition $(\lambda_1\ge\lambda_2\ge\cdots\ge\lambda_{m})$. As in \cref{n:knd}, $\lambda$ has an $(n-k)\times k$ bounding rectangle and $d$ denotes the number of outer corners, or 
removable boxes, of $\lambda$.
So for example, 
if $\lambda = (4,4,2)$, then $k=4, n=7$, and $d=2$, while 
if $\lambda = (4,3,2)$, then $k=4$, $n=7$, and $d=3$.
In the notation from \cref{r:R(lambda)} the removable boxes are those boxes $b_i$ from the rim, see~\cref{d:frozens}, for which 
$i\in\Rout(\lambda)=\{\rho_1,\rho_3,\dotsc,\rho_{2d-1}\}$. 
We label the removable boxes of $\lambda$  
by the `quantum parameters' $q_1,\dotsc, q_d$, counting from top to bottom.
Thus the box $b_{\rho_{2\ell-1}}$ is labelled by the parameter $q_\ell$. 

Recall the set of frozen rectangles $\Fr(\lambda)$ from \cref{d:frozens}.
Namely $\Fr(\lambda)$ consists of $\mu_n=\emptyset$  together with the rectangles $\mu_1,\dotsc, \mu_{n-1}$ such that the southeast corner box of $\mu_i$ is the $i$-th box $b_i$ of the rim of $\lambda$. 

\begin{defn}\label{d:addable}
Consider a pair of Young diagrams $\mu$ and $\lambda$. 
We say a box $b$ is an \emph{addable box} for $\mu$ if $b$ does not lie in $\mu$, and the union $\mu\cup b$ is a Young diagram. We use the notation $\mu\sqcup b$ for the union of $\mu$ and $b$ when $b$ is such an addable box. 
We say the box $b$ is an \emph{addable box} for $\mu$ in $\lambda$ if it is an addable box for $\mu$ and additionally lies in $\lambda$.

Given a rectangular Young diagram $\mu \in \Rect(\lambda)$, we also let 
$\mu^-$ denote the rectangle obtained from $\mu$ by removing the rim.
\end{defn}

Our first version of the definition of the superpotential is as follows.

\begin{defn}[Canonical formula for the superpotential]\label{d:LG1}
Let $\lambda$ be a Young diagram with set $\Rout(\lambda)=\{\rho_1<\rho_3<\dotsc<\rho_{2d-1}\} \subset [n]$ of outer corner labels, compare \cref{r:R(lambda)}. 
We define the \emph{superpotential} of 
	$\check{X}_\lambda$ 
	to be the  regular function on 
$\check{X_{\lambda}^{\circ}}$, 
	depending on parameters $q_1,\dotsc, q_d$, which is given by 
\begin{equation}
	\label{e:Wq1}
	W^{\lambda} 
	= 
	\sum_{i=\rho_{2\ell-1} \in \Rout(\lambda)}	q_\ell \frac{p_{\mu_i^-}}{p_{\mu_i}}+
	\sum_{i\in [n]\setminus \Rout(\lambda)}
	\left(\sum_{b\in\mu_{i-1}\cup\mu_{i+1}} \frac{p_{\mu_i\sqcup b}}{p_{\mu_i}}\right).
\end{equation}
Here the first sum is over all $i\in\Rout(\lambda)$, so that $\ell$ ranges from $1$ to $d$.
	The sum inside the brackets on the right hand side is over all boxes $b$  which  lie in the union $\mu_{i-1}\cup\mu_{i+1}$ and are addable to $\mu_i$. 
	Here we think of $p_\mu$ as having degree $|\mu|$ and  
	$q_\ell$ as having degree $|\mu_{i}| - |\mu_{i}^-|+1$ so that the formula for $W^{\lambda}$ is homogeneous of degree $1$.

\end{defn}

\begin{remark}
The only elements which appear in the denominators in $W^{\lambda}$ are 
	the $p_{\mu_i}$ defining the divisor $\checkD_{\ac}^{\lambda}$, and this is precisely the
divisor	which we removed when defining 
$\check{X_{\lambda}^{\circ}}$
	(see \cref{d:openX}). 
Therefore $W^{\lambda}$ is indeed a regular function on 
$\check{X_{\lambda}^{\circ}}$.
Moreover, as we will show in \cref{prop:univpos},
	$W^{\lambda}$ is a \emph{universally positive} element of 
the $\mathcal{A}$-cluster structure on the open Schubert variety 
described in \cite{SSW}, see \cref{thm:rectangles}.
	This means that when we restrict $W^{\lambda}$
	to any $\mathcal{A}$-cluster torus, we will obtain
	a Laurent polynomial with positive coefficients.
\end{remark}

We now rewrite the superpotential in a slightly different way. 
Recall the notation $\sh(b)$ from \cref{d:frozens}.
The advantage of this next formula is that it generalizes
in a straightforward manner to skew shaped positroid varieties, see 
\cref{sec:skew}.

\begin{definition}\label{d:2recs}
Given a partition $\lambda$ in a $(n-k) \times k$ bounding rectangle,
we number its rows from $1$ to $n-k$ from top to bottom,
and its columns from $1$ to $k$ from left to right, as in the indexing
	of a matrix, see \cref{fig:SchubExample}.
For $1\leq i \leq n-k-1$, find the maximal-width rectangle of height $2$
that is contained in rows $i, i+1$ of $\lambda$.  
Let $d_i$ and $c_i$ denote its northeast and southwest corner, respectively.
Similarly, for $1\leq j \leq k-1$, find the maximal-height rectangle
of width $2$ contained in columns $j, j+1$ of $\lambda$.
Let $d^j$ and $c^j$ denote its southwest and northeast corner, respectively.
\end{definition}

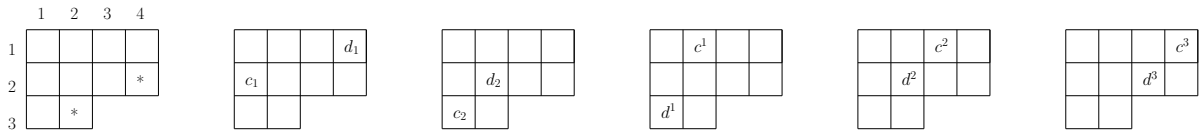
\begin{figure}[h]
\setlength{\unitlength}{1.3mm}
\begin{center}
\resizebox{0.15\textwidth}{!}{\begin{picture}(50,35)
% Young Diagram
% horizontal lines
  \put(5,32){\line(1,0){36}}
  \put(5,23){\line(1,0){36}}
  \put(5,14){\line(1,0){36}}
  \put(5,5){\line(1,0){18}}
% vertical lines
  \put(5,5){\line(0,1){27}}
  \put(14,5){\line(0,1){27}}
  \put(23,5){\line(0,1){27}}
  \put(32,14){\line(0,1){18}}
  \put(41,14){\line(0,1){18}}
% first row
	\put(8,35){\huge{$1$}}
	\put(17,35){\huge{$2$}}
	\put(26,35){\huge{$3$}}
	\put(35,35){\huge{$4$}}
	\put(0,25){\huge{$1$}}
	\put(0,15){\huge{$2$}}
	\put(0,5){\huge{$3$}}
% arrows
	 \put(35,16){\huge{*}}
	 \put(17,7){\huge{*}}
	\end{picture}}
	\hspace{0.1cm}
\resizebox{0.15\textwidth}{!}{\begin{picture}(50,35)
% Young Diagram
% horizontal lines
  \put(5,32){\line(1,0){36}}
  \put(5,23){\line(1,0){36}}
  \put(5,14){\line(1,0){36}}
  \put(5,5){\line(1,0){18}}
% vertical lines
  \put(5,5){\line(0,1){27}}
  \put(14,5){\line(0,1){27}}
  \put(23,5){\line(0,1){27}}
  \put(32,14){\line(0,1){18}}
  \put(41,23){\line(0,1){9}}
  \put(41,14){\line(0,1){18}}
% first row
	 \put(35,26){\huge{$d_1$}}
% second row
	 \put(8,17){\huge{$c_1$}}
% third row
% arrows
	\end{picture}}
	\hspace{0.1cm}
\resizebox{0.15\textwidth}{!}{\begin{picture}(50,35)
% Young Diagram
% horizontal lines
  \put(5,32){\line(1,0){36}}
  \put(5,23){\line(1,0){36}}
  \put(5,14){\line(1,0){36}}
  \put(5,5){\line(1,0){18}}
% vertical lines
  \put(5,5){\line(0,1){27}}
  \put(14,5){\line(0,1){27}}
  \put(23,5){\line(0,1){27}}
  \put(32,14){\line(0,1){18}}
  \put(41,23){\line(0,1){9}}
  \put(41,14){\line(0,1){18}}
% first row
% second row
	 \put(17,17){\huge{$d_2$}}
% third row
	 \put(8,8){\huge{$c_2$}}
% arrows
	\end{picture}}
	\hspace{0.1cm}
\resizebox{0.15\textwidth}{!}{\begin{picture}(50,35)
% Young Diagram
% horizontal lines
  \put(5,32){\line(1,0){36}}
  \put(5,23){\line(1,0){36}}
  \put(5,14){\line(1,0){36}}
  \put(5,5){\line(1,0){18}}
% vertical lines
  \put(5,5){\line(0,1){27}}
  \put(14,5){\line(0,1){27}}
  \put(23,5){\line(0,1){27}}
  \put(32,14){\line(0,1){18}}
  \put(41,23){\line(0,1){9}}
  \put(41,14){\line(0,1){18}}
% first row
	 \put(17,26){\huge{$c^1$}}
% second row
% third row
	 \put(8,8){\huge{$d^1$}}
% arrows
	\end{picture}}
	\hspace{0.1cm}
\resizebox{0.15\textwidth}{!}{\begin{picture}(50,35)
% Young Diagram
% horizontal lines
  \put(5,32){\line(1,0){36}}
  \put(5,23){\line(1,0){36}}
  \put(5,14){\line(1,0){36}}
  \put(5,5){\line(1,0){18}}
% vertical lines
  \put(5,5){\line(0,1){27}}
  \put(14,5){\line(0,1){27}}
  \put(23,5){\line(0,1){27}}
  \put(32,14){\line(0,1){18}}
  \put(41,23){\line(0,1){9}}
  \put(41,14){\line(0,1){18}}
% first row
	 \put(26,26){\huge{$c^2$}}
% second row
	 \put(17,17){\huge{$d^2$}}
% third row
% arrows
	\end{picture}}
	\hspace{0.1cm}
\resizebox{0.15\textwidth}{!}{\begin{picture}(50,35)
% Young Diagram
% horizontal lines
  \put(5,32){\line(1,0){36}}
  \put(5,23){\line(1,0){36}}
  \put(5,14){\line(1,0){36}}
  \put(5,5){\line(1,0){18}}
% vertical lines
  \put(5,5){\line(0,1){27}}
  \put(14,5){\line(0,1){27}}
  \put(23,5){\line(0,1){27}}
  \put(32,14){\line(0,1){18}}
  \put(41,23){\line(0,1){9}}
  \put(41,14){\line(0,1){18}}
% first row
	 \put(26,17){\huge{$d^3$}}
	 \put(35,26){\huge{$c^3$}}
% second row
% third row
% arrows
	\end{picture}}
\end{center}
     \caption{We can compute the superpotential using width $2$ and height $2$ rectangles. Here $\lambda=(4,4,2)$ as in \cref{ex:2recs}. The leftmost diagram indicates the 
      numbering of rows and columns, and each 
      $\star$ denotes an outer corner.
      The other diagrams indicate 
      the corners $c_i, d_i$,
      and $c^i, d^i$ of the rectangles from \cref{d:2recs}.}
\label{fig:SchubExample}
\end{figure}

\begin{proposition}\label{p:superpotential}
We have that 
\begin{equation}\label{e:another}
W^{\lambda} 
	= 
	\sum_{b=b_{\rho_{2\ell-1}} \in \Rout(\lambda)}	
	q_\ell \frac{p_{\sh(b)^-}}{p_{\sh(b)}}+
	\frac{p_{\ydiagram{1}}}{p_{\emptyset}}+
	\sum_{i=1}^{n-k-1} \frac{p_{\sh(d_i) \cup \sh(c_i)}}{p_{\sh(d_i)}}
	+ \sum_{j=1}^{k-1} \frac{p_{\sh(d^j) \cup \sh(c^j)}}{p_{\sh(d^j)}}.
\end{equation}
\end{proposition}
\begin{proof}
To see that \eqref{e:another} agrees with \eqref{e:Wq1},
first note that the first sums of both are identical.
Meanwhile 
	the term 
	from the second sum 
	of \eqref{e:Wq1} 
	in the case that $i=n$ (and hence
	$\mu_n = \emptyset$)
	is exactly 
	the term 
	$\frac{p_{\ydiagram{1}}}{p_{\emptyset}}$ from \eqref{e:another}.
Finally, the remaining terms 
	from the second sum 
	of \eqref{e:Wq1}  correspond to the sums over $i$ and $j$ in \eqref{e:another},
	where we note that each $\sh(d_i)$ and $\sh(d^j)$ correspond to 
	a $\mu_i$ where 
	${i\in [n]\setminus \Rout(\lambda)}$. 
\end{proof}

\begin{example}\label{ex:2recs}
	For the case $\lambda = (4,4,2)$ (shown in \cref{fig:SchubExample}),
	\cref{p:superpotential} tells us that 
	$$W^{\lambda}=
        W^{\lambda}_{\mathbf q}=
 \frac{q_1 p_{\ydiagram{3}}}{p_{\ydiagram{4,4}}} +
        \frac{q_2 p_{\ydiagram{1,1}}}{p_{\ydiagram{2,2,2}}} + 
        \frac{p_{\ydiagram{1}}}{p_{\emptyset}}+ 
        \frac{p_{\ydiagram{4,1}}}{\p_{\ydiagram{4}}} +
        \frac{p_{\ydiagram{2,2,1}}}{\p_{\ydiagram{2,2}}}+
        \frac{ p_{\ydiagram{2,1,1}}}{p_{\ydiagram{1,1,1}}} + 
        \frac{p_{\ydiagram{3,2}}}{\p_{\ydiagram{2,2}}}+
        \frac{p_{\ydiagram{4,3}}}{\p_{\ydiagram{3,3}}}.
	$$

\end{example}
For another example, see \cref{sec:example}.

We next  analyse the different types of summands that occur in $W^\lambda$. 

\begin{defn}\label{d:frozendec} Note that the rim of $\lambda$ consists of outer corner boxes, indexed by $\Rout(\lambda)$, inner corner boxes, indexed by $\Rin(\lambda)$, and two other kinds of boxes which we may think of as belonging to vertical, respectively horizontal, segments of the rim. We define four disjoint subsets of $\Fr(\lambda)$ using this division  of the rim: 

\begin{itemize}
	\item $\Fr_0(\lambda)$ consists of the rectangles $\mu_i \in \Fr(\lambda)$
such that $b_i$ is a removable box of $\lambda$. Equivalently $\Fr_0(\lambda)=\{\mu_\rho\mid\rho\in\Rout(\lambda)\}$.
	\item $\Fr_{1,E}(\lambda)$ consists of the rectangles $\mu_i \in \Fr(\lambda)$ such that $b_i$ has a box to the east of it in the rim of $\lambda$ but not to the south. That is, $b_i$ belongs to a {\it horizontal segment} of the rim. 	
	\item $\Fr_{1,S}(\lambda)$ consists of the rectangles $\mu_i \in \Fr(\lambda)$ such that $b_i$ has a box to the south of it in the rim of $\lambda$, but no box to the east of it. That is $b_i$ belongs to a {\it vertical segment} of the rim.
	\item $\Fr_2(\lambda)$ consists of the rectangles $\mu_i \in \Fr(\lambda)$
		such that $b_i$ is an inner corner box of the rim, so that $b_i$ has both a box to the east of it and a box to the south in the rim of $\lambda$. In other words $\Fr_2(\lambda)=\{\mu_\rho\mid\rho\in\Rin(\lambda)\}$.
\end{itemize}
Clearly we have $\Fr(\lambda) = \{\emptyset\}\cup
\Fr_0(\lambda) \cup \Fr_{1,E}(\lambda) \cup \Fr_{1,S}(\lambda) \cup \Fr_2(\lambda) $. The numerical index ($0$, $1$ or $2$) of each subset of frozen rectangles $\mu$ indicates the number of addable boxes to $\mu=\mu_i$ in $\mu_{i-1}\cup\mu_{i+1}$ for $\mu_i$ in this subset. 
Note that these addable boxes to $\mu_i$ are in bijection with 
the boxes
in the \textit{rim} of $\lambda$ that touch $\mu_i$ and lie directly to the south or east of $b_i$.
See \cref{fig:4cases}.  
\end{defn}

\begin{figure}[h]
\centering
\includegraphics[height=.8in]{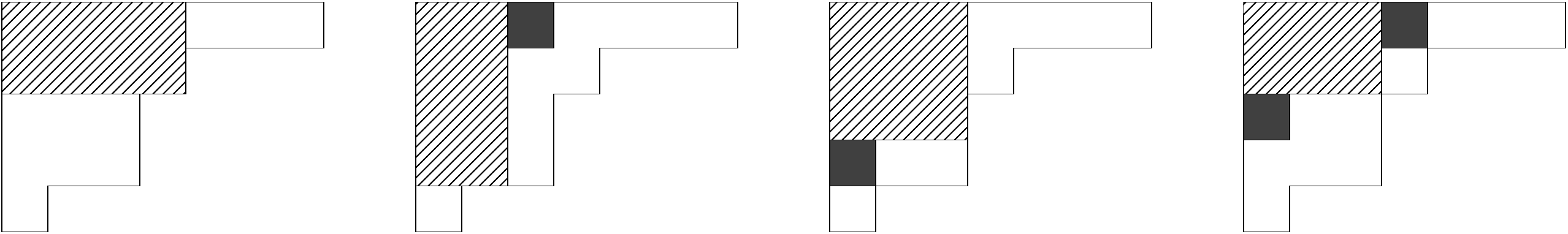}
	\caption{The shaded rectangles provide 
	examples of rectangles from the sets  
	 $\Fr_0(\lambda)$,
	$\Fr_{1,E}(\lambda)$, 
	 $\Fr_{1,S}(\lambda)$, 
	 $\Fr_2(\lambda)$, while the solid black boxes 
	 depict addable boxes.}
\label{fig:4cases}
\end{figure}

If $\mu$ is a rectangle, we let $\mu^{\square}$ denote the Young diagram obtained
from $\mu$ by adding a new box to the right of the first row of $\mu$.  Similarly, let 
$_{\square}\mu$ denote the Young diagram obtained from $\mu$ by adding a new box
at the bottom of the leftmost column of $\mu$. 
Note that if $\mu\in \Fr_0(\lambda)$ then $\mu=\mu_\rho$ for some $\rho=\rho_{2\ell-1}$ in $\Rout(\lambda)$. In this case we define $q(\mu):=q_\ell$. 
Then we have the following reformulation of \eqref{e:Wq1}. 

\begin{proposition}[Rim-indexed formula for the superpotential]
\begin{equation}
\label{e:Wq2}
W^{\lambda} = 
	\frac{p_{\Box}}{p_{\emptyset}} 
+
	\sum_{\mu\in \Fr_{1,E}(\lambda)} \frac{p_{\muright}}{p_{\mu}}
	+ \sum_{\mu\in \Fr_{1,S}(\lambda)} \frac{p_{\mudown}}{p_{\mu}}
	+ \sum_{\mu\in \Fr_2(\lambda)} \frac{p_{\muright}+p_{\mudown}}{p_{\mu}}
	+	\sum_{\mu\in \Fr_0(\lambda)} q(\mu) \frac{p_{\mu^{-}}}{p_{\mu}}.
\end{equation}
\end{proposition}

\begin{example}
Suppose $\lambda = (4,4,2)$. Then $k=4$ and  $n=7$, and we have $\checkX_{\lambda} \subset Gr_4((\C^7)^*)$. In this case
\begin{align*}
\Fr_0(\lambda)=
	\big\{\ydiagram{4,4}, \ydiagram{2,2,2}\,\big\}, \quad
\Fr_{1,E}(\lambda)=
	\big\{\ydiagram{1,1,1}, \ydiagram{3,3}\big\}, \quad
	\Fr_{1,S}(\lambda)=\big\{\ydiagram{4}\big\},\quad
\Fr_{2}(\lambda)=\big\{\ydiagram{2,2} \big\},
\end{align*} 
and the superpotential on $\dualSchub$ is given by the expression
\[
	W^{\lambda}=
        \frac{p_{\ydiagram{1}}}{p_{\emptyset}} + 
        \frac{ p_{\ydiagram{2,1,1}}}{p_{\ydiagram{1,1,1}}} + 
        \frac{p_{\ydiagram{4,3}}}{\p_{\ydiagram{3,3}}} + 
        \frac{p_{\ydiagram{4,1}}}{\p_{\ydiagram{4}}} +
	\frac{\left(
	p_{\ydiagram{3,2}}+
	p_{\ydiagram{2,2,1}}
	\right)}{\p_{\ydiagram{2,2}}}+
 \frac{q_1 p_{\ydiagram{3}}}{p_{\ydiagram{4,4}}} +
        \frac{q_2 p_{\ydiagram{1,1}}}{p_{\ydiagram{2,2,2}}}.
\]
\end{example}

We now give one more equivalent way of expressing the superpotential on $\dualSchub$.

\begin{definition}\label{def:NW}
Let $\lambda, k, n$ be as in  \cref{n:knd}, so that 
 $\lambda$ has $n-k$ rows and $\lambda_1=k$.  Recall the notation
	$\BB^{\NW}(\lambda)$ for the northwest boundary of $\lambda$, see  \cref{def:BB}.
	Let us define a map, 
\[
\frozen: \BB^{\NW}(\lambda)\to \Fr(\lambda),
\]
by setting $\frozen(b)$ to be the minimal element $\mu$ of $\Fr(\lambda)$ such that 
	$b$ is an addable box for $\mu$ in $\lambda$, see 
\cref{d:addable}.
 In particular if $b$ is the top left hand corner box then $\frozen(b)=\emptyset$. Note that the map $\frozen$ is clearly not surjective nor is it in general injective. 
\end{definition}

Using the map $\frozen$ we can give the following equivalent description of the superpotential.

\begin{proposition}[Northwest-border-indexed formula for the superpotential]
\label{p:3rdpotential}
Let $\lambda$ be a Young diagram  as in 
\cref{def:NW}.
We have that 
	\begin{equation}\label{eq:3rdpotential}
		W^{\lambda} = \sum_{b\in \BB^{\NW}(\lambda)}  \frac{p_{\frozen(b) \sqcup b}}{p_{\frozen(b)}} +  \sum_{\mu\in \Fr_0(\lambda)} q(\mu)\frac{p_{\mu^-}}{p_{\mu}}.
	\end{equation}
\end{proposition}

\begin{proof} It suffices to show that this function 
	\eqref{eq:3rdpotential} 
	is made up of the same terms as the one given in \eqref{e:Wq2}. We consider the terms according to their denominators $p_\mu$, for which there are five cases. If $\mu=\mu_i$ lies in $\Fr_{1,S}(\lambda)$ then $\mu_{i-1}\subset\mu_i\subset\mu_{i+1}$ are rectangles of the same width but differing height. In this case there is a unique box $b$ for which $\mu$ is the  minimal rectangle in $\Fr(\lambda)$ such which $b$ is an addable box for $\mu$, and this $b$ necessarily lies in the first column of $\lambda$. The term associated to the box $b$ in \eqref{eq:3rdpotential} agrees with the term associated to $\mu$ in \eqref{e:Wq2}. Similarly, if $\mu=\mu_i$ lies in $\Fr_{1,E}(\lambda)$ then $\mu_{i-1}\supset\mu_i\supset\mu_{i+1}$ have the same height but differing width. In this case again $\mu=\frozen(b)$ only for a single box $b$, and now this box lies in the first row of $\lambda$. The term associated to this $b$ in \eqref{eq:3rdpotential} agrees with the term associated to $\mu$ in \eqref{e:Wq2}. If $\mu=\mu_i\in\Fr_{2}(\lambda)$ then there are two boxes, $b_S$ and $b_E$, one in the first column, and one in the first row of $\lambda$, which are addable to $\mu$ and for which $\mu$ is minimal. The sum of the terms associated to $b_S$ and $b_E$ in  \eqref{eq:3rdpotential} agree with the term associated to $\mu$ in  \eqref{e:Wq2}.
The last `non-quantum' term in \eqref{e:Wq2} is $\frac{p_{\ydiagram{1}}}{p_{\emptyset}}$ and this corresponds to the term in \eqref{eq:3rdpotential}, which is associated to top left hand corner box $b=(1,1)$.  
Finally, if $\mu\in\Fr_{0}(\lambda)$ then it is not in the image of the map $\frozen$ and only contributes terms involving the quantum parameters. These terms agree in \eqref{eq:3rdpotential} and  \eqref{e:Wq2}. 
\end{proof}

\begin{definition}\label{d:Wis}
From \cref{p:3rdpotential} we  see that there are precisely $(n-1)+d$ 
terms in the superpotential: $(n-1)$ terms from the boxes 
$b'_1,\dots,b'_{n-1}$ in the northwest boundary
	$\BB^{\NW}(\lambda)$
of $\lambda$, and $d$ terms corresponding to the outer corners of $\lambda$.
	Recall that the outer corner boxes of $\lambda$ are labeled by $\Rout(\lambda)=\{\rho_1, \rho_3,\dotsc,\rho_{2d-1}\}$, see \cref{r:R(lambda)}, and accordingly the $\ell$-th element of $\Fr_0(\lambda)$ is $\mu_i$ for $i=\rho_{2\ell-1}$, compare~\cref{d:frozendec}. 
We set
\begin{equation}
	W'_i:=\frac{p_{\frozen(b'_i)\sqcup b'_i}}{p_{\frozen(b'_i)}}\quad\text{for $i=1,\dotsc, n-1$} \qquad \text{ and }\qquad 
W_\ell:=\frac{p_{\mu_{\rho_{2\ell-1}}^-}}{p_{\mu_{\rho_{2\ell-1}}}}\quad\text{for $\ell=1,\dotsc, d$.}
\end{equation}
Note that we have
\begin{equation}
\label{e:Wq3} 
		W^{\lambda} = \sum_{i=1}^{n-1}  W'_i +  \sum_{\ell=1}^{d} q_\ell W_\ell.
	\end{equation}
An example specifying the $W_i'$ and the $W_\ell$ was given in \eqref{e:IntroExample}\end{definition}

\begin{remark}\label{rem:GHKK}
We remark that \cite[Corollary 9.17]{GHKK} has a very general construction for a
superpotential associated to a cluster variety,
and in the case of the Grassmannian
$Gr_{2,5}$,
\cite[Section 7]{BFMN} shows that
the Marsh-Rietsch superpotential at $q=1$ agrees with the superpotential of \cite{GHKK}.
It would be interesting to extend this comparison in the case of open Schubert varieties.
\end{remark}

\subsection{The superpotential in terms  of the rectangles cluster}\label{s:rectW}

When restricted to a particular torus, 
the superpotential can also be expressed as a Laurent polynomial 
which is encoded by a diagram, generalising the early Laurent polynomial mirror constructions from \cite{Givental:fullflag,BC-FKvSGrass,BC-FKvS} as well as \cite{EHX}. Our 
\cref{fig:superpotential} shows this diagram in an example.  The general formula is given in 
\cref{prop:super2}.

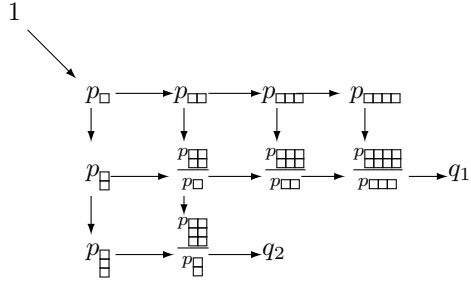
\begin{figure}[h]
\setlength{\unitlength}{1.3mm}
\begin{center}
	\begin{picture}(50,28)(0,8)
% extra node
         \put(0,34){$\color{black}1$}
% first row
	 \put(8,26){$p_{\ydiagram{1}}$}
	 \put(17,26){$p_{\ydiagram{2}}$}
		   \put(26,26){$p_{\ydiagram{3}}$}
	 \put(35,26){$\color{black}p_{\ydiagram{4}}$}
% second row
	 \put(8,18){$p_{\ydiagram{1,1}}$}
	 \put(17,18){$\frac{p_{\ydiagram{2,2}}}{p_{\ydiagram{1}}}$}
	 \put(26,18){$\frac{p_{\ydiagram{3,3}}}{p_{\ydiagram{2}}}$}
	 \put(35,18){$\frac{p_{\ydiagram{4,4}}}{p_{\ydiagram{3}}}$}
	 \put(45,18){$q_1$}
% third row
	 \put(8,10){$\color{black}p_{\ydiagram{1,1,1}}$}
	 \put(17,10){$\color{black}\frac{p_{\ydiagram{2,2,2}}}{p_{\ydiagram{1,1}}}$}
	 \put(26,10){$q_2$}
 % arrows
         \put(2,33){{\vector(1,-1){5}}}
         \put(11,26.5){{\vector(1,0){6}}}
         \put(11,10){{\vector(1,0){6}}}
         \put(20.5, 26.5){{\vector(1,0){5.5}}}
         \put(10.5,18){{\vector(1,0){6}}}
         \put(8.5,16){{\vector(0,-1){4}}}
         \put(18,16){{\vector(0,-1){2}}}
         \put(8.5,25){{\vector(0,-1){3.5}}}
         \put(18,25){{\vector(0,-1){3.5}}}
         \put(27.5,25){{\vector(0,-1){3.5}}}
         \put(36.5,25){{\vector(0,-1){3.5}}}
         \put(20.5,18){{\vector(1,0){5.5}}}
         \put(20.5,10){{\vector(1,0){5.5}}}
         \put(29.5, 26.5){{\vector(1,0){4.5}}}
         \put(30,18){{\vector(1,0){4.5}}}
	 \put(41,18){{\vector(1,0){4}}}
 \end{picture}
\end{center}
	\caption{A quiver encoding the superpotential for $X_{\lambda}$ with $\lambda = (4,4,2)$.}
\label{fig:superpotential}
  \end{figure}

\begin{notation}\label{emptyis1}
Let $p_{i \times j}$ denote the Pl\"ucker coordinate indexed by the Young diagram
which is an $i \times j$ rectangle.  If $i=0$ or $j=0$ then we 
set $p_{i \times j} = p_{\emptyset} = 1$.
\end{notation}

\begin{definition}\label{def:superquiver}
	Let $\lambda, k, n, d$ be as in \cref{n:knd}.
We label rows of $\lambda$ from top to bottom, and columns from left to right.
We refer to the box in row $i$ and column $j$ as $(i,j)$.
Let $i_1<\dots< i_d$ denote the rows containing the outer corners of $\lambda$.
We define a labeled quiver $Q_{\lambda}$ as follows:
\begin{itemize}
	\item If $(i,j)$ is a box of $\lambda$, we associate a 
		vertex $v(i,j)$ of $Q_{\lambda}$ and label it by 
		 $p_{i \times j}/ p_{(i-1) \times (j-1)}$.
	 \item If $(i,j)$ and $(i,j+1)$ are boxes of $\lambda$,
		 we add an arrow $v(i,j)\to v(i,j+1)$.
	 \item If $(i,j)$ and $(i+1,j)$ are boxes of $\lambda$,
		 we add an arrow $v(i,j)\to v(i+1,j)$.
	 \item We add one extra vertex $v_0$ of $Q_{\lambda}$, labeled $1$,
		 together with an arrow $v_0 \to v(1,1)$.
	 \item For each outer corner in row $i_\ell$, we add an extra vertex $v_{\ell}$ 
		 labeled $q_{\ell}$, 
		 together with an arrow $v(i_{\ell}, \lambda_{i_{\ell}}) \to v_{\ell}$.
\end{itemize}
	Let $A(Q_{\lambda})$ denote the set of arrows of $Q_{\lambda}$, and for each 
	arrow $a:v\to v'$ in $A(Q_{\lambda})$, let $p(a)$ denote the Laurent monomial 
	in Pl\"ucker coordinates obtained
	by dividing the label of $v'$ by the label of $v$.

\end{definition}

See  \cref{fig:superpotential} for an example.
%of the quiver $Q_{\lambda}$ associated
%to $\lambda = (4,4,2)$.  
If $a$ is the arrow from $\frac{p_{\ydiagram{4,4}}}{p_{\ydiagram{3}}}$ 
to $q_1$, then 
$p(a) = \frac{q_1 p_{\ydiagram{3}}}{p_{\ydiagram{4,4}}}$.

\begin{proposition}[Expansion of the superpotential in the rectangles cluster]\label{prop:super2}
	Let $\lambda, k, n, d$ be as in \cref{n:knd}.
	Let $\mathbb T^{\lambda}_{\rect}$ be the subset of $\dualSchub$ where
	 $p_{i \times j} \neq 0$ for all $i \times j \subseteq \lambda$.
	When we restrict $W^{\lambda}$ to
	$\mathbb T^{\lambda}_{\rect}$
	(a cluster torus for the $\mathcal{A}$-cluster
	structure for the open Schubert variety, see
	\cref{sec:Acluster})
	we obtain $$W^{\lambda}_\rect = \sum_{a\in A(Q_{\lambda})} p(a).$$
\end{proposition}

For example, when $\lambda = (4,4,2)$, we obtain
\begin{align*}\label{super}
	W^{\lambda} &= 
	\frac{\pyo}{p_{\emptyset}}+ 
\frac{\pyz}{\pyo} + 
\frac{\pyt}{\pyz}+
\frac{\pytz}{\pyt}+
\frac{\pyzz }{\pyo \ \pyoo}+
\frac{\pytt \ \pyo}{\pyz \ \pyzz}+
	\frac{p_{\ydiagram{4,4}} p_{\ydiagram{2}}}
	{\pyt \ \pytt}+
		\frac{q_1 \pyt}{p_{\ydiagram{4,4}}}+
		\frac{\pttt}{\pyoo \pooo} + 
		\frac{q_2 \pyoo}{\pttt}\\
		&+	\frac{\pyoo}{\pyo}+
\frac{\pyzz  }{\pyo \ \pyz}+
\frac{\pytt }{\pyz \ \pyt}+
	\frac{p_{\ydiagram{4,4}} }{\pyt \ p_{\ydiagram{4}}}+
 \frac{\pooo}{\pyoo}+
		\frac{\pttt \pyo}{\pyoo \pyzz}.
\end{align*}

To prove \cref{prop:super2} we first verify the following lemma.

\begin{lemma}\label{lem:Plucker}
Recall that  
	$\mathbb\checkX 
	=Gr_{k}((\C^n)^*)$, 
	 with Pl\"ucker coordinates
indexed by partitions contained in a $(n-k)\times k$ rectangle,
and  
let $i$ and $m$ be positive integers such that $i<n-k$ and $m \leq k$.
Then  
	\begin{equation}\label{eq:first}
		\sum_{j=1}^m \frac{p_{(i+1) \times j} \ p_{(i-1) \times (j-1)}}{p_{i \times (j-1)}\ p_{i \times j}} 
		= \frac{p_{_{\square}(i \times m)}}{p_{i \times m}}
	\end{equation}
	where $_{\square}(i \times m)$ is the 
	 Young diagram $(m,m,\dots, m, 1)$, i.e.
	 an $i \times m$ rectangle with a box appended at the bottom of the 
	 leftmost column.

Let $j$ and $h$ be positive integers such that $h \leq n-k$ and $j<k$.
Then we have that 
	\begin{equation}\label{eq:second}
		\sum_{i=1}^{h} \frac{p_{i \times (j+1)}\ p_{(i-1)\times (j-1)}}{p_{i \times j}\ p_{(i-1) \times j}} 
		= \frac{p_{(h \times j)^{\square}}}{p_{h \times j}} 
	\end{equation}	
	where ${(h \times j)^{\square}}$ is the 
	 Young diagram $(j+1,j,j,\dots,j)$, i.e. 
	 an $h \times j$ rectangle with a box appended at the right of the topmost row.
\end{lemma}
\begin{proof}
We will see that \eqref{eq:first} follows easily by induction on $m$, using 
 the three-term Pl\"ucker relation.   When $m=1$ there is nothing to prove.
Now suppose \eqref{eq:first} is true for a fixed $m$.  Then we want to show that 
it is true for $m+1$.  Using induction, it is enough to show that 
	\begin{equation}\label{eq:Plucker1}
		  \frac{p_{_{\square}(i \times m)}}{p_{i \times m}} + 
		 \frac{p_{(i+1) \times (m+1)} \ p_{(i-1) \times m}}{p_{i \times m}\ p_{i \times (m+1)}}  = 
		  \frac{p_{_{\square}(i \times (m+1))}}{p_{i \times (m+1)}}.
	\end{equation}
	But this is precisely a three-term Pl\"ucker relation.

The proof of \eqref{eq:second} can be obtained from \eqref{eq:first}
	by working in the dual Grassmannian.
\end{proof}

The proof of \cref{prop:super2} 
follows from \cref{lem:Plucker}: we simply sum the contributions of all arrows in 
a given row and all arrows in a given column of $Q_{\lambda}$.  
That produces the formula 
\eqref{e:Wq3} 
for $W^{\lambda}$: in particular, for $0<i<n-k$,
$W'_{n-k-i}$ has the form 
$\frac{p_{_{\square}(i \times m)}}{p_{i \times m}},$
while for $0<j<k$,
$W'_{n-k+j}$ 
has the form $\frac{p_{(h \times j)^{\square}}}{p_{h \times j}}.$

\section{Geometry of the Schubert variety $X_\lambda$ and its boundary divisor}\label{s:GeometrySchubert}

In this section we recall the positroid stratification of the Schubert variety and use it to describe the irreducible components of the boundary divisor $D^\lambda_{\ac}$ of $X_{\lambda}$ defined in \eqref{e:boundarydivisor}. We will also express the homology classes of the irreducible components of the boundary divisor in terms of the Schubert basis, see \cref{p:PositroidHomology}.
We will furthermore describe which  divisors supported on the boundary are Cartier. 
These results will be used later in the proof of \cref{t:maingen}.

We start by collecting together some facts about the geometry of Schubert varieties $X_\lambda$, see \cite{Manivel,LakBrown15, Humphreys, Springer} for reference. We freely use notations from Section~\ref{s:SchubandOpenSchub}. 
\begin{enumerate}
\item The Schubert variety $X_\lambda$ has an algebraic cell decomposition into Schubert cells given by $X_\lambda=\bigsqcup_{\mu\subseteq\lambda}\Omega_{\mu}$. Their closures $X_\mu=\overline{\Omega}_\mu$ are the Schubert varieties contained in $X_\lambda$. 
The associated fundamental homology classes $[X_{\mu}]$ form a $\Z$-basis of  $H_*(X_{\lambda},\Z)$ with $[X_{\mu}]$ having degree $2|\mu|$. The homology group $H_{2k}(X_{\lambda},\Z)$ is isomorphic to the Chow group $A_k(X_\lambda)$ of $X_\lambda$. 
\item The cap product defines a perfect pairing between homology and cohomology and we denote by $\sigma^{\mu}$ the cohomology class in $H^*(X_{\lambda},\Z)$ dual to $[X_{\mu}]$. 
\item There are $d$ Schubert divisors in $X_\lambda$, where $d$  is the number of outer corners in $\lambda$. We denote these by $D_1,\dotsc, D_d$, where $D_i$ is the Schubert divisor $X_{\mu}$ with $\mu$ obtained by removing the $i$-th outer corner from $\lambda$ (counting from the NE corner to the SW corner). These are generally only Weil divisors. Their linear equivalence classes form a basis of the divisor class group $\Cl(X_\lambda)$, which is isomorphic to $H_{2|\lambda|-2}(X_\lambda)$. 
\item  
The divisor $\sum_{i=1}^d D_i$ is Cartier, and is an ample divisor corresponding to the Pl\"ucker embedding of $X_{\lambda}$. Namely, it is precisely equal to $\{P_{\lambda}=0\}$. 
\item 
The map $\Pic(Gr_{n-k}(\C^n))\to \Pic(X_\lambda)$ defined by restriction of line bundles is an isomorphism. 
Therefore, $\Pic(X_\lambda)\cong \Z$. Its generator is the line bundle $\mathcal O(\sum_{i=1}^d D_i)$ corresponding to the Pl\"ucker embedding, and the Schubert class $\sigma^{\square}\in H^2(X_\lambda,\Z)$ is the first Chern class of $\mathcal O(\sum_{i=1}^d D_i)$. The first Chern class map gives an isomorphism between $\Pic(X_{\lambda})$ and $H^2(X_\lambda,\Z)$.
\end{enumerate}

\subsection{Positroids and $\Le$-diagrams}\label{s:LeDivisors}
The Schubert variety $X_\lambda$ has a natural stratification that is finer than the Schubert cell decomposition called the positroid stratification.
Open positroid varieties are examples of projected open Richardson varieties, which 
were studied by Lusztig \cite{Lusztig3} and Rietsch \cite{RietschThesis} in the 
context of total positivity. Independently, Postnikov introduced the positroid decomposition of the totally nonnegative Grassmannian \cite{Postnikov}
and gave many combinatorially explicit ways to describe the strata. Knutson-Lam-Speyer studied the corresponding stratification in the complex 
Grassmannian \cite{KLS}.

Consider $GL_n(\C)$ with its upper- and lower-triangular Borel subgroups $B_+$ and $B_-$,respectively. Let $P_{n-k}$ be $(n-k)$-th the maximal parabolic subgroup of $GL_n$, so that we have the homogeneous space description $\mathbb X=GL_n(\C)/P_{n-k}$ of the Grassmannian  containing $X_\lambda$. Let 
\[
\pi:GL_{n}(\C)/B_+\to  GL_n(\C)/P_{n-k}=\mathbb X
\]
be the projection map. Identify the Weyl group $W=S_n$ as the group of permutation matrices in $GL_n(\C)$ and write $W_{P_{n-k}}$ for its associated parabolic subgroup, that is, the subgroup generated by the simple reflections $s_i=(i,i+1)$ where $i\ne n-k$.   
The set $W^{P_{n-k}}$ of minimal coset representatives consists of all Grassmannian permutations with unique descent in position $n-k$. Equivalently, $w\in W^{P_{n-k}}$ if every reduced expression for $w$ in terms of simple reflections $s_i$ ends in $s_{n-k}$.

\begin{remark}\label{r:Young2} There is a standard bijection between $W^{P_{n-k}}$ and the set of Young diagrams that fit into an $(n-k)\times k$ rectangle. Namely fill the boxes of the $(n-k)\times k$ rectangle by simple reflections where the box in row $i$ and column $j$ is filled with $s_{n-k+j-i}$. For the Young diagram $\mu$ we then associate the Weyl group element $w_\mu$ obtained by reading the entries of $\mu$ row by row from right to left bottom to top. The resulting product of simple reflections is $w_\mu$, and moreover it forms a reduced expression, so that we also see that the length $\ell(w_\mu)$ of $w_\mu$ is given by~$|\mu|$. We set $w_\emptyset=e$, the identity element of $W$. 
\end{remark}

\begin{defn} \label{d:projRich} 
	For any pair of permutations $v,w$ in $W=S_n$ with $v\le w$ for the Bruhat order, define the associated \emph{open Richardson variety} $\mathcal R_{v,w}$ in the full flag variety to be the intersection of opposite Bruhat cells, 
$$
\mathcal R_{v,w}=B_-vB_+\cap B_+wB_+/B_+.
$$
We have that 
$\mathcal R_{v,w}\subset \overline{\mathcal R_{v',w'}}$ 
whenever $v'\le v\le w\le w'$.
If $w\in W^{P_{n-k}}$ then 
\[X^\circ_{(v,w)}:=\pi(w_0 \mathcal  R_{v,w})
\] 
is an isomorphic image of $\mathcal R_{v,w}$ and we call it the 
	\emph{projected open Richardson variety}   or \emph{open positroid variety}
	in $\mathbb X$ associated to $(v,w)$.  
	(The terminology is justified by the fact that this variety is an open positroid variety
	in the sense of \cref{def:posvariety}, as shown in 
	 \cite[Theorem 5.9]{KLS}.)
	We call its closure $X_{(v,w)}$ the \textit{positroid variety} associated to $(v,w)$.  
\end{defn}

Open positroid varieties are smooth and irreducible, because this holds for the open Richardson
varieties by Kleiman transversality. Moreover the dimension of $\mathcal R_{v,w}$ and hence of $X^\circ_{(v,w)}$ and $X_{(v,w)}$ is given by $\ell(w)-\ell(v)$. See \cite{KL} and also \cite{KLS}. The closed positroid varieties are unions of open positroid strata; see \cite{RietschClosure}
for the precise description of which open positroid strata
comprise a given closed positroid variety.
The positroid stratification of the Schubert variety $X_\lambda$ is given by
\begin{equation}\label{e:PositroidStrat}
X_\lambda=\bigsqcup_{\mu\subseteq \lambda}\left(\bigsqcup_{v\in W, v\le w_{\mu}} X^\circ_{(v,w_\mu)}\right).
\end{equation}

\begin{remark}\label{r:maximalpositroid}
Note that the pair $(e,w_\lambda)$ gives rise to the unique full-dimensional open  positroid variety $X^\circ_{(e,w_\lambda)}$ in $X_\lambda$, and this positroid stratum coincides with the  \emph{open Schubert variety} $\openSchub = 
X^\circ_{(e,w_\lambda)}$ 
from \cref{d:openX}, see \cref{rem:defopenSchubert}.
	The positroid variety $X_{(e,w_\lambda)}$ defined as its closure is just  $X_{\lambda}$.
\end{remark}

\subsection{Positroid divisors} 
In order to describe the boundary divisor $D^{\lambda}_{\ac}$ and the individual divisors $\{P_{\mu_i}=0\}$ contained in the boundary, we now focus on the codimension $1$ positroid strata.   These \textit{positroid divisors} come in two types. The Schubert varieties $D_1,\dots, D_d$ are the first immediate examples of positroid divisors for $X_\lambda$, and then we have positroid divisors of the form $X_{(s_i,w_\lambda)}$. We now use the fact that positroid strata are in bijection with $\Le$-diagrams to label the positroid divisors of $X_\lambda$ combinatorially. 

\begin{lemma}\label{lem:codim1}
The positroid divisors which are contained in $X_\lambda$ are 
precisely the positroid varieties whose $\Le$-diagrams are the following:
\begin{itemize}
\item The filling by all $+$'s of 
a Young diagram obtained from $\lambda$ by removing an outer corner;
\item A filling of the Young diagram $\lambda$ 
	in which each box contains a $+$
	except for one box; necessarily 
		that box must be in the leftmost column
		of $\lambda$ or the topmost row of $\lambda$.
	\end{itemize}
In particular, if $k$ and $n$ are minimal such that $\lambda \subseteq (n-k)\times k$,
	and $\lambda$ has $d$ outer corners, then there are $d+(n-1)$
	positroid divisors contained in $X_{\lambda}$.
\end{lemma}
See \cref{fig:codim} for an example 
illustrating \cref{lem:codim1}.

\begin{proof} 
Recall from \cref{def:BB}
that $b'_i$ denotes 
the $i$-th box in the northwest border of the Young diagram $\lambda$,
counting from the bottom upwards.
The bijection between the pairs $(v,w)$ from \eqref{e:PositroidStrat} and $\Le$-diagrams  is given in \cite[Section 19]{Postnikov}. It is straightforward to deduce \cref{lem:codim1} using this bijection and \cite{RietschClosure}. Explicitly, the first kind of $\Le$-diagram, which involves removing an outer corner of $\lambda$ to obtain a smaller Young diagram $\mu$,  corresponds to the positroid variety $X_{(e,w_\mu)}$, that is, to the Schubert divisor $X_\mu$. For 
	the second kind of $\Le$-diagram, if we let $b'_i$ denote
	the box containing the unique $0$, then this $\Le$-diagram
	corresponds to the codimension $1$ positroid variety $X_{(s_i,w_\lambda)}$.
\end{proof}

Following 
\cref{lem:codim1}
we may index positroid divisors either by pairs of Weyl group elements or by $\Le$-diagrams. For convenience, we will also denote the $n-1+d$ positroid divisors in $X_{\lambda}$ as follows. Recall that each Schubert divisor in $X_\lambda$ relates to removing a single box (outer corner) of $\lambda$. Let us write $\lambda^-_\ell:=\lambda \setminus b_{\rho_{2\ell+1}}$ for the Young diagram with $\ell$-th outer corner removed (using notation from  \cref{r:R(lambda)}).

\begin{definition}\label{def:positroiddivisors}
Let 
	\[\begin{array}{ccll}\label{e:Dis}
D_\ell&:=&X_{\lambda^-_\ell},& \ell=1,\dotsc, d,\\
D'_i&:=& X_{(s_i,w_\lambda)}, & i=1,\dotsc, n-1.
\end{array}
	\]
Equivalently,
$D_{\ell}$ is 
associated to the $\Le$-diagram whose
 Young diagram $\lambda_\ell^-$ is 
	obtained by removing the outer corner box $b_{\rho_{2\ell-1}}$ 
	from $\lambda$, and whose boxes are all filled with $+$'s.
And $D_i'$ is associated to the $\Le$-diagram whose 
Young diagram $\lambda$ contains a $0$ in  box $b'_i$ 
	and a $+$ in every other box. 
\end{definition}

\begin{figure}[h]
\centering
\setlength{\unitlength}{1.3mm}
\begin{center}
	\resizebox{0.11\textwidth}{!}{\begin{picture}(42,35)
\put(5,32){\line(1,0){36}}
  \put(5,23){\line(1,0){36}}
  \put(5,14){\line(1,0){36}}
  \put(5,5){\line(1,0){18}}
% vertical lines
  \put(5,5){\line(0,1){27}}
  \put(14,5){\line(0,1){27}}
  \put(23,5){\line(0,1){27}}
  \put(32,14){\line(0,1){18}}
  \put(41,23){\line(0,1){9}}
  \put(41,14){\line(0,1){18}}
% first row
   \put(8,26){\Huge{$0$}}
   \put(17,26){\Huge{$+$}}
         \put(26,26){\Huge{$+$}}
         \put(35,26){\Huge{$+$}}
% second row
         \put(8,17){\Huge{$+$}}
         \put(17,17){\Huge{$+$}}
         \put(26,17){\Huge{$+$}}
         \put(35,17){\Huge{$+$}}
% third row
         \put(8,8){\Huge{$+$}}
         \put(17,8){\Huge{$+$}}
        \put(17,-4){\Huge{$D'_3$}}
% arrows
        \end{picture}}
        \hspace{0.06cm}
\resizebox{0.11\textwidth}{!}{\begin{picture}(42,35)
\put(5,32){\line(1,0){36}}
  \put(5,23){\line(1,0){36}}
  \put(5,14){\line(1,0){27}}
  \put(5,14){\line(1,0){36}}
  \put(5,5){\line(1,0){18}}
% vertical lines
  \put(5,5){\line(0,1){27}}
  \put(14,5){\line(0,1){27}}
  \put(23,5){\line(0,1){27}}
  \put(32,14){\line(0,1){18}}
  \put(41,23){\line(0,1){9}}
  \put(41,14){\line(0,1){18}}
% first row
   \put(8,26){\Huge{$+$}}
   \put(17,26){\Huge{$0$}}
         \put(26,26){\Huge{$+$}}
         \put(35,26){\Huge{$+$}}
% second row
         \put(8,17){\Huge{$+$}}
         \put(17,17){\Huge{$+$}}
         \put(26,17){\Huge{$+$}}
         \put(35,17){\Huge{$+$}}
% third row
         \put(8,8){\Huge{$+$}}
         \put(17,8){\Huge{$+$}}
        \put(17,-4){\Huge{$D'_4$}}
% arrows
        \end{picture}}
        \hspace{0.06cm}
\resizebox{0.11\textwidth}{!}{\begin{picture}(42,35)
\put(5,32){\line(1,0){36}}
  \put(5,23){\line(1,0){36}}
  \put(5,14){\line(1,0){27}}
  \put(5,14){\line(1,0){36}}
  \put(5,5){\line(1,0){18}}
% vertical lines
  \put(5,5){\line(0,1){27}}
  \put(14,5){\line(0,1){27}}
  \put(23,5){\line(0,1){27}}
  \put(32,14){\line(0,1){18}}
  \put(41,23){\line(0,1){9}}
  \put(41,14){\line(0,1){18}}
% first row
   \put(8,26){\Huge{$+$}}
   \put(17,26){\Huge{$+$}}
         \put(26,26){\Huge{$0$}}
         \put(35,26){\Huge{$+$}}
% second row
         \put(8,17){\Huge{$+$}}
         \put(17,17){\Huge{$+$}}
         \put(26,17){\Huge{$+$}}
         \put(35,17){\Huge{$+$}}
% third row
         \put(8,8){\Huge{$+$}}
         \put(17,8){\Huge{$+$}}
        \put(17,-4){\Huge{$D'_5$}}
% arrows
        \end{picture}}
        \hspace{0.06cm}
\resizebox{0.11\textwidth}{!}{\begin{picture}(42,35)
\put(5,32){\line(1,0){36}}
  \put(5,23){\line(1,0){36}}
  \put(5,14){\line(1,0){27}}
  \put(5,14){\line(1,0){36}}
  \put(5,5){\line(1,0){18}}
% vertical lines
  \put(5,5){\line(0,1){27}}
  \put(14,5){\line(0,1){27}}
  \put(23,5){\line(0,1){27}}
  \put(32,14){\line(0,1){18}}
  \put(41,23){\line(0,1){9}}
  \put(41,14){\line(0,1){18}}
% first row
   \put(8,26){\Huge{$+$}}
   \put(17,26){\Huge{$+$}}
         \put(26,26){\Huge{$+$}}
         \put(35,26){\Huge{$0$}}
% second row
         \put(8,17){\Huge{$+$}}
         \put(17,17){\Huge{$+$}}
         \put(26,17){\Huge{$+$}}
         \put(35,17){\Huge{$+$}}
% third row
         \put(8,8){\Huge{$+$}}
         \put(17,8){\Huge{$+$}}
        \put(17,-4){\Huge{$D'_6$}}
% arrows
        \end{picture}}
        \hspace{0.06cm}
\resizebox{0.11\textwidth}{!}{\begin{picture}(42,35)
\put(5,32){\line(1,0){36}}
  \put(5,23){\line(1,0){36}}
  \put(5,14){\line(1,0){27}}
  \put(5,14){\line(1,0){36}}
  \put(5,5){\line(1,0){18}}
% vertical lines
  \put(5,5){\line(0,1){27}}
  \put(14,5){\line(0,1){27}}
  \put(23,5){\line(0,1){27}}
  \put(32,14){\line(0,1){18}}
  \put(41,23){\line(0,1){9}}
  \put(41,14){\line(0,1){18}}
% first row
   \put(8,26){\Huge{$+$}}
   \put(17,26){\Huge{$+$}}
         \put(26,26){\Huge{$+$}}
         \put(35,26){\Huge{$+$}}
% second row
         \put(8,17){\Huge{$0$}}
         \put(17,17){\Huge{$+$}}
         \put(26,17){\Huge{$+$}}
         \put(35,17){\Huge{$+$}}
% third row
         \put(8,8){\Huge{$+$}}
         \put(17,8){\Huge{$+$}}
        \put(17,-4){\Huge{$D'_2$}}
% arrows
        \end{picture}}
        \hspace{0.06cm}
\resizebox{0.11\textwidth}{!}{\begin{picture}(42,35)
\put(5,32){\line(1,0){36}}
  \put(5,23){\line(1,0){36}}
  \put(5,14){\line(1,0){27}}
  \put(5,14){\line(1,0){36}}
  \put(5,5){\line(1,0){18}}
% vertical lines
  \put(5,5){\line(0,1){27}}
  \put(14,5){\line(0,1){27}}
  \put(23,5){\line(0,1){27}}
  \put(32,14){\line(0,1){18}}
  \put(41,23){\line(0,1){9}}
  \put(41,14){\line(0,1){18}}
% first row
   \put(8,26){\Huge{$+$}}
   \put(17,26){\Huge{$+$}}
         \put(26,26){\Huge{$+$}}
         \put(35,26){\Huge{$+$}}
% second row
         \put(8,17){\Huge{$+$}}
         \put(17,17){\Huge{$+$}}
         \put(26,17){\Huge{$+$}}
         \put(35,17){\Huge{$+$}}
% third row
         \put(8,8){\Huge{$0$}}
         \put(17,8){\Huge{$+$}}
        \put(17,-4){\Huge{$D'_1$}}
% arrows
        \end{picture}}
        \hspace{0.06cm}
\resizebox{0.11\textwidth}{!}{\begin{picture}(42,35)
\put(5,32){\line(1,0){36}}
  \put(5,23){\line(1,0){36}}
  \put(5,14){\line(1,0){27}}
  \put(5,5){\line(1,0){18}}
% vertical lines
  \put(5,5){\line(0,1){27}}
  \put(14,5){\line(0,1){27}}
  \put(23,5){\line(0,1){27}}
  \put(32,14){\line(0,1){18}}
  \put(41,23){\line(0,1){9}}
% first row
   \put(8,26){\Huge{$+$}}
   \put(17,26){\Huge{$+$}}
         \put(26,26){\Huge{$+$}}
         \put(35,26){\Huge{$+$}}
% second row
         \put(8,17){\Huge{$+$}}
         \put(17,17){\Huge{$+$}}
         \put(26,17){\Huge{$+$}}
% third row
         \put(8,8){\Huge{$+$}}
         \put(17,8){\Huge{$+$}}
        \put(17,-4){\Huge{$D_1$}}
% arrows
        \end{picture}}
        \hspace{0.06cm}
\resizebox{0.11\textwidth}{!}{\begin{picture}(42,35)
\put(5,32){\line(1,0){36}}
  \put(5,23){\line(1,0){36}}
  \put(5,14){\line(1,0){36}}
  \put(5,5){\line(1,0){9}}
% vertical lines
  \put(5,5){\line(0,1){27}}
  \put(14,5){\line(0,1){27}}
  \put(23,14){\line(0,1){18}}
  \put(32,14){\line(0,1){18}}
  \put(41,23){\line(0,1){9}}
  \put(41,14){\line(0,1){18}}
% first row
   \put(8,26){\Huge{$+$}}
   \put(17,26){\Huge{$+$}}
         \put(26,26){\Huge{$+$}}
         \put(35,26){\Huge{$+$}}
% second row
         \put(8,17){\Huge{$+$}}
         \put(17,17){\Huge{$+$}}
         \put(26,17){\Huge{$+$}}
         \put(35,17){\Huge{$+$}}
% third row
         \put(8,8){\Huge{$+$}}
        \put(17,-4){\Huge{$D_2$}}
% arrows
        \end{picture}}
\end{center}
	\caption{The codimension 1 positroids
	contained in $\dualSchub$ for $\lambda = (4,4,2)$.
	These correspond (in order) to the summands of 
	the superpotential in 
	\eqref{ex:8superpotential}.
	} 
\label{fig:codim}
\end{figure}
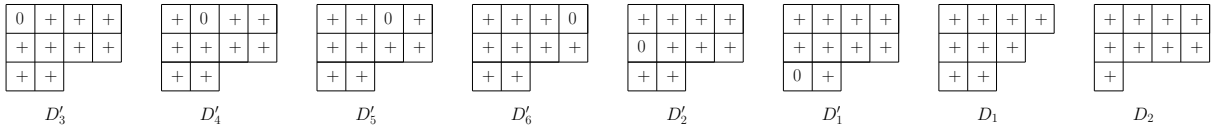

\begin{prop}\label{p:PmuVSPositroids}
Let $\mu$ be a frozen rectangle for $X_{\lambda}$. 
If the SE corner of $\mu$ is a removable box $b_{\rho_{2\ell-1}}$ in $\lambda$, then
\[
(P_\mu)=D_{\ell}+\sum_{b'_i\in \operatorname{add}(\mu)} D'_i
\] 
where $\operatorname{add}(\mu)$ denotes the set of boxes $b'_i$ from the NW border of $\lambda$ that can be added to $\mu$.

If the SE corner of $\mu$ is not removable, or if $\mu=\emptyset$, then
\[
(P_\mu)=\sum_{b'_i\in \operatorname{add}(\mu)} D'_i.
\] 
\end{prop}

\begin{proof}
Recall the two equivalent descriptions of the open Schubert variety $\openSchub$, \cref{d:openX} and 
\cref{d:projRich}, compare \cref{r:maximalpositroid}. By the first, the frozen Pl\"ucker coordinate $P_\mu$ does not vanish on $\openSchub$. By the second, $\openSchub=X^\circ_{(e,w_{\lambda})}$ and its complement is the union of the positroid divisors.
We therefore have that the divisor $(P_\mu)$ must be a linear combination of the boundary divisors $D_\ell$ and $D'_i$. 

Consider a partition $\lambda$, and a frozen rectangle
$\mu$.  Note that 
there will be at most two addable boxes for $\mu$ in $\lambda$.
 The left of \cref{fig:DivisorFlows} shows an example. 
\begin{figure}[ht]
\begin{center}
 \includegraphics[height=1in]{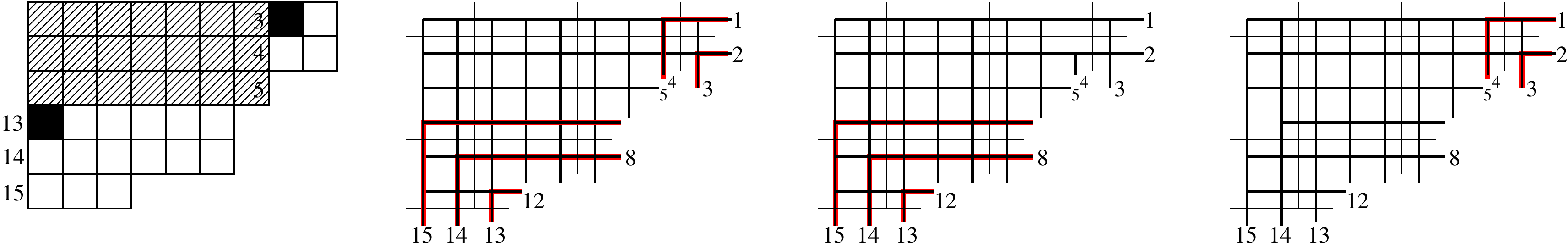}
	\caption{
	From left to right, we have:
	 the partition $\lambda
	 = (9,9,7,6,6,3)$, 
	  with the frozen rectangle 
	  $\mu = (7,7,7)$ (associated to the subset
	  $\{3,4,5,13,14,15\}$)
	 highlighted inside it, together with the two
	addable boxes for $\mu$;
	 the rectangles network associated to 
	$X_{\lambda}$ (all edges are oriented left or down), 
	together with the unique flow for $P_{\mu}$;
	 the networks associated to the two $\Le$-diagrams 
	obtained from $\lambda$ by placing a $0$ in a
	box $b_i'\in \add(\mu)$.}
 \label{fig:DivisorFlows}
\end{center}
\end{figure}
Consider the rectangles network associated to $X_\lambda$, as in the second diagram in \cref{fig:DivisorFlows}. 
It is clear by inspection that there is a \textit{unique}
flow ending at $\mu$, in line with the fact that $P_{\mu} \neq 0$ on $X^{\circ}_{e,w_\lambda}$. 
(We refer to the collection of paths in this flow
as ``packed,'' since they are as close together as possible.)
However, if the southeast corner of 
	$\mu$ is a removable box in $\lambda$,
	and we remove that box, obtaining
	a partition $\lambda'$, then 
	there will no longer be a flow ending
	at $\mu$, so $P_{\mu}$ will vanish
	on $X_{\lambda'}$.

Now consider a positroid divisor $D'_i$, whose $\Le$-diagram has shape $\lambda$,
and contains a unique $0$, where that 
 $0$ lies in 
 an addable box $b_i'\in \add(\mu)$ for $\mu$ in $\lambda$.
The corresponding rectangles network is shown in the 
two rightmost diagrams in \cref{fig:DivisorFlows}; clearly 
there will no longer be a flow ending at $\mu$, so 
$P_{\mu}$ will vanish on $D'_i$. On the other hand,
any other Schubert divisor in $X_{\lambda}$
will have $P_{\mu} \neq 0$, because there will still be
a  ``packed flow'' ending at $\mu$, analogous to the one shown
in the second diagram of \cref{fig:DivisorFlows}.
And any other positroid divisor in $X_{\lambda}$
will have $P_{\mu} \neq 0$, because the packed
flow shown in the second diagram of the figure
will still be a valid flow in the corresponding 
rectangles network for the positroid divisor.
\end{proof}

\begin{remark}\label{r:specialdivisor}
The special positroid divisor $D'_{n-k}$, which corresponds
to a $\Le$-diagram of shape $\lambda$ whose unique $0$ 
is in the northwest-most box, agrees with
	$\{P_{\emptyset}=0\}$. 
\end{remark}

\begin{cor}\label{c:ac} The divisor $D^\lambda_{\ac}$ defined in \eqref{e:boundarydivisor} can be written in terms of irreducible divisors as
\[
\Dac^\lambda = D_1 + \dots + D_d+D'_1+\dotsc+ D'_{n-1},
\]
and $\Dac^\lambda$ is an anti-canonical divisor of $X_\lambda$.
\end{cor}

\begin{proof}  We have $\Dac^\lambda=\bigcup_{i=1}^n\{P_{\mu_i}=0\}$ by definition, see \eqref{e:boundarydivisor}. Each divisor $\{P_{\mu_i}=0\}$ is a union of irreducible positroid divisors as described explicitly in \cref{p:PmuVSPositroids}. Moreover, from this explicit description we see that each positroid divisor arises as irreducible divisor contained in some $\{P_{\mu_i}=0\}$. This implies the formula for $\Dac^\lambda$. The fact that $\Dac^\lambda$ is an anticanonical divisor now follows from \cref{lem:codim1}
and \cite[Lemma 5.4]{KLS:ProjRich}. 
\end{proof}

This sum of positroid divisors is also described for projected Richardson varieties more generally in \cite[Lemma~5.4]{KLS:ProjRich} and it gives a distinguished anti-canonical divisor, see also \cite{Brion:GeomFlag}.

We observe a close relationship between the form of the superpotential $W^\lambda$ and the irreducible components of the anticanonical divisor $\Dac^\lambda$ of $X_\lambda$. Namely, the following proposition follows directly from 
\cref{lem:codim1}
and the formula \eqref{eq:3rdpotential} for the superpotential.

\begin{proposition}\label{p:LeDivisors}
The summands of the superpotential $W^{\lambda}$  
from \eqref{eq:3rdpotential} are in natural bijection with the positroid divisors in $X_\lambda$.
More specifically, for $b\in \BB_{\lambda}^{\NW}$, the 
	term $\frac{p_{\frozen(b) \cup b}}{p_{\frozen(b)}}$ in the superpotential
	is naturally associated to the $\Le$-diagram obtained from $\lambda$ by 
	putting a $0$ in box $b$ (and putting a $+$ in every other box). 
	And the term 
	$\frac{p_{\muminus}}{p_{\mu}}$
	for 
	${\mu\in \Fr_0(\lambda)}$ is naturally associated to the $\Le$-diagram 
  filled with all $+$'s, whose shape is	obtained
	from $\lambda$ by removing the box which is the outer corner of $\mu$.
\end{proposition}

\begin{example}
	The $\Le$-diagrams for the eight
	codimension $1$ positroid
	varieties contained in $\dualSchub$ for $\lambda = (4,4,2)$ shown in 
	\cref{fig:codim} 
	 correspond
to the eight terms of the superpotential 
	\begin{equation}\label{ex:8superpotential}
	W^{\lambda}=
        \frac{p_{\ydiagram{1}}}{p_{\emptyset}}+
        \frac{ p_{\ydiagram{2,1,1}}}{p_{\ydiagram{1,1,1}}}+
        \frac{p_{\ydiagram{3,2}}}{\p_{\ydiagram{2,2}}}+
        \frac{p_{\ydiagram{4,3}}}{\p_{\ydiagram{3,3}}} +
	\frac{p_{\ydiagram{4,1}}}{\p_{\ydiagram{4}}} +
        \frac{p_{\ydiagram{2,2,1}}}{\p_{\ydiagram{2,2}}}+
	 \frac{q_1 p_{\ydiagram{3}}}{p_{\ydiagram{4,4}}} +
        \frac{q_2 p_{\ydiagram{1,1}}}{p_{\ydiagram{2,2,2}}}.
	\end{equation}
\end{example}
For another example, see \cref{sec:example}. The next goal of this section is to express the positroid divisor homology classes $[D'_i]$ in terms of the basis of Schubert classes $[D_1],\dotsc,[D_d]$.

\subsection{Homology classes of positroid divisors}\label{s:PositroidHomology}
Recall the notation from \cref{d:Wis}  by which the NW border boxes of $\lambda$ are denoted by $b'_1,\dotsc, b'_{n-1}$ counting clockwise from the bottom left-hand corner. As seen in Section~\ref{s:LeDivisors}, the $\Le$-diagram associated to $D_i'$ is the Young diagram $\lambda$ filled with a $0$ in the box $b'_i$ and a $+$ in every other box. The $\Le$-diagram associated to the Schubert divisor $D_\ell$ is the Young diagram $\lambda_\ell^-$ obtained by removing the outer corner box $b_{\rho_{2\ell-1}}$ from $\lambda$. 

The following proposition will be proved in Appendix~\ref{a:positroidhomology}.

\begin{prop}\label{p:PositroidHomology} 
For each NW boundary box $b'_i$ 
	we consider the set of indices for removable corner boxes 
\[
SE(b'_i):=\{\ell\mid \text{The box $b_{\rho_{2\ell-1}}$ is 
	weakly southeast of $b_i'$}\}. 
\]
Then the homology class  
of the positroid divisor $D_i'=X_{(s_i,\lambda)}$ is expressed in terms of the Schubert classes $[D_\ell]=[X_{\lambda^-_\ell}]$ by
\[
[D_i']=\sum_{\ell\in SE(b'_i)} [D_\ell].
\]
\end{prop}

Note that the positroid divisor $D'_{n-k}$ associated to the upper left-hand corner box $b'_{n-k}$ has all outer corners {weakly} southeast of it, so that $SE(b'_{n-k})=\{1,\dots,d\}$ and $[D_{n-k}']=\sum_{\ell=1}^d [D_\ell]$.
Therefore in particular $D'_{n-k}$ is Cartier.

As an immediate corollary we can  characterise when the 
boundary divisor $\Dac^\lambda$ is Cartier. This recovers the well-known characterisation of which Schubert varieties are Gorenstein, 
see \cite{WooYong,BrownLak:Schubert,Perrin:Schubert, Svanes}.
\begin{cor}\label{c:Gorenstein}
The divisor $\Dac^\lambda=\sum_{\ell=1}^d D_\ell+\sum_{i=1}^{n-1} D_i'$ of $X_\lambda$ has homology class given by
\begin{equation}\label{e:DacFormula}
[\Dac^\lambda]=\sum_{\ell=1}^d n_\ell[D_\ell],
\end{equation}
where 
	\begin{equation}\label{eq:nell}
n_\ell= 1+\#\{\text {boxes $b_i'$  weakly northwest of the removable box $b_{\rho_{2\ell+1}}$}\}.  
	\end{equation}
The Schubert variety $X_\lambda$ is Gorenstein if and only if $n_1=\dotsc=n_d$, or equivalently, if and only if the removable boxes $b_{2\rho_\ell-1}$ all lie on the same anti-diagonal. 
\end{cor} 

Note that $n_\ell$ agrees with the degree given to the quantum parameter $q_\ell$ in \cref{d:LG1}.

\begin{proof} The formula \eqref{e:DacFormula} for $[\Dac^\lambda]$ follows immediately from \cref{p:PositroidHomology}. The  divisor $\Dac^\lambda$ is Cartier if and only if it is linearly equivalent to a multiple of $\sum_{\ell=1}^d D_\ell$, which we can detect from the homology class using that the Class group of $X_\lambda$  equals the homology $H_{2|\lambda|-2}(X_\lambda,\Z)$. 
By \cref{c:ac},
	$\Dac^\lambda$ is an anti-canonical divisor, so  
	we have that $X_\lambda$ is indeed Gorenstein if and only if $n_1=\dotsc=n_d$.
\end{proof}

More generally, we can characterise which divisors supported on the boundary of $X_\lambda$ are anti-canonical and which are Cartier.   
\begin{cor}\label{c:Drr} 
	Let $(\mathbf r,\mathbf r')\in \Z^{d}\times \Z^{n-1}$ and 
\[
D_{(\mathbf r,\mathbf r')}=\sum_{\ell=1}^{d} r_\ell D_\ell+\sum_{i=1}^{n-1} r'_i D'_i.
\]
For any removable box $b_{\rho_{2\ell-1}}$ in $\lambda$ consider  
$NW(b_{\rho_{2\ell-1}}):=\{i\mid \text {$b_i'$ is weakly northwest of the box $b_{\rho_{2\ell-1}}$}\}$. 
Let 
\begin{equation}\label{e:Rj}
R_\ell:=r_\ell+\sum_{i\in NW(b_{\rho_{2\ell-1}})}r'_i. 
\end{equation}
The divisor
	$D_{(\mathbf r,\mathbf r')}$ is an anti-canonical divisor if and only if $R_\ell=n_\ell$ for each $\ell=1,\dotsc, d$, where $n_{\ell}$
	is defined in \eqref{eq:nell}.
The divisor $D_{(\mathbf r,\mathbf r')}$ is Cartier if and only if $R_1=R_2=\dotsc=R_d$. In this case the divisor is linearly equivalent to $R\sum_{\ell=1}^d D_\ell$ where $R=R_\ell$, and $D_{(\mathbf r,\mathbf r')}$ is ample if and only if $R>0$.
\end{cor}

\begin{proof} By the formula in \cref{p:PositroidHomology} we have $[D_{(\mathbf r,\mathbf r')}]=\sum_{\ell=1}^d R_\ell[D_\ell]$ where $R_\ell$ is as in \eqref{e:Rj}. The Corollary follows.
\end{proof}

 Note that the formula in \cref{c:Gorenstein} is the special case of the one in the proof of \cref{c:Drr} where all $r_\ell$ and $r'_i$ have been set equal to $1$.
 
\begin{remark}\label{r:representative} For any fixed choice of $\mathbf r'\in\Z^{n-1}$ there is a unique representative of the form $D_{(\mathbf r,\mathbf r')}$ in each homology class of $H_{2|\lambda|-2}(X_\lambda,\Z)$. Namely, in the class $\sum_\ell m_\ell[ D_\ell]$ this is the divisor $D_{(\mathbf r,\mathbf r')}$ with 
$\mathbf r=(r_\ell)_{\ell=1}^d$ given by $r_\ell=m_\ell-\sum_{i\in NW(b_{\rho_{2\ell+1}})}r'_i$. If $\mathbf r'=0$, this recovers the usual choice of divisor $D_{(\mathbf m,\mathbf 0)}=\sum m_\ell D_\ell$ representing $\sum_\ell m_\ell[ D_\ell]$. If $\mathbf r'=(1,\dotsc,1)$ we obtain the nice representation of the anti-canonical divisor as $D_{(\mathbf 1,\mathbf 1)}$, and a general divisor class $\sum_\ell m_\ell [D_\ell]$ is  then represented by $D_{((m_1-n_1+1,\dotsc, m_d-n_d+1),\mathbf 1)}$. 
\end{remark}

Let us now restrict our attention to Cartier boundary divisors. Consider the line bundle $\mathcal O(R)$ on $X_\lambda$. The divisors of global (meromorphic) sections of $\mathcal O(R)$ with support in the boundary are in bijection with $\Z^{n-1}$ via
\[
\mathbf r'=(r'_1,\dotsc, r'_{n-1})\mapsto D_{(\mathbf r,\mathbf r')},
\] 
with $\mathbf r=(r_1,\dotsc, r_d)$ given by 
\begin{equation}\label{e:Cartier-rviaR}
r_\ell=R-\sum_{i\in NW(b_{\rho_{2\ell-1}})}r'_i. 
\end{equation}
On the other hand, we can also construct Cartier boundary divisors  by using the frozen Pl\"ucker coordinates $P_{\mu_i}$ as in \cref{p:PmuVSPositroids}. We now extend this relationship to describe more general Cartier boundary divisors in terms of the divisors $(P_{\mu_i})$.

Recall our notations for the SW rim and corner boxes of $\lambda$ and the frozen rectangles, see \cref{d:frozens} and \cref{r:R(lambda)}. The following Corollary summarises the relationship between the frozen Pl\"ucker variables and the boundary divisors.  
\cref{c:DrrVSPmu} and 
	\cref{c:RationalPmu}
will be used in the proof of \cref{t:maingen}.

\begin{corollary}\label{c:DrrVSPmu}
Consider the linear map $\varphi:\Z^{n}\to\Z^{n-1}$ defined in terms of standard bases by 
\[
\varphi(e_j):=\sum_{i\in\operatorname{add}(\mu_j)} e_i
\]
The map $\varphi$ is a surjection with kernel spanned by $\sum_{s=1}^{2d-1} (-1)^s e_{\rho_{s}}$. 

If $\varphi(m_1,\dotsc,m_n)=\mathbf r'$, then 
\begin{equation}\label{e:PmuVSPositroids}
\sum_{j=1}^n m_j(P_{\mu_j})=D_{(\mathbf r,\mathbf r')}.
\end{equation}
with $\mathbf r$  given by $r_\ell=(\sum_{j=1}^{n}m_j)-\sum_{i\in NW(b_{\rho_{2\ell-1}})}r'_i$, for all $1\le \ell\le d$
\end{corollary}
In the proof below we also give an explicit construction of a  right inverse to $\varphi$.
\begin{proof}
It is straightforward to see that $\sum_{s=1}^{2d-1} (-1)^s e_{\rho_{s}}$ lies in the kernel of $\varphi$. Moreover it is a primitive vector. Now, given any element $\mathbf r'\in\Z^{n-1}$ we can construct an element $\mathbf m\in\Z^n$ that maps to it explicitly as follows. If $i\notin \mathcal R(\lambda)$, meaning $i$ labels neither an inner or outer corner of the rim, then  $m_i=r'_s$ where $b'_s$ is the unique box that is addable only to $\mu_i$. For example, if $b_i$ is on a horizontal part of the rim, then $b'_s$ is on the N border. If $i=n$ then we set $m_n=r'_{n-k}$ corresponding to $b'_{n-k}$ being addable only to $\mu_n$. For the first outer corner $b_{\rho_1}$, we then choose to let $m_{\rho_{1}}=0$. We then determine the $m_i$ for the remaining $i\in R(\lambda)$ in order. If $i=\rho_{2j}$ corresponding to an inner corner $b_{\rho_{2j}}$, find the addable box to $\mu_{\rho_{2j}}$ that lies along the W border and call it $b'_s$. We then set 
\[
m_{\rho_{2j}}=r'_s-\sum_{\rho_{2j-1}\le i<\rho_{2j}} m_i.
\]
Here, as we are going in order, the $m_i$ in the sum have already been expressed in terms of $\mathbf r'$.
Similarly if $i=\rho_{2j+1}$ corresponding to an outer corner $b_{\rho_{2j+1}}$, we find the addable box to $\mu_{\rho_{2j+1}}$ that lies along the N border and call it $b'_s$. Then we set
\[
m_{\rho_{2j+1}}=r'_k-\sum_{\rho_{2j}\le p<\rho_{2j+1}}m_p.
\]
This recursion constructs an $\mathbf m=(m_i)_i$ such that $\varphi(\mathbf m)=\mathbf r'$. It follows that $\varphi$ is surjective. Therefore also its kernel must have rank $1$, and be generated by its primitive element $\sum_{s=1}^{2d-1} (-1)^s e_{\rho_{s}}$.

Finally, $\sum m_j(P_{\mu_j})$ is the divisor of the meromorphic section $\prod P_{\mu_j}^{m_j}$ of  $\mathcal O(R)$ where $R=\sum_{j=1}^n m_j$. The identity \eqref{e:PmuVSPositroids} follows from \cref{c:Drr} and \cref{p:PmuVSPositroids}, see also \eqref{e:Cartier-rviaR}.
\end{proof}

\begin{corollary}\label{c:RationalPmu}
Consider the linear map $\tilde\varphi$ defined by restriction of $\varphi$ to $M=\{\mathbf m\in \Z^{n}\mid \sum_{j=1}^n m_j=0\}$. 
The map $\tilde\varphi$ is an isomorphism. If $\tilde\varphi(\mathbf m)=\mathbf r'$ then the rational function $f=\prod_{j=1}^n P_{\mu_j}^{m_j}$ has divisor $(f)=D_{(\mathbf r,\mathbf r')}$ with $\mathbf r$ given by $r_\ell=-\sum_{i\in NW(b_{\rho_{2\ell-1}})}r'_i$.
\end{corollary}

\begin{proof} For any $\mathbf r'\in\Z^{n-1}$ there is an $\mathbf m\in\Z^n$ with $\varphi(\mathbf m)=\mathbf r'$, by \cref{c:DrrVSPmu}. Let $m:=\sum_{j=1}^n m_j$ and consider $\mathbf k:=\sum_{k=1}^{2d-1} (-1)^k e_{\rho_{k}}$, the generator of the kernel of $\varphi$. The coordinates of $\mathbf k$ sum to $-1$ since there is one more outer corner than there are inner corners. It follows that $\mathbf m + m \mathbf k$ lies in $M$ and is the unique preimage of $\mathbf r'$ in $M$. Therefore $\tilde\varphi$ is a bijection. The remainder is a restatement of \cref{c:DrrVSPmu} for the case where $\sum_{j=1}^{n} m_j=0$. 
\end{proof}

\section{The Newton-Okounkov body of a Schubert variety}\label{sec:NO}

In this section we define the {\it Newton-Okounkov  body} $\NO_{G}^{\lambda}(D)$
 associated to an ample divisor in $X_{\lambda}$, 
 along with a choice of transcendence basis $\TBG$ of $\C(X_{\lambda})$, 
see \cref{Xseed}.
 The theory of Newton-Okounkov bodies was  developed  in 
\cite{KavehKhovanskii,KK2, LazarsfeldMustata, Anderson},  building on 
\cite{Ok96, Okounkov:symplectic,Ok03}. 
 A key property of a Newton-Okounkov  body associated to a divisor $D$ is that its Euclidean volume encodes the volume of $D$, i.e.\ the asymptotics of $\dim(H^0(\X,\mathcal O(rD)))$ as $r\to\infty$.

Fix a  labeled $\mathcal{X}$-seed $\Sigma^{\mathcal X}_G$  for $X_{\lambda}$.
To define the Newton-Okounkov  body $\NO^{\lambda}_G(D)$  we first construct a valuation $\val_G$ on $\C(X_{\lambda})$ from the transcendence basis $\TBG$.

\begin{defn}[The valuation $\val_G$] \label{de:val}
Given a general 
	$\mathcal{X}$-seed  $\Sigma^{\mathcal X}_G$ for $X_{\lambda}$,
we fix a total order $<$ on the parameters $x_{\mu} \in \TBG$,
where $\mathcal{P}_G$ is the index set for the parameters.
	This order extends
 to a term order on monomials in the parameters $\TBG$
which is lexicographic with respect to $<$. For example if $x_\mu<x_\nu$ then $x_{\mu}^{a_1} x_{\nu}^{a_2}<x_{\mu}^{b_1} x_{\nu}^{b_2}
$ if either $a_1<b_1$, or if $a_1=b_1$ and $a_2<b_2$.
We use the multidegree of the lowest degree summand to define a valuation
\begin{equation} \label{eq:valuation}
\val_{G}:\C(X_{\lambda})\setminus\{0\}\to \Z^{{\mathcal P}_G}.
\end{equation}
Explicitly, let $f$ be  a polynomial in the  Pl\"ucker coordinates
for $X_{\lambda}$. We use  \cref{thm:Laurent}
to write $f$ uniquely as a Laurent polynomial in $\TBG$.
We then choose the  lexicographically minimal term $\prod_{\mu\in{\mathcal P}_G}x_\mu^{a_\mu}$ and define $\val_G(f)$ to be the associated exponent vector $(a_\mu)_\mu\in \Z^{{\mathcal P}_G}$.
In general for
$(f/g) \in \C(X_{\lambda}) \setminus\{0\}$ (here $f,g$ are polynomials
in the Pl\"ucker coordinates), the valuation is defined by
$\val_G(f/g) = \val_G(f) - \val_G(g)$.
Note however
        that we will only  be applying $\val_G$ to functions whose $\mathcal X$-cluster expansions are
        Laurent.
\end{defn}

\begin{defn}[The Newton-Okounkov body $\NO^{\lambda}_G(D)$]\label{def:NObody}
Let $D\subset X_{\lambda}$ be a divisor in the complement of 
$\openX_{\lambda}$, that is we have $D=D_{(\mathbf{r,r'})}=\sum_{\ell=1}^d r_\ell D_\ell
+\sum_{i=1}^{n-1} r'_i D'_i$. 
Denote by $L_{rD}$, the subspace of $\C(X_{\lambda})$ given by
\[
L_{rD}:=H^0(X_{\lambda},\mathcal O(rD)).
\]
By abuse of notation we write $\val_G(L)$ for $\val_G(L \setminus \{0\})$. We define the {\it Newton-Okounkov  body} associated to $\val_G$ and the divisor $D$ by
\begin{equation}\label{e:NOviaNOGr}
\NO^{\lambda}_G(D)=
\overline{\operatorname{ConvexHull}
        \left(\bigcup_{r=1}^{\infty} \frac{1}{r}
\val_G(L_{rD})\right)}.
\end{equation}
\end{defn}

The following result will be useful for the proof of \cref{thm:rectanglesproof}.
\begin{theorem}\cite[Corollary 3.2]{KavehKhovanskii}\label{t:KK}
The dimension 
$\dim \Delta_G^{\lambda}(D)$
of the Newton-Okounkov body 
equals the dimension $|\lambda|$ of $X_{\lambda}$,
and 
the volume
$\Vol(\Delta_G^{\lambda}(D))$ of the Newton-Okounkov body 
equals
$\frac{1}{|\lambda|!}$ times 
the degree of $X_{\lambda}$ in its Pl\"ucker embedding.
\end{theorem}

\begin{rem}[Preferred divisor $D$]\label{rem:special}
There are two interesting choices for $D$ in \cref{def:NObody}, namely $D=D'_{n-k}= \{P_{\emptyset}=0\}$, and  
	$D=D_1+\dotsc +D_{d}= \{P_{\lambda}=0\}$, which are linearly equivalent and correspond to the Pl\"ucker embedding. 
	\emph{Our preferred
	choice will be $D=D_{(\mathbf{1,0})}=D_1+\dotsc +D_{d}$.} This divisor equals to $\{P_\lambda=0\}$ and is the natural generalisation of the divisor used in \cite{RW}. Let us now fix $D=D_{(\mathbf{1,0})}$ and set
	\[
L_1 :=H^0(X_{\lambda}, \mathcal{O}(D)) = \left < \frac{P_{\mu}}{P_{\lambda}} \ \vert \ \mu \subseteq \lambda \right >.
\]

Recall that the ample line bundles on $X_\lambda$ all arise by restriction from ample line bundles on the Grassmannian. Combined with \cite[Theorem 3]{RamananRamanathan} it follows that any projective embedding of the Schubert variety $X_\lambda$ is projectively normal.
	In 
	the setting of the Pl\"ucker embedding we therefore have that
 $H^0(X_{\lambda},\mathcal O(rD))$ 
is the linear subspace of $\C(X_{\lambda})$ described as follows
\begin{equation}\label{e:projnormal}
L_r:=L_{r,\, D}=\left<\frac{M}
{(P_{\lambda})^r} \ \vert \
M\in\mathcal{M}_{r}%
\right>,
\end{equation}
where $\mathcal{M}_r$ is the set of all degree $r$ monomials in the
	Pl\"ucker coordinates of $X_{\lambda}$,
	see also \cite{Lakshmibai}. 
	
We will refer to $\NO^{\lambda}_G(D)$ simply as $\NO^{\lambda}_G$ when  the choice $D=D_1+\dotsc + D_d$ is made.
\end{rem}

\begin{remark}
For simplicity of notation we will usually write $\val_G(M)$ 
for $\val_G(M/P_\lambda^r)$ in the setting of \cref{rem:special}. Thus we may write $\val_G(P_\mu)$ instead of $\val_G(P_\mu/P_\lambda)$ and talk about the valuation of a Pl\"ucker coordinate.	
\end{remark}

We consider $\Delta^\lambda_G$ to be our fundamental Newton-Okounkov polytope for the Schubert variety $X_\lambda$ with choice of cluster $G$. For a general ample boundary divisor $D=D_{(\mathbf{r,r'})}$, the Newton-Okounkov polytope $\Delta_G^\lambda(D_{(\mathbf{r,r'})})$ is obtained from $\Delta^\lambda_G$ by dilation and translation. 

\begin{lemma}\label{l:DeltaShift} Let $R\in\Z_{>0}$, and suppose  $D=D_{(\mathbf{r,r'})}$ is linearly equivalent to $R D_{(\mathbf{1,0})}$.  Then $\Delta_G^\lambda(D_{(\mathbf{r,r'})})$ is a translate of the dilation $R\Delta_G^\lambda$ of $\Delta_G^\lambda$.
\end{lemma}

\begin{proof} This is a straightforward consequence of our choice of conventions. Since $D_{(\mathbf{r,r'})}$ is linearly equivalent to $R D_{(\mathbf{1,0})}$ we have a rational function $f$ on $X_\lambda$ with divisor $(f)=D_{(\mathbf{r,r'})}-RD_{(\mathbf{1,0})}$ and for any $m\in\Z_{>0}$ an isomorphism
\[
\begin{array}{ccc}
H^0(X_\lambda,\mathcal O(mRD_{(\mathbf{1,0})}))&\longrightarrow &H^0(X_\lambda,\mathcal O(mD_{(\mathbf{r,r'})})).\\
M&\mapsto & f^{-m} M
\end{array}
\]
The image of $H^0(X_\lambda,\mathcal O(mD_{(\mathbf{r,r'})}))\setminus\{0\}$ under $\frac{1}m\val_G$ is therefore the translate by $-\val_G(f)$ of the image of $H^0(X_\lambda,\mathcal O(mRD_{(\mathbf{1,0})}))\setminus\{0\}$, since  
\[\frac{1}m\val_G(f^{-m}M)=-\val_G(f)+\frac{1}{m}\val_G(M).
\]
This implies that 
$\Delta_G^\lambda(D_{(\mathbf{r,r'})})=\Delta_G(RD_{(\mathbf{1,0})})-\val_G(f)=R\Delta_G^\lambda - \val_G(f)$.
\end{proof}

Starting from the divisor $D=D_1+\dotsc+D_d$ 
we now introduce a set of lattice polytopes $\conv_G^{\lambda}(r)$ related to $\Delta^\lambda_G$.

\begin{defn}[The polytope $\conv_G^{\lambda}(r)$]
\label{d:NOrG}  For each 
$\mathcal{X}$-seed  $\Sigma^{\mathcal X}_G$ for $X_{\lambda}$
 and related valuation $\val_G$,
we define lattice polytopes $\conv_G^{\lambda}(r)$ 
in $\R^{\mathcal P_{G}}$ by
\[
	\conv_G^{\lambda}(r):=\operatorname{ConvexHull}\left(\val_G(L_r)\right),
\]
for $r\in\Z_{>0}$ and $L_r$ as in \eqref{e:projnormal}.
	When $r=1$, we also write $\conv_G^{\lambda}:=\conv_G^{\lambda}(1)$.
\end{defn}

The lattice polytope $\conv_G^{\lambda}$ (resp. $\conv_G^{\lambda}(r)$)
is what $\val_G$ associates to the 
divisor $D=D_1+\dotsc +D_d$ (resp. $rD$) directly, 
without taking account of asymptotic behaviour.
Since we have effectively fixed $D$ to be the divisor $D_1+\dotsc +D_d$
when constructing the polytopes $\Conv_G^{\lambda}(r)$, 
we don't indicate dependence on a divisor
in the notation $\Conv_G^{\lambda}(r)$.

\begin{remark}
Note that we used a total order $<$ on the parameters in order to
define $\val_G$, and different choices give slightly differing valuation maps. However
$\NO^{\lambda}_G$ and the polytopes ${\conv}_G^{\lambda}(r)$, 
will turn out not to depend on our choice of total order, and that choice will not enter into our proofs.
\end{remark}

\begin{lem}[Version of {\cite[Lemma from Section 2.2]{Ok96}}]
\label{l:okounkovlemma}
Consider $\C(X_{\lambda})$ with the valuation $\val_G$ from \cref{de:val}.  For  any finite-dimensional linear subspace $L$ of $\C(X_{\lambda})$,
 the cardinality of the image $\val_G(L)$ equals the dimension of $L$. In particular, the cardinality of the set $\val_G(L_r)$ equals
 the dimension of the vector space $L_r$ from \eqref{e:projnormal}.
\end{lem}

\begin{example}\label{ex:NO}
We now take $r=1$ and $\lambda=(4,4,2)$,
and compute some vertices of the polytope ${\conv}^{\lambda}_G$ 
associated to the $\mathcal{X}$-cluster
chart $G=G^{\rect}_{\lambda}$
from  \cref{SchubertNetwork}, by 
computing the valuations
the nonzero  Pl\"ucker coordinates.  We get the lattice points
shown in Table~\ref{table:Dynkin}. 
In fact, the lattice points $\val_G(P_{ijk})$ are all distinct, and thus \cref{l:okounkovlemma} implies that we obtain from them the entire image $\val_G(L_1)$. As a consequence, ${\conv}^{\lambda}_G$ is the convex hull of these points. 
It will follow from \cref{thm:rectanglesproof}
that in this example,
$\conv^{\lambda}_G = \Delta^{\lambda}_G$.
\end{example}

\begin{center}
    \begin{table}[h]
	    \begin{tabular}{| c || p{.7cm} | p{.7cm} | p{.7cm} | p{.7cm} | p{.7cm} | p{.7cm} | p{.7cm} | p{.7cm} | p{.7cm} | p{.7cm} |}
            \hline
	     Pl\"ucker & $\ydiagram{1}$ & $\ydiagram{2}$ & $\ydiagram{3}$ & $\ydiagram{4}$ & $\ydiagram{1,1}$
		    & $\ydiagram{2,2}$ & $\ydiagram{3,3}$ & $\ydiagram{4,4}$ & $\ydiagram{1,1,1}$ & $\ydiagram{2,2,2}$ \\
                \hline
                \hline
		    $\val_G(P_{125})$ & $0$ & $0$ & $0$ & $0$ & $0$ & $0$ & $0$ & $0$ &  $0$ & $0$\\ \hline
		    $\val_G(P_{126})$ & $0$ & $0$ & $0$ & $0$ & $0$ & $0$ & $0$ & $0$ &  $0$ & $1$\\ \hline
		    $\val_G(P_{127})$ & $0$ & $0$ & $0$ & $0$ & $0$ & $0$ & $0$ & $0$ &  $1$ & $1$\\ \hline
		    $\val_G(P_{135})$ & $0$ & $0$ & $0$ & $0$ & $0$ & $0$ & $0$ & $1$  &  $0$ & $0$\\ \hline
		    $\cdots$ &  &  &  &  &  &  &  &  &   & \\ \hline
		    $\val_G(P_{467})$ & $0$ & $1$ & $1$ & $1$ & $1$ & $1$ & $2$ &$2$ &  $1$ & $2$\\ \hline
		    $\val_G(P_{567})$ & $1$ & $1$ & $1$ & $1$ & $1$ & $2$ & $2$ &$2$ &  $1$ & $2$\\ \hline
                \end{tabular}
\vspace{0.2cm}
 \caption{The valuations $\val_G(P_J)$
	    of some of the flow polynomials one obtains from 
the network chart 
	    from  \cref{SchubertNetwork}.}
                \label{table:Dynkin}
        \end{table}
\end{center}

For another example, see \cref{sec:example}.

%\Comment{Note that in previous paper we used the divisor $P_{\max} = 0$
%but here $P_{\max} = 0$ is not an irreducible divisor.
%Because for the Schubert variety $X_{\lambda}$, $\max=\lambda$ and hence
%it appears on the LHS of some Pl\"ucker relation in general.  Whereas
%$P_{\emptyset}$ never does.}

\section{Polytopes via tropicalization and the superpotential polytope}\label{sec:trop}

In this section we briefly review how, given an 
$\mathcal{A}$-cluster variety $\mathbb X$, together with 
a universally positive Laurent polynomial 
	$h\in\C[\mathbb{X}]$ 
 and a choice of cluster $\mathcal{C}$, one 
can construct a polyhedron.  Moreover when one applies a mutation to the cluster $\mathcal{C}$,
this polyhedron transforms via tropicalized cluster mutation.  
We will then use this construction to associate a polytope to the superpotential for Schubert 
varieties.

\subsection{Polyhedra associated to positive Laurent polynomials}

\begin{definition}[Universally positive]\label{d:universallypositive}
We say that a Laurent polynomial is \emph{positive} if all of its coefficients are in $\R_{>0}$. 
	If $\mathbb{X}$ is an $\mathcal A$-cluster variety, we say that an element 
	$h\in\C[\mathbb{X}]$ 
	is  {\it universally positive}  (for the $\mathcal A$-cluster structure) if for every $\mathcal A$-cluster,
	 the expansion $\mathbf h^G$ of $h$ in that cluster is a positive Laurent polynomial. Similarly 
	$f\in\C[\mathbb{X}\times\C_{q_1}^* \times \dots \times \C_{q_d}^*]$ is called universally positive if for every cluster, its expansion $\mathbf f^G$ in that cluster adjoin $\{q_1,\dots,q_d\}$ 
	is given by a positive Laurent polynomial. 
\end{definition}

\begin{definition}[naive Tropicalisation]\label{def:Tropicalisation}
To any Laurent polynomial $\mathbf h$ in variables $X_1,\dotsc, X_N$ with coefficients in $\R_{>0}$ we associate a piecewise linear map $\Trop(\mathbf{h}):\R^N \to \R$
 called the {\it tropicalisation} of $\mathbf h$ as follows. We set $\Trop(X_i)(y_1,\dotsc, y_N)=y_i$. If $\mathbf{h_1}$ and $\mathbf{h_2}$     are two positive Laurent polynomials, and $a_1,a_2\in\R_{>0}$, then we impose the condition that
\begin{equation}\label{eq:min-etc}
\Trop(a_1\mathbf{h_1}+a_2\mathbf{h_2}) = 
\min(\Trop(\mathbf{h_1}),\Trop(\mathbf{h_2})),\text{ and }
\Trop(\mathbf{h_1}\mathbf{h_2}) = 
\Trop(\mathbf{h_1}) + \Trop(\mathbf{h_2}).
\end{equation}
This defines $\Trop(\mathbf h)$ for all positive Laurent polynomials $\mathbf h$, by induction.
\end{definition}

\begin{definition}\label{def:troppoly}
Let $\hh_1,\dots,\hh_m$ be positive Laurent polynomials in variables $X_1,\dots,X_N$, 
let $\hh=\hh_1 + \dots + \hh_m$, and choose $\rr=(r_1,\dots,r_m)\in \Z^m$.
We define an associated polyhedron 
	$\Q^\hh(\rr) \subset \R^N$ by 
	the following  inequalities in terms of variables $v = (v_1,\dots,v_N)$:
	$\Trop(\hh_i)(v) + r_i \geq 0$ for all $1 \leq i \leq m$.
\end{definition}

We are most interested in the case where 
	 $\mathbb{X}$ is an $\mathcal A$-cluster variety, and $h\in\C[\mathbb{X}]$ is a sum
	 of {\it universally positive} Laurent polynomials.

\begin{definition}\label{def:troppoly2}
 Let $\mathbb{X}$ be an $\mathcal A$-cluster variety of dimension $N$, 
 and $h\in \C[\mathbb{X}]$ an element which is given as a sum  $h_1+ \dots + h_m$,
 where each $h_i$ is a 
	 {\it universally positive} Laurent polynomial.  If $G$ indexes a cluster seed for $\mathbb{X}$,
	 we let $\hh^G$ and $\hh_1^G,\dots, \hh_m^G$ denote the Laurent polynomials
	 obtained by expressing $h$ and $h_1,\dots, h_m$ in the variables of that cluster.
	 Given positive integers $\rr = (r_1,\dots, r_m)$, 
	we then let $\Q^h_G(\rr) \subset \R^N$ be the polyhedron defined  in terms of variables $v = (v_1,\dots,v_N)$ by 
	the following  inequalities:
	$\Trop(\hh_i^G)(v) + r_i \geq 0$ for all $1 \leq i \leq m$.
        In other words, we have 
	\begin{equation}
		\Q^h_G(\rr) = \bigcap_i \{v\in \R^{N} \ \vert \ \Trop(\hh_i^G)(v) + r_i \geq 0\}.
        \end{equation}
\end{definition}

\begin{defn}[Tropicalized $\mathcal A$-cluster mutation]\label{d:PsiGG'}
Let $\mathbb{X}$ be an $\mathcal A$-cluster variety of dimension $N$.  
Suppose $\Sigma^{\mathcal A}_G$ and $\Sigma^{\mathcal A}_{G'}$ are general $\mathcal{A}$-cluster 
seeds for $\mathbb{X}$, with quivers $Q(G)$ and $Q(G')$,
which are related by a single mutation at a vertex $\nu_i$. 
Let the cluster variables for $\Sigma^{\mathcal A}_G$ be  indexed by
$\mathcal P_G=\{\nu_1,\dots, \nu_N\}$. 
  We define a map 
 $\Psi_{G,G'}:\R^{\mathcal P_G}\to \R^{\mathcal P_{G'}}$ by
	$(v_{\nu_1},v_{\nu_2}, \dots,v_{\nu_N}) \mapsto
(v_{\nu_1}, \dots,v_{\nu_{i-1}}, v_{\nu_i'},v_{\nu_{i+1}},\dotsc, v_{\nu_N})$, where
\begin{equation}\label{e:tropmut}
v_{\nu'_i} = \min(\sum_{\nu_j \to \nu_i} v_{\nu_j}, \sum_{\nu_i \to \nu_j} v_{\nu_j})\, - \, v_{\nu_i}, 
\end{equation}
and the sums are over arrows in the quiver $Q(G)$ pointing towards $\nu_i$ or away from $\nu_i$, respectively. We call $\Psi_{G,G'}$ a \emph {tropicalized $\mathcal A$-cluster mutation}.
\end{defn}

The following result from \cite{RW} was stated for the polytope associated to the superpotential for the Grassmannian.  However
 the statement and proof hold for general universally positive elements $h=h_1+ \dots + h_m$ in a cluster algebra.

\begin{proposition}[{\cite[Corollary 11.16]{RW}}]\label{prop:tropmutation}\
We use the notation of \cref{def:troppoly2}.
If the cluster seed $\Sigma^{\mathcal A}_{G'}$ is related to $\Sigma^{\mathcal A}_G$ by a single mutation 
at vertex $\nu$, then the  {tropicalized $\mathcal{A}$-cluster mutation} $\Psi_{G,G'}$ restricts to a bijection
\[
\Psi_{G,G'}: \Q^h_G(\rr)\to \Q^h_{G'}(\rr).
\]\qed
\end{proposition}

\subsection{The superpotential polytope}

In this section we
introduce one of the main polytopes of this paper: the superpotential polytope.

\begin{proposition}\label{prop:univpos}
The superpotential $W^{\lambda}$ from \cref{d:LG1} is a universally positive element
of the cluster algebra in the sense of \cref{def:rectseed}.
\end{proposition}

\begin{proof}
Clearly the denominators which appear in \cref{d:LG1} are all frozen variables
for the cluster algebra.  So it suffices to show that every Pl\"ucker coordinate
appearing in the numerator of each $W_i$ and $W'_i$ in \eqref{e:Wq1}
is a cluster variable.\footnote{In the case of the Grassmannian, every 
Pl\"ucker coordinate is a cluster variable, but this is not true in the 
case of Schubert varieties, so we do need to prove that the Pl\"ucker coordinates
appearing in \eqref{e:Wq1} are cluster variables!}
If we can do this,
then by the positivity of the Laurent phenomenon \cite{LeeSchiffler, GHKK}, 
each $W_i$ and $W'_i$ is an example of a universally positive element of $\C[\openSchub]$.

We use the initial seed 
	$\Sigma_{\lambda}^{\rect}$
	for the $\mathcal{A}$-cluster structure 
on 
	$\dualSchub^{\circ}$
	from \cref{thm:rectangles}.
Note that the expression for the superpotential in \cref{prop:super2}
is a Laurent polynomial in Pl\"ucker coordinates indexed by 
rectangles; so it is a Laurent polynomial in the cluster
variables from the seed of \cref{thm:rectangles}.

Now we use \cref{lem:Plucker}, which shows how to combine
the summands from \cref{prop:super2} to obtain the 
Laurent monomials from \eqref{e:Wq1} (or equivalently
\eqref{e:Wq2}) in the definition  of the 
superpotential.  We claim that each 
three-term Pl\"ucker relation \eqref{eq:Plucker1} is
in fact a cluster relation.
Indeed, if we start from the initial seed
	$\Sigma_{\lambda}^{\rect}$
	and mutate at  $p_{i \times 1}$, we get
	$$p_{i \times 1} p_{_{\square}(i \times 2)} = 
	p_{(i-1)\times 1} p_{(i+1)\times 2} + 
	p_{i \times 2} p_{(i+1)\times 1},$$
	which is equivalent to \eqref{eq:Plucker1} for $m=1$.  
	If we continue by mutating at 
	$p_{i \times 2}$, then $p_{i \times 3}$, etc then 
	the cluster relations we obtain will be precisely
	the relations from \eqref{eq:Plucker1}.  This shows
	by induction that each Pl\"ucker coordinate 
	$p_{_{\square}(i \times m)}$ is a cluster variable.
	A similar argument shows that 
	$p_{(\ell \times j)^{\square}}$ is a cluster variable.
\end{proof}

\begin{definition}\label{d:B-modelPolytope}
Let $\lambda$ be a Young diagram contained in an $(n-k)\times k$ rectangle,
with $d$ the number of removable boxes in $\lambda$.
The Schubert variety  
	$\dualSchub$ has dimension $N$, where 
$N$ is the number of boxes of $\lambda$, and each cluster
for 
	$\dualSchub^{\circ}$
	contains $N+1$ cluster variables,
including $p_{\emptyset}=1$ (recall \cref{emptyis1}).
Recall that the superpotential $W^{\lambda}$ for 
	$\dualSchub^{\circ}$
has $d+(n-1)$ summands, i.e. 
$W^{\lambda} = W_1 + \dots + W_d + W'_1 + \dots + W'_{n-1}$, and 
let $\rr=(r_1,\dots,r_d)\in \mathbb{N}^d$
and $\rr'=(r'_1,\dots,r'_{n-1})\in \mathbb{N}^{n-1}$.
Given a seed $\Sigma_G^{\mathcal{A}}$ for 
	$\dualSchub^{\circ}$
we use \cref{def:troppoly} to define the \emph{superpotential polytope}
\[
	\Q_G^\lambda(\rr,\rr'):= \Q_G^{W^{\lambda}}(\rr, \rr').
\]
	
Concretely, if we let the cluster variables (besides $p_{\emptyset}=1$)
for $\Sigma_G^{\mathcal{A}}$ be indexed by 
$\{\nu_1,\dots,\nu_N\}$, then 
$\Q_G^\lambda(\rr,\rr')$ is the polyhedron defined by 
the following  inequalities in terms of variables $v = (v_{\nu_1},\dots,v_{\nu_N})$:
	\begin{align}
		\Trop(\mathbf{W}_\ell^G)(v) + r_\ell & \geq 0 \text{ for all } 1 \leq \ell \leq d, \\
		\Trop({{\mathbf{W}_i}'}^G)(v) + r'_i & \geq 0 \text{ for all }1 \leq i \leq n-1. 
	\end{align}
	We also let
	\begin{equation} \label{eq:mainsuperpolytope}
		\Q_G^{\lambda}:=\Q_G^{\lambda}(\mathbf{1},\mathbf{0})
	\end{equation}
	denote the superpotential polytope in the case that 
	$r_1 = \dots = r_d=1$, and $r'_1 = \dots = r'_{n-1}=0$.
\end{definition}

\begin{example}\label{ex:21}
Let $\lambda=(2,1)$
and let $G = G_{\rect}$.
The superpotential is
\begin{equation}\label{super:21}
W_{\rect}^{\lambda} = q_1 \frac{p_{\emptyset}}{p_{\ydiagram{2}}}
+q_2 \frac{p_{\emptyset}}{p_{\ydiagram{1,1}}} + \frac{p_{\ydiagram{2}}}{p_{\ydiagram{1}}} + 
\frac{p_{\ydiagram{1,1}}}{p_{\ydiagram{1}}}+
\frac{p_{\ydiagram{1}}}{p_{\emptyset}},
\end{equation}
where $p_{\emptyset}=1$ (recall \cref{emptyis1}).
We obtain the superpotential polytope 
	$\Gamma_{\rect}^{\lambda} :=
	\Gamma_{G_{\rect}}^{\lambda}$
from \eqref{eq:mainsuperpolytope}
by 
tropicalizing.
That is, $\Gamma_{\rect}^{\lambda}$ is cut out by the inequalities
$$1 -y_{\ydiagram{2}} \geq 0, \hspace{.4cm}
1-y_{\ydiagram{1,1}} \geq 0, \hspace{.4cm}
y_{\ydiagram{2}}-y_{\ydiagram{1}} \geq 0, \hspace{.4cm}
y_{\ydiagram{1,1}}-y_{\ydiagram{1}} \geq 0, \hspace{.4cm}
y_{\ydiagram{1}} \geq 0.$$
\end{example}

Applying \cref{prop:tropmutation} to the superpotential polytope, we obtain the following result.
\begin{corollary} \label{c:GammaMutation}
If  $\Sigma^{\mathcal A}_{G'}$ is related to $\check\Sigma^{\mathcal A}_G$ by a cluster mutation 
	at vertex $\nu$, then we have that the {tropicalized $\mathcal{A}$-cluster mutation} $\Psi_{G,G'}$ restricts to a bijection
\[
\Psi_{G,G'}: 
	\Q_G^\lambda(\rr,\rr') \to 
	\Q_{G'}^\lambda(\rr,\rr'). 
\]
\end{corollary}

\subsection{Balanced tropical points}\label{s:balanced} 
The results of this subsection constitute a technical tool that 
will be used in \cref{s:GenD} (the reader may feel free to skip it on a first reading
of the paper).
In particular, we observe here 
that while the tropicalised mutation map $\Psi_{G,G'}$ from \cref{d:PsiGG'} is  piecewise-linear, there is a special subspace of $\R^{\mathcal P_G}$ where $\Psi_{G,G'}$ is linear in a strong way. 

\begin{defn}[{\cite[Definition~15.8]{RW}}]
Let $v\in \R^{\mathcal P_G}$ with $v=(v_{\nu_1},v_{\nu_2}, \dots,v_{\nu_N})$, then $v$ is called {\it balanced at $\nu_i$} if and only if for the coordinate $v_{\nu_i}$ we have
\[
\sum_{j;\nu_j \to \nu_i} v_{\nu_j}=\sum_{j;\nu_i \to \nu_j} v_{\nu_j},
\]
where we are using notation from \cref{d:PsiGG'}. We say that $v$ is \textit{balanced} 
if $v$ is balanced at $\nu_i$ for every~$i$. 
The property of being balanced is invariant under mutation, see \cite[Proposition~4.8]{Akhtar:PolygonalQuivers}.
\end{defn}  

\begin{lemma}\label{l:balanced}
Suppose $v,w\in\R^{\mathcal P_G}$ and $v$ is balanced. Then
\[
\Psi_{G,G'}(v+w)=\Psi_{G,G'}(v)+\Psi_{G,G'}(w).
\]
\end{lemma}

\begin{proof} Suppose $v=(v_{\nu_1},v_{\nu_2}, \dots,v_{\nu_N})$ and $w=(w_{\nu_1},w_{\nu_2}, \dots,w_{\nu_N})$. Let us assume that $G'$ is obtained from $G$ by mutation at a single vertex $\nu_i$. Then we only need to check the $\nu'_i$ coordinate of $\Psi_{G,G'}(v+w)$ agrees with the $\nu'_i$ coordinate of $\Psi_{G,G'}(v)+\Psi_{G,G'}(w)$. We have
\[
\Psi_{G,G'}(v+w)_{\nu'_i}=\min\left(\sum_{\nu_j \to \nu_i} (v_{\nu_j}+w_{\nu_j}), \sum_{\nu_i \to \nu_j} (v_{\nu_j}+w_{\nu_j})\right)\, - \, v_{\nu_i} - w_{\nu_i}.
\]
Since $v$ is balanced at $\nu_i$, we can replace  $\sum_{\nu_i \to \nu_j} v_{\nu_j}$ by $\sum_{\nu_j \to \nu_i} v_{\nu_j}$ and rewrite the right-hand side to get 
\[
\min(\sum_{\nu_j \to \nu_i}w_{\nu_j}, \sum_{\nu_i \to \nu_j} w_{\nu_j})\, +(\sum_{\nu_i \to \nu_j} v_{\nu_j}) - \, v_{\nu_i} - w_{\nu_i}=\Psi_{G,G'}(w)\, +(\sum_{\nu_i \to \nu_j} v_{\nu_j}) - \, v_{\nu_i}=\Psi_{G,G'}(w)_{\nu'_i}+\Psi_{G,G'}(v)_{\nu'_i}.
\]
Any mutation $\Psi_{G,G'}$ is obtained by repeated application of such mutations at different vertices $\nu_i$. The lemma follows.  
\end{proof}

We also have the following $\mathcal X$-cluster interpretation of balanced elements. This is \cite[Proposition~15.9]{RW} applied to our setting.

\begin{lemma}\label{l:monomialXmutation}
Let $P^G=\prod_{\nu\in\mathcal P_G} x_{\nu}^{v_{\nu}}$ 
	be a monomial in the network parameters $\TB(G)=\{x_{\nu}\mid\nu\in\mathcal P_G\}$. Consider the $\mathcal X$-mutation of $P^G$ at a vertex $\nu$ and call it $P^{G'}$. 

The mutation $P^{G'}$ is again a monomial if and only if the exponent vector $v$ is balanced at $\nu$. Moreover, in that case it is the monomial in $\TB(G')=\{x'_{\eta}\mid\eta\in\mathcal P_{G'}\}$ with exponent vector $v'$ given by 
\begin{equation}\label{e:balancedMutation}
v'_\eta=
\begin{cases} 
(\sum_{\mu \to \nu} v_\mu) - v_\nu, & \eta=\nu,\\
v_\eta, & \eta \ne \nu.
\end{cases}
\end{equation}
Note that since $v$ was balanced, this is an instance of tropicalied $\mathcal{A}$-cluster mutation. \qed
\end{lemma}

\begin{corollary}\label{c:balanced} Suppose 
$P\in\C[X_\lambda^\circ]$ is a regular function 
	which doesn't vanish
	on $X_{\lambda}^\circ$, and consider its valuation $\val_G(P)$ associated to an $\mathcal X$-cluster torus with coordinates $\TB(G)$. Then 
for any $v\in\R^{\mathcal P_G}$ we have
\[
\Psi_{G,G'}(v+\val_G(P))=\Psi_{G,G'}(v)+\Psi_{G,G'}(\val_G(P)).
\]
\end{corollary}
\begin{proof}
Since $P$ is a regular function on $X_\lambda^\circ$ it must 
be a Laurent polynomial
in $\mathcal X$-cluster coordinates, and since it is nonvanishing, it must
be a single Laurent monomial. 
Therefore it is given by a monomial (with some scalar coefficient) in terms of the coordinates $\TB(G)$ of the cluster torus associated to $G$, and the same thing holds for any cluster torus obtained from this one by mutation. By \cref{l:monomialXmutation} it follows that $\val_G(P)$ is balanced. The statement of the corollary now follows from~\cref{l:balanced}.
\end{proof}

\section{The superpotential polytope for the rectangles seed and order polytopes}\label{sec:rectangles}

When $G = G_{\lambda}^{\rect}$ is the rectangles seed, the superpotential polytope
has a particularly explicit description.  In fact after performing a unimodular 
change of variables (so that we obtain the ``superpotential polytope in vertex coordinates'')
it becomes an \emph{order polytope}.

\subsection{The superpotential polytope for the rectangles seed}
When $G = G_{\lambda}^{\rect}$ is the rectangles seed,  we can use \cref{prop:super2}
to obtain the following inequality description of
$\Gamma^{\lambda}_G(\mathbf{r},\mathbf{r'}) = 
\Gamma^{\lambda}_{\rect}(\mathbf{r},\mathbf{r'})$.
Note that this description can also be easily read off the quiver from
\cref{fig:superpotential}.
\begin{lemma}\label{l:rectangles}
	Let  $\mathbf{r}=(r_1,\dots,r_d)$ and $\mathbf{r'}=(r'_1,\dots,r'_{n-1})$. Recall that $\lambda$ has $n-k$ rows and $k$ columns. 
	The superpotential polytope $\Gamma^{\lambda}_{\rect}(\mathbf{r},\mathbf{r'})$ 
is cut out by the following inequalities:
	\begin{align}
		v_{\mu} - v_{\muminus} & \leq r_\ell\ \text{ whenever }\mu 
		\text{ labels the $\ell$th outer corner of }\lambda  \label{eq:1gen}\\
		0 & \leq v_{1 \times 1}+r'_{n-k} \label{eq:2gen} \\
		v_{i\times j} - v_{(i-1)\times (j-1)} &\leq
		v_{(i+1) \times j}  - v_{i\times (j-1)}+r'_{n-k-i}
\ \text{ for }1 \leq i, 
		1\leq j, \text{ and } ((i+1) \times j) \subseteq \lambda \label{eq:3gen}\\
v_{i\times j}-v_{(i-1)\times (j-1)} &\leq
		v_{i\times (j+1)}  - v_{(i-1)\times j} + r'_{n-k+j}
\ \text{ for } 1\leq i,
		1 \leq j, \text{ and } (i \times (j+1)) \subseteq \lambda. \label{eq:4gen}
	\end{align}
\end{lemma}

\subsection{The superpotential for the rectangles seed in vertex coordinates}\label{s:vertexcoordinates}

The Laurent polynomial superpotential $W_{\rect}^{\lambda}$
may be encoded in a labeled quiver 
following \cref{s:rectW}
(see e.g. \cref{fig:superpotential}),
 but with new variables that are associated directly to vertices.   
See the left of \cref{fig:superpotential-442} for the example where
$\lambda = (4,4,2)$.

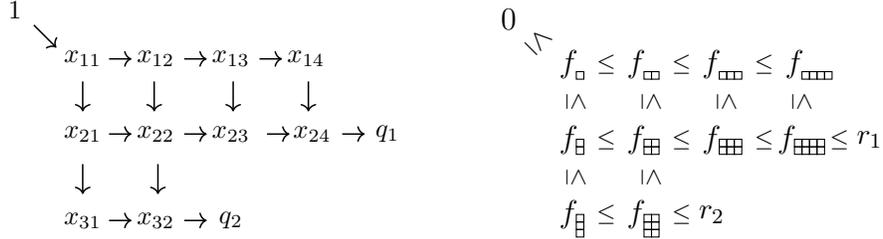
\begin{figure}[h]
\begin{center}
\begin{tikzpicture}
\node at (0.2,1.9) {$1$};
\draw[->, line width=0.2mm] (0.45, 1.7) -- (0.75, 1.4);

\node at (1.1,1.25) {$x_{11}$};
\draw[->, line width=0.2mm] (1.45,1.25) -- (1.75,1.25);

\node at (2.07,1.25) {$x_{12}$};
\draw[->, line width=0.2mm] (2.45,1.25) -- (2.75,1.25);

\node at (3.07,1.25) {$x_{13}$};
\draw[->, line width=0.2mm] (3.45,1.25) -- (3.75,1.25);

\node at (4.07,1.25) {$x_{14}$};

\draw[->, line width=0.2mm] (1.1,0.95) -- (1.1,0.55);
\draw[->, line width=0.2mm] (2.05,0.95) -- (2.05,0.55);
\draw[->, line width=0.2mm] (3.1,0.95) -- (3.1,0.55);
\draw[->, line width=0.2mm] (4.1,0.95) -- (4.1,0.55);

\node at (1.1,0.25) {$x_{21}$};
\draw[->, line width=0.2mm] (1.45,0.25) -- (1.75,0.25);

\node at (2.07,0.25) {$x_{22}$};
\draw[->, line width=0.2mm] (2.45,0.25) -- (2.75,0.25);

\node at (3.07,0.25) {$x_{23}$};
\draw[->, line width=0.2mm] (3.55,0.25) -- (3.85,0.25);

	\node at (4.15,0.25) {$x_{24}$};
\draw[->, line width=0.2mm] (4.55,0.25) -- (4.85,0.25);

\node at (5.15,0.25) {$q_1$};

\draw[->, line width=0.2mm] (1.1,-0.15) -- (1.1,-0.55);
\draw[->, line width=0.2mm] (2.1,-0.15) -- (2.1,-0.55);

\node at (1.1,-0.9) {$x_{31}$};
\draw[->, line width=0.2mm] (1.45,-0.9) -- (1.75,-0.9);

	\node at (2.07,-0.9) {$x_{32}$};
\draw[->, line width=0.2mm] (2.45,-0.9) -- (2.75,-0.9);

\node at (3.07,-0.9) {$q_2$};
\end{tikzpicture}\hspace{1cm}
\begin{tikzpicture}
\node at (0.2,1.9) {\Large 0};
\node at (0.6,1.6) {\large \begin{turn}{315} $\leq$ \end{turn}};

\node at (1,1.3) {\Large $f$};
\draw (1.1,1.1) rectangle (1.2, 1.2);
\node at (1.5, 1.3) {\small $\leq$};

\node at (1.9,1.3) {\Large $f$};
\draw (2,1.1) rectangle (2.1, 1.2);
\draw (2.1,1.1) rectangle (2.2, 1.2);
\node at (2.5, 1.3) {\small $\leq$};

\node at (2.9,1.3) {\Large $f$};
\draw (3,1.1) rectangle (3.1, 1.2);
\draw (3.1,1.1) rectangle (3.2, 1.2);
\draw (3.2,1.1) rectangle (3.3, 1.2);
\node at (3.6, 1.3) {\small $\leq$};

\node at (4,1.3) {\Large $f$};
\draw (4.1,1.1) rectangle (4.2, 1.2);
\draw (4.2,1.1) rectangle (4.3, 1.2);
\draw (4.3,1.1) rectangle (4.4, 1.2);
\draw (4.4,1.1) rectangle (4.5, 1.2);

\node at (1.1,0.8) {\small \begin{turn}{270} $\leq$ \end{turn}};
\node at (2.1,0.8) {\small \begin{turn}{270} $\leq$ \end{turn}};
\node at (3.1,0.8) {\small \begin{turn}{270} $\leq$ \end{turn}};
\node at (4.1,0.8) {\small \begin{turn}{270} $\leq$ \end{turn}};

\node at (1,0.3) {\Large $f$};
\draw (1.1,0.2) rectangle (1.2, 0.3);
\draw (1.1,0.1) rectangle (1.2, 0.2);
\node at (1.5, 0.3) {\small $\leq$};

\node at (1.9,0.3) {\Large $f$};
\draw (2,0.2) rectangle (2.1, 0.3);
\draw (2,0.1) rectangle (2.1, 0.2);
\draw (2.1,0.2) rectangle (2.2, 0.3);
\draw (2.1,0.1) rectangle (2.2, 0.2);
\node at (2.5, 0.3) {\small $\leq$};

\node at (2.9,0.3) {\Large $f$};
\draw (3,0.2) rectangle (3.1, 0.3);
\draw (3,0.1) rectangle (3.1, 0.2);
\draw (3.1,0.2) rectangle (3.2, 0.3);
\draw (3.1,0.1) rectangle (3.2, 0.2);
\draw (3.2,0.2) rectangle (3.3, 0.3);
\draw (3.2,0.1) rectangle (3.3, 0.2);
\node at (3.6, 0.3) {\small $\leq$};

\node at (3.9,0.3) {\Large $f$};
\draw (4,0.2) rectangle (4.1, 0.3);
\draw (4,0.1) rectangle (4.1, 0.2);
\draw (4.1,0.2) rectangle (4.2, 0.3);
\draw (4.1,0.1) rectangle (4.2, 0.2);
\draw (4.2,0.2) rectangle (4.3, 0.3);
\draw (4.2,0.1) rectangle (4.3, 0.2);
\draw (4.3,0.2) rectangle (4.4, 0.3);
\draw (4.3,0.1) rectangle (4.4, 0.2);
\node at (4.6, 0.3) {\small $\leq$};

\node at (5, 0.3) {\large $r_1$};

\node at (1.1,-0.2) {\small \begin{turn}{270} $\leq$ \end{turn}};
\node at (2.1,-0.2) {\small \begin{turn}{270} $\leq$ \end{turn}};

\node at (1,-0.7) {\Large $f$};
\draw (1.1,-0.8) rectangle (1.2,-0.7);
\draw (1.1,-0.9) rectangle (1.2,-0.8);
\draw (1.1,-1) rectangle (1.2,-0.9);
\node at (1.5, -0.7) {\small $\leq$};

\node at (1.9,-0.7) {\Large $f$};
\draw (2,-0.8) rectangle (2.1,-0.7);
\draw (2,-0.9) rectangle (2.1,-0.8);
\draw (2,-1) rectangle (2.1,-0.9);
\draw (2.1,-0.8) rectangle (2.2,-0.7);
\draw (2.1,-0.9) rectangle (2.2,-0.8);
\draw (2.1,-1) rectangle (2.2,-0.9);
\node at (2.5, -0.7) {\small $\leq$};

\node at (2.9, -0.7) {\large $r_2$};
\end{tikzpicture}
\caption{The quiver 
encoding the superpotential
and the associated poset $P(\lambda)$.}
\label{fig:superpotential-442}
\end{center}
\end{figure}

The associated Laurent polynomial superpotential written in
vertex coordinates is
\begin{equation}\label{eq:442}
	\rW_{\rect}^{(4,4,2)}=x_{11}+\frac{x_{21}}{x_{11}}+\frac{x_{31}}{x_{21}}+\frac{x_{22}}{x_{12}}+\frac{x_{32}}{x_{22}}+\frac{x_{23}}{x_{13}}+\frac{x_{24}}{x_{14}}+\frac{x_{12}}{x_{11}}+\frac{x_{22}}{x_{21}}+\frac{x_{32}}{x_{31}}+\frac{x_{23}}{x_{22}}+\frac{x_{13}}{x_{12}}+
	\frac{x_{14}}{x_{13}}+\frac{x_{24}}{x_{23}}+\frac{q_{1}}{x_{24}}+\frac{q_{2}}{x_{32}}.
\end{equation}

The following is a translation to vertex coordinates of the polytope $\Gamma_\rect^\lambda(\mathbf r,\mathbf r')$ from \cref{l:rectangles}.

\begin{definition}\label{def:super2}
Let $\rr=(r_1,\dots,r_d)\in \mathbb{N}^d$
and $\rr'=(r'_1,\dots,r'_{n-1})\in \mathbb{N}^{n-1}$.
Recall that $\lambda$ is contained in an $(n-k) \times k$ rectangle.
We define $\rGamma_{\rect}^\lambda(\rr,\rr')$, the \emph{superpotential polytope (in vertex coordinates)}
 by the following inequalities obtained using the tropicalisation of $\rW_{\rect}^{\lambda}$:
 \[
\begin{array}{cll}
	f_{i \times j}  & \leq r_{\ell} &\text{ whenever } x_{ij}
	\text{ is adjacent to }q_\ell,  \\
	0 & \leq f_{1 \times 1} +r'_{n-k}, & \\
	f_{i \times j} &\leq
	f_{(i+1) \times j} + r'_{n-k-i}
	\ &\text{ for } 1\leq i<n-k,
		1 \leq j\leq k, \text{ and } ((i+1) \times j) \subseteq \lambda, \\
	f_{i \times j} &\leq
	f_{i \times (j+1)} +r'_{n-k+j}
	\ &\text{ for }1 \leq i \leq n-k,
		1\leq j <k, \text{ and } (i \times (j+1)) \subseteq \lambda. 
        \end{array}      
\]
	We write $\overline{\Gamma}_{\rect}^{\lambda}$ for 
	$\rGamma_{\rect}^\lambda(\mathbf{1},\mathbf{0})$.
\end{definition}

If we set $r'_1 = \dots = r'_{n-1}=0$, 
then most of the inequalities 
in \cref{def:super2}  
correspond to the cover relations in the poset $P(\lambda)$
of rectangles contained in $\lambda$, shown 
at the right of 
\cref{fig:superpotential-442}.  (The inequalities involving the constants $0, r_1, r_2$ are additional 
inequalities of the superpotential polytope, and will be discussed in~\cref{s:order}.)

\begin{definition}\label{def:posetlambda}
Given a Young diagram $\lambda$, 
let $S_{\lambda}$ denote 
 the set of all 
boxes of $\lambda$.
We obtain a poset $P({\lambda})$ on $S_{\lambda}$ by 
identifying each box $b$ with 
$\Rect(b)$ (see \cref{def:rectseed}), and saying that 
$b \leq b'$ whenever $\Rect(b) \subseteq \Rect(b')$.  
The box $c_0$ in the NW corner of $\lambda$ is the (unique) minimal element and the outer corners $c_1\dotsc, c_d$
of $\lambda$
	are the maximal elements of this poset. Here $c_\ell=b_{\rho_{2\ell+1}}$ in the notation from \cref{s:superpotential}, and $c_0=b'_{n-k}$.  
For convenience, we will sometimes use  $\Rect(b)$ instead of $b$ when referring to the associated element of $P(\lambda)$, as for example in the 
	right of 
\cref{fig:superpotential-442}.
\end{definition}

\begin{example}\label{ex:442}
		When $r_1=\dots=r_d=1$ and $r'_1=\dots=r'_{n-1}=1$, then the superpotential polytope (in vertex coordinates)
	$\rGamma_{\rect}^{(4,4,2)}(\mathbf{1},\mathbf{1})$
is defined by the inequalities
$$1+f_{1 \times 1} \geq 0, \
1+f_{2 \times 1}-f_{1 \times 1} \geq 0, \
1+f_{3 \times 1}-f_{2 \times 1} \geq 0, \
1+f_{2 \times 2}-f_{1 \times 2} \geq 0,\ \dots  ,
1-f_{2 \times 4} \geq 0, \
1-f_{3 \times 2} \geq 0.$$

\end{example}

We now explain how our two versions 
	of the superpotential polytope 
	are related.

\begin{definition}
	We say that two integral polytopes 
  $\Ppoly_1 \subset \R^n$ and $\Ppoly_2 \subset \R^m$ are \emph{integrally
	equivalent}\footnote{Sometimes the terminology  
 \emph{isomorphic} or
	\emph{unimodularly equivalent} is used synonymously.}
	if there is an affine transformation
   $\phi:\R^n \to \R^m$ whose restriction to $\Ppoly_1$
    is a bijection $\phi:\Ppoly_1 \to \Ppoly_2$  that preserves the $\Z$-lattices,
    i.e. $\phi$ induces a bijection between
     $\Z^n \cap \aff(\Ppoly_1)$ and $\Z^m \cap \aff(\Ppoly_2)$, where
     $\aff(\cdot)$ denotes the affine span.  The map $\phi$
      is then called an \emph{integral equivalence} between the two polytopes.
\end{definition}
We note that integrally equivalent polytopes 
have the same Ehrhart polynomials
and  volume. 

\cref{p:unimodular} shows that 
our two versions 
$\Gamma_{\rect}^\lambda(\rr,\rr')$   and
$\rGamma_{\rect}^\lambda(\rr,\rr')$ of
the superpotential
polytope 
are integrally equivalent.
The proof follows immediately by comparing 
\cref{l:rectangles} to \cref{def:super2}.
\begin{proposition}\label{p:unimodular}
The map $F:\R^{\mathcal P_{G^{\rect}_{\lambda}}}\to \R^{\mathcal P_{G^{\rect}_{\lambda}}}$ defined by
\[
(v_{i\times j})\mapsto (f_{i \times j}) = ( v_{i \times j} - v_{(i-1)\times (j-1)})
\]
is a unimodular linear transformation, with inverse 
$v_{i \times j} = f_{i \times j} + f_{(i-1) \times (j-1)} + 
f_{(i-2) \times (j-2)} + \dots$.
 Moreover, 
$F(\Gamma_{\rect}^\lambda(\rr,\rr'))=
\rGamma_{\rect}^\lambda(\rr,\rr')$, and hence the  polytopes
$\Gamma_{\rect}^\lambda(\rr,\rr')$  and
$\rGamma_{\rect}^\lambda(\rr,\rr')$ are integrally equivalent.
\end{proposition}

\subsection{The connection with order polytopes}\label{s:order}
In this section we briefly review some background on order polytopes, then 
explain how our superpotential polytopes in vertex coordinates
are related to order polytopes.  These results will be a crucial tool in identifying
the superpotential polytope with the Newton-Okounkov body, 
see
 \cref{thm:rectanglesproof}.

\begin{definition}\label{def:order}\cite{order}
	Let $(P, \leq)$ be a (finite) poset, and let $a \lessdot b$ denote the cover relations.
	The \emph{order polytope} $\OO(P)$ 
	of $P$ is the subset of $\R^P$ consisting
	of points $\mathbf{f} = (f_a)_{a\in P}$ such that
	\begin{align}
		0 &\leq f_a \leq 1 \text{ for all } a\in P\label{e:OrderPolytope1}\\
		f_a & \leq f_b \text{ whenever } a<b \text{ in }P.\label{e:OrderPolytope2}
	\end{align}
\end{definition}

If we set $r_1=r_2=1$ in the right of 
\cref{fig:superpotential-442},
then the resulting inequalities define
 an order polytope.

An immediate observation from \cite[Section 1]{order} is a characterization
of facets of $\OO(P)$.
\begin{lemma}\label{lem:facets}
	There are three types of facets for an order polytope $\OO(P)$. Namely,
	\begin{enumerate}
	\item the set of $f\in \OO(P)$ satisfying $f_x=0$ for some fixed minimal $x\in P$,
\item the set of $f\in \OO(P)$ satisfying $f_x=1$ for some fixed maximal $x\in P$,
\item the set of $f\in \OO(P)$ satisfying $f_x=f_y$ for some fixed $x,y$ such that $x\lessdot y$.
	\end{enumerate}	
\end{lemma}

Recall that a \emph{filter} (or \emph{dual order ideal}) $J$ of $P$
is a subset 
of $P$ such that whenever $a\in J$ and $b \geq a$, then also $b\in J$.
Let $\chi_J:P \to \R$ denote the characteristic function of $J$; i.e.
$$\chi_J(a) = \begin{cases}  1, \text{ if }x\in J\\
	  0, \text{ if }x\notin J
\end{cases}$$

\begin{proposition}\label{p:StanleyVertices}\cite[Corollary 1.3 and Theorem 4.1]{order} 
The vertices of $\OO(P)$ are precisely the characteristic
functions
	$\chi_J$ of filters $J$ of $P$.  All lattice points of $\OO(P)$ are vertices.
\end{proposition}

Let $e(P)$ denote the number of linear extensions of the poset $P$.

\begin{theorem}\cite[Corollary 4.2]{order}\label{thm:volume}
Let $P$ be a poset with $n$ elements.  Then the volume of the order
polytope $\OO(P)$ is $e(P)/n!$.
\end{theorem}

We now focus on the case that our poset equals 
the poset $P(\lambda)$ from \cref{def:posetlambda}.
\begin{defn}\label{def:markedorder}
Choose nonnegative numbers
$\mathbf{r} = (r_1, \dots, r_d)$, and let 
$\OO^{\mathbf{r}}(\lambda)$
denote the subset of $\R^{P(\lambda)}$ 
defined by the inequalities ~
\eqref{e:OrderPolytope2}  
for the poset $P=P(\lambda)$ 
	together with 
\begin{equation}
0\le f_{c_0} \quad \text{ and }\quad  f_{c_\ell}\le  r_\ell \quad \text{ for  $\ell=1,\dotsc, d$.} 
\end{equation}
See 
		the right of \cref{fig:superpotential-442}
		for an example.
The polytope 
$\OO^{\mathbf{r}}(\lambda)$ 
is an example of a \emph{marked order polytope}
 \cite{marked}.
When each $r_i = 1$, 
$\OO^{\mathbf{r}}(\lambda)$ 
recovers the order polytope
of the poset $P({\lambda})$, which we denote by $\OO(\lambda)$. 
\end{defn}

\begin{remark}\label{rem:marked}
Clearly  $\OO^{\mathbf{r}}(\lambda)$ is full-dimensional if and only if 
each $r_i$ is positive.
\end{remark}

\begin{lemma}\label{lem:faces}
Let $\nu \subseteq \lambda$ be partitions.
Then $\OO(\nu)$ is a face of 
$\OO(\lambda)$.
\end{lemma}
\begin{proof}
Here we view $\OO(\nu)$ as lying in the vector space $\R^{P(\lambda)}$ via the inclusion $\R^{P(\nu)}\hookrightarrow \R^{P(\lambda)}$	 that sets the coordinates $f_b$ with $b\notin \nu$ to $1$. If $\nu$ is obtained from $\lambda$ by removing a single box $b$, then this box was a maximal element of $P(\lambda)$ and $\OO(\nu)$ is one of the facets of $\OO(\lambda)$ described in \cref{lem:facets}. The lemma now follows for general $\nu$ by recursion.   
\end{proof}

It is well known that the linear extensions of the poset $P(\lambda)$
are in bijection with the standard Young tableaux of shape $\lambda$.
Therefore we obtain the following result as an application of Theorem~\ref{thm:volume}.
\begin{corollary}\label{cor:volume}
If we let $|\lambda|$ denote
the number of boxes of $\lambda$, then the volume of the polytope $\OO(\lambda)$
is $\frac{1}{|\lambda|!}$ times the 
number of standard Young tableaux of 
shape $\lambda$.
\end{corollary}
There are $252$ linear extensions 
of the poset 
in the right of 
\cref{fig:superpotential-442}, or equivalently, there
are $252$ standard Young tableaux of shape $(4,4,2)$.
Hence
the volume of the corresponding order polytope 
is $\frac{252}{10!}$.

\begin{proposition}\label{p:unimodular2} 
The superpotential 
	polytope $\rGamma^{\lambda}_{{\rect}}({\mathbf{r},\mathbf{0})}$
 from \cref{def:super2}	coincides with the marked order polytope 
$\OO^{\mathbf{r}}(\lambda)$.
	When 
	$r_1=\dots=r_d=1$, 
	 the superpotential polytope $\rGamma^{\lambda}_{{\rect}} = 
	 \rGamma^{\lambda}_{{\rect}}(\mathbf{1},\mathbf {0})$ agrees with
	 the order polytope $\OO(\lambda)$, and hence 
	the volumes of 
	 $\rGamma^{\lambda}_{{\rect}}$ and of
	 $\Gamma^{\lambda}_{{\rect}}$
	 equal $\frac{1}{|\lambda|!}$ times the number
	 of standard Young tableaux of shape~$\lambda$.
\end{proposition}

	\begin{proof}
The first statement follows
from the definitions.
The statement about the volume of $\Q^{\lambda}_{{\rect}}$ follows from 
 \cref{cor:volume}.
\end{proof}

We now generalize \cref{p:unimodular2}  and show that each superpotential polytope
$\rGamma_{\rect}^{\lambda}(\rr,\rr')$
 is a translation of a marked order polytope. 

\begin{definition}\label{def:al}
Let us associate to $\lambda$ an arrow-labeling of the quiver $Q_\lambda$ from \cref{def:superquiver}, in which the arrow pointing to $q_\ell$ is labelled $r_\ell$, the arrow with source $1$ is labelled $r'_{n-k}$, the vertical arrows in row $i$ from the bottom are all labelled $r'_i$, and the horizontal arrows in the  $j^{th}$ column are labelled $r'_{n-k+j}$. 
\end{definition}
 An example of the arrow-labeling from \cref{def:al} is shown at the left of 
\cref{f:Dtilde}.

\cref{prop:markedorder}  describes $\rGamma_{\rect}^{\lambda}(\rr,\rr')$ as a 
translated marked order polytope, with marking and translation vector encoded in this 
arrow-labeling, and follows immediately from 
	 \cref{def:markedorder} and 
	\cref{def:super2}.

\begin{figure}
	\[\adjustbox{scale=.8,center}{\begin{tikzcd}
	{\star_0} \\
	& \bullet & \bullet & \bullet & \bullet & \bullet & \bullet \\
	& \bullet & \bullet & \bullet & \bullet & \bullet & \bullet \\
	& \bullet & \bullet & \bullet & \bullet & \bullet & \bullet & {\star_1} \\
	& \bullet & \bullet & \bullet & \bullet & {\star_2} \\
	& \bullet & \bullet \\
	& \bullet & \bullet & {\star_{3}} \\
	& \bullet \\
	& \bullet & {\star_4}
	\arrow["{r_8'}", from=1-1, to=2-2]
	\arrow["{r_9'}", from=2-2, to=2-3]
	\arrow["{r'_7}", from=2-2, to=3-2]
	\arrow["{r'_{10}}", from=2-3, to=2-4]
	\arrow["{r'_{7}}", from=2-3, to=3-3]
	\arrow["{r'_{11}}", from=2-4, to=2-5]
	\arrow["{r'_{7}}", from=2-4, to=3-4]
	\arrow["{r'_{12}}", from=2-5, to=2-6]
	\arrow["{r'_{7}}", from=2-5, to=3-5]
	\arrow["{r'_{13}}", from=2-6, to=2-7]
	\arrow["{r'_{7}}", from=2-6, to=3-6]
	\arrow["{r'_{7}}", from=2-7, to=3-7]
	\arrow["{r'_9}", from=3-2, to=3-3]
	\arrow["{r_6'}", from=3-2, to=4-2]
	\arrow["{r'_{10}}", from=3-3, to=3-4]
	\arrow["{r'_{6}}", from=3-3, to=4-3]
	\arrow["{r'_{11}}", from=3-4, to=3-5]
	\arrow["{r'_{6}}", from=3-4, to=4-4]
	\arrow["{r'_{12}}", from=3-5, to=3-6]
	\arrow["{r'_{6}}", from=3-5, to=4-5]
	\arrow["{r'_{13}}", from=3-6, to=3-7]
	\arrow["{r'_{6}}", from=3-6, to=4-6]
	\arrow["{r'_{6}}", from=3-7, to=4-7]
	\arrow["{r'_9}", from=4-2, to=4-3]
	\arrow["{r_5'}", from=4-2, to=5-2]
	\arrow["{r'_{10}}", from=4-3, to=4-4]
	\arrow["{r'_{5}}", from=4-3, to=5-3]
	\arrow["{r'_{11}}", from=4-4, to=4-5]
	\arrow["{r'_{5}}", from=4-4, to=5-4]
	\arrow["{r'_{12}}", from=4-5, to=4-6]
	\arrow["{r'_{5}}", from=4-5, to=5-5]
	\arrow["{r'_{13}}", from=4-6, to=4-7]
	\arrow["{r_1}", from=4-7, to=4-8]
	\arrow["{r'_9}", from=5-2, to=5-3]
	\arrow["{r'_4}", from=5-2, to=6-2]
	\arrow["{r'_{10}}", from=5-3, to=5-4]
	\arrow["{r'_4}", from=5-3, to=6-3]
	\arrow["{r'_{11}}", from=5-4, to=5-5]
	\arrow["{r_2}", from=5-5, to=5-6]
	\arrow["{r'_9}", from=6-2, to=6-3]
	\arrow["{r_3'}", from=6-2, to=7-2]
	\arrow["{r_3'}", from=6-3, to=7-3]
	\arrow["{r'_9}", from=7-2, to=7-3]
	\arrow["{r_2'}", from=7-2, to=8-2]
	\arrow["{r_3}"{pos=0.6}, from=7-3, to=7-4]
	\arrow["{r_1'}", from=8-2, to=9-2]
	\arrow["{r_4}"', from=9-2, to=9-3]
	\end{tikzcd} \hspace{.1cm}
	\begin{tikzcd}
	\star_0 \\
	& \bullet & \bullet & \bullet & \bullet & \bullet & \bullet \\
	& \bullet & \bullet & \bullet & \bullet & \bullet & \bullet \\
	& \bullet & \bullet & \bullet & \bullet & \bullet & \bullet & {\star_1} \\
	& \bullet & \bullet & \bullet & \bullet \\
	& \bullet & \bullet \\
	& \bullet & \bullet && {\quad \star_{2,3}} \\
	& \bullet \\
	& \bullet & {\star_4}
	\arrow["{r_8'}", from=1-1, to=2-2]
	\arrow["{r_9'}", from=2-2, to=2-3]
	\arrow["{r'_7}", from=2-2, to=3-2]
	\arrow["{r'_{10}}", from=2-3, to=2-4]
	\arrow["{r'_{7}}", from=2-3, to=3-3]
	\arrow["{r'_{11}}", from=2-4, to=2-5]
	\arrow["{r'_{7}}", from=2-4, to=3-4]
	\arrow["{r'_{12}}", from=2-5, to=2-6]
	\arrow["{r'_{7}}", from=2-5, to=3-5]
	\arrow["{r'_{13}}", from=2-6, to=2-7]
	\arrow["{r'_{7}}", from=2-6, to=3-6]
	\arrow["{r'_{7}}", from=2-7, to=3-7]
	\arrow["{r'_9}", from=3-2, to=3-3]
	\arrow["{r_6'}", from=3-2, to=4-2]
	\arrow["{r'_{10}}", from=3-3, to=3-4]
	\arrow["{r'_{6}}", from=3-3, to=4-3]
	\arrow["{r'_{11}}", from=3-4, to=3-5]
	\arrow["{r'_{6}}", from=3-4, to=4-4]
	\arrow["{r'_{12}}", from=3-5, to=3-6]
	\arrow["{r'_{6}}", from=3-5, to=4-5]
	\arrow["{r'_{13}}", from=3-6, to=3-7]
	\arrow["{r'_{6}}", from=3-6, to=4-6]
	\arrow["{r'_{6}}", from=3-7, to=4-7]
	\arrow["{r'_9}", from=4-2, to=4-3]
	\arrow["{r_5'}", from=4-2, to=5-2]
	\arrow["{r'_{10}}", from=4-3, to=4-4]
	\arrow["{r'_{5}}", from=4-3, to=5-3]
	\arrow["{r'_{11}}", from=4-4, to=4-5]
	\arrow["{r'_{5}}", from=4-4, to=5-4]
	\arrow["{r'_{12}}", from=4-5, to=4-6]
	\arrow["{r'_{5}}", from=4-5, to=5-5]
	\arrow["{r'_{13}}", from=4-6, to=4-7]
	\arrow["{r_1}", from=4-7, to=4-8]
	\arrow["{r'_9}", from=5-2, to=5-3]
	\arrow["{r_4'}", from=5-2, to=6-2]
	\arrow["{r'_{10}}", from=5-3, to=5-4]
	\arrow["{r'_4}", from=5-3, to=6-3]
	\arrow["{r'_{11}}", from=5-4, to=5-5]
	\arrow["{r_2}", from=5-5, to=7-5]
	\arrow["{r'_9}", from=6-2, to=6-3]
	\arrow["{r_3'}", from=6-2, to=7-2]
	\arrow["{r_3'}", from=6-3, to=7-3]
	\arrow["{r'_9}", from=7-2, to=7-3]
	\arrow["{r_2'}", from=7-2, to=8-2]
	\arrow["{r_3}"{pos=0.6}, from=7-3, to=7-5]
	\arrow["{r_1'}", from=8-2, to=9-2]
	\arrow["{r_4}"', from=9-2, to=9-3]
\end{tikzcd}}\]
\caption {At left: an arrow-labeling 
that encodes the translation vector from \cref{prop:markedorder}.
At right: the same arrow-labeling on a related quiver.
In \cref{s:toric}, these arrow-labelings will be used to encode
the Weil divisor
$D_{(\mathbf r,\mathbf {r'})}$ in $X_\lambda$, 
and the Weil divisor $\widetilde D_{(\mathbf r,\mathbf{r'})}$ in $Y(\overline{\FF_\lambda})$.  
Here $\lambda=(6,6,6,4,2,2,1,1)$, $n=14$, $k=6$, and $d=4$.
\label{f:Dtilde}}
\end{figure}
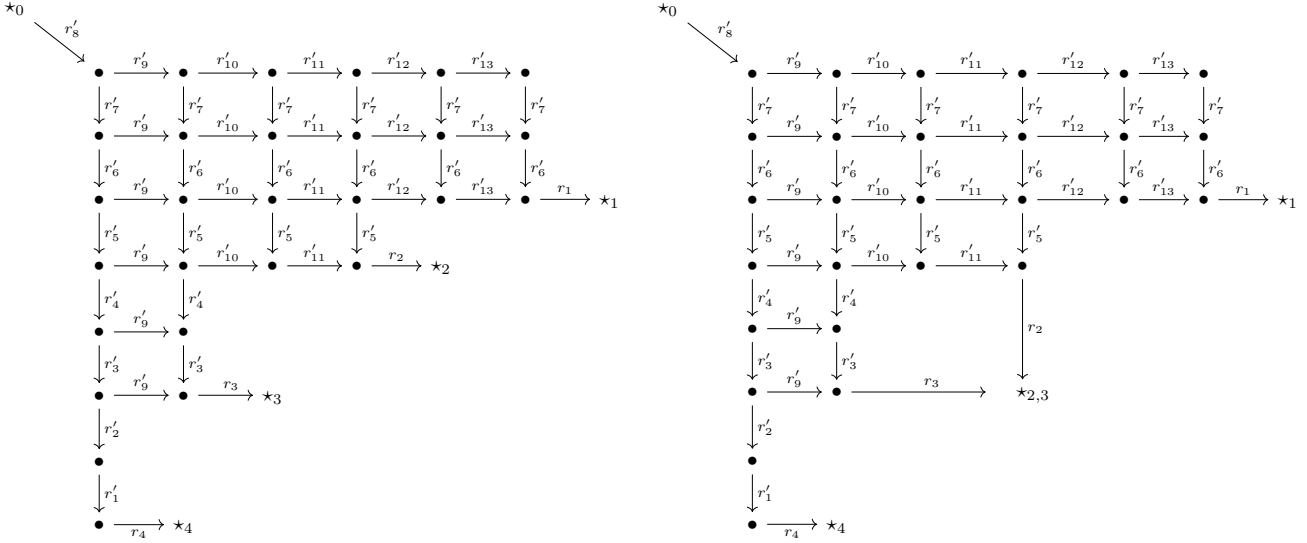

\begin{proposition}\label{prop:markedorder}
Let $\rr$ and $\rr'$ be as in \cref{def:super2}.
	Let $\nn(\rr,\rr'):=(\eta_1,\dots,\eta_d)$, where $\eta_i$ is the sum of the edge weights from 
	a/any path from $\star_0$ to $\star_i$ in the arrow-labeling of $Q_\lambda$ 
	from \cref{def:al}, see  
	\cref{f:Dtilde}.
	Recall that $S_{\lambda}$ denotes the set of all boxes of $\lambda$,
	and let $\uu(\rr,\rr')=\{u_b\}_{b\in S_{\lambda}} \in \Z^{S_{\lambda}}$,
	where $u_b$ is the sum of the edge weights from a/any path from $\star_0$ to $b$.
	Then $$\rGamma_{\rect}^{\lambda}(\rr,\rr') 
	= \OO^{\nn(\rr,\rr')}(\lambda)-
	\uu(\rr,\rr'),$$ that is, 
	the superpotential polytope $\rGamma_{\rect}^{\lambda}(\rr,\rr')$ is a translate
	of the marked order polytope
	$\OO^{\nn(\rr,\rr')}(\lambda)$.
\end{proposition}

Note that when $(\mathbf r,\mathbf r')=(\mathbf 1,\mathbf 1)$ we have $\mathbf n(\mathbf 1,\mathbf 1)=(n_1,\dots, n_d)$, where $n_\ell$ is as in \cref{d:LG1} and \cref{c:ac}.

\begin{definition}\label{def:IDP}
  A lattice polytope $P$ is said to
be \emph{integrally closed} 
or have
the \emph{integer decomposition property} (IDP)
if every lattice point in its $r$th dilation $rP$ 
is a sum of $r$ lattice points in $P$.
\end{definition}

\begin{corollary}\label{cor:IDP}
Let $\rr=(r_1,\dots,r_d)\in \mathbb{N}^d$
and $\rr'=(r'_1,\dots,r'_{n-1})\in \mathbb{N}^{n-1}$.
Then the superpotential polytope 
$\rGamma^{\lambda}_{{\rect}}(\mathbf{r},\mathbf{r'})$ 
has the integer decomposition property.
\end{corollary}
\begin{proof}
\cite[Corollary 2.3]{FangFourier} says that every marked order polytope has the integer decomposition property.
This property is preserved under translation by an integer vector, 
so the result now follows from \cref{prop:markedorder}.
Alternatively, it follows from \cite[Theorem~2.4]{HPPS}, see \cite[Remark 5.6]{RWReflexive}, that superpotential polytopes
(as defined in \cite[Definition 5.2]{RWReflexive}) have the integer decomposition property.
\end{proof}

 We can obtain a description of the vertices of the order polytope 
 $\OO(\lambda) = 
\rGamma_{\rect}^{\lambda}$ 
 by using
\cref{p:StanleyVertices} and noting that 
the filters of $P(\lambda)$ are precisely the complements of Young diagrams $\mu\subseteq 
\lambda$.  This will allow us to describe the vertices of the superpotential
polytope $\Gamma_{\rect}^{\lambda}$ in 
	\cref{l:VerticesOfGamma}.
 
 \begin{corollary}\label{c:VerticesOfO}
For $\mu\subseteq\lambda$ let $\chi_{\mu^c}=(f_b)\in\R^{P(\lambda)}$ be given by 
\begin{equation*}
f_b=\begin{cases} 0 &
 \text{ if the box $b$ lies in $\mu$, }\\ 
 1 & \text{ if $b\notin\mu$.} 
 \end{cases}
\end{equation*}
The vectors $\chi_{\mu^c}$ are precisely the vertices of the polytope $\OO(\lambda)$.
\end{corollary}

\begin{definition}\label{d:MaxDiagVectors}\cite[Definition 14.3]{RW} 
Given two partitions $\mu$ and $\nu$, we let 
$\nu \setminus \mu$ denote the corresponding 
	\emph{skew diagram}, i.e. the set of boxes remaining
	if we justify both $\nu$ and $\mu$ at the top-left of 
	a rectangle, then remove from $\nu$ any boxes that 
	are in $\mu$.  We let 
	 $\maxdiag(\nu\setminus \mu)$ denote the maximum
	 number of boxes of $\nu \setminus \mu$ that lie 
	 along any diagonal (with slope $-1$) of the rectangle.
\end{definition}

\begin{proposition}\label{l:VerticesOfGamma}
For any $\mu\subseteq\lambda$ let 
$v_{i\times j}=\maxdiag((i\times j)\setminus\mu)$
and define a vector 
$\mathbf v_\mu=(v_{i\times j})_{i,j}\in \R^{P(\lambda)}$.
Then the vectors $\mathbf v_\mu$ 
are all distinct and the set 
\[
\{\mathbf v_\mu\mid\mu\subseteq\lambda\}
\] 
is the set of vertices of 
	$\Q^{\lambda}_{{\rect}}$, 
	which coincides with the set of 
lattice points of 
	$\Q^{\lambda}_{{\rect}}$.
\end{proposition}
\begin{proof}
The first part of the proposition
follows from \cref{c:VerticesOfO} by applying the inverse of the transformation $F$ in Proposition~\ref{p:unimodular}. Namely, $F\inv(\chi_{\mu^c})=\mathbf v_{\mu}$. Note that whenever the skew shape $(i\times j)\setminus\mu$ is non-empty it contains the box $(i,j)$, and the associated diagonal precisely has length $\maxdiag((i\times j)\setminus\mu)$.  The fact that each lattice point of the polytope is a vertex follows from 
	\cref{p:StanleyVertices}.
\end{proof}

\section{The proof that the Newton-Okounkov body equals the superpotential polytope}\label{sec:polytopescoincide}

In this section we will prove \cref{t:mainIntro}, which says 
that the Newton-Okounkov body $\Delta^{\lambda}_G$
associated to the $\mathcal{X}$-cluster chart indexed by $G$
(see \cref{rem:special}) 
equals the superpotential polytope $\Gamma^{\lambda}_G$
associated to the $\mathcal{A}$-cluster chart indexed by $G$
(see \eqref{eq:mainsuperpolytope}).
We start by proving this result in the case of the 
rectangles chart.  We then explain how to prove the result 
in the case of a general
cluster chart.

\subsection{The proof that $\Delta^{\lambda}_G=\Gamma^{\lambda}_G$ for the rectangles cluster}\label{s:proof}

Fix a Schubert variety $X_{\lambda}$.
In this section we will prove the following result.

\begin{theorem}\label{thm:rectanglesproof}
When $G = G_{\lambda}^{\rect}$, we have
that the Newton-Okounkov body 
$\Delta^{\lambda}_{\rect}:=\Delta^{\lambda}_{G_{\lambda}^{\rect}}$ is 
equal to 
the superpotential polytope
$\Gamma^{\lambda}_{\rect}:=\Gamma^{\lambda}_{G_{\lambda}^{\rect}}$. 
\end{theorem}

Recall 
from \cref{d:NOrG}
the polytope 
$\conv_G^{\lambda}$, 
which is the convex hull of the points
$\val_G(P_\mu)$ for $\mu\subseteq\lambda$.

\begin{proposition}\label{c:ConvEqualsGamma}
Suppose 
$G=G^\rect_{\lambda}$. 
Then for $\mu \subseteq \lambda$, we have 
	$\val_G(P_\mu)=\mathbf v_\mu$, where
	$\mathbf v_{\mu}$ is defined in \cref{l:VerticesOfGamma}.
Moreover the polytopes $\conv_G^{\lambda}$ and $\Q^\lambda_{G}$ agree, and their 
lattice points are precisely the vertices 
	$\val_G(P_\mu)=\mathbf v_\mu$ for $\mu \subseteq \lambda$.
\end{proposition}

\begin{proof} The proof of the first statement is 
completely analogous to the proof in the Grassmannian case,
namely the proof of  \cite[Proposition 14.4]{RW}. 
It then follows from \cref{l:VerticesOfGamma} that the polytopes agree
and that every lattice point is a vertex. 
\end{proof}

\begin{proposition}\label{c:VolumeOfNOBody}
The Newton-Okounkov body $\Delta^\lambda_{\rect}$ has volume equal to $\frac{1}{|\lambda|!}$ times the number
 of standard Young tableaux of shape $\lambda$.
\end{proposition}
\begin{proof}
By \cref{t:KK}, the volume of the Newton-Okounkov body $\Delta_{G}^\lambda$ is equal to $\frac{1}{|\lambda|!}$ times the degree of the Schubert variety $X_\lambda$. 
It is well-known 
that 
the degree of the Schubert variety $X_{\lambda}$ in its Pl\"ucker embedding is equal to the number of standard Young tableaux
	of shape $\lambda$ \cite{LS}.  The result follows.
\end{proof}

We are now in a position to prove \cref{thm:rectanglesproof}.
\begin{proof}[Proof of \cref{thm:rectanglesproof}]
Let $G=G^{\rect}_\lambda$. We have proved that
\[\Q^\lambda_\rect\,=\,\conv_G^{\lambda}\,\subseteq\, \Delta^\lambda_\rect,
\]
where the equality is due to \cref{c:ConvEqualsGamma} and the inclusion is just a result of the definition of the Newton-Okounkov convex body. On the other hand $\Q^{\lambda}_\rect$ and $\Delta^\lambda_\rect$ have the same volume, by comparing 
	\cref{p:unimodular2} 
	and \cref{c:VolumeOfNOBody}. 
Therefore $\Q^\lambda_\rect \subseteq \Delta^\lambda_\rect$ is an inclusion of a polytope in a convex set of the same volume. This implies that the inclusion must be an equality. It follows that $ \Q^\lambda_\rect= \Delta^\lambda_\rect$. 
\end{proof}

\begin{proposition}\label{cor:IDP2} 
For any  $\lambda$,
$\Delta^\lambda_\rect$ has the integer decomposition property.
If $\nu \subseteq \lambda$,
then $\Delta^{\nu}_{\rect}$ embeds as a face
into $\Delta^{\lambda}_{\rect}$.
\end{proposition}
\begin{proof}
To see that
$\Delta^{\nu}_{\rect}$ appears as a face
of $\Delta^{\lambda}_{\rect}$, we use
the fact that
$\OO(\nu)$ is a face of $\OO(\lambda)$ (by
\cref{lem:faces}), and that
 $\Delta^{\nu}_{\rect}$ and
 $\Delta^{\lambda}_{\rect}$ are integrally equivalent to
$\OO(\nu)$ and $\OO(\lambda)$
(by
\cref{p:unimodular} and
\cref{thm:rectanglesproof}).
The fact that $\Delta^\lambda_\rect$ has the integer decomposition property  now follows from \cref{cor:IDP}.
\end{proof}

\subsection{A detailed example: the Schubert variety $X_{(2,1)}$}\label{sec:example}

Consider the Schubert variety $X_{\lambda}$  where $\lambda=(2,1)$.
Its rectangles network is shown at the left of 
\cref{fig:positroidsubvariety}.  If we compute the valuations of 
Pl\"ucker coordinates, expressed in this network chart, we obtain the 
lattice points shown in 
\cref{table:Schub21}.  

By
\cref{c:ConvEqualsGamma}, the convex hull  
$\conv_{\rect}^{\lambda}$ 
of these lattice points
equals the superpotential polytope $\Gamma_{\rect}^{\lambda}$; we 
can check this directly by comparing with 
the superpotential polytope 
$\Gamma_{\rect}^{\lambda}$
from 
\cref{ex:21}.
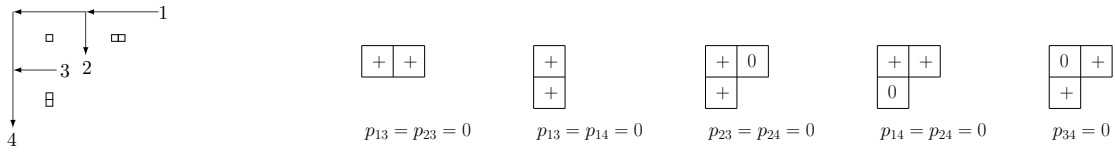
\begin{figure}[h]
\centering
\setlength{\unitlength}{1.3mm}
\begin{center}
	\resizebox{0.25\textwidth}{!}{\begin{picture}(42,30)
 \put(22,19){$1$}
 \put(11.5,11.5){$2$}
 \put(8.5,11){$3$}
 \put(1.2,1.5){$4$}
% first row
         \put(6.5,16){$\ydiagram{1}$}
         \put(15.5,16){$\ydiagram{2}$}
% second row
         \put(6.5,8){$\ydiagram{1,1}$}
         \put(22, 20){{\vector(-1,0){10}}}
         \put(12, 20){{\vector(-1,0){10}}}
         \put(8, 12){{\vector(-1,0){6}}}
         \put(2,20){{\vector(0,-1){16}}}
         \put(12,20){{\vector(0,-1){6}}}
\end{picture}}
		\hspace{.3cm}
       \resizebox{0.12\textwidth}{!}{\begin{picture}(42,30)
\put(5,32){\line(1,0){18}}
  \put(5,23){\line(1,0){18}}
%  \put(5,14){\line(1,0){9}}
% vertical lines
  \put(5,23){\line(0,1){9}}
  \put(14,23){\line(0,1){9}}
  \put(23,23){\line(0,1){9}}
% first row
   \put(8,26){\Huge{$+$}}
   \put(17,26){\Huge{$+$}}
% second row
        % \put(8,17){\Huge{$+$}}
% third row
		\put(6,6){\Huge{$p_{13}=p_{23}=0$}}
        \end{picture}}
    \hspace{0.1cm}
		\resizebox{0.12\textwidth}{!}{\begin{picture}(42,30)
\put(5,32){\line(1,0){9}}
  \put(5,23){\line(1,0){9}}
  \put(5,14){\line(1,0){9}}
% vertical lines
  \put(5,14){\line(0,1){18}}
  \put(14,14){\line(0,1){18}}
%  \put(23,23){\line(0,1){9}}
% first row
   \put(8,26){\Huge{$+$}}
%   \put(17,26){\Huge{$+$}}
% second row
         \put(8,17){\Huge{$+$}}
% third row
		\put(6,6){\Huge{$p_{13}=p_{14}=0$}}
        \end{picture}}
    \hspace{0.1cm}
        \resizebox{0.12\textwidth}{!}{\begin{picture}(42,30)
\put(5,32){\line(1,0){18}}
  \put(5,23){\line(1,0){18}}
  \put(5,14){\line(1,0){9}}
% vertical lines
  \put(5,14){\line(0,1){18}}
  \put(14,14){\line(0,1){18}}
  \put(23,23){\line(0,1){9}}
% first row
   \put(8,26){\Huge{$+$}}
		\put(17,26){\Huge{$0$}}
% second row
         \put(8,17){\Huge{$+$}}
% third row
		\put(6,6){\Huge{$p_{23}=p_{24}=0$}}
        \end{picture}}
    \hspace{0.1cm}
        \resizebox{0.12\textwidth}{!}{\begin{picture}(42,30)
\put(5,32){\line(1,0){18}}
  \put(5,23){\line(1,0){18}}
  \put(5,14){\line(1,0){9}}
% vertical lines
  \put(5,14){\line(0,1){18}}
  \put(14,14){\line(0,1){18}}
  \put(23,23){\line(0,1){9}}
% first row
   \put(8,26){\Huge{$+$}}
   \put(17,26){\Huge{$+$}}
% second row
         \put(8,17){\Huge{$0$}}
% third row
		\put(6,6){\Huge{$p_{14}=p_{24}=0$}}
        \end{picture}}
    \hspace{0.1cm}
        \resizebox{0.12\textwidth}{!}{\begin{picture}(42,30)
\put(5,32){\line(1,0){18}}
  \put(5,23){\line(1,0){18}}
  \put(5,14){\line(1,0){9}}
% vertical lines
  \put(5,14){\line(0,1){18}}
  \put(14,14){\line(0,1){18}}
  \put(23,23){\line(0,1){9}}
% first row
   \put(8,26){\Huge{$0$}}
   \put(17,26){\Huge{$+$}}
% second row
         \put(8,17){\Huge{$+$}}
% third row
		\put(6,6){\Huge{$p_{34}=0$}}
        \end{picture}}
\end{center}
\caption{At left: the rectangles network
$G^{\rect}_{\lambda}$
 for the
Schubert variety associated to 
the partition 
$\lambda=(2,1)$.  To its right
are the $\Le$-diagrams for 
  the five codimension 1 positroids
     contained in $X_{(2,1)}$. The first two are Schubert divisors: they
      are the two irreducible components obtained when one 
     sets the frozen variable $p_{13}$ equal to $0$.}
\label{fig:positroidsubvariety}
\end{figure}

\begin{center}
	\begin{table}[h]
		\begin{tabular}{| c||p{.7cm}| p{.7cm} | p{.7cm}|}
			\hline
			Pl\"ucker & $\ydiagram{1}$ & $\ydiagram{2}$ & $\ydiagram{1,1}$ \\
			\hline
			\hline
			$\val(P_{13})$ & $0$ & $0$ & $0$\\ \hline
			$\val(P_{14})$ & $0$ & $0$ & $1$\\ \hline
			$\val(P_{23})$ & $0$ & $1$ & $0$\\ \hline
			$\val(P_{24})$ & $0$ & $1$ & $1$\\ \hline
			$\val(P_{34})$ & $1$ & $1$ & $1$\\ \hline
		\end{tabular}
		\vspace{0.2cm}
		\caption{The valuations of the 
		nonvanishing Pl\"ucker coordinates
		for $X_{(2,1)}$.}
		\label{table:Schub21}
	\end{table}
\end{center}

Note that the superpotential  
\begin{equation*}
W_{\rect}^{\lambda} = q_1 \frac{p_{\emptyset}}{p_{\ydiagram{2}}}
+q_2 \frac{p_{\emptyset}}{p_{\ydiagram{1,1}}} + \frac{p_{\ydiagram{2}}}{p_{\ydiagram{1}}} + 
\frac{p_{\ydiagram{1,1}}}{p_{\ydiagram{1}}}+
\frac{p_{\ydiagram{1}}}{p_{\emptyset}},
\end{equation*}
from 
\eqref{super:21} 
has five summands, which correspond to 
the five codimension $1$ positroid strata whose $\Le$-diagrams are shown in \cref{fig:positroidsubvariety}.

This Schubert variety $X_{(2,1)}$ is in fact a (Gorenstein Fano) toric variety, and $D^{(2,1)}_{\ac}$ is its toric boundary divisor. It has a small desingularisation to a toric variety of Picard rank $2$, 
but its own Picard rank is equal to $1$, suggesting the identification of $q_1$ and $q_2$ to a single quantum parameter $q$. There will be a Gorenstein toric Fano variety in the background also for general $X_\lambda$, whose  superpotential agrees with $W^\lambda_\rect$ up to possible identification of certain quantum parameters. In \cref{s:toric}
we will explain this connection and describe the superpotential of the associated toric variety using a  specific
\emph{starred quiver}
$\widetilde{Q}(\lambda)$
whose sink vertices correspond to  Picard group generators.

\subsection{The theta function basis and the 
proof that $\Delta^\lambda_G=\Gamma^\lambda_G$ for arbitrary seeds}\label{sec:arbitraryseeds}

In this section we start by showing
that there is a theta function basis for the coordinate 
ring of the cluster $\mathcal{X}$-variety $X_{\lambda}^{\circ}$.
We then use properties of the theta basis together with our results 
about the rectangles seed 
to prove in \cref{thm:main} that for any choice of cluster,
the Newton-Okounkov body coincides with the corresponding
superpotential polytope.

In order to show that the theta basis exists, we need to first 
prove some  results on optimized seeds.

\begin{definition}\cite[Definition 9.1 and Lemma 9.2]{GHKK}
For a cluster algebra coming from a quiver, a seed is 
\emph{optimized for a frozen variable} if and only
if in the quiver for this seed, all arrows between mutable vertices and 
the given frozen vertex point towards the
given frozen vertex.
\end{definition}

\begin{lemma}\label{lem:optimized}
Consider the cluster structure associated to a Schubert variety 
$\openSchub$.
 Every frozen variable has an optimized seed.
\end{lemma}

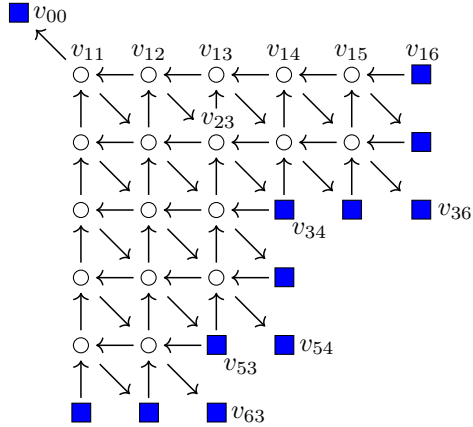
\begin{figure}[ht]

\begin{tikzpicture}
\node at (0.6,8.11) {$v_{00}$};
\filldraw[fill=blue] (0.05,8) rectangle (0.3,8.25);
\draw[->, line width=0.2mm] (0.8,7.5) -- (0.4,7.9);

\node at (1.1,7.6) {$v_{11}$};
\draw (1,7.3) circle (0.1cm);
\draw[->, line width=0.2mm] (1.7,7.3) -- (1.2,7.3);
\draw[->, line width=0.2mm] (1,6.6) -- (1,7.1);
\draw[->, line width=0.2mm] (1.25,7.05) -- (1.65,6.65);

\node at (1.9,7.6) {$v_{12}$};
\draw (1.9,7.3) circle (0.1cm);
\draw[->, line width=0.2mm] (2.6,7.3) -- (2.1,7.3);
\draw[->, line width=0.2mm] (1.9,6.6) -- (1.9,7.1);
\draw[->, line width=0.2mm] (2.15,7.05) -- (2.5,6.7);

\node at (2.8,7.6) {$v_{13}$};
\draw (2.8,7.3) circle (0.1cm);
\draw[->, line width=0.2mm] (3.5,7.3) -- (3,7.3);
\draw[->, line width=0.2mm] (2.8,6.85) -- (2.8,7.1);
\draw[->, line width=0.2mm] (3.05,7.05) -- (3.45,6.65);

\node at (3.7,7.6) {$v_{14}$};
\draw (3.7,7.3) circle (0.1cm);
\draw[->, line width=0.2mm] (4.4,7.3) -- (3.9,7.3);
\draw[->, line width=0.2mm] (3.7,6.6) -- (3.7,7.1);
\draw[->, line width=0.2mm] (3.95,7.05) -- (4.35,6.65);

\node at (4.6,7.6) {$v_{15}$};
\draw (4.6,7.3) circle (0.1cm);
\draw[->, line width=0.2mm] (5.3,7.3) -- (4.8,7.3);
\draw[->, line width=0.2mm] (4.6,6.6) -- (4.6,7.1);
\draw[->, line width=0.2mm] (4.85,7.05) -- (5.25,6.65);

\node at (5.55,7.6) {$v_{16}$};
\filldraw[fill=blue] (5.4,7.18) rectangle (5.65,7.43);

\draw (1,6.4) circle (0.1cm);
\draw[->, line width=0.2mm] (1.7,6.4) -- (1.2,6.4);
\draw[->, line width=0.2mm] (1,5.7) -- (1,6.2);
\draw[->, line width=0.2mm] (1.25,6.15) -- (1.65,5.75);

\draw (1.9,6.4) circle (0.1cm);
\draw[->, line width=0.2mm] (2.6,6.4) -- (2.1,6.4);
\draw[->, line width=0.2mm] (1.9,5.7) -- (1.9,6.2);
\draw[->, line width=0.2mm] (2.15,6.15) -- (2.55,5.75);

\node at (2.83,6.7) {$v_{23}$};
\draw (2.8,6.4) circle (0.1cm);
\draw[->, line width=0.2mm] (3.5,6.4) -- (3,6.4);
\draw[->, line width=0.2mm] (2.8,5.7) -- (2.8,6.2);
\draw[->, line width=0.2mm] (3.05,6.15) -- (3.45,5.75);

\draw (3.7,6.4) circle (0.1cm);
\draw[->, line width=0.2mm] (4.4,6.4) -- (3.9,6.4);
\draw[->, line width=0.2mm] (3.7,5.7) -- (3.7,6.2);
\draw[->, line width=0.2mm] (3.95,6.15) -- (4.35,5.75);

\draw (4.6,6.4) circle (0.1cm);
\draw[->, line width=0.2mm] (5.3,6.4) -- (4.8,6.4);
\draw[->, line width=0.2mm] (4.6,5.7) -- (4.6,6.2);
\draw[->, line width=0.2mm] (4.85,6.15) -- (5.25,5.75);

\filldraw[fill=blue] (5.4,6.28) rectangle (5.65,6.53);

\draw (1,5.5) circle (0.1cm);
\draw[->, line width=0.2mm] (1.7,5.5) -- (1.2,5.5);
\draw[->, line width=0.2mm] (1,4.8) -- (1,5.3);
\draw[->, line width=0.2mm] (1.25,5.25) -- (1.65,4.85);

\draw (1.9,5.5) circle (0.1cm);
\draw[->, line width=0.2mm] (2.6,5.5) -- (2.1,5.5);
\draw[->, line width=0.2mm] (1.9,4.8) -- (1.9,5.3);
\draw[->, line width=0.2mm] (2.15,5.25) -- (2.55,4.85);

\draw (2.8,5.5) circle (0.1cm);
\draw[->, line width=0.2mm] (3.5,5.5) -- (3,5.5);
\draw[->, line width=0.2mm] (2.8,4.8) -- (2.8,5.3);
\draw[->, line width=0.2mm] (3.05,5.25) -- (3.45,4.85);

\filldraw[fill=blue] (3.58,5.38) rectangle (3.83,5.63);
\filldraw[fill=blue] (4.48,5.38) rectangle (4.73,5.63);
\filldraw[fill=blue] (5.4,5.38) rectangle (5.65,5.63);
\node at (4.04,5.23) {$v_{34}$};
\node at (5.98,5.48) {$v_{36}$};

\draw (1,4.6) circle (0.1cm);
\draw[->, line width=0.2mm] (1.7,4.6) -- (1.2,4.6);
\draw[->, line width=0.2mm] (1,3.9) -- (1,4.4);
\draw[->, line width=0.2mm] (1.25,4.35) -- (1.65,3.95);

\draw (1.9,4.6) circle (0.1cm);
\draw[->, line width=0.2mm] (2.6,4.6) -- (2.1,4.6);
\draw[->, line width=0.2mm] (1.9,3.9) -- (1.9,4.4);
\draw[->, line width=0.2mm] (2.15,4.35) -- (2.55,3.95);

\draw (2.8,4.6) circle (0.1cm);
\draw[->, line width=0.2mm] (3.5,4.6) -- (3,4.6);
\draw[->, line width=0.2mm] (2.8,3.9) -- (2.8,4.4);
\draw[->, line width=0.2mm] (3.05,4.35) -- (3.45,3.95);

\filldraw[fill=blue] (3.58,4.48) rectangle (3.83,4.73);

\draw (1,3.7) circle (0.1cm);
\draw[->, line width=0.2mm] (1.7,3.7) -- (1.2,3.7);
\draw[->, line width=0.2mm] (1,3) -- (1,3.5);
\draw[->, line width=0.2mm] (1.25,3.45) -- (1.65,3.05);

\draw (1.9,3.7) circle (0.1cm);
\draw[->, line width=0.2mm] (2.6,3.7) -- (2.1,3.7);
\draw[->, line width=0.2mm] (1.9,3) -- (1.9,3.5);
\draw[->, line width=0.2mm] (2.15,3.45) -- (2.55,3.05);

\node at (3.13,3.42) {$v_{53}$};
\filldraw[fill=blue] (2.68,3.58) rectangle (2.93,3.83);
%\draw[->, line width=0.2mm] (2.8,3) -- (2.8,3.5);

\node at (4.14,3.68) {$v_{54}$};
\filldraw[fill=blue] (3.58,3.58) rectangle (3.83,3.83);

\filldraw[fill=blue] (0.88,2.68) rectangle (1.13,2.93);
\filldraw[fill=blue] (1.78,2.68) rectangle (2.03,2.93);
\filldraw[fill=blue] (2.68,2.68) rectangle (2.93,2.93);
\node at (3.23,2.78) {$v_{63}$};
\end{tikzpicture}
	\caption{
The labeled quiver $\Sigma_{\lambda}^{\rect}$ for $\openSchub$
where $\lambda = (6,6,4,4,4,3)$. 
\label{fig:opt}}
\end{figure}

\begin{proof}
For every frozen variable of $\Sigma_{\lambda}^{\rect}$,
we give a sequence of mutations that will produce a quiver
which is optimized for that frozen variable.  We label
vertices of the quiver $v_{r,c}$ where $r$ and $c$ denote
the row and column where the vertex is located, see \cref{fig:opt}.

	We divide up the frozen variables of $\Sigma_{\lambda}^{\rect}$
	into four groups:
	\begin{itemize}
	\item $v_{0,0}$;
	\item those corresponding to the outer corners of 
		$\lambda$, e.g.
	$v_{3,6}$, $v_{5,4}$, $v_{6,3}$ in \cref{fig:opt};
	\item those corresonding to the inner corners of $\lambda$,
		e.g. $v_{3,4}$, $v_{5,3}$ in \cref{fig:opt};
	\item all other frozen variables.
	\end{itemize}

	Note that the quiver $\Sigma_{\lambda}^{\rect}$ is already 
	optimized for the frozen variable $v_{0,0}$,
	 and for the frozen variables
	corresponding to the outer corners of $\lambda$.

	If $v_{i,j}$ is a frozen variable corresponding to an inner
	corner of $\lambda$, then we produce an optimized seed for 
	$v_{i,j}$ by mutating each mutable vertex in the diagonal
	of slope $-1$ containing $v_{i,j}$, from northwest to southeast.
	So for instance in \cref{fig:opt}, to find an optimized seed for 
	$v_{3,4}$, we mutate at $v_{1,2}$ then $v_{2,3}$.

Each remaining frozen variable is either the rightmost vertex in its
row (with at least one mutable vertex to the left), e.g.
$v_{1,6}$, $v_{2,6}$, $v_{4,4}$, 
or the bottommost vertex in its column (with at least one mutable
vertex above), e.g. $v_{3,5}$, $v_{6,2}$, $v_{6,1}$. 

To find a seed which is optimized for a frozen variable that is rightmost
in its row, we just mutate each mutable vertex in that row from left to right.
So in \cref{fig:opt}, to find an optimized seed for $v_{1,6}$,
we mutate $v_{1,1}, v_{1,2}, v_{1,3}, v_{1,4}$ then $v_{1,5}.$

To find a seed which is optimized for a frozen variable that is bottommost 
in its column, we just mutate each mutable vertex in that column 
from top to bottom.
So in \cref{fig:opt}, to find an optimized seed for $v_{6,2}$,
we mutate $v_{1,2}, v_{2,2}, v_{3,2}, v_{4,2}$ then $v_{5,2}.$

We leave it as an exercise for the reader to check that these simple mutation
sequences produce the requisite optimized seeds.
\end{proof}

\begin{theorem}\label{thm:theta}
There is a theta function basis $\mathcal{B}(X_{\lambda}^{\circ})$ for the coordinate ring
$\C[\widehat{X}_{\lambda}^{\circ}]$ of the affine
cone over the cluster $\mathcal{X}$-variety
$X_{\lambda}^{\circ}$, which restricts to a 
theta function basis
$\mathcal{B}(X_{\lambda})$ for the homogeneous
coordinate ring $\C[\widehat{X}_{\lambda}]$
of the Schubert variety.
And $\mathcal{B}(X_{\lambda})$ restricts to a basis
$\mathcal{B}_r$ of the degree $r$ component
of the homogeneous coordinate ring, for every positive
$r$.
\end{theorem}

\begin{proof}
Gross-Hacking-Keel-Kontsevich \cite[Theorem 0.3]{GHKK} showed that canonical bases of global regular
``theta" functions exist for a formal version of cluster varieties, and in many cases
(when ``the full Fock-Goncharov conjecture holds"), these extend to bases for regular functions
on the actual cluster varieties.  
They pointed out 
in \cite[Proposition 8.28 and Cor 8.30]{GHKK} 
that the full Fock-Goncharov 
conjecture holds if
there is a \emph{maximal green sequence}
(or more generally a \emph{green-to-red sequence}).
In the case of open positroid varieties
(of which $X_{\lambda}^{\circ}$ is an example),
a green-to-red sequence was found in
 \cite[Theorem 1.2]{FordSer}.
Therefore 
we indeed have a theta function basis 
$\mathcal{B}(X_{\lambda}^\circ)$ for 
the coordinate ring $\C[\widehat{X}_{\lambda}^\circ]$ 
of the affine cone over the cluster $\mathcal{X}$-variety
$X_{\lambda}^{\circ}$.

We now use our result (\cref{lem:optimized})
that every frozen variable has an optimized seed to 
show that there is also a theta function basis
$\mathcal{B}(X_{\lambda})$
for the  coordinate ring
$\C[\widehat{X}_{\lambda}]$ of the affine cone over the Schubert variety.
We use results from \cite[Section 9]{GHKK}
which give conditions for when  \cite[Theorem 0.3]{GHKK} extends to partial compactifications of cluster
varieties coming from frozen variables.  In particular, 
\cite[Corollary 9.17]{GHKK} says that if every frozen variable 
has an optimized seed, then the theta basis for the cluster variety 
restricts to a theta basis for its partial compactification 
coming from the frozen variables.  While \cite[Section 9]{GHKK} 
works in the setting of cluster $\mathcal{A}$-varieties, 
the \emph{twist} automorphism \cite[Theorem 7.1]{MullerSpeyer} for open Schubert varieties
maps the cluster $\mathcal{X}$-tori to the cluster $\mathcal{A}$-tori,
so  we can apply 
\cite[Corollary 9.17]{GHKK} to our theta function basis for the 
cluster $\mathcal{X}$-variety $X_{\lambda}^{\circ}$.
In particular, it 
 restricts to a basis 
$\mathcal{B}(X_{\lambda})$
for $\C[\widehat{X}_{\lambda}]$.

Finally $\mathcal{B} (X_{\lambda})$
restricts to a basis of $L_r$ because it is compatible with the
one-dimensional
torus action (which is overall scaling in the
Pl\"ucker embedding).
\end{proof}

The reason that it is useful to have the theta basis is the following result.
\begin{theorem} \cite[Theorem 16.15]{RW} \label{thm:pointed}
Fix a cluster $\mathcal{X}$-variety and an arbitrary $\mathcal{X}$-chart.
Then every element of the theta function basis can be written as a pointed Laurent
polynomial in the variables of the $\mathcal{X}$-chart.  Moreover the exponents
of the leading terms are all distinct.
\end{theorem}

\begin{corollary}\label{cor:bijval}
Suppose that $G$ and $G'$ index two $\mathcal{X}$-seeds connected by a single mutation.
Then the tropicalized $\mathcal{A}$-cluster mutation
$\Psi_{G,G'}$ is a bijection
$$\Psi_{G,G'}: \val_G(L_r) \to \val_{G'}(L_r),$$
where $L_r$ is the linear subspace defined in 
\eqref{e:projnormal}.
\end{corollary}
\begin{proof}
This follows from  the proof of \cite[Lemma 16.16 and Lemma 16.17]{RW}, 
	using \cref{thm:pointed}, 
as well as the facts that
	$\conv_{\rect}^{\lambda}=\Q^\lambda_{\rect}$ (\cref{c:ConvEqualsGamma})
	and $\Gamma_{\rect}^{\lambda}$ has the integer decomposition property  (\cref{cor:IDP}).
\end{proof}

\begin{theorem}\label{thm:main}
Let $\Sigma_G^{\mathcal{X}}$ be an arbitrary
$\mathcal{X}$-cluster seed for the open Schubert variety $X^{\circ}_{\lambda}$.
Then
 the Newton-Okounkov body
$\Delta^{\lambda}_{G}$
is a rational polytope
with   lattice points 
$\{\val_G(P_{\mu}) \ \vert \  \mu \subseteq \lambda\}$, 
	and it coincides with 
the superpotential polytope
$\Gamma^{\lambda}_{G}.$
\end{theorem}

\begin{proof}
We know from \cref{c:ConvEqualsGamma} that \cref{thm:main}
holds when $G = G_{\rect}$.
We also know from \cref{c:GammaMutation} that 
if $G$ and $G'$ are related by mutation at vertex $\nu$,
	then the tropicalized $\mathcal{A}$-cluster mutation
	$\Psi_{G,G'}$ restricts to a bijection
	\[
\Psi_{G,G'}:
        \Q_G^\lambda %(\rr,\rr') \to 
	\to 
        \Q_{G'}^\lambda.  %(\rr,\rr').
\]

Since the Newton-Okounkov body is given by 
\begin{equation*}
\Delta_G^{\lambda} =
\overline{\operatorname{ConvexHull}
\left(\bigcup_r \frac{1}{r}
\val_G(L_r)\right)}, 
\end{equation*}
\cref{cor:bijval} implies that 
	$\Psi_{G,G'}$ is a bijection 
	\[
\Psi_{G,G'}:
        \Delta_G^\lambda %(\rr,\rr') \to 
	\to 
        \Delta_{G'}^\lambda.  %(\rr,\rr').
\]
The fact that the two polytopes agree now follows.

Finally since the tropicalized $\mathcal{A}$-cluster mutation
maps lattice points to lattice points, it follows that the set of lattice points for both $\Delta_G^{\lambda}$ and $\Gamma_G^{\lambda}$ is 
$\{\val_G(P_{\mu}) \ \vert \  \mu \subseteq \lambda\}$.
\end{proof}

The following result is a consequence of
 \cite[Section 17]{RW} (which builds on work of \cite{Anderson})
 and our result from \cref{thm:main}
 that our Newton-Okounkov bodies $\Delta_G^{\lambda}$ are rational polytopes.
 Note that in order to use the results of \cite{Anderson} we need to work with 
 an ample divisor. Our preferred divisor $D=D_{(\mathbf{1},\mathbf{0})}=\{P_\lambda=0\}$ is an example. In contrast, $D_{(\mathbf{1},\mathbf{1})}$ is only ample if $X_\lambda$ is Gorenstein. 

%\Comment{We need to use a Cartier divisor to use Dave Anderson's work
%in \cite{Anderson}}

\begin{cor}\cite[Corollary 17.11]{RW}\label{cor:degeneration}
Let $\Sigma_G^{\mathcal X}$ be an arbitrary $\mathcal X$-cluster seed
for $X_{\lambda}$, and consider the corresponding  Newton-Okounkov body
$\Delta_G^{\lambda}$.
Let $\RG$ denote the minimal positive integer such that
the dilated polytope
$\RG\NO_G$ has the \emph{integer decomposition property}.  (This exists
since $\Delta_G^{\lambda}$ is a rational polytope.)
Then we have a flat degeneration of $X_{\lambda}$
to the normal projective toric variety
associated to the polytope $\RG\Delta_G$ (i.e. to
        the Newton-Okounkov body associated to the rescaled divisor).
\end{cor}

In the special case that $G=G_{\rect}$, \cref{cor:degeneration}
plus \cref{p:unimodular}
recovers the following result
of Gonciulea and Lakshmibai.
\begin{cor}\cite[Theorem 7.34]{GL}
We have a flat degeneration of $X_{\lambda}$ to
the \emph{Hibi toric variety} associated to the order polytope
$\OO(\lambda)$ 
 of the poset $P(\lambda)$
of rectangles contained in $\lambda$.  
\end{cor}

We note that the above degeneration was also a key ingredient in  the work of Miura
\cite{Miura:minuscule}, who studied the mirror symmetry of smooth complete
intersection Calabi-Yau $3$-folds in minuscule Gorenstein Schubert varieties by 
degenerating the ambient Schubert varieties to Hibi toric varieties.

\begin{remark} \label{rem:Khovanskii}
We know  from  \cref{cor:IDP2}
that in the case of the rectangles cluster $G=G_{\lambda}^{\rect}$,
the Newton-Okounkov body
 $\Delta_{\rect}^{\lambda}$ has the integer decomposition property,
and its lattice points are precisely
the valuations of Pl\"ucker coordinates.
It now follows as in the proof of \cite[Corollary 17.10]{RW}
that the Pl\"ucker coordinates from the rectangles seed of 
$X_{\lambda}$ form a \emph{Khovanskii} or \emph{SAGBI basis} 
(as in \cite[Definition 17.1]{RW})
of the homogeneous coordinate ring of $X_{\lambda}$.  More formally,
let
	$R_{\lambda}:=\bigoplus_j t^j L_j$, where
	$L_j=H^0(X_{\lambda},\mathcal O(jD))$ 
	as in \eqref{e:projnormal}, and 
	$D=\{P_{\lambda}=0\}$ is our preferred ample divisor corresponding to the 
	Pl\"ucker embedding;
	note that $R_{\lambda}$ is isomorphic to the homogeneous 
	coordinate ring of $X_{\lambda}$. 
	Consider the extended valuation $\overline \val_G:R_{\lambda} \setminus\{0\}\to 
	\Z\times \Z^{\mathcal P_G}$ defined by
	\begin{eqnarray}\label{e:extendedval}
		\overline{\val_G}: R_{\lambda}\setminus\{0\}&\to &\Z\times \Z^{\mathcal P_G},\\
\sum t^j f^{(j)} &\mapsto & \left(j_0,\val_G (f^{(j_0)})\right),
\end{eqnarray}
where $j_0=\max\{j\mid f^{(j)}\ne 0\}$.
Then
the set  $\{tP_{\mu}/P_\lambda\mid \mu\subseteq \lambda\}$ is a Khovanskii basis for
	$(R_{\lambda}, \overline{\val}_G)$.%
\end{remark}

\section{The max-diagonal formula for lattice points}\label{s:MaxDiag}

In this section we will prove a  ``max diagonal'' formula for valuations
of Pl\"ucker coordinates in Schubert varieties, see \cref{p:latticepoints-comp}.
This result will be used in the proof of 
\cref{t:maingen}, and it 
generalizes 
our previous result  \cite[Theorem 15.1]{RW} in the  Grassmannian setting.
Our proof of 
\cref{p:latticepoints-comp}
uses the geometry of how Newton-Okounkov bodies for Schubert varieties
sit inside Newton-Okounkov bodies for the Grassmannian, and it uses
the flow polynomials for general plabic graphs from \eqref{eq:Plucker2}.

\begin{theorem}\label{p:latticepoints-comp}
Suppose $\kappa\subseteq\nu$ indexes a  Pl\"ucker coordinate $P_\kappa$ for $X_\nu$. For any reduced plabic graph $G$ for $X_\nu$ and any $\eta\in\mathcal P_G(\nu)$ we have the formula
\begin{equation}\label{e:GenMaxDiag}
\val_G(P_\kappa)_\eta=\maxdiag(\eta\setminus\kappa),
\end{equation}
where $\maxdiag$ is as in \cref{d:MaxDiagVectors}.
If $\nu\subset\lambda$, then there exists a seed $G'$ of  $X_\lambda$ extending $G$, such that we have an embedding $\iota:\Delta^\nu_G\to \Delta^\lambda_{G'}$ 
satisfying 
\begin{equation}\label{e:iotaonvalG}
\iota(\val_G(P_\kappa))=\val_{G'}(P_\kappa)
\end{equation}
for all Pl\"ucker coordinates $P_\kappa$ of $X_\nu$.
Moreover, the set of lattice points of the face $F_{G'}^\nu=\iota(\Delta^\nu_G)$ of $\Delta^\lambda_{G'}$ is precisely the set
$
\{\val_{G'}(P_\kappa)\mid \kappa\subseteq \nu\}.
$
\end{theorem} 

\cref{p:latticepoints-comp} is a generalization of  
\cref{thm:maxdiag} below,
 which was our previous
result in the Grassmannian case.
\begin{theorem}\cite[Theorem 15.1]{RW}\label{thm:maxdiag}
Let $G$ be a reduced plabic graph for the Grassmannian, let $\kappa$ be a 
partition indexing a Pl\"ucker coordinate $P_\kappa$, and let $\eta \in \mathcal{P}_G$.  Then we have
\begin{equation}
\val_G(P_\kappa)_{\eta} = \maxdiag(\eta \setminus \kappa),
\end{equation}
where $\maxdiag$ is as in \cref{d:MaxDiagVectors}.
\end{theorem}

 We recall the following result from \cite{RW}.
\begin{theorem}\cite[Theorem 13.1]{RW}\label{thm:flowmutate}
Suppose that $G$ and $G'$ are reduced plabic graphs, which are related by a single
move, and let ${\kappa}$ be a partition indexing a Pl\"ucker coordinate that is nonzero on the 
open positroid variety $X_G^{\circ}$.  
If $G$ and $G'$ are related by one of the moves (M2) or (M3), then
$\val_G(P_{\kappa}) = \val_{G'}(P_{\kappa})$.
If $G$ and $G'$ are related by move (M1), then
	\begin{equation*}
		\val_{G'}(P_{\kappa}) = \Psi_{G,G'}(\val_G(P_{\kappa})),
	\end{equation*}
	for $\Psi_{G,G'}$ the tropicalized $\mathcal{A}$-cluster mutation from 
	\cref{d:PsiGG'}.
\end{theorem}

\begin{remark}\label{r:MaxDiagMutationProof}
One possible approach to proving the identity  \eqref{e:GenMaxDiag} 
in the case of a Schubert variety $X_\nu$  is to follow the proof of \cref{thm:maxdiag}
given in \cite{RW}, using \cref{thm:flowmutate} (which was proved
in \cite{RW} for the case of the full Grassmannian but whose proof
applies to the general case). We will sketch such a proof below, then give
an alternative proof.
\end{remark}

\begin{proof}[Proof sketch of 
	\cref{e:GenMaxDiag}]
We know by \cref{c:ConvEqualsGamma} that 
the formula holds for  the rectangles cluster.  
One can 
then provide an explicit construction of a tropical point of $\check X_\lambda$ whose (tropical) Pl\"ucker coordinates can all be evaluated and shown to be given by the max-diag formula. This implies that the  max-diag formula is compatible with  tropicalised $\mathcal A$-cluster mutation.  	
Then one can use \cref{thm:flowmutate}, which says that valuations of flow polynomials are compatible with the tropicalised $\mathcal A$-cluster mutation.
Thus the formula \eqref{e:GenMaxDiag} that we already proved for $G=G_{\rect}$ holds true for every Pl\"ucker cluster.
\end{proof}

We now give a different proof of \cref{p:latticepoints-comp} 
that makes use of the 
compatibility of the cluster structures of the different Schubert varieties under inclusion.

Recall from \cref{cor:IDP2} that $\Delta_{\rect}^{\lambda}$ has the integer decomposition property, and 
that for $\nu \subseteq \lambda$, $\Delta_{\rect}^{\nu}$ embeds as a face of $\Delta_{\rect}^{\lambda}$.
Having the integer decomposition property, and even being integral, is a special feature associated to the rectangles cluster that doesn't necessarily hold for general $G$, see~\cite[Section~9]{RW}. The property that $\Delta^\nu_{\rect}$ is naturally identified with a face of $\Delta^\lambda_{\rect}$ does, however, generalise beyond the rectangles seed, as we will now see.

\begin{notation} \label{def:extend}
Throughout this section we will let $\nu \subseteq \lambda$ be 
partitions, $G$  an arbitrary seed for $X_{\nu}$, and $G'$  the seed 
for $X_{\lambda}$ such that $G$ is a restricted seed obtained from $G'$ (see
\cref{lem:restrictedseed}).
\end{notation}

Let us denote the ambient spaces for $\Delta^\lambda_{G'}$ and $\Delta^\nu_{G}$ by $\R^{\mathcal P_{G'}(\lambda)}$ and $\R^{\mathcal P_{G}(\nu)}$, respectively. Note that we have a natural inclusion $\mathcal P_{G}(\nu)\subseteq\mathcal P_{G'}(\lambda)$ and associated to it a standard projection $\pi:\R^{\mathcal P_{G'}(\lambda)}\to\R^{\mathcal P_{G}(\nu)}$. 

\begin{prop}\label{p:FaceProp}
Let $\nu\subset\lambda$ be partitions and $G$ a seed for $X_\nu$. Then there is a seed $G'$ for $X_\lambda$ such that we have an  embedding $\iota:\Delta^\nu_G\hookrightarrow\Delta^\lambda_{G'}$ identifying $\Delta^\nu_G$ unimodularly with a face $F_{G'}^\nu$ of $\Delta^\lambda_{G'}$.
The inverse map $F_{G'}^\nu\to \Delta^\nu_{G}$ is obtained as a restriction of the coordinate projection $\pi:\R^{\mathcal P_{G'}(\lambda)}\to\R^{\mathcal P_{G}(\nu)}$.
\end{prop}

\begin{proof}
It suffices to prove this proposition in the case that $\nu$ is obtained from $\lambda$ by the removal of a single box. The proposition then follows for general $\nu\subset \lambda$ by induction. So let us assume $\nu=\lambda\setminus b$ for $b=b_{\rho_{2\ell+1}}$ the $\ell$-th outer corner of $\lambda$. We 
freely use the identity $\Delta^\lambda_G=\Q^\lambda_G$ from 
	Theorem~\ref{thm:main} 
	to describe $\Delta^\lambda_G$ in terms of facet inequalities.
	
We begin by considering the summand, 
\begin{equation}\label{e:WjFaceProp}
q_\ell W_\ell
	=q_\ell\frac{p_{\mu_{\rho_{2\ell+1}}^-}}{p_{\mu_{\rho_{2\ell+1}}}}
	=q_\ell\frac{p_{\mu_{\rho}^-}}{p_{\mu_{\rho}}},
\end{equation}
of the superpotential $W^\lambda$ (cf. \cref{d:LG1}) corresponding to the removable box $b$,
	where we use $\rho$ to denote $\rho_{2\ell+1}$.
	Note that $\mu_{\rho}$ is the rectangle containing the outer corner $b$ of $\lambda$ in its SE corner and is the unique rectangle in $\lambda$ that does not lie in $\nu$. The rectangle $\mu_{\rho}^-$ is obtained by removing the rim from $\mu_{\rho}$. Note that $\mu_{\rho}^-$ corresponds to a frozen variable for the $\mathcal A$-cluster structure on $X_\nu$.  
By \cref{lem:restrictedseed}, there exists a 
cluster $G'$ of $X_\lambda$ such that $G$ is obtained from $G'$
by restriction.  Therefore $G'$ contains both $p_{\mu_{\rho}}$ and $p_{\mu_{\rho}^-}$, that is,
\[\mu_{\rho},\mu_{\rho}^-\in\mathcal P_{G'}(\lambda).\] 
Therefore the expression \eqref{e:WjFaceProp} is a $G$-cluster expansion for $q_j W_j$ and gives rise to an inequality, namely
\begin{equation}\label{e:inequalityWj}v_{\mu_{\rho}}- v_{\mu_{\rho}^-}\le 1,
\end{equation}
on the points $v$ of $\Delta^\lambda_G$. 

We have $\nu=\lambda\setminus b_{\rho}$ and $\mathcal P_G(\nu)=\mathcal P_{G'}(\lambda)\setminus\{\mu_{\rho}\}$. Let us map $\R^{\mathcal P_G(\nu)}$ into  $\R^{\mathcal P_{G'}(\lambda)}=\R^{\mathcal P_G(\nu)}\times \R$
via $v\mapsto ( v,v_{\mu_{\rho}^-}+1)$. We use this embedding to identify $\R^{\mathcal P_G(\nu)}$ unimodularly with the affine hyperplane in $\R^{\mathcal P_{G'}(\lambda)}$ defined by 
\begin{equation}\label{e:FacetHyperplane}
v_{\mu_{\rho}}- v_{\mu_{\rho}^-}= 1.
\end{equation}
If $G=G_\rect$ it is straightforward to see that this affine hyperplane cuts out a facet in $\Delta^\lambda_\rect$, 
	and this facet is precisely the one identified with $\Delta^\nu_\rect$ in \cref{cor:IDP2}. 
	
For the general $G$ case we first apply \cref{c:GammaMutation}, which says that 
a sequence of tropical $\mathcal A$-cluster mutations  takes $\Delta_\rect^\lambda$ bijectively to $\Delta_G^\lambda$ (where we are using also Theorem~\ref{thm:main}). Since none of these mutations affect the coordinates 
$v_{\mu_{\rho}}$ and $v_{\mu_{\rho}^-}$, by our assumptions on $G'$, we have that their composition restricts to give a piecewise linear bijection $\Psi$ between
\[
\Delta^\nu_\rect\ \hat{=}\ \Delta_\rect^\lambda \cap \{v_{\mu_{\rho}}- v_{\mu_{\rho}^-}=1\} \qquad \text{ and }\qquad
\Delta_{G'}^\lambda \cap \{v_{\mu_{\rho}}- v_{\mu_{\rho}^-}=1\}.
 \]
Since this bijection preserves dimension, it follows that $\Delta_{G'}^\lambda \cap \{v_{\mu_{\rho}}- v_{\mu_{\rho}^-}=1\}$ has codimension $1$ in $\Delta_{G'}^\lambda$. Since we know that $\Delta_{G'}^\lambda$ satisfies the inequality~\eqref{e:inequalityWj}, we can deduce that $\Delta_{G'}^\lambda \cap \{v_{\mu_{\rho}}- v_{\mu_{\rho}^-}=1\}$  is a facet of $\Delta_{G'}^\lambda$. Using the compatibility result from Lemma~\ref{l:freezingcommutesmutation} we see that the image of $\Delta^\nu_\rect$ under $\Psi$ is in fact also naturally identified with $\Delta^\nu_G$. Therefore 
\[
\Delta^\nu_G\ \hat=\ \Delta_{G'}^\lambda \cap \{v_{\mu_{\rho}}- v_{\mu_{\rho}^-}=1\},
\]
so that $\Delta^\nu_G$ is indeed identified with a facet $F_{G'}^\nu$ of $\Delta_{G'}^\lambda$. 

Finally, recall that we had identified the ambient space of $\Delta^\nu_G$ with an affine subspace of $\R^{\mathcal P_{G'}(\lambda)}$ via the embedding $\R^{\mathcal P_G(\nu)}\hookrightarrow \R^{\mathcal P_{G'}(\lambda)}$ given by $v\mapsto ( v,v_{\mu_\rho})$ where $v_{\mu_\rho}=v_{\mu_{\rho}^-}+1$. Thus we obtain the embedding that we call $\iota$, that sends $\Delta^\nu_G$ isomorphically to a face $F_{G'}^\nu$ of $\Delta^\lambda_{G'}$. It follows directly from the formula for $\iota$ that the inverse map, $F_{G'}^\nu\overset\sim\to \Delta^\nu_G$ is indeed the restriction of the projection map that forgets the coordinate~$v_{\mu_\rho}$. 
\end{proof}

We note that the above proof provides a concrete, recursive construction of the embedding of the polytope $\Delta^\nu_G$ into $\Delta^\lambda_{G'}$.

\begin{proof}[Proof of 
	\cref{p:latticepoints-comp}] To prove this theorem we must verify \eqref{e:iotaonvalG}.
Let us assume, as in the proof of \cref{p:FaceProp}, that $\nu=\lambda\setminus b$ and prove the statement recursively. We use the notations from that proof as needed. Note that $G$ is now given in terms of a plabic graph.
\begin{comment}
Note that if $G=G_{\rect}$, then the special lattice point $\val_{G'_\rect}(P_\lambda)\in \Delta^\lambda_{G'_\rect}$ lies outside of the face $F^\nu_{G'_\rect}$ isomorphic to $\Delta_{G_\rect}^\nu$. Explicitly, we see this as follows. By construction (or by the Max-Diag formula from \cref{p:FaceProp})  we have that $\val_{G'_\rect}(P_\lambda)=(0,\dotsc,0)$. This shows directly that the point does not lie on the hyperplane \eqref{e:FacetHyperplane} defining the facet $F^\nu_{G'_{\rect}}$. 

Since the mutations we are allowed to perform to go from the rectangles cluster of $X_\lambda$ to $G'_\rect$ do not affect the coordinates $v_{\mu_{\rho}}$ and $v_{\mu_{\rho}^-}$ involved in \eqref{e:FacetHyperplane}, we have that 
\[
\val_{G'}(P_\lambda)\notin F_{G'}^\nu.
\]
In fact, the mutation formulas imply that $\val_{G'}(\lambda)$ has all coordinates equal to $0$. 

Analogously, we see that $\val_{G'}(P_\nu)$ has coordinate $v_{\mu_\rho}=1$, and $v_{\mu_\rho^-}$ and all other coordinates are equal to~$0$. Therefore, $\val_{G'}(P_\nu)\in F^{\nu}_{G'}$ and 
\[
\val_{G'}(P_\nu)=\iota(\val_G(P_\nu))
\]
since $\val_G(P_\nu)=(0,\dotsc, 0)$ and so this precisely adds a coordinate $1$.

We can also deduce that $\Delta^\nu_G$ and $\Delta^\lambda_{G'}$ have $n_\lambda:=\#\{\mu\mid\kappa\subseteq\lambda\}$ and $n_\nu=n_\lambda-1$ lattice points,  respectively, from the analogous fact for the rectangles cluster. This is because mutation is bijective on lattice points. 
\end{comment}

Let $J_\nu=(m_1,\dotsc, m_{n-k})$ denote the labels of the vertical 
steps in the path $\pathsw{\nu}$ associated to $\nu$ as in \cref{s:Young}.
Then $J_\nu$ is the index set for the lexicographically minimal nonvanishing Pl\"ucker
coordinate on $X_{\nu}$, and by \cref{acyclic},
	there is a unique acyclic perfect orientation $\O$ for $G$ with source set $J_{\nu}$.

Let $\kappa\subseteq\nu$ and 
consider the flow polynomial $P^G_\kappa((x_\eta)_{\eta\in\mathcal P_G(\nu)})$ 
expression for $P_\kappa$ (see \cref{def:flow} and \cref{thm:Talaska}) associated to 
the perfect orientation $\O$ of $G$.
	Recall that $\val_G(P_\kappa)_\eta$ is the degree of $x_\eta$ in the minimal degree term of  $P^G_\kappa((x_\eta)_{\eta\in\mathcal P_G(\nu)})$.

Since $\lambda$ is obtained by adding a box to $\nu$ we have that $J_\lambda$ is of the form
\[
J_\lambda=(m_1,\dotsc,m_{r-1},m_r-1,m_{r+1}\dotsc, m_{n-k}). 
\]    
The seed $G'$ for $X_{\lambda}$ is now also given by a plabic graph. 
Namely, we can construct $G'$ along with an acyclic perfect orientation $\O'$  
from $G$ and $\O$ by extending the  wires labelled $m_r-1$ and $m_r$ to 
create an extra bounded region as in \cref{f:networkextension}.
This is now our plabic graph $G'$,  together with an acyclic perfect orientation $\O'$ 
with sources at $J_\lambda$. 

Note that there is a rectangle $\mu_{\rho}^-$ in $\mathcal P_G(\nu)$ which 
labels the new bounded region in $G'$. 
Moreover, $\mathcal P_{G'}(\lambda)=\mathcal P_{G}(\nu)\cup \{\mu_\rho\}$, and
the new label $\mu_\rho$ labels the region southeast of the region labeled~$\mu^-_\rho$, 
as in the figure.
 
\begin{figure}[h]
\centering
   \includegraphics[height=1.8in]{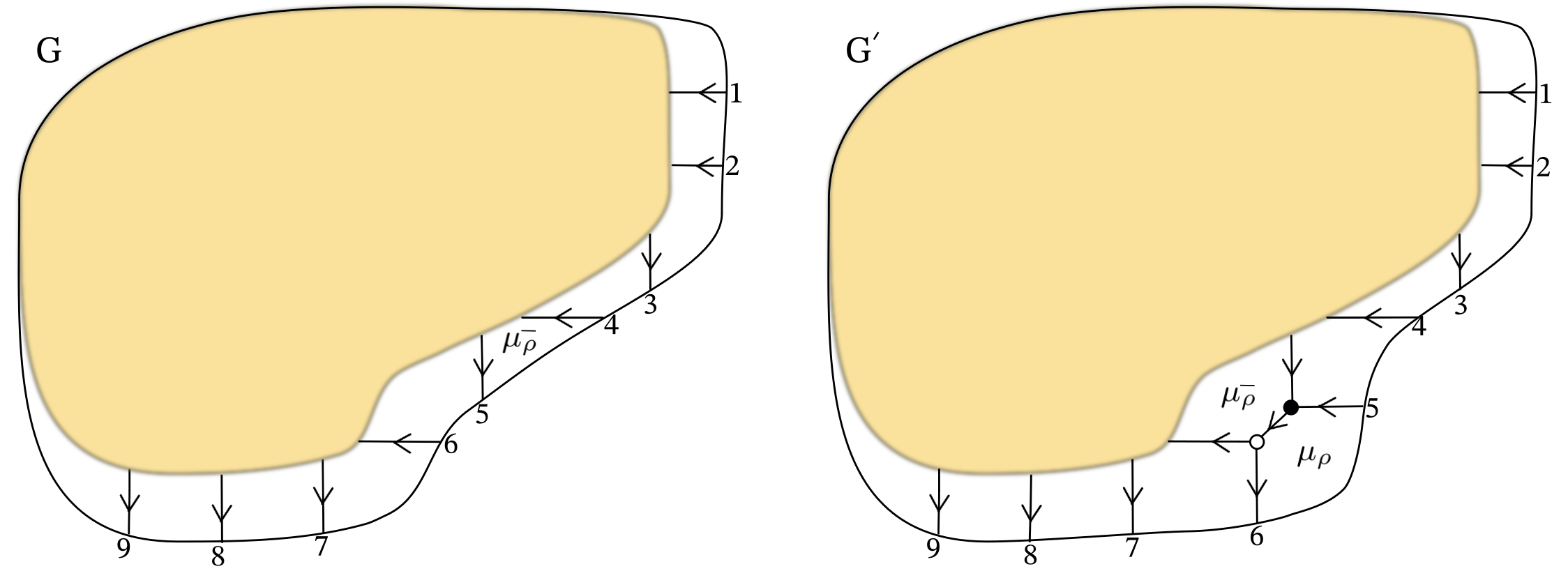}
     \caption{The construction of the perfect orientation 
	  of the plabic graph $G'$ for $X_\lambda$ from the perfect orientation
	   of $G$ for $X_\nu$. Here $m_r=6$, and 
	   $\nu=(5,5,4,3)$ and $\lambda=(5,5,4,4)$ so that 
	   $J_\nu=\{1,2,4,6\}$ and $J_\lambda=\{1,2,4,5\}$.}
\label{f:networkextension}
   \end{figure}

Consider a flow $\mathcal F$ in $\O$ 
from the sources $J_\nu$ to $J=J_\kappa$. It is straightforward to check that 
there is a unique way to extend $\mathcal F$ to 
a flow $E_{G,G'}(\mathcal F)$ 
 in $\O'$ from $J_\lambda$ to $J$. 
For example in  \cref{f:networkextension}, if $5,6\notin J$ then the extended flow has a path entering at $5$ and continuing to the left. If $5\notin J, 6\in J$, then the extended flow has a path entering at $5$ and turning straight away to exit at $6$. If $5\in J$ the new flow will have a stationary path at $5$. In this case either $6\in J$, in which case the extended flow has a path coming vertically down and exiting at  $6$, or $6\notin J$ in which case  the path coming vertically down turns right at the white vertex 
and continues in $\O'$. 

We note that the extension $E_{G,G'}(\mathcal F)$ 
always includes a path that separates the ${\mu^-_\rho}$ region 
from the ${\mu_\rho}$ region. Conversely, all flows in $\O'$ 
from $J_\lambda$ to $J$ with such a path arise in this way as extensions. 

Consider the weight $\wt(\mathcal F)$ 
of a flow $\mathcal F$ in $\O$,
\begin{equation}\label{e:basicmonomial}
	\wt(\mathcal F)	=\prod_{\eta\in\mathcal P_{G}(\nu)}x_\eta^{c_\eta}.
\end{equation}
For the extension $E_{G,G'}(\mathcal F)$ we then get the weight 
\begin{equation}\label{e:extendedmonomial}
	\wt(E_{G,G'}(\mathcal F))
	=\prod_{\eta\in\mathcal P_{G'}(\lambda)}x_\eta^{c_\eta}=\left(\prod_{\eta\in\mathcal P_{G}(\nu)}x_\eta^{c_\eta}\right)x_{\mu_\rho}^{(c_{\mu^-_\rho}+1)}.
\end{equation}
Here the exponent of the new variable $x_{\mu_\rho}$ is  $c_{\mu_\rho}=c_{\mu^-_\rho}+1$, because of the path in $E_{G,G'}(\mathcal F)$ separating $x_{\mu^-_\rho}$ from $x_{\mu_\rho}$. Note in particular that $c_{\mu^-_\rho}<c_{\mu_\rho}$ for $E_{G,G'}(\mathcal F)$.

Let us write $\mathcal F_{\min}$ for the minimal flow in $\O$, 
which is the one contributing the minimal order term to the flow polynomial $P^G_\kappa$.
\vskip .2cm
\noindent{\it Claim:~}
The extension $E_{G,G'}(\mathcal F_{\min})$ is the minimal flow in $\O'$
from $J_\lambda$ to $J$, 
contributing the  minimal order term to the flow polynomial $P^{G'}_\kappa$.
\vskip .2cm

We now prove this claim. Note that a flow in $\O'$
from $J_\lambda$ to $J$ need not be an extension of a flow in $\O$ 
from $J_\nu$ to $J$. The key point is to prove that the minimal one is such an extension. 
Let us denote the minimal flow in $\O'$ from $J_\lambda$ to $J$ by $\mathcal F'_{\min}$ and its weight by 
\[
\wt({\mathcal F'_{\min}})
=\prod_{\eta\in\mathcal P_{G'}(\lambda)}x_\eta^{d_\eta}.
\]
We compare this weight to the weight 
for $E_{G,G'}(\mathcal F)$ from \eqref{e:extendedmonomial}.
Note that $d_{\mu_\rho}$ depends only on $J_\lambda$ and $J$, as it counts the number of paths in the flow that enter before $m_r-1$ and exit after $m_r-1$. Therefore $d_{\mu_\rho}=c_{\mu_\rho}$. By the minimality of $\mathcal F'_{\min}$  we have that $d_{\mu_\rho^-}\ \le\ c_{\mu_{\rho}^-}$. It follows, using also \eqref{e:extendedmonomial}, that
\[
d_{\mu_\rho^-}\ \le\ c_{\mu_{\rho}^-}\ <\ c_{\mu_\rho}\ =\ d_{\mu_\rho}.
\] 
This implies that the flow $\mathcal F'_{\min}$ must contain a path separating ${\mu_\rho^-}$ from ${\mu_\rho}$. If $\mathcal F'_{\min}$ contains such a path then it is an extension of a flow from $J_\nu$ to $J$ in the $\O$ (that flow being the restriction of $\mathcal F'_{\min}$). The Claim now follows. Namely, since $\mathcal F'_{\min}$ is minimal and is an extension, it must be the extension of the minimal flow~$\mathcal F_{\min}$ in $\O$.
\vskip .2cm
We now have that 
$\wt({\mathcal F_{\min}})$ 
and $\wt({E_{G,G'}(\mathcal F_{\min})})$ 
are the leading terms of $P^G_\kappa$ and 
 $P^{G'}_\kappa$, respectively. Comparing the formulas~\eqref{e:basicmonomial} and \eqref{e:extendedmonomial} we see that,
\begin{equation}\label{e:iotanadpi}
\iota(\val_G(P_\kappa))= \val_{G'}(P_\kappa)\quad \text{ and }\quad \pi(\val_{G'}(P_\kappa))=\val_G(P_\kappa),
\end{equation}
where $\iota$ is as constructed in the proof of \cref{p:FaceProp}, and $\pi$ is the coordinate projection that forgets the $v_{\mu_\rho}$-coordinate. 

The theorem now follows. Namely, the Max-Diag formula follows from the full Grassmannian case proved in \cite{RW} by recursively applying the second part of \eqref{e:iotanadpi}. Equation~\eqref{e:iotaonvalG} follows recursively from the first part of~\eqref{e:iotanadpi}. Finally, combining Proposition~\ref{p:FaceProp} and Theorem~\ref{thm:main} with equation \eqref{e:iotaonvalG} gives the desired description of the lattice points of the face $F^{\nu}_{G'}$.
\end{proof}

\begin{remark}
With the notations as in the proof above we have the following relationship 
	between the flow polynomial for $P_{\kappa}$ in $\O'$ and 
	 the flow polynomial for $P_{\kappa}$ in $\O.$ 
	 Namely, set 
\[
x'_\eta=\begin{cases}
x_\eta & \text{for $\eta\in\mathcal P_G(\lambda)\setminus\{\mu_\rho^-\}$}\\
\frac{x_{\mu_\rho^-}}{x_{\mu_\rho}}&\text{for $\eta=\mu_\rho^-$.}
\end{cases}
\]
Then
\[
P^G_\kappa((x_\eta)_{\eta\in\mathcal P_G(\nu)})=
\lim_{x_{\mu_\rho}\to\infty}\left(\frac{1}{x_{\mu_\rho}}P^{G'}_\kappa((x'_\eta)_{\eta\in\mathcal P_G(\lambda)})\right).
\] 
Note that the change of coordinates is engineered so that $x'_{\mu_\rho}x'_{\mu_\rho^-}=x_{\mu_\rho^-}$. This formula follows from the fact that the monomials in $P^{G'}_\kappa((x'_\eta))$ that don't come from monomials in $P^G_\kappa$ are precisely those for which the exponents of $x'_{\mu_\rho}$ and $x'_{\mu_{\rho}^-}$ are equal. 
\end{remark}

\section{Generalisation of \cref{t:mainIntro} to arbitrary ample boundary divisors}\label{s:GenD}

We can now generalise \cref{thm:main} (which showed that the 
Newton-Okounkov body $\Delta_G^{\lambda}$ equals the superpotential polytope
$\Gamma_G^{\lambda}$) as follows. Namely, let us set $\Delta^{\lambda}_G(\mathbf r,\mathbf r'):=\Delta^{\lambda}_G(D_{(\mathbf r,\mathbf r')})$ and instead of fixing $D_{(\mathbf 1,\mathbf 0)}$, we allow arbitrary ample divisors $D_{(\mathbf r,\mathbf r')}$ with support contained in $D_{\ac}$. 
\begin{thm}\label{t:maingen}
 Consider $\mathbf{r}=(r_1,\dots,r_d)\in\Z^d$ and $\mathbf{r'}=(r'_1,\dots,r'_{n-1})\in \Z^{n-1}$ such that the associated divisor $D_{(\mathbf r,\mathbf r')}$ is ample. Then we have that 
$$\Delta^{\lambda}_G(\mathbf r,\mathbf r')= \Gamma^{\lambda}_G(\mathbf r,\mathbf r').$$
\end{thm}

This section is devoted to the proof of this theorem. But first we record a corollary. Namely the generalisation of \cref{thm:main} also implies a generalisation of \cref{cor:degeneration}.

\begin{cor}
	\label{cor:degenerationgen}
Let $\Sigma_G^{\mathcal X}$ be an arbitrary $\mathcal X$-cluster seed
	for $X_{\lambda}$, and suppose the divisor $D_{(\mathbf r,\mathbf r')}$ is ample.  Consider the Newton-Okounkov body
$\Delta_G^{\lambda}(D_{(\mathbf r,\mathbf r')})$.
We have that $\Delta_G^{\lambda}(D_{(\mathbf r,\mathbf r')})$ is a rational polytope. 
There exists a flat degeneration of $X_{\lambda}$
to the normal projective toric variety
associated to the  polytope $\Delta_G(\RG D_{(\mathbf r,\mathbf r')})$, for some  positive integer $\RG$, where we may take $\RG=1$ if $\Delta_G^{\lambda}(D_{(\mathbf r,\mathbf r')})$ has the integer decomposition property.
\end{cor}

\begin{proof}[Proof of the Corollary]
By Theorem~\ref{t:maingen} $\NO_G(D_{(\mathbf r,\mathbf r')})$ is a rational polytope. We may choose $\RG$ to be the minimal positive integer such that
the dilated polytope $\RG\NO_G(D_{(\mathbf r,\mathbf r')})$ has the \emph{integer decomposition property}. Note that $\RG\NO_G( D_{(\mathbf r,\mathbf r')})=\NO_G(\RG D_{(\mathbf r,\mathbf r')})$. Now \cite{Anderson} applies to $\NO_G(\RG D_{(\mathbf r,\mathbf r')})$ giving the required flat degeneration.
\end{proof}

\subsection{Varying $(\mathbf r,\mathbf r')$ in $\Gamma^\lambda_G(\mathbf r,\mathbf {r'})$ }\label{s:GammaShifts}

Let us consider an ample boundary divisor $D_{(\mathbf r,\mathbf r')}$ linearly equivalent to $RD_{(\mathbf 1,\mathbf 0)}$. Recall that $\mathbf r$ is completely determined by $\mathbf r'$ and $R$, via the formula
\begin{equation}\label{e:rr'Cartier}
r_j=R-\sum_{b'_{i}\in NW(b_{\rho_{2j-1}})}r'_i
\end{equation}
from \cref{c:Drr}, see also  \eqref{e:Cartier-rviaR}. 

For any rectangle $\nu \subseteq \lambda$ we introduce the 
constant-along-the-diagonals filling of the boxes of $\nu$ where the box $b'_i$ in the NW rim of $\lambda$, if it lies in $\nu$, is filled by $r'_i$. Let $r'_{(\nu)}$ be the sum of all the entries of the boxes of $\nu$. In other words, 
\begin{equation}\label{e:vrectcomponents}
	r'_{(\nu)}=\sum_{i=1}^{n-1}m_i(\nu) r'_i
\end{equation}
where $m_i(\nu)$ is the number of boxes in $\nu$ that lie in the diagonal containing $b'_i$.

We define a vector $v_{\rect}(\mathbf r')\in\R^{\mathcal P_{G_{\lambda}^{\rect}}}$ by
  \begin{equation}\label{e:vGr'}
	  v_\rect(\mathbf r'):=(r'_{(\nu)})_{\nu\in\mathcal P_{G_{\lambda}^{\rect}}}. 
  \end{equation}

\begin{proposition}\label{p:GammaShiftRect} Suppose that $R>0$ and $(\mathbf r,\mathbf r')$ satisfy \eqref{e:rr'Cartier}. We have that
\[
\Gamma^\lambda_{\rect}(\mathbf r,\mathbf r')=R\Gamma^\lambda_\rect(\mathbf 1,\mathbf 0)-v_\rect(\mathbf r').
\]
\end{proposition}

\begin{proof} 
Let $G=G_{\lambda}^{\rect}$ for the duration of this proof. 
 \cref{l:rectangles} gives us the following explicit description of $R\Gamma^\lambda_G(\mathbf 1,\mathbf 0)$.
\begin{align}
		v_{\mu} - v_{\muminus} & \le R\quad \text{ whenever }\mu 
		\text{ labels an outer corner of }\lambda  \label{eq:1}\\
		0 & \leq v_{1 \times 1} \label{eq:2} \\
		v_{(i-1)\times j} - v_{(i-2)\times (j-1)} &\leq
v_{i \times j}  - v_{(i-1)\times (j-1)}
\ \text{ for }2 \leq i, 
		1\leq j, \text{ and } (i \times j) \subseteq \lambda \label{eq:3}\\	
v_{i\times (j-1)}-v_{(i-1)\times (j-2)} &\leq
v_{i\times j}  - v_{(i-1)\times (j-1)}
\ \text{ for } 1\leq i,
		2 \leq j, \text{ and } (i \times j) \subseteq \lambda. \label{eq:4}
	\end{align}
	
	The polytope $R\Gamma^\lambda_G(\mathbf 1,\mathbf 0)-v_G(\mathbf r')$ can be described by shifting each coordinate $v_\nu$ of a point in $R\Gamma^\lambda_G(\mathbf 1,\mathbf 0)$ to $v'_\nu=v_\nu- r'_{(\nu)}$. 

	Consider first a frozen index $\mu\in\mathcal P_G$ associated to the outer corner box $b_{\rho_{2\ell-1}}$. Using the definition of the $r'_{(\nu)}$ and the inequality \eqref{eq:1} we get 
\[
v'_\mu-v'_{\muminus}=v_\mu-v_{\muminus}-r'_\mu+r'_{\muminus}=v_\mu-v_{\muminus}-\sum_{b'_{i}\in NW(b_{\rho_{2\ell-1}})}r'_i\le R -\sum_{b'_{i}\in NW(b_{\rho_{2\ell-1}})}r'_i,
\]
Therefore, 
 using \eqref{e:rr'Cartier}, the equivalent inequality to \eqref{eq:1} for the translated polytope becomes
\begin{equation} v'_{\mu} - v'_{\muminus} \le r_\ell, \label{eq:1'}
 \end{equation}
 which agrees with the inequality \eqref{eq:1gen} for $\Gamma^\lambda_G(\mathbf r,\mathbf{r'})$.
  
Since $v'_{1\times 1}=v_{1\times 1}-r'_{n-k}$ we get the inequality
\begin{equation}-r'_{n-k}\le  v'_{1\times 1}  \label{eq:2'}
 \end{equation}
 for $v'_{1\times 1}$, equivalent to \eqref{eq:2gen}.
 
We then have
\begin{equation*}
\begin{array}{ll}
v'_{i \times j}  - v'_{(i-1)\times (j-1)}&=  v_{i \times j}  - v_{(i-1)\times (j-1)} - r'_{i\times j}+r'_{(i-1)\times (j-1)}\\
&=v_{i \times j}  - v_{(i-1)\times (j-1)} -r'_{n-k-i+1}-r'_{n-k-i+2}+\dotsc- r'_{n-k+j-1}\\
v'_{(i-1) \times j}  - v'_{(i-2)\times (j-1)}&=  v_{(i-1) \times j}  - v_{(i-2)\times (j-1)}- r'_{(i-1)\times j}+r'_{(i-2)\times (j-1)}\\
&=v_{(i-1) \times j}  - v_{(i-2)\times (j-1)}- r'_{n-k-i+2}-r'_{n-k-i+3}-\dotsc - r'_{n-k+j-1}
	\end{array}
	\end{equation*}
	and therefore 
	\begin{multline*}
	\left(v'_{i \times j}  - v'_{(i-1)\times (j-1)}\right)-\left(v'_{(i-1) \times j}  - v'_{(i-2)\times (j-1)}\right)=\\
	\left(v_{i \times j}  - v_{(i-1)\times (j-1)}\right)-\left(v_{(i-1) \times j}  - v_{(i-2)\times (j-1)}\right)-r'_{n-k-i+1}.
	\end{multline*}
	Therefore the equivalent inequality to \eqref{eq:3} for the translated polytope is 
	\begin{equation}
	v'_{(i-1)\times j} - v'_{(i-2)\times (j-1)} \leq
v'_{i \times j}  - v'_{(i-1)\times (j-1)}+r'_{n-k-i+1}
\ \text{ for }2 \leq i, 
		1\leq j, \text{ and } (i \times j) \subseteq \lambda, \label{eq:3'}
		\end{equation}
		which agrees with the inequality~\eqref{eq:3gen} for $\Gamma^\lambda_G(\mathbf r,\mathbf r')$. 
		The completely analogous calculation shows the inequality \eqref{eq:4gen} is the shifted version of \eqref{eq:4}. Thus we have shown the statement of the lemma.
\end{proof}

\subsection{The proof of \cref{t:maingen}}\label{s:proof2}
The proof of \cref{t:maingen} requires the following lemma. 
\begin{lem}\label{l:recvaluationf} Consider an ample boundary divisor $D_{(\mathbf r,\mathbf r')}$ of degree $R$ in $X_\lambda$. Suppose $f$ is a rational function on $X_\lambda$ with $(f)=D_{(\mathbf r,\mathbf r')}-R D_{(\mathbf 1,\mathbf 0)}$. Then 
\[\val_{G^\lambda_\rect}(f)=v_\rect(\mathbf r').
\]
for $v_{\rect}$ as defined in \eqref{e:vGr'}.
\end{lem}

\begin{proof}
Let $G=G^\lambda_{\rect}$. 
Set $\tilde {\mathbf r}=\mathbf r-R\mathbf 1$, so that
$(f)=D_{(\mathbf r,\mathbf r')}-R D_{(\mathbf 1,\mathbf 0)}=D_{(\tilde{\mathbf r},\mathbf r')}$. Recall the map $\varphi:\Z^n\to\Z^{n-1}$ defined in \cref{c:DrrVSPmu}. We have 
\begin{equation}\label{e:DviaPmu}
D_{(\tilde{\mathbf r},\mathbf r')}=\sum_{j=1}^n m_j (P_{\mu_j}),
\end{equation}
where $\mathbf m=(m_1,\dotsc, m_n)\in \Z^n$ is determined by $\varphi(\mathbf m)=\mathbf r'$ and $\sum_{j=1}^n m_j=0$, see Corollary~\ref{c:RationalPmu}. As a consequence of \eqref{e:DviaPmu} we have
\begin{equation}\label{e:fviaPmu}
f=c\prod_{j=1}^nP_{\mu_j}^{m_j},
\end{equation}
for some nonzero constant $c$. We define the following linear map
\[
\begin{array}{cccl}
\mathcal V:&\Z^n &\to & \Z^{\mathcal P_{G}}\\
&(m_j)_{j=1}^n &\mapsto & \sum_{j=1}^n m_j\val_G(P_{\mu_j}).
\end{array}
\]
It follows from \eqref{e:fviaPmu} that $\val_{G}(f)=\mathcal V(\mathbf m)$. 

Now we consider the linear map $v_{\rect}:\Z^{n-1}\to \Z^{\mathcal P_G}$ defined componentwise via the formula \eqref{e:vrectcomponents}, 
\begin{equation}\label{e:vrectcomponents2}
v_{\rect}((r'_i)_i)_\nu=\sum_{i=1}^{n-1}m_i(\nu) r'_i,
\end{equation}
where $m_i(\nu)$ is the number of boxes in $\nu$ that lie in the diagonal containing $b'_i$. 

We claim that the following diagram commutes. 
\[\begin{tikzcd}
	{\mathbb Z^n} \\
	{\mathbb Z^{n-1}} & {\mathbb Z^{\mathcal P_{G}}}
	\arrow["\varphi"', from=1-1, to=2-1]
	\arrow["{v_{\operatorname {rec}}}"', from=2-1, to=2-2]
	\arrow["{\mathcal V}", from=1-1, to=2-2]
\end{tikzcd}.\]
All the maps are linear, and therefore it suffices to check commutativity on a basis of $\Z^n$. We have that $\mathcal V(e_j)=\val_{G^\lambda_{\rect}}(P_{\mu_j})$, which, using 
	\cref{p:latticepoints-comp},
	is given by
\[
\mathcal V(e_j)_\nu=\val_{G^\lambda_{\rect}}(P_{\mu_j})_\nu=\maxdiag(\nu\backslash \mu_j).
\]
On the other hand, let 
\[
\mathbf r'_{(j)}:=
\varphi(e_j)=\sum_{i\in\operatorname{add}(\mu_j)} e_i.
\] 
Consider any rectangle $\nu\in\mathcal P^G$. If $\mu_j$ has more than one addable box, then at most one of these addable boxes can lie in $\nu$, since $\nu$ can be taller than $\mu_j$ or wider than $\mu_j$ but not both. Using this observation and the definition of $v_{\rect}$ from~\eqref{e:vrectcomponents2}, we see that $v_{\rect}(\mathbf r'_{(j)})_\nu=\maxdiag(\nu\backslash \mu_j)$. So we have shown that $v_{\rect}(\varphi(e_j))=\mathcal V(e_j)$, and therefore the diagram commutes. 

Finally, we have 
\[
\val_G(f)=\mathcal V(\mathbf m)=v_{\rect}(\varphi(\mathbf m))=v_{\rect}(\mathbf r'),
\]
which concludes the proof.
\end{proof}

We note that the above Lemma is purely about the $\mathcal X$-variety $X_\lambda$ and its Cartier boundary divisors (made up of positroids and Schubert divisors). However, the formula was inspired by the calculation in 
\cref{p:GammaShiftRect} on the mirror side.

\begin{proof}[Proof of \cref{t:maingen}]

Let $f$ be the rational function on $X_\lambda$ from \cref{l:recvaluationf} with $(f)=D_{(\mathbf {r},\mathbf {r'})}-RD_{(\mathbf 1,\mathbf 0)}$.  Then
\begin{equation}\label{e:DeltaGenGShift}
\Delta^\lambda_G(D_{(\mathbf {r},\mathbf {r'})})=R \Delta^\lambda_G-\val_{G}(f),
\end{equation}   
	by the proof of \cref{l:DeltaShift}. We have that $f$ is a Laurent monomial in the $P_{\mu_j}$, and $\Psi_{G,G'}(\val_{G}(P_{\mu_j}))=\val_{G'}(P_{\mu_j})$, see 
\cref{thm:flowmutate}.
Moreover,  $f$ and the $P_{\mu_j}$ are regular functions which do not vanish
on $X_\lambda^\circ$, making translation by their valuations compatible with mutation, see~\cref{c:balanced}. We obtain that $\Psi_{G,G'}(\val_{G}(f))=\val_{G'}(f)$ for any choice of $G, G'$ and we obtain the useful identity
\[
\Psi_{G^\lambda_\rect,G}(R \Delta^\lambda_\rect-\val_{G^\lambda_\rect}(f))=R\Psi_{G^\lambda_\rect,G}(\Delta^\lambda_\rect)-\val_{G}(f).
\] 
Using Theorem~\ref{thm:main}, we may reformulate this to   
\begin{equation}\label{e:DeltaGenGShift2}
\Psi_{G^\lambda_\rect,G}(R \Delta^\lambda_\rect-\val_{G^\lambda_\rect}(f))
=R\Delta^\lambda_G-\val_{G}(f).
\end{equation}   

We now recall that $\Delta^\lambda_G=\Gamma^\lambda_G$ and  $\val_{G_\rect}(f)=v_\rect(\mathbf r')$, by \cref{thm:main} and \cref{l:recvaluationf}, respectively. Therefore, we can make replacements on the left-hand side of \eqref{e:DeltaGenGShift2} and we find that
\[
\Psi_{G^\lambda_\rect,G}(R \Gamma^\lambda_\rect-v_\rect(\mathbf r'))=R\Delta^\lambda_G-\val_{G}(f). 
\]
Thanks to \cref{p:GammaShiftRect} the left-hand side above is the mutation of a superpotential polytope, namely 
\[
	R\Gamma^\lambda_{\rect}-v_{\rect}(\mathbf r') = 
\Gamma^\lambda_{\rect}({\mathbf {r},\mathbf {r'}}).
\]
Finally, using \cref{c:GammaMutation} and the identity \eqref{e:DeltaGenGShift} we obtain
\[
\Gamma_G^\lambda(\mathbf r,\mathbf r')= \Psi_{G^\lambda_\rect,G}(\Gamma_\rect^\lambda(\mathbf r,\mathbf r'))=
\Psi_{G^\lambda_\rect,G}(R \Gamma^\lambda_\rect-v_\rect(\mathbf r'))=R\Delta^\lambda_G-\val_{G}(f)=\Delta^\lambda_G(D_{(\mathbf {r},\mathbf {r'})}).
\] 
\end{proof}

 Consider the positroid divisor $D_{(\mathbf 0,\mathbf {r'})}=D'_{n-k}$, where $\mathbf r'=\delta_{i,n-k}$, so  $r'_{n-k}=1$ and all other $r'_i=0$. The divisor $D_{n-k}'$ is a distinguished irreducible Cartier divisor in $X_\lambda$. We have the following special case of Theorem~\ref{t:maingen}. 

\begin{corollary} For the ample positroid divisor $D'_{n-k}=D_{(\mathbf 0,\mathbf\delta_{i,n-k})}$ we have the following description of the Newton-Okounkov convex body $\Delta_G^\lambda(D'_{n-k})$ as a superpotential polytope,
\[
\Delta_G^\lambda(D'_{n-k})=\Gamma^\lambda_G(\mathbf 0,\mathbf\delta_{i,n-k}).
\]
\end{corollary}

\begin{remark}
Related to this corollary we propose the following alternative choice of a superpotential for $X_\lambda$,
\[
\sum_{j=0}^d W_j + \sum_{i=0}^{n-1} q^{\delta_{i,n-k}}W'_i,
\]
which has a single parameter $q$ in keeping with the rank of the Picard group of $X_\lambda$.
\end{remark}

\section{A Gorenstein Fano toric variety 
constructed from  $W_{\rect}^{\lambda}$}\label{s:toric}

%\LW{Anderson's results are for Cartier divisors.}

One of the main properties of our superpotential $W^\lambda$ is that it encodes in one compact formula a multitude of toric degenerations of the Schubert variety $X_\lambda$ via
the superpotential polytopes, see 
\cref{t:maingen}
and \cref{cor:degenerationgen}. 
More specifically, 
we get one toric degeneration from each choice of cluster chart, and it is
encoded in the corresponding  Laurent expansion of $W^\lambda$. 
This extends to the Schubert setting a key property of the Grassmannian superpotential \cite{RW}, see also \cite{SW1}. In this section we focus on the Laurent expansion of the superpotential $W^\lambda$ in 
the rectangles cluster, which encodes the degeneration from \cite{GL} of the Schubert variety $X_{\lambda}$ to a `Gelfand-Tsetlin' toric variety. 
If $X_\lambda$ is Gorenstein, then so is its toric degeneration.  
We show that when $X_\lambda$ is \textit{not} Gorenstein, 
its toric degeneration nevertheless has a canonical small partial resolution to a  
Gorenstein toric Fano variety. Moreover it has a small toric desingularisation in all cases.

The idea of Laurent polynomial mirrors  relating to toric degenerations goes back to \cite{BC-FKvSGrass,BC-FKvS}, who constructed Laurent polynomial mirrors for partial flag manifolds generalising  \cite{EHX,Givental:fullflag} and related them to the toric degenerations from \cite{GL}. Subsequently, 
this kind of approach was taken in a variety of settings such as in  \cite{Galkin,ILP, Kalashnikov}, see also \cite[Conjecture 9]{KP:whyandhow} and \cite{CKPT}. Toric degenerations also play a role in the construction of superpotentials using Floer theory, see \cite[Theorem~1]{Nishinou}, and \cite[Theorem 4.4]
{BGM22}. Note that \cite[Theorem~1]{Nishinou} requires the existence of a small resolution of the central toric fiber.  

In the above references, smooth 
varieties are degenerated to Gorenstein Fano toric varieties for the
purpose of applying mirror symmetry. 
But degenerations of Gorenstein (singular) Schubert varieties have been used in \cite{Miura:minuscule} for studying quantum periods of smooth Calabi-Yau $3$-folds contained in them, giving another kind of application of a Laurent 
polynomial superpotential for a singular variety. This work uses the degeneration of \cite{GL}, coinciding with ours as in  \cref{def:toricvarieties2}, and thus further supports $W^\lambda$ being called the superpotential for $X_\lambda$ in the Gorenstein case.

Let us start by 
defining the toric varieties of interest.  
We introduce two toric varieties, both related to the superpotential 
expressed in terms of the rectangles cluster.

\begin{defn}\label{def:toricvarieties}
Let 
$\CC_\lambda$ denote the (inner) normal fan of the superpotential 
polytope $\Gamma_{\rect}^{\lambda} = \Delta^\lambda_\rect$ (for the rectangles cluster)
and let $Y(\CC_\lambda)$ denote the associated toric variety. 

Let $\NP(W^{\lambda}_{\rect})$ denote 
the Newton polytope of the Laurent polynomial  superpotential
$W_{\rect}^{\lambda}(q_i=1)$
after specializing
each $q_i=1$.
Let $\FF_\lambda$ denote the face fan 
of the Newton polytope 
$\NP(W^\lambda_{\rect})$ 
and let 
$Y(\FF_\lambda)$ be its associated toric variety.

Note that $Y(\CC_\lambda)$ comes with a projective embedding via the polytope $\Gamma_{\rect}^{\lambda} = \Delta^\lambda_\rect$ by default and we will usually consider $Y(\CC_\lambda)$ as projective toric variety via this embedding. We may also write $\mathbb P_{\Delta}$ for the projective variety associated to a polytope $\Delta$ if it is useful to include $\Delta$ in the notation. 
\end{defn}

We 
have a toric degeneration of our Schubert variety $X_\lambda$ to the projective toric variety $Y(\CC_\lambda)$, as a special case of \cref{cor:degeneration};
this recovers the `Gelfand-Tsetlin' toric degeneration of the Schubert variety $X_\lambda$ constructed by \cite{GL}, as already pointed out.

The section is now organised as follows. The first main result of this section is 
that the polytope $\NP(W^\lambda_{\rect})$ is reflexive and terminal, see \cref{c:reflexiveterminal},
and hence $Y(\FF_\lambda)$ is Gorenstein Fano with at most terminal singularities, 
see \cref{thm:Gorenstein}.
The second main result of this section, \cref{cor:small}, is that 
$Y(\FF_\lambda)$ is a small partial desingularization of the toric 
degeneration $Y(\CC_\lambda)$ of the Schubert variety 
$X_{\lambda}$. 
Our third main result is a description of the group $\operatorname{Cart}_T(Y(\FF_\lambda))$ of torus-invariant Cartier divisors, 
and of the Picard group $\Pic(Y(\FF_\lambda))$, see  
	\cref{t:toricCartier}, \cref{p:toricCartierAmple} and \cref{c:toricPicard}. We do this by constructing a new poset $\widetilde{P}(\lambda)$, 
extending the poset $P(\lambda)$ from \cref{def:posetlambda}, whose maximal elements we show determine a basis of $\Pic(Y(\FF_\lambda))$. We also describe the ample cone of $Y(\FF_\lambda)$.   
We discuss connections to marked order polytopes and flow polytopes and consider analogues $\widetilde D_{(\mathbf r,\mathbf r')}$ in $Y(\FF_\lambda)$ of the boundary divisors $D_{(\mathbf r,\mathbf r')}\subset X_\lambda$. 
See \cref{fig:degeneration} for a depiction of some of these relationships.

\subsection{The Newton polytope of the superpotential and the superpotential
polytope}\label{sec:reflexive}

Just as we did in \cref{s:vertexcoordinates}, it will be convenient for us to 
 work with the superpotential
$\rW_{\rect}^{\lambda}$ in vertex coordinates; we will then 
study the associated superpotential polytope and Newton polytope and 
the relations between them.
These two  polytopes lie in dual vector spaces.

\begin{defn}\label{def:Mlambda}
We write 
	$ \mathbf N^\lambda_{\R}$ 
	for the  vector space (isomorphic to $\R^{|\lambda|}$) 
containing the Newton polytope of $\rW_{\rect}^{\lambda}$, and 
  $\mathbf M^\lambda_{\R}$ 
	for the dual vector space
containing the superpotential polytope (see \cref{def:super2}).
	$ \mathbf N^\lambda_{\R}$ and 
  $\mathbf M^\lambda_{\R}$  have
	coordinates $e_{i \times j}$ and 
$ f_{i \times j}$ indexed by 
the rectangles $i \times j \subseteq \lambda$. 
Let $\mathbf N^\lambda_\Z\subset \mathbf N^\lambda_{\R}$ be the $\Z$-lattice where the coordinates $e_{i\times j}\in\Z$, and 
	let $\mathbf M^\lambda_\Z$ be the dual $\Z$-lattice. 
	We may identify $\mathbf M^\lambda_\Z$ with $\Z^{P(\lambda)}$. 
\end{defn}

We now give an analogue of \cref{def:toricvarieties} which uses vertex coordinates.
\begin{defn}\label{def:toricvarieties2} 
Let 
$\overline{\CC}_\lambda$ denote the (inner) normal fan 
 of the superpotential 
polytope $\rGamma_{\rect}^{\lambda}$ 
and let $Y(\overline{\CC}_\lambda)$ denote the associated toric variety. Let $\NP(\rW_{\rect}^{\lambda})$ denote
	the Newton polytope 
	of the Laurent polynomial 
	$\rW_{\rect}^{\lambda}(q_i=1)$, 
	after specializing
each $q_i=1$.
Let $\overline{\FF_\lambda}$ denote the face fan 
of 
$\NP(\overline{W}^\lambda_{\rect})$ 
and let 
$Y(\overline{\FF_\lambda})$ be its associated toric variety.
Both fans $\overline\CC_\lambda$ and $\overline{\FF_{\lambda}}$, lie in $\mathbf N_\R^\lambda$, and 
$\mathbf M^\lambda_\Z$ is the character lattice of the torus acting on $Y(\overline{\FF_\lambda})$ and  $Y(\overline{\CC_\lambda})$, compare \cref{def:Mlambda}.  
\end{defn}

\begin{example} \label{ex:442again}
  We continue \cref{ex:442}, which uses the 
	superpotential from \eqref{eq:442}.
The Newton polytope
is the convex hull of the points
$$\{e_{1 \times 1}, e_{2 \times 1}-e_{1 \times 1}, e_{3 \times 1}-e_{2 \times 1}, e_{2 \times 2}-e_{1 \times 2},\dots,  -e_{2 \times 4}, -e_{3 \times 2}\}
	\subset \mathbf N^{\lambda}_{\R}.$$
\end{example}

\begin{remark}\label{rem:variablechange}
As in \cref{p:unimodular}, we have a unimodular change of 
variables between 
the Newton polytope 
$\NP(\rW_{\rect}^{\lambda})$ in 
vertex coordinates and the usual
Newton polytope 
	$\NP(W_{\rect}^{\lambda})$.
Therefore to understand 
$\NP(W_{\rect}^{\lambda})$, 
	it suffices 
	to work with 
$\NP(\rW_{\rect}^{\lambda})$. 
This allows us to work more directly in terms of posets and quivers and apply results from \cite{RWReflexive} to the study of the toric variety $Y(\FF_\lambda)$.  
\end{remark}

\begin{definition}\label{def:starredquiver}
	We define a \emph{starred quiver} to be 
	a quiver $Q$ with vertices $\V=\nV \sqcup \V_{\star}$
(where $\nV=\{\vv_1,\dots,\vv_n\}$ for $n \geq 1$ and
$\V_{\star}=\{\star_1,\dots,\star_{\ell}\}$ for $\ell \geq 1$ are called the
        \emph{(normal) vertices} and
\emph{starred vertices}),
and arrows $\Arr(Q) \subseteq (\nV \times \nV) \sqcup (\nV \times \V_{\star}) \sqcup (\V_{\star} \times \nV)$. We will always assume the graph underlying $Q$ to be connected, and $Q$ to have at least one starred vertex. 
\end{definition}

Let $P$ be a finite, ranked poset, where we assume minimal elements of $P$ to all have the same rank, as in \cite{RWReflexive}, while maximal elements may be of varying ranks. 
We also assume that the Hasse diagram is connected. Let us write $\rk:P\to\Z$ for the rank function on $P$. We construct two starred quivers out of $P$. The first one is defined simply, see \cref{d:QPmax} below. The second one will be defined later in \cref{d:canonicalquiver}.

\begin{defn}\label{d:QPmax}
Suppose $P$ is a finite, ranked poset as above, with maximal elements denoted $m_1,\dotsc, m_d$ and minimal elements of rank $1$. Let $n_j$ denote the rank of $m_j$. We define an extension of $P$ denoted by $P_{\max}$ by adjoining one minimal element $\star_0$ (of rank $0$), and for every $m_j$ a new maximal element $\star_j$ covering $m_j$ (of rank $n_j+1$). Note that $P_{\max}$ is again a ranked poset. 

We now associate to $P$ a starred quiver $Q_{P_{\max}}$ with vertex sets $\V_\bullet=P$ and $\V_\star=\{\star_0,\star_1,\dotsc,\star_d\}$. 
Thus $\V=\V_\bullet\sqcup\V_\star$ agrees with the set of elements of  $P_{\max}$. For every covering relation in $P_{\max}$ we introduce an arrow pointing from the smaller to the larger element. In other words, $Q_{P_{\max}}$ is constructed out of the Hasse diagram of $P_{\max}$ by orienting the edges from smaller to larger, and designating the minimal and maximal elements of $P_{\max}$ as starred vertices.     
\end{defn}

\begin{notation}\label{notation:Q}
Given a partition $\lambda$,  let 
$Q_\lambda$ denote the starred quiver
	$Q_{P(\lambda)_{\max}}$.
\end{notation}

Note that $Q_\lambda$ agrees with the quiver constructed in \cref{def:superquiver}, where we declare the vertices labeled $1$ and $q_i$ as starred vertices. 
For example, the quiver at the left of 
\cref{fig:superpotential-442} corresponds to the starred quiver shown in  \cref{fig:starredquiver}.

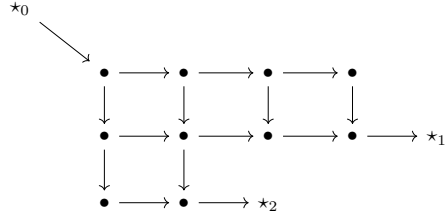
\begin{figure}[h] 
\begin{center}
\[\adjustbox{scale=.8,center}{\begin{tikzcd}
	{\star_0} \\
	& \bullet & \bullet & \bullet & \bullet \\
	& \bullet & \bullet & \bullet & \bullet & {\star_1} \\
	& \bullet & \bullet & {\star_2}
	\arrow[from=1-1, to=2-2]
	\arrow[from=2-2, to=2-3]
	\arrow[from=2-2, to=3-2]
	\arrow[from=2-3, to=2-4]
	\arrow[from=2-3, to=3-3]
	\arrow[from=2-4, to=2-5]
	\arrow[from=2-4, to=3-4]
	\arrow[from=2-5, to=3-5]
	\arrow[from=3-2, to=3-3]
	\arrow[from=3-2, to=4-2]
	\arrow[from=3-3, to=3-4]
	\arrow[from=3-3, to=4-3]
	\arrow[from=3-4, to=3-5]
	\arrow[from=3-5, to=3-6]
	\arrow[from=4-2, to=4-3]
	\arrow[from=4-3, to=4-4]
\end{tikzcd}}\]
\hspace{1cm}
	\caption{The starred quiver $Q_\lambda$ for $\lambda=(4,4,2)$
	\label{fig:starredquiver}}
\end{center}
\end{figure}

\begin{definition}[Root polytope]\label{def:Root} 
Let $Q$ be a starred quiver 
with arrows $\Arr(Q)$
and vertices
$\mathcal V_{\bullet}=\{\vv_1,\dots,\vv_n\}$ and $\mathcal V_{\star}=\{\star_1,\dots,\star_{\ell}\}$. 
We write $\mathbf N_\R$ (or $\mathbf N^Q_{\R}$) for the  vector space $\R^{\mathcal V_\bullet}$  that will contain the root polytope $\Root(Q)$.  We also write $\mathbf M_\R$ (or $\mathbf M^Q_{\R}$) for the dual vector space. For $Q=Q_\lambda$ coming from the poset $P(\lambda)$ we recover the vector spaces $\mathbf N^{\lambda}_\R$ and $\mathbf M^{\lambda}_\R$ from \cref{def:Mlambda}. 
	
We now identify $\mathbf N_\R=\R^{\mathcal V_\bullet}$ with $\R^n$.  Let $e_i$ denote the 
standard basis vector in $\R^n$ with a $1$ in position $i$ and $0$'s elsewhere. We associate a point $u_a\in \mathbf N_\R$ to each arrow $a$ as follows:
\begin{itemize}
        \item if $a: \vv_i \to \vv_j$, $u_a:=e_j-e_i$;
        \item if $a: \star_i \to \vv_j$, $u_a:= e_j$; and
        \item if $a: \vv_i \to \star_j$, $u_a:= -e_i$.
\end{itemize}
We then define the \emph{root polytope} $\Root(Q)\subset \mathbf N_\R$ to be the convex hull of the points $u_a$,
$$\Root(Q) := \conv\{u_a \ \vert \ a\in \Arr(Q)\}.
$$
\end{definition}
\begin{remark}\label{r:rootNewton}
The Newton polytope 
	$\NP(\rW_{\rect}^{\lambda})$ 
equals  the root polytope $\Root(Q_\lambda)$.
\end{remark}

\begin{definition}\label{def:reflexive}
Suppose that $\Poly\subset \R^n$ is a lattice polytope of full dimension $n$
which contains the origin $0$ in its interior. Then
the \emph{polar dual polytope} of $\Poly$ is
$$\Poly^*:= \{y \in (\R^n)^* \ \vert \ \langle y,x \rangle \geq -1
  \text{ for all }x\in \Poly\},$$
where $\langle y,x \rangle$ is the pairing between 
	$(\R^n)^*$ and $\R^n$.

	A full-dimensional lattice polytope $\Poly\subset \R^n$ with the origin in its interior is called \emph{reflexive} (or \emph{Gorenstein Fano}) if 
its polar dual is also a lattice polytope.
It is called \textit{terminal} if its vertices and~$0$ are the only  lattice points contained in~$\Poly$ (with $0$ in the interior). 
\end{definition}

\begin{definition}\label{d:strongconn}
We say that a starred quiver $Q$ is \emph{strongly connected} if after identifying all of the starred vertices, there is an oriented path from any vertex to any other vertex.
\end{definition}

We can now make use of \cite[Theorem A]{RWReflexive}, which says that the root polytope $\Root(Q)$ of any strongly connected  quiver or starred quiver $Q$  
is reflexive and terminal.\footnote{In the case of 
root polytopes associated to strongly connected 
starred quivers with no starred vertices -- 
also known as \emph{edge polytopes}
of strongly connected directed graphs -- this result
is stated in \cite[Proposition 1.4]{Higashitani}
(the latter reference does not provide a proof,
but says the proof is similar to that of 
  \cite[Proposition 3.2]{Matsui}).} Note that the starred quiver $Q_{P_{\max}}$ associated to a ranked poset $P$ is automatically strongly connected. 
\begin{prop}\label{c:reflexiveterminal}
 The polytopes  $\NP(W_{\rect}^{\lambda})$ and 
$\NP(\rW_{\rect}^{\lambda})=\Root(Q_\lambda)$  
  are reflexive and terminal.  
\end{prop}
\begin{proof}
Since the quiver $Q_\lambda$ is strongly connected, the root polytope $\Root(Q_\lambda)$   
is reflexive and terminal by \cite[Theorem A]{RWReflexive}.  
Since, $\NP(\rW_{\rect}^{\lambda})=\Root(Q_\lambda)$ agrees with $\NP(W_{\rect}^{\lambda})$ up to a unimodular change of coordinates, $\NP(W_{\rect}^{\lambda})$ is also reflexive and terminal.   
\end{proof}

 Using \cref{c:reflexiveterminal}, we obtain the following.
\begin{corollary}\label{thm:Gorenstein}
The toric variety $Y(\FF_\lambda)$ associated to the face fan 
of $\NP(W_{\rect}^{\lambda})$ is Gorenstein Fano,
with at most terminal singularities.
\end{corollary}

We now want to relate some of our superpotential polytopes to the root polytope $\Root(Q_\lambda)$.
The following statement is immediate from the definitions, cf \cref{def:super2} and \cref{rem:variablechange}.
\begin{lemma}\label{l:easyreflexive}
When $r_1=\dots=r_d=1$ and $r'_1=\dots=r'_{n-1}=1$, the resulting 
superpotential polytope (in vertex coordinates)
	$\rGamma_{\rect}^{\lambda}(\mathbf{1},\mathbf{1})$ 
	is polar dual to the root polytope $\Root(Q_\lambda)$.
It follows that
$\Gamma_{\rect}^{\lambda}(\mathbf{1},\mathbf{1})$  is reflexive and that the inequalities for 
$\Gamma_{\rect}^{\lambda}(\mathbf{1},\mathbf{1})$ listed in \cref{l:rectangles}
are precisely the facet inequalities.
\end{lemma}

\begin{prop}\label{prop:refine}
The face fan $\FF_\lambda$
	of the Newton polytope $\NP(W^{\lambda}_\rect)$
	refines the normal fan $\CC_{\lambda}$
	of 
	$\Gamma^{\lambda}_{\rect}$, 
	and both fans have the same set of rays.
\end{prop}
\begin{proof}
\cite[Theorem D]{RWReflexive} says that given any finite ranked poset $P$,
	the face fan of the root polytope $\Root(Q_{P_{\max}})$
	of the starred quiver  $Q_{P_{\max}}$
associated to $P$ refines the (inner) normal fan of the 
	order polytope $\OO(P)$, and the rays of the two fans coincide.\footnote{Note that 
	the fan we refer to as $\overline{\FF}_{\lambda}$ here is denoted by
	$\FF_{\lambda}$ in \cite{RWReflexive}.}
	In our setting, 
$\NP(\rW_{\rect}^{\lambda})$ is exactly the root polytope
	$\Root(Q_\lambda)$ 
	associated to the starred quiver of the 
	poset $P(\lambda)$; and the superpotential polytope
	$\rGamma^{\lambda}_{\rect}$ (in vertex 
	coordinates) coincides with the order polytope
	$\OO(P(\lambda))$. 
Therefore the face fan $\overline{\FF_{\lambda}}$  of the Newton  polytope  $\NP(\rW_{\rect}^{\lambda})	=\Root(Q_\lambda)$  refines the normal fan $\overline{\CC_{\lambda}}$ of the superpotential polytope	$\rGamma^{\lambda}_{\rect}$ (cf \cref{def:toricvarieties2}) and both fans have the same set of rays. The result for $\FF_\lambda$ now follows by using the unimodular change of variables as in \cref{rem:variablechange}.
\end{proof}

We can now interpret \cref{prop:refine} geometrically and use \cite[Theorem E]{RWReflexive} to obtain a toric  desingularisation of $Y(\CC_\lambda)$.

\begin{corollary}\label{cor:small}
The Gorenstein toric  Fano variety $Y(\FF_\lambda)$ 
is a small partial desingularization of the toric variety $Y(\CC_\lambda)$. 
Moreover 
there exists a  
small toric desingularisation 
$$Y(\widehat{\FF}_{\lambda}) \to 
	Y(\FF_{\lambda}) \to Y(\CC_{\lambda}).$$
	of $Y(\CC_{\lambda})$ via $Y(\FF_{\lambda})$.
\end{corollary}

\begin{proof}
By \cref{prop:refine}, we have a partial desingularization 
$Y(\FF_{\lambda}) \to Y(\CC_{\lambda})$ because the first fan refines the second, and it is small because the two fans have the same rays.
Then 
\cite[Theorem E]{RWReflexive} says that we have a 
small crepant toric desingularization 
$Y(\widehat{\overline{\FF}}_{\lambda}) \to 
Y(\overline{\FF}_{\lambda})$ of $Y(\overline{\FF}_{\lambda})$. 
Via our unimodular change of variables, this 
gives a small crepant toric desingularization 
$Y(\widehat{{\FF}}_{\lambda}) \to 
Y({\FF}_{\lambda})$. The composition is a small toric desingularisation of $Y(\CC_\lambda)$.
\end{proof}

See \cref{fig:degeneration} for a summary of  the relationships between the various varieties
we have been discussing.

\begin{remark} We mention one further perspective on the polytope $\overline\Gamma^{\lambda}(\mathbf 1,\mathbf 1)$ arising from  \cref{l:easyreflexive}. Namely, we may construct a quiver $Q_{\lambda,\star}$ by identifying all of the starred vertices in $Q_\lambda$. Thus $Q_{\lambda,\star}$ is a strongly connected starred quiver with a single starred vertex. We may again view $Q_{\lambda,\star}$ as embedded in the plane and we note that  $\Root(Q_{\lambda,\star})=\Root(Q_{\lambda})$. By taking the planar dual of $Q_{\lambda,\star}$ we construct a planar acyclic quiver that we denote $Q^\vee_{\lambda,\star}$. From \cite[Theorem~3.4]{RWReflexive} we then obtain a direct description of $\overline\Gamma^{\lambda}(\mathbf 1,\mathbf 1)$ as a \textit{flow polytope} for $Q_{\lambda,\star}^\vee$, where the `weight' is chosen in a canonical way, see \cite[Definition~3.1]{RWReflexive}. Moreover, the variety $Y(\FF_\lambda)$ can thereby be realised as the Fano \textit{toric quiver moduli space} for the quiver $Q^\vee_{\lambda,\star}$ (with dimension vector $(1,1,\dotsc,1)$). See also \cite[Remark 5.11]{RWReflexive}. 

This approach was taken in the context of describing mirrors of Fano quiver varieties in \cite{Kalashnikov} and for Grassmannians themselves~\cite{CDK}. There the dual quiver is called the \emph{ladder quiver}, inspired by the ladder diagram from \cite{BC-FKvS},
		see also \cite{AltmannvanStraten}.
\end{remark}

\subsection{The Cartier divisors and Picard group of 
$Y(\FF_\lambda)$}\label{sec:Cartier}

In this section we determine the group of torus-invariant Cartier divisors and the Picard group of  
$Y(\FF_\lambda)$
using 
results from \cite{RWReflexive}.

\begin{defn} \label{d:equivalence}
Let $P$ be a finite, ranked poset, and let $P_{\max}$ be its `maximal' extension 
as in \cref{d:QPmax}, 
with new maximal elements $\star_1,\dotsc,\star_d$ and minimal element $\star_0$. 
Let $Q_{P_{\max}}=(\V=\V_\bullet\sqcup\V_\star,\Arr)$ be its associated starred quiver, with 
	$\V_\star=\{\star_0,\star_1,\dotsc,\star_d\}$. 
We define an equivalence relation $\sim$ on $\V_\star$ by letting $\star_i\sim \star_j$  if and only if there exists a $\Z$-labeling  $M:\Arr\to\Z$ with the following properties.
\begin{enumerate}
\item The labels $M(a)$ all lie in $\Z_{\ge -1}$. 
\item The sum of labels along any oriented path from $\star_0$ to 
	any $\star_\ell \in \V_\star$ is equal to $0$. We call a labeling satisfying this condition a \textit{$0$-sum arrow labeling}.
\item 
	The vertices $\star_i$ and $\star_j$ lie in the same connected component of 
		the graph on $\V$ obtained from $Q_{P_{\max}}$ by forgetting the orientation of the arrows and removing all of the edges with labels in $\Z_{\ge 0}$.
\end{enumerate}
 See the left of \cref{f:facetlabeling} for an example.  
\end{defn}

\begin{remark}\label{r:facetlabeling}  A labeling of arrows satisfying   (1) and (2) from \cref{d:equivalence} is called a \textit{face arrow-labeling} in \cite{RWReflexive}, because such arrow-labelings 
	correspond to faces of $\Root(Q)$. Namely, the face associated to such a labeling is the convex hull of the vertices of $\Root(Q)$ corresponding to arrows labeled by $-1$. Moreover, if the labeling is maximal in the sense that the set of arrows labeled $-1$ is maximal by inclusion among face labelings, then the labeling corresponds to a facet of $\Root(Q)$, and is called a \textit{facet arrow-labeling}. The `connected components' associated to a facet labeling as in  \cref{d:equivalence}.(3) are also called \emph{facet components}.
\end{remark}

We have the following lemma.
\begin{lemma}[{\cite[Lemma 5.8]{RWReflexive}}]\label{l:canonicalextension} Suppose $P$ is a finite, ranked poset with $P_{\max}=P\sqcup\{\star_0,\star_1,\dotsc,\star_d\}$, as above. If two elements $\star_i$ and $\star_j$ are equivalent under the equivalence relation from \cref{d:equivalence}, then they have the same rank. Therefore, we have a well-defined ranked poset structure on the quotient $\widetilde{P}:=P_{\max}/\sim$, and $P$ is a subposet of $\widetilde{P}$. We call this new poset $\widetilde{P}$ the  \emph{canonical extension} of $P$. 
\end{lemma}

\begin{defn}\label{d:canonicalquiver} For a finite, ranked poset $P$ with its canonical extension $\widetilde{P}$, as defined in \cref{l:canonicalextension}, the starred quiver $Q_{\widetilde{P}}=(\widetilde{\V}=\widetilde{\V}_\bullet\sqcup\widetilde{\V}_\star,\Arr)$ associated to $\widetilde{P}$ is called the \textit{canonical quiver} for $P$. If $\{\star_{k}\mid k\in K\}$ is an equivalence class in $\V_\star$ we write $\star_K$ for the associated starred vertex in $Q_{\widetilde{P}}$.

In the case of $P=P(\lambda)$ we have 
	$Q_\lambda:=
	Q_{P(\lambda)_{\max}}$,
	and we write $Q_{\widetilde{\lambda}}$ for the canonical quiver $Q_{\widetilde{P}(\lambda)}$. 
\end{defn}
	For $\lambda=(6,6,6,4,2,2,1,1)$,  
	the quiver $Q_\lambda$
	and the quiver $Q_{\widetilde{\lambda}}$ are shown at the left and 
	right of 
\cref{f:facetlabeling}.

\begin{remark}\label{r:arrowsandroots}  The quiver $Q_{\widetilde{P}}$ has the same $\bullet$-vertices as $Q_{P_{\max}}$, namely $\widetilde{\V}_\bullet=P(\lambda)$. 
The difference between the two quivers $Q_{\widetilde{P}}$ and  $Q_{P_{\max}}$ is confined to the sink $\star$-vertices, where certain sink $\star$-vertices of  $Q_{P_{\max}}$ are identified in $Q_{\widetilde{P}}$. It follows that the arrow sets of $Q_{\widetilde{P}}$ and  $Q_{P_{\max}}$ are in 
	natural bijection. This fact also implies that $\Root(Q_{\widetilde{P}})=\Root(Q_{P_{\max}})$. 
\end{remark} 

\begin{remark}\label{r:Weil} Since $\Root(Q_{\lambda})$ is reflexive and terminal by \cref{c:reflexiveterminal}, its vertices and therefore the rays of its face fan $\overline{\FF_\lambda}$ are in 
natural bijection with the arrows of $Q_{\lambda}$. By \cref{r:arrowsandroots} these are also in bijection with the the arrows of $Q_{\widetilde{\lambda}}$. We now focus on $Q_{\widetilde \lambda}$, which is the more useful quiver for describing the toric \textit{Cartier} divisors. We also replace  $Y(\overline{\FF_\lambda})$ by $Y(\FF_\lambda)$ again, via \cref{rem:variablechange}.
\end{remark}
\begin{defn}\label{d:Dtilde}
Let us use the notation $\widetilde D_a$ for the irreducible toric Weil divisor in  $Y(\FF_\lambda)$ associated to an arrow $a$ in $Q_{\widetilde\lambda}$, see \cref{r:Weil}. We obtain a Weil divisor $\sum_{a\in \Arr(Q_{\widetilde{\lambda}})} c_a \widetilde D_a$ in $Y(\FF_\lambda)$ 
 for every arrow labeling $\mathbf c\in\Z^{\Arr(Q_{\widetilde \lambda})}$. 
Given $(\mathbf r,\mathbf {r'})\in\Z^{d}\times\Z^{n-1}$ recall the arrow labeling for $Q_\lambda$  introduced in \cref{def:al}. This arrow-labeling also gives us an arrow-labeling $\mathbf c(\mathbf r,\mathbf r')$ of $Q_{\widetilde\lambda}$, as illustrated in \cref{f:Dtilde}. We consider the associated toric Weil divisor 
\begin{equation}\label{e:toricWeil}
\widetilde D_{(\mathbf r,\mathbf r')}:=
\sum_{a\in \Arr(Q_{\widetilde{\lambda}})} c_a(\mathbf r,\mathbf r') \widetilde D_a,
\end{equation}
as an  analogue in 	$Y(\FF_\lambda)$ of the divisor $D_{(\mathbf r,\mathbf r')}$ in the Schubert variety $X_\lambda$.
\end{defn}

\begin{figure}
	\[\adjustbox{scale=.7,center}{\begin{tikzcd}
	\star_0 & {} \\
	& \bullet & \bullet & \bullet & \bullet & \bullet & \bullet \\
	& \bullet & \bullet & \bullet & \bullet & \bullet & \bullet \\
	& \bullet & \bullet & \bullet & \bullet & \bullet & \bullet & \star_1 \\
	& \bullet & \bullet & \bullet & \bullet & \star_2 \\
	& \bullet & \bullet \\
	& \bullet & \bullet & \star_3 \\
	& \bullet \\
	& \bullet & \star_4
	\arrow["7", from=1-1, to=2-2]
	\arrow["{-1}", from=2-2, to=2-3]
	\arrow["{-1}", from=2-2, to=3-2]
	\arrow["{-1}", from=2-3, to=2-4]
	\arrow["{-1}", from=2-3, to=3-3]
	\arrow["{-1}", from=2-4, to=2-5]
	\arrow["{-1}", from=2-4, to=3-4]
	\arrow["{-1}", from=2-5, to=2-6]
	\arrow["{-1}", from=2-5, to=3-5]
	\arrow["{-1}", from=2-6, to=2-7]
	\arrow["{-1}", from=2-6, to=3-6]
	\arrow["{-1}", from=2-7, to=3-7]
	\arrow["{-1}", from=3-2, to=3-3]
	\arrow["{-1}", from=3-2, to=4-2]
	\arrow["{-1}", from=3-3, to=3-4]
	\arrow["{-1}", from=3-3, to=4-3]
	\arrow["{-1}", from=3-4, to=3-5]
	\arrow["{-1}", from=3-4, to=4-4]
	\arrow["{-1}", from=3-5, to=3-6]
	\arrow["{-1}", from=3-5, to=4-5]
	\arrow["{-1}", from=3-6, to=3-7]
	\arrow["{-1}", from=3-6, to=4-6]
	\arrow["{-1}", from=3-7, to=4-7]
	\arrow["{-1}", from=4-2, to=4-3]
	\arrow["{-1}", from=4-2, to=5-2]
	\arrow["{-1}", from=4-3, to=4-4]
	\arrow["{-1}", from=4-3, to=5-3]
	\arrow["{-1}", from=4-4, to=4-5]
	\arrow["{-1}", from=4-4, to=5-4]
	\arrow["{-1}", from=4-5, to=4-6]
	\arrow["{-1}", from=4-5, to=5-5]
	\arrow["{-1}", from=4-6, to=4-7]
	\arrow["0", from=4-7, to=4-8]
	\arrow["{-1}", from=5-2, to=5-3]
	\arrow["{-1}", from=5-2, to=6-2]
	\arrow["{-1}", from=5-3, to=5-4]
	\arrow["{-1}", from=5-3, to=6-3]
	\arrow["{-1}", from=5-4, to=5-5]
	\arrow["{-1}", from=5-5, to=5-6]
	\arrow["{-1}", from=6-2, to=6-3]
	\arrow["{-1}", from=6-2, to=7-2]
	\arrow["{-1}", from=6-3, to=7-3]
	\arrow["{-1}", from=7-2, to=7-3]
	\arrow["{-1}", from=7-2, to=8-2]
	\arrow["{-1}", from=7-3, to=7-4]
	\arrow["{-1}", from=8-2, to=9-2]
	\arrow["0", from=9-2, to=9-3]
	\end{tikzcd} \hspace{.3cm}
	\begin{tikzcd}
	{\star_0} \\
	& \bullet & \bullet & \bullet & \bullet & \bullet & \bullet \\
	& \bullet & \bullet & \bullet & \bullet & \bullet & \bullet \\
	& \bullet & \bullet & \bullet & \bullet & \bullet & \bullet & {\star_1} \\
	& \bullet & \bullet & \bullet & \bullet \\
	& \bullet & \bullet \\
	& \bullet & \bullet && {\ \ \star_{2,3}} \\
	& \bullet \\
	& \bullet & {\star_4}
	\arrow[from=1-1, to=2-2]
	\arrow[from=2-2, to=2-3]
	\arrow[from=2-2, to=3-2]
	\arrow[from=2-3, to=2-4]
	\arrow[from=2-3, to=3-3]
	\arrow[from=2-4, to=2-5]
	\arrow[from=2-4, to=3-4]
	\arrow[from=2-5, to=2-6]
	\arrow[from=2-5, to=3-5]
	\arrow[from=2-6, to=2-7]
	\arrow[from=2-6, to=3-6]
	\arrow[from=2-7, to=3-7]
	\arrow[from=3-2, to=3-3]
	\arrow[from=3-2, to=4-2]
	\arrow[from=3-3, to=3-4]
	\arrow[from=3-3, to=4-3]
	\arrow[from=3-4, to=3-5]
	\arrow[from=3-4, to=4-4]
	\arrow[from=3-5, to=3-6]
	\arrow[from=3-5, to=4-5]
	\arrow[from=3-6, to=3-7]
	\arrow[from=3-6, to=4-6]
	\arrow[from=3-7, to=4-7]
	\arrow[from=4-2, to=4-3]
	\arrow[from=4-2, to=5-2]
	\arrow[from=4-3, to=4-4]
	\arrow[from=4-3, to=5-3]
	\arrow[from=4-4, to=4-5]
	\arrow[from=4-4, to=5-4]
	\arrow[from=4-5, to=4-6]
	\arrow[from=4-5, to=5-5]
	\arrow[from=4-6, to=4-7]
	\arrow[from=4-7, to=4-8]
	\arrow[from=5-2, to=5-3]
	\arrow[from=5-2, to=6-2]
	\arrow[from=5-3, to=5-4]
	\arrow[from=5-3, to=6-3]
	\arrow[from=5-4, to=5-5]
	\arrow[from=5-5, to=7-5]
	\arrow[from=6-2, to=6-3]
	\arrow[from=6-2, to=7-2]
	\arrow[from=6-3, to=7-3]
	\arrow[from=7-2, to=7-3]
	\arrow[from=7-2, to=8-2]
	\arrow[from=7-3, to=7-5]
	\arrow[from=8-2, to=9-2]
	\arrow[from=9-2, to=9-3]
	\end{tikzcd}}\]
\caption{At left: a facet arrow-labeling of $Q_\lambda$ showing the equivalence of $\star_2$ and $\star_3$ as per \cref{d:equivalence}.
At right: 
the canonical quiver $Q_{\widetilde{\lambda}}$.  Here $\lambda=(6,6,6,4,2,2,1,1)$. 
\label{f:facetlabeling}}
\end{figure}
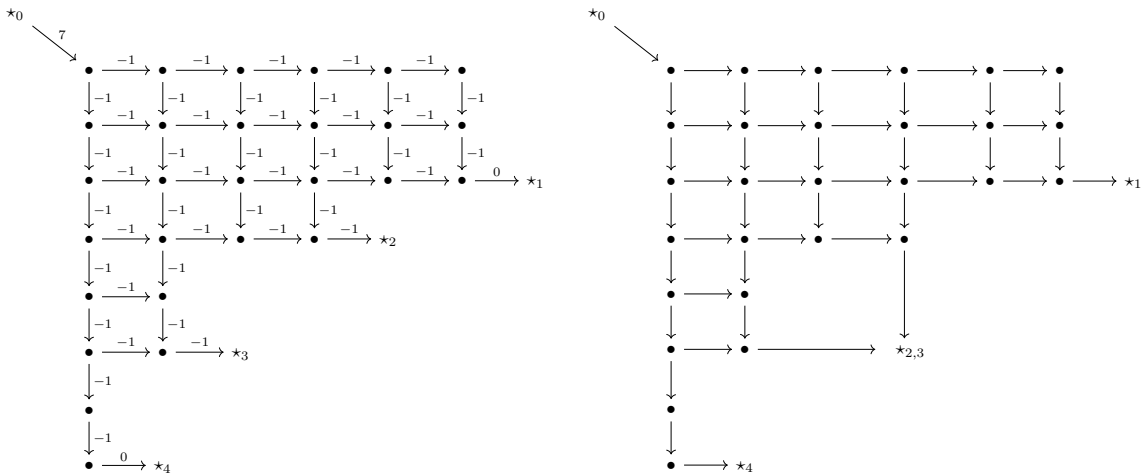

\begin{remark}\label{rem:510}
Note that in terms of the arrow labeling of $Q_{\lambda}$ from \cref{def:al}, the condition from \cref{c:Drr} for the divisor $D_{(\mathbf r,\mathbf r')}$ in the Schubert variety $X_\lambda$ to be Cartier can be reinterpreted as follows. Namely, $D_{(\mathbf r,\mathbf r')}$ is Cartier in $X_\lambda$ if and only if for all oriented paths from $\star_0$ to any $\star_k$, the sum of the arrow labels is the same, independently of $k$. Note that this sum is $R$ if $D_{(\mathbf r,\mathbf r')}$ has degree $R$. 
\end{remark}
The following result is an application of \cite[Theorem~5.18]{RWReflexive}.

\begin{thm}\label{t:toricCartier}
Let $\lambda$ be a partition and $Y(\FF_\lambda)$ the associated toric variety
from \cref{def:toricvarieties}.
Consider the canonical quiver $Q_{\widetilde{\lambda}}$ of $P(\lambda)$, 
from \cref{d:canonicalquiver},
and let $\mathbf c\in\Z^{\Arr(Q_{\widetilde{\lambda}})}$ be an arrow labeling for $Q_{\widetilde{\lambda}}$. Call $\mathbf c$ an \emph{independent-sum  arrow labeling} for $Q_{\widetilde\lambda}$ if  the sums 
\[
s_{\pi_K}(\mathbf c):=\sum_{a\in\pi_K} c_a,
\]
associated to oriented paths $\pi_K$ in $Q_{\widetilde\lambda}$ from $\star_0$ to $\star_K$, depend only on the endpoint $\star_K$ of the path.   The toric Weil divisor $\sum_{a\in \Arr(Q_{\widetilde{\lambda}})} c_a \widetilde D_a$ in $Y(\FF_\lambda)$ associated to  $\mathbf c$ is Cartier if and only if $\mathbf c$ is an independent-sum arrow labeling for $Q_{\widetilde{\lambda}}$.
	
	In particular, if the divisor $D_{(\mathbf r,\mathbf r')}$ is Cartier in $X_\lambda$, then $\widetilde D_{(\mathbf r,\mathbf r')}$  from \eqref{e:toricWeil} is Cartier 
	in 	$Y(\FF_{\lambda})$.
\end{thm}

\begin{proof}
The path independence condition in $Q_{\widetilde\lambda}$ is equivalent to the Cartier condition for $Y(\overline\FF_\lambda)$, as proved in \cite[Theorem~5.18]{RWReflexive}, which translates directly to $Y(\FF_\lambda)$. 
For the second part of the theorem, 
	recall from \cref{rem:510} that 
a boundary divisor $D_{(\mathbf r,\mathbf {r'})}$ is Cartier in the Schubert variety $X_\lambda$
if and only if,
when we use the associated arrow labeling of 
$Q_\lambda$ (or equivalently of $Q_{\widetilde\lambda}$),
for all paths from a starred vertex to a starred vertex, the sum of the arrow labels 
is the same. 
	This implies that
 $\widetilde D_{(\mathbf r,\mathbf {r'})}$ is also Cartier in $Y(\FF_\lambda)$ by the first part.  
\end{proof}

\begin{example}\label{ex:66642211}
Let $\lambda=(6,6,6,4,2,2,1,1)$, so that $n=14$, $k=6$, and $d=4$.
Then $X_\lambda$ is a Schubert variety in $Gr_8(\C^{14})$. There are $d=4$ many Schubert divisors in $X_\lambda$ and $n-1=13$ remaining positroid divisors. Therefore we have boundary Weil divisors $D_{(\mathbf r,\mathbf {r'})}$ in $X_\lambda$ indexed by $4+13$ parameters. We consider the toric divisor $\widetilde D_{(\mathbf r,\mathbf {r'})}$ of $Y(\FF_{\lambda})$ from \cref{d:Dtilde},
	with arrow labeling  shown in \cref{f:Dtilde}.

We can now read off information about both $D_{(\mathbf r,\mathbf {r'})}$ and $\widetilde D_{(\mathbf r,\mathbf {r'})}$ from this arrow-labeling. 
\begin {itemize}
\item The divisor $D_{(\mathbf r,\mathbf {r'})}$ in the Schubert variety is Cartier of degree $R$ if and only if the sum of arrow labels along each path from $\star_0$ to a sink vertex 
	$\star_K$ is equal to $R$. 
	In particular, $D_{(\mathbf r,\mathbf {r'})}$ is Cartier if and only if the sum of labels along a path from $\star_0$ to $\star_K$ is independent of $K$ and the path taken.
\item  The toric divisor $\widetilde D_{(\mathbf r,\mathbf {r'})}$ is Cartier in $Y(\FF_\lambda)$, in this example, 
	if and only 
	if 
each	path from $\star_0$ to $\star_{2,3}$ has the same sum of arrow labels.
Note that the path independence is automatic for paths ending at $\star_1$, and for paths ending in $\star_4$. Also for general $\lambda$, we only obtain relations on $(\mathbf r,\mathbf r')$ whenever there a starred vertex $\star_K$ in $Q_{\widetilde\lambda}$ corresponding to a non-trivial equivalence class of $P(\lambda)_{\max}$.  The Cartier condition for $\widetilde D_{(\mathbf r,\mathbf {r'})}$ is therefore precisely  $r'_{10}+r'_{11}+r_2=r'_4+r'_3+r_3$.
\end{itemize}

\end{example}

We can now give a more direct description of the Cartier divisors as well as the Picard group of $Y(\FF_\lambda)$ in terms of the poset $\widetilde P(\lambda)$. 
\begin{defn} \label{d:tagging}
Let $P$ be a ranked poset with a unique minimal element $\star_0$.
We call a $\Z$-valued function $s: P\to \Z$
with $s(\star_0)=0$ a \emph{normalized tagging} of $P$. 
\end{defn}

\begin{remark}\label{r:tagging} A normalised tagging $s$ of $\widetilde P(\lambda)$ determines an independent-sum arrow labeling $\mathbf c$ of $Q_{\widetilde \lambda}$, by setting $c_a:=s(\vv')-s(\vv)$ 
if $a$ is an arrow from $\vv$ to $\vv'$. Thus a normalised tagging for $\widetilde P(\lambda)$  determines a Cartier divisor for $Y(\FF_\lambda)$, see \cref{t:toricCartier}. Conversely, a toric Cartier divisor $\sum_{a\in\Arr(Q_{\widetilde\lambda})}c_a\widetilde D_a$ comes from a unique normalised tagging that is defined by setting $s(\vv):=\sum_{a\in\pi}c_a$, where $\pi$ is any oriented path from $\star_0$ to $\vv$. This bijection between normalised taggings and toric Cartier divisors is a particular example of \cite[Remark 5.19]{RWReflexive}. Note that the rank function on $\widetilde P(\lambda)$, viewed as a normalised tagging, corresponds to the toric boundary divisor $\sum_{a\in\Arr(Q_{\widetilde \lambda})} D_a$.   \end{remark}

We will now describe the Picard group of $Y(\FF_\lambda)$. Let us write $\widetilde{P}(\lambda)=P(\lambda)\sqcup\{\star_0\}\sqcup\{\star_{K_1},\dotsc,\star_{K_b}\}$, where $\{\star_{K_1},\dotsc,\star_{K_b}\}$ is the set of maximal elements in $\widetilde {P}(\lambda)$. We may identify the $\Z$-lattice of normalized taggings of $\widetilde P(\lambda)$ with $\Z^{P(\lambda)\sqcup\{\star_{K_1},\dotsc,\star_{K_b} \}}$, and with the group of toric Cartier divisors $\operatorname{CDiv}_T(Y(\FF_\lambda))$, see \cref{r:tagging}. Let  $\mathbf M^\lambda_\Z\cong\Z^{P(\lambda)}$ be the character lattice of the torus that acts on $Y(\overline{\FF_\lambda})$, see 
\cref{def:toricvarieties2}, which agrees with the character group for the torus acting on $Y(\FF_\lambda)$ after the coordinate change from~\cref{rem:variablechange}.

\begin{cor}\label{c:toricPicard}  The Picard rank of $Y(\FF_\lambda)$ is equal to the number of maximal elements of $\widetilde P(\lambda)$. Namely, we have the following commutative diagram  of  short exact sequences
\[
\begin{array}{ccccccccc}
0&\longrightarrow &\mathbf M_\Z^\lambda&\longrightarrow &\operatorname{CDiv}_T(Y(\FF_\lambda))
&\longrightarrow &
\Pic(Y(\FF_\lambda))&\longrightarrow &0\\
\|&&\ \downarrow\cong &&\downarrow\cong &&\downarrow &&\|\\
0&\longrightarrow &\Z^{P(\lambda)}&\longrightarrow &\Z^{P(\lambda)\sqcup\{\star_{K_1},\dotsc,\star_{K_b} \}}&\longrightarrow &
\Z^{\{\star_{K_1},\dotsc,\star_{K_b} \}}&\longrightarrow &\ 0,
\end{array}
\] 
whereby the third vertical map $\Pic(Y(\FF_{\lambda})) \to\Z^{\{\star_{K_1},\dotsc,\star_{K_b}\}}$ is an isomorphism.	
Explicitly, the composition $\operatorname{CDiv}_T(Y(\FF_\lambda))\to  \Z^{\{\star_{K_1},\dotsc,\star_{K_b} \}}$ takes the Cartier divisor $\sum_{a\in \Arr(Q_{\widetilde{\lambda}})} c_a \widetilde D_a$ in $Y(\FF_\lambda)$ to the vector $(s_{\pi_{K_i}}(\mathbf c))_{i=1}^b$,
 where $s_{\pi_K}(\mathbf c):=\sum_{a\in\pi_K} c_a$, as in \cref{t:toricCartier}.
\end{cor}
\begin{proof}
This description of the Picard group is a consequence of \cite[Theorem~5.18]{RWReflexive} and \cite[Remark~5.19]{RWReflexive}, which apply to arbitrary ranked posets $P$, applied to the case $P=P(\lambda)$. 
\end{proof}
\begin{remark}
Recall that the Picard group of the Schubert variety $X_\lambda$ has rank $1$. The same holds true for its toric degeneration $Y(\mathcal N_\lambda)$.
Meanwhile, by \cref{c:toricPicard} the rank 
	of the Picard group
	of the partial desingularization
	$Y(\FF_\lambda)$ of $Y(\mathcal N_\lambda)$ is equal to the number of maximal elements in the canonical extension $\widetilde{P}(\lambda)$ of $P(\lambda)$. 
	In particular, 
	$1 \leq \operatorname{rank}(\Pic(Y(\FF_\lambda)))\leq d$, where 
	$d$ is the number of removable boxes in $\lambda$. Interpreting $1$ and $d$ as Betti numbers of $X_\lambda$, we have $b_2(X_\lambda)\le 
	\operatorname{rank}(\Pic(Y(\FF_\lambda)))
	\le b_{2|\lambda|-2}(X_\lambda)$. 

Note that both extreme cases in the inequality above can be realized.
Recall that $-K_{X_\lambda}=[D_{(\mathbf 1,\mathbf 1)}]=\sum_{\ell=1}^{d} n_\ell [D_\ell]$,
where $n_\ell$ is described in \eqref{eq:nell}, and
equals the rank 
of $\star_\ell$
 in $P(\lambda)_{\max}$.  
 If $X_\lambda$ is Gorenstein 
	(as is the case of the Schubert variety $X_{(2,1)}$ considered in \cref{sec:example}) then the $n_\ell$ all coincide. In this case, $P(\lambda)$ is graded and $Y(\mathcal F_\lambda)=Y(\mathcal N_\lambda)$, see \cite[Theorem~4.15]{RWReflexive}. In particular $Y({\FF_\lambda})$ has Picard rank~$1$. 

Suppose on the other hand that $X_\lambda$ is maximally far away from being Gorenstein 
in the sense that the $n_\ell$ are pairwise distinct.  Then no two starred vertices of $Q_{\lambda}$ 
	are equivalent, see \cref{l:canonicalextension}. Therefore  $\widetilde{P}(\lambda)=P(\lambda)_{\max}$ and we have that $\Pic(Y(\mathcal F_\lambda))\cong \Z^d$. 

Our  \cref{ex:66642211} with $\lambda=(6,6,6,4,2,2,1,1)$ 
lies in between these two extreme cases. We have that $d=4$ and $\Pic(Y(\mathcal F_\lambda))\cong \Z^3$. 

Finally, we note that the full resolution $Y(\widehat\FF_{\lambda})$ of $Y(\CC_\lambda)$ always has Picard rank $d$, compare~\cite[Proposition~5.25]{RWReflexive}.
\end{remark}

We now give a description of the nef cone and the ample cone of $Y(\FF_\lambda)$. 

\begin{defn}\label{d:Gammatilde}
	Suppose $s$ is a normalised tagging of $\widetilde P(\lambda)$. Let us denote by $\widetilde\Gamma(s) \subseteq \mathbf M^{\lambda}_\R$ the polytope  given by inequalities 
\begin{equation}\label{e:momentineqY}
\langle y,u_a\rangle + s(h(a))-s(t(a))\ge 0, \qquad a\in \Arr(Q_{\widetilde P(\lambda)})
\end{equation} 
	on $y\in\mathbf M^\lambda_\R$, where $h(a)$ and $t(a)$ are head and tail of the arrow $a$, respectively, and $u_a$ is as in \cref{def:Root} 
	This is the polytope in vertex coordinates associated to the toric  divisor $\widetilde D(s)$ as in \cite[Section 4.3]{CLS}. If $s$ is the normalised tagging associated to the arrow labeling 
	from \cref{def:al}, then $\widetilde\Gamma(s)$ equals the superpotential
	polytope $\overline{\Gamma}^\lambda_\rect(\mathbf r,\mathbf r')$.  
\end{defn}
\begin{defn}\label{d:nef} 
Consider the poset $\widetilde P(\lambda)$ and the quiver $Q_{\widetilde \lambda}$ associated to its Hasse diagram. 
	We define a new poset structure on the set $\widetilde P(\lambda)_\star:=\{\star_{K_1},\dotsc, \star_{K_b}\}$ of maximal elements of $\widetilde P(\lambda)$ as follows. Namely, we call $\star_K, \star_{K'}$  \textit{comparable} if there exists a $0$-sum arrow labeling of $Q_{\widetilde\lambda}$ with labels in $\Z_{\ge -1}$ and an unoriented simple path between $\star_K$ and $\star_{K'}$ such that precisely one of the arrows in the path has a nonnegative label. The partial order is generated by setting $\star_K<\star_{K'}$ if $\star_K$ and $\star_{K'}$ are comparable and $\rk(\star_K)< \rk(\star_{K'})$.

\end{defn}
\begin{remark}\label{r:PosArrow}
Note that if $\star_K$ and $\star_{K'}$ are comparable and $\rk(\star_K)<\rk(\star_{K'})$, then the unique arrow with a nonnegative label in the path between $\star_K$ and $\star_{K'}$ from the definition above must be oriented towards $\star_{K'}$. This follows immediately from the $0$-sum condition on the labeling and the fact that all the other arrows are labeled $-1$. 
\end{remark}
\begin{rem}\label{r:StarPoset}
	The comparability  condition on $\star_K,\star_{K'}$ used in \cref{d:nef}  could also be described as saying that  for some facet labeling, the starred vertices $\star_K$ and $\star_{K'}$ lie in facet components that are `adjoining' (connected by a single arrow), 
	compare \cref{r:facetlabeling}. We observe that the poset $\widetilde P(\lambda)_\star$ has a unique minimal element. Namely, there is a facet arrow-labeling of $Q_{\lambda}$ for which the arrows labeled $-1$ are precisely the arrows between normal vertices and those pointing to a minimal-rank starred vertex $\star_k$. An example of such a labeling is shown in \cref{f:facetlabeling} on the left. This facet arrow-labeling connects all of the minimal-rank  sink vertices $\star_k$ of $Q_\lambda$, so that they are identified to a single vertex in $Q_{\widetilde \lambda}$, as shown in \cref{f:facetlabeling} on the right. Moreover via this labeling we see that the other starred vertices are all comparable to this minimal-rank starred vertex, in the partial order of $\widetilde P(\lambda)_\star$. This makes it the unique minimal element of the poset $\widetilde P(\lambda)_\star$. 
\end{rem}  
  
\begin{prop}\label{p:toricCartierAmple} 
Let $s:\widetilde P(\lambda)\to \Z$ be a normalised tagging and $\widetilde D(s)=\sum_a c_a\widetilde D_a$ the corresponding toric Cartier divisor of $Y(\FF_\lambda)$, as in \cref{r:tagging}. 
\begin{enumerate}
\item The polytope $\widetilde \Gamma(s)$ associated to $\widetilde D(s)$ is full-dimensional  if and only if $s(\star_K)>0$ 
	for all $\star_K\in\widetilde P(\lambda)_\star$. 
\item $\widetilde D(s)$ is \emph{nef} if and only if  $s(\star_K)\ge 0$ 
	for all $\star_K\in\widetilde P(\lambda)_\star$,
		and $s(\star_K)\le s(\star_{K'})$ whenever $\star_K\le \star_{K'}$ 
		in the partial order on
		$\widetilde P(\lambda)_\star$ from
\cref{d:nef}. 
\item $\widetilde D(s)$ is \emph{ample} if and only if $s(\star_K)> 0$ 
	for all $\star_K\in\widetilde P(\lambda)_\star$,
	and $s(\star_K)< s(\star_K')$ whenever $\star_K <\star_{K'}$ in 
		$\widetilde P(\lambda)_\star$.  
\end{enumerate} 
If $\widetilde D(s)$ is ample then it is also very ample. 
\end{prop}

\begin{proof} Setting $s(\vv)=0$ for normal vertices $\vv\in P(\lambda)$ (while leaving the $s(\star_K)$ as they are) replaces $\widetilde D(s)$ by a linearly equivalent divisor by \cref{c:toricPicard}, and 
amounts to a shift of the polytope. Therefore we may assume $s(\vv)=0$ for $\vv\in P(\lambda)$. Let us write $s_K$ for $s(\star_K)$. 
	Now the $s_K$ define a marking $\mathbf s$ of $\widetilde P(\lambda)$, and $\widetilde\Gamma(s)$ can be interpreted as the marked order polytope $\mathbb O^{\mathbf s}(\widetilde P(\lambda))$, so that (1) follows, see   \cref{rem:marked}.

To prove (2) and (3)	we need to understand the Cartier datum of $\widetilde D(s)=\sum_{a}c_a(s)\widetilde D_a$. Recall that by our choices above we have that $c_a(s)=s_K$ if the head $h(a)=\star_K$, and $c_a(s)=0$ otherwise.  Suppose $\sigma$ is a maximal cone in $\overline{\FF_\lambda}$ and let us denote by $M_\sigma$ its associated facet arrow-labeling, see \cref{r:facetlabeling}. The vertices of $Q_{\widetilde\lambda}$ are decomposed into a disjoint union of 
	\textit{facet components} as in the end of \cref{r:facetlabeling}, and we have precisely one facet component  for each starred vertex \cite[Lemma~5.14]{RWReflexive}. Let $m_\sigma\in M^\lambda_{\R}$ be defined 
	by setting  $m_{\sigma,\vv}=s_K$  
	if $\vv\in P(\lambda)$ lies in the facet component of $\star_K$. Then we have that  
	\begin{equation}\label{eq:innerproduct}
\langle m_\sigma,u_a\rangle= \begin{cases}
0 & \text{if $M_\sigma(a)=-1$ and $h(a)\in P(\lambda)$,}\\
-s_K & \text{if $M_\sigma(a)=-1$ and $h(a)=\star_K$.}
\end{cases} 
	\end{equation}
Thus for every primitive ray generator $u_a$ in $\sigma$ we have $\langle m_\sigma,u_a\rangle=-c_a(s)$, so that $(m_\sigma)_\sigma$ is the Cartier datum for $\widetilde D(s)$. 
We now recall that a Cartier divisor in a complete toric variety is nef if and only if it is basepoint free \cite[Theorem 6.3.12]{CLS}, and this is equivalent to the condition that  $m_\sigma\in\widetilde\Gamma(s)$ for all maximal cones $\sigma$, see \cite[Proposition 6.1.1]{CLS}. Recall also that the ample cone is the interior of the nef cone, therefore (3) will follow from (2).
	We proceed to prove (2) using the above  characterisation of the nef property. 

	Let us assume that the condition from (2) holds for $s$ and show that $\widetilde D(s)$ is then nef. Pick some maximal cone $\sigma$ and consider  $m_\sigma$ as above. If $a$ is an arrow with  $M_\sigma(a)=-1$, then by \eqref{eq:innerproduct}, the inequality  \eqref{e:momentineqY} holds for $y=m_\sigma$ as an equality (keeping in mind $s(\vv)=0$ and $s(\star_K)=s_K$). Otherwise, if $M_\sigma(a)\ge 0$, then $h(a)$ and $t(a)$ lie in different facet components for $M_\sigma$. Suppose $t(a)$ lies in the facet component of $\star_K$ and $h(a)$ in the facet component of $\star_{K'}$.  We have $m_{\sigma,t(a)}=s_K$ 
	and $m_{\sigma,h(a)}=s_{K'}$ by construction of $m_\sigma$. Here $\star_K$ may equal to $\star_0$ in which case we set $s_0=0$. Altogether, we see that
\[
\langle m_\sigma,u_a\rangle=\begin{cases}-s_K&\text{if $h(a)=\star_{K'}$,}\\
s_{K'}-s_K&\text{otherwise,}\end{cases}\qquad\text{ and }\qquad
	s(h(a))-s(t(a))=\begin{cases}s_{K'}&\text{if $h(a)=\star_{K'}$,}\\
		0&\text{otherwise. }
\end{cases} 
\]
	Now recall that we have $\rk(\star_K)<\rk(\star_{K'})$ by \cref{r:PosArrow}. Thus  the assumption in (2) implies  $s_K\le s_{K'}$. This implies that $m_\sigma$ satisfies the inequality \eqref{e:momentineqY} for this arrow $a$, since $
\langle m_\sigma,u_a\rangle+s(h(a))-s(t(a))=s_{K'}-s_K$, by the above. Now we have proved that $m_\sigma$ lies in $\widetilde \Gamma(s)$ for all maximal cones $\sigma$. It follows that $\widetilde D(s)$ is nef. The proof that $\widetilde D(s)$ being nef implies the condition in (2) is obtained by the analogous arguments in reverse.  

\begin{comment}
ALTERNATIVE proof of ampleness: We may also see that the inequalities in (3) imply ampleness directly. Suppose $\sigma$ and $\sigma'$ are maximal cones separated by a wall. Let $M_\sigma$ and $M_{\sigma'}$ be the associated facet labelings, and consider the face-labeling $M_{\sigma\cap\sigma'}$ which is $-1$ on all arrows $a$ with $M_{\sigma}(a)=M_{\sigma'}(a)=-1$, and which is sub-maximal (dominated only by $M_{\sigma}$ and $M_{\sigma'}$). Pick a ray that is contained in $\sigma$ but not in $\sigma'$, say. For the associated arrow $a$, either $h(a)$ or $t(a)$ is a normal vertex $\vv$.    This element $\vv\in P(\lambda)$ lies in one facet component for $M_\sigma$ (say the facet component of $\star_K$) and another facet component for $M_{\sigma'}$ (say the facet component of $\star_{K'}$). If $\star_K$ and $\star_{K'}$ are both maximal elements then $\star_K$ and $\star_{K'}$ are comparable via a path involving the arrow $a$. Now for example if $\star_K<\star_{K'}$ we have $s(\star_K)<s(\star_{K'})$ and vice versa in the other case. Either way, $s(\star_K)\ne s(\star_{K'})$. Since  $m_{\sigma,\vv}=s(\star_K)$ and $m_{\sigma'}(\vv)=s(\star_{K'})$ we have that $m_{\sigma}\ne m_{\sigma'}$. Similarly if one of the starred vertices, say $\star_K$, is equal to $\star_0$, using that   $s(\star_{K'})>0$. This implies that $\widetilde D(s)$ is ample. 
\end{comment}

Finally, recall that $\mathbb O^{\mathbf s}(\widetilde P(\lambda))$ has the integer decomposition property (IDP) by \cite[Corollary 2.3]{FangFourier}, compare \cref{cor:IDP}. This implies that the polytope $\widetilde\Gamma(s)$, and thus the divisor $\widetilde D(s)$ (if ample), is  very ample, see~\cite[Proposition~2.2.18]{CLS}. 
\end{proof}

\begin{remark}
Recall \cref{c:Drr} along with associated notation. By this corollary, the boundary divisor $D_{(\mathbf r,\mathbf r')}$ in the Schubert variety $X_\lambda$ is ample if and only if the quantities 
\[
R_\ell:=\sum r_\ell+\sum_{i\in NW(b_{\rho_{2\ell-1}})}r'_i 
\]
are independent of $\ell=1,\dotsc, d$, and given by a positive integer $R$. We now consider the analagous toric divisor $\widetilde D_{(\mathbf r,\mathbf {r'})}$ in $Y(\FF_{\lambda})$ from \cref{d:Dtilde}.  \cref{p:toricCartierAmple} implies that whenever $D_{(\mathbf r,\mathbf r')}$ is ample, then  $\widetilde D_{(\mathbf r,\mathbf {r'})}$ is big and nef; namely $\widetilde D_{(\mathbf r,\mathbf {r'})}=\widetilde D(s)$ for a normalised tagging $s$ with $s(\star_K)=R$ for all~$\star_K\in \widetilde P(\lambda)_\star$, so that both (1) and (2) from \cref{p:toricCartierAmple} apply. 
Note that the polytope $\widetilde \Gamma(s)$ associated  to $\widetilde D(s)$ equals to $\overline\Gamma^\lambda_\rect(\mathbf r,\mathbf r')$. All together it follows that, when $D_{(\mathbf r,\mathbf r')}$ is ample, we obtain via $\widetilde D_{(\mathbf r,\mathbf {r'})}$ a proper birational morphism to the projective toric variety $\mathbb P_{\Gamma^\lambda_\rect(\mathbf r,\mathbf r')}$ associated to
$\Gamma^\lambda_\rect(\mathbf r,\mathbf r')$,  
\begin{equation}\label{e:GenDesing}
Y(\FF_{\lambda})\to \mathbb P_{\Gamma^\lambda_\rect(\mathbf r,\mathbf r')}.
\end{equation}
 Since $\Gamma^\lambda_{\rect}(\mathbf r,\mathbf r')$ is in fact a translation of a dilation of $\Gamma^\lambda_{\rect}(\mathbf 1,\mathbf 0)$, as we saw in \cref{p:GammaShiftRect}, we have that $\mathbb P_{\Gamma^\lambda_\rect(\mathbf r,\mathbf r')}$ is isomorphic to  $Y(\CC_\lambda)$ and \eqref{e:GenDesing} generalises the partial desingularisation $Y(\FF_\lambda)\to Y(\CC_\lambda)$ from \cref{cor:small}. The morphism \eqref{e:GenDesing} is an isomorphism (and  $\widetilde D_{(\mathbf r,\mathbf {r'})}$ is ample) if and only if  $X_\lambda$ was Gorenstein. 

Note  that $\Gamma^\lambda_{\rect}(\mathbf r,\mathbf r')$ agrees with the Newton-Okounkov convex body $\Delta_\rect^\lambda(\mathbf r,\mathbf r')$ of $X_\lambda$ associated to $D_{(\mathbf r,\mathbf r')}$, see \cref{t:maingen}. Moreover, $\Gamma^\lambda_{\rect}(\mathbf r,\mathbf r')$ has the integer decomposition property by  \cref{cor:IDP}. Therefore we have a toric degeneration of $X_\lambda$ to $\mathbb P_{\Gamma^\lambda_\rect(\mathbf r,\mathbf r')}$ given by \cref{cor:degenerationgen}.
\end{remark}

\section{The superpotential for skew shaped positroid varieties}\label{sec:skew}

In this section we define \emph{skew shaped positroid varieties} in 
a Grassmannian,
and we define a superpotential associated to each one.
We then outline how the proofs of our main results for Schubert varieties
extend to this setting.  We note that some of our results also 
extend to the setting of \emph{positroid varieties}, but these will 
be studied separately.

\subsection{Skew shaped positroid varieties}

In this section we define skew shaped positroid varieties.

\begin{notation}
Let $\nu \subseteq \lambda$ be partitions such that 
the $(n-k) \times k$ rectangular Young diagram
is the minimal rectangle containing $\lambda$.
Let $d$ denote the number of removable boxes in $\lambda$.
\end{notation}

The following generalizes \cref{d:frozens}.
\begin{definition}[Frozen shapes for $\lambda/\nu$]\label{d:skewfrozens}
Consider our skew partition $\lambda/\nu$ 
with its bounding $(n-k)\times k$ rectangle.
As in \cref{d:frozens},
we denote the $i$-th
 box in the rim of $\lambda$ by $b_i$ (numbered from northeast to southwest), 
and 
 write $\Rect(b)$ for the maximal rectangle 
whose lower right hand corner is the box $b$.
We also define 
	\begin{equation}\label{eq:shape}
		\sh(b) = \Rect(b) \cup \nu \hspace{.2cm} \text{ and } \hspace{.2cm}
		\sh(b)^-:=\Rect(b)^- \cup \nu,
\end{equation}
where (as in \cref{d:addable}) $\Rect(b)^-$ is the rectangle
obtained from $\Rect(b)$ by removing the rim.

Note that  if $\nu = \emptyset$,
$\sh(b) = \Rect(b)$.
Let \[\mu_i:=\sh(b_i) \text{ and } \mu_n:=\nu.\]
 We let $\Fr(\lambda/\nu) = 
\{\mu_1,\dots, \mu_{n-1}, \mu_n\} \subseteq \mathcal{P_{\lambda/\nu}}$, and call
 the elements of $\Fr(\lambda/\nu)$ the \emph{frozen shapes} 
for $\lambda/\nu$. 
We treat the indices modulo $n$ so that $\mu_{n+1}=\mu_1$.
\end{definition}

\begin{definition}\label{d:SkewSchubertCell}
The \emph{open skew shaped positroid variety} 
$X^{\circ}_{\lambda/\nu}$ is defined to be 
\begin{equation*}
X_{\lambda/\nu}^{\circ}:= 
	\{x\in Gr_{n-k}(\C^n) \ | \ P_{\mu}(x)=0 \text{ unless }
	           \nu \subseteq \mu \subseteq \lambda, \text{ and } P_{\mu_i}(x)\ne 0 \ \forall i\in[n] \}.
\end{equation*}
\end{definition}
We similarly have an (open) skew shaped positroid variety
in the Langlands dual Grassmannian, denoted
 $\check{X}_{\lambda/\nu}$.
The dimension of these varieties is $|\lambda/\nu|$, the 
number of boxes in the skew Young diagram $\lambda/\nu$.

\subsection{The definition of the superpotential for skew shaped positroid varieties}

Let $\BB^{\SE}_{\out}(\lambda/\nu)$ denote the set of southeast ``outer'' corners of the 
skew-shape $\lambda/\nu$, that is, the set of boxes $b$ of 
$\lambda/\nu$ such that $b$ has no boxes below it or to its 
right.
And we let $\BB^{\NW}_{\out}(\lambda/\nu)$ denote the northwest ``outer'' corners of the 
skew-shape $\lambda/\nu$, that is, the  boxes $b'$ of 
$\lambda/\nu$ such that $b'$ has no boxes above it or to its left.

\begin{definition}
Given a skew shape $\lambda/\nu$ in a bounding rectangle of size $(n-k) \times k$,
we
number its  rows from $1$ to $n-k$ from top to bottom,
and the columns from $1$ to $k$ from left to right.
For $1\leq i \leq n-k-1$, find the maximal-width rectangle $R_i$ of height $2$
that is contained in rows $i, i+1$ of $\lambda/\nu$ (if one exists).
If $R_i$ exists, let $d_i$ and $c_i$ denote its
northeast and southwest corner,
respectively.
Similarly, for $1\leq j \leq k-1$, find the maximal-height rectangle $R^j$
of width $2$ contained in columns $j, j+1$ of $\lambda/\nu$ (if one exists).
If $R^j$ exists, 
let $d^j$ and $c^j$ denote its southwest and northeast corner, respectively.
\end{definition}

The following definition of superpotential $W^{\lambda/\nu}$
for a skew shaped positroid variety generalizes our formula from \cref{p:superpotential}
in the Schubert variety case.
We use the notation from \eqref{eq:shape}.

\begin{definition}[Canonical formula for the superpotential]\label{d:skewLG1}
Let $\lambda/\nu$ be a skew shape.  We define
\begin{equation}\label{e:anotherSchub}
W^{\lambda/\nu} 
        = 
	\sum_{b=b_{\rho_{2\ell-1}} \in \BB^{\SE}_{\out}(\lambda/\nu)}  
        q_\ell \frac{p_{\sh(b)^-}}{p_{\sh(b)}}+
	\sum_{b'\in \BB^{\NW}_{\out}(\lambda/\nu)} \frac{p_{\sh(b')}}{p_{\nu}} 
+        \sum_{i=1}^{n-k-1} 
	\frac{p_{\sh(d_i) \cup \sh(c_i)}}{p_{\sh(d_i)}}
        + \sum_{j=1}^{k-1} 
	\frac{p_{\sh(d^j) \cup \sh(c^j)}}{p_{\sh(d^j)}},
\end{equation}
	where if for some $i$ (respectively, $j$) the rectangle 
	$R_i$ (respectively, $R^j$) does not exist, then
	the corresponding term 
	$\frac{p_{\sh(d_i) \cup \sh(c_i)}}{p_{\sh(d_i)}}$
	(respectively, 
	$\frac{p_{\sh(d^j) \cup \sh(c^j)}}{p_{\sh(d^j)}}$)
	above is understood to be $0$.
\end{definition}

\begin{figure}[h]
\centering
\setlength{\unitlength}{1.3mm}
\begin{center}
\resizebox{0.12\textwidth}{!}{\begin{picture}(50,35)
% Young Diagram
% horizontal lines
  \put(32,32){\line(1,0){18}}
  \put(14,23){\line(1,0){36}}
  \put(5,14){\line(1,0){27}}
  \put(5,5){\line(1,0){27}}
  \put(5,-4){\line(1,0){27}}
% vertical lines
  \put(5,-4){\line(0,1){18}}
  \put(14,-4){\line(0,1){27}}
  \put(23,-4){\line(0,1){27}}
  \put(32,-4){\line(0,1){36}}
  \put(41,23){\line(0,1){9}}
  \put(50,23){\line(0,1){9}}
% first row
        \put(8,35){\Huge{$1$}}
        \put(17,35){\Huge{$2$}}
        \put(26,35){\Huge{$3$}}
        \put(35,35){\Huge{$4$}}
        \put(44,35){\Huge{$5$}}
        \put(0,25){\Huge{$1$}}
        \put(0,15){\Huge{$2$}}
        \put(0,5){\Huge{$3$}}
        \put(0,-4){\Huge{$4$}}
% arrows
         \put(44,25){\Huge{*}}
         \put(26,-4){\Huge{*}}
        \end{picture}}
        \hspace{0.1cm}
\resizebox{0.12\textwidth}{!}{\begin{picture}(50,35)
% Young Diagram
% horizontal lines
  \put(32,32){\line(1,0){18}}
  \put(14,23){\line(1,0){36}}
  \put(5,14){\line(1,0){27}}
  \put(5,5){\line(1,0){27}}
  \put(5,-4){\line(1,0){27}}
% vertical lines
  \put(5,-4){\line(0,1){18}}
  \put(14,-4){\line(0,1){27}}
  \put(23,-4){\line(0,1){27}}
  \put(32,-4){\line(0,1){36}}
  \put(41,23){\line(0,1){9}}
  \put(50,23){\line(0,1){9}}
% first row
        \put(8,35){\Huge{$1$}}
        \put(17,35){\Huge{$2$}}
        \put(26,35){\Huge{$3$}}
        \put(35,35){\Huge{$4$}}
        \put(44,35){\Huge{$5$}}
        \put(0,25){\Huge{$1$}}
        \put(0,15){\Huge{$2$}}
        \put(0,5){\Huge{$3$}}
        \put(0,-4){\Huge{$4$}}
% arrows
         \put(35,25){\Huge{*}}
         \put(17,16){\Huge{*}}
         \put(8,7){\Huge{*}}
        \end{picture}}
        \hspace{0.1cm}
\resizebox{0.12\textwidth}{!}{\begin{picture}(50,35)
% Young Diagram
% horizontal lines
  \put(32,32){\line(1,0){18}}
  \put(14,23){\line(1,0){36}}
  \put(5,14){\line(1,0){27}}
  \put(5,5){\line(1,0){27}}
  \put(5,-4){\line(1,0){27}}
% vertical lines
  \put(5,-4){\line(0,1){18}}
  \put(14,-4){\line(0,1){27}}
  \put(23,-4){\line(0,1){27}}
  \put(32,-4){\line(0,1){36}}
  \put(41,23){\line(0,1){9}}
  \put(50,23){\line(0,1){9}}
% first row
        \put(8,35){\Huge{$1$}}
        \put(17,35){\Huge{$2$}}
        \put(26,35){\Huge{$3$}}
        \put(35,35){\Huge{$4$}}
        \put(44,35){\Huge{$5$}}
        \put(0,25){\Huge{$1$}}
        \put(0,15){\Huge{$2$}}
        \put(0,5){\Huge{$3$}}
        \put(0,-4){\Huge{$4$}}
% arrows
         \put(26,16){\Huge{$d_2$}}
         \put(17,7){\Huge{$c_2$}}
        \end{picture}}
        \hspace{0.1cm}
\resizebox{0.12\textwidth}{!}{\begin{picture}(50,35)
% Young Diagram
% horizontal lines
  \put(32,32){\line(1,0){18}}
  \put(14,23){\line(1,0){36}}
  \put(5,14){\line(1,0){27}}
  \put(5,5){\line(1,0){27}}
  \put(5,-4){\line(1,0){27}}
% vertical lines
  \put(5,-4){\line(0,1){18}}
  \put(14,-4){\line(0,1){27}}
  \put(23,-4){\line(0,1){27}}
  \put(32,-4){\line(0,1){36}}
  \put(41,23){\line(0,1){9}}
  \put(50,23){\line(0,1){9}}
% first row
        \put(8,35){\Huge{$1$}}
        \put(17,35){\Huge{$2$}}
        \put(26,35){\Huge{$3$}}
        \put(35,35){\Huge{$4$}}
        \put(44,35){\Huge{$5$}}
        \put(0,25){\Huge{$1$}}
        \put(0,15){\Huge{$2$}}
        \put(0,5){\Huge{$3$}}
        \put(0,-4){\Huge{$4$}}
% arrows
         \put(26,7){\Huge{$d_3$}}
         \put(8,-2){\Huge{$c_3$}}
       %  \put(26,16){\Huge{$d_2$}}
       %  \put(17,7){\Huge{$c_2$}}
        \end{picture}}
        \hspace{0.1cm}
\resizebox{0.12\textwidth}{!}{\begin{picture}(50,35)
% Young Diagram
% horizontal lines
  \put(32,32){\line(1,0){18}}
  \put(14,23){\line(1,0){36}}
  \put(5,14){\line(1,0){27}}
  \put(5,5){\line(1,0){27}}
  \put(5,-4){\line(1,0){27}}
% vertical lines
  \put(5,-4){\line(0,1){18}}
  \put(14,-4){\line(0,1){27}}
  \put(23,-4){\line(0,1){27}}
  \put(32,-4){\line(0,1){36}}
  \put(41,23){\line(0,1){9}}
  \put(50,23){\line(0,1){9}}
% first row
        \put(8,35){\Huge{$1$}}
        \put(17,35){\Huge{$2$}}
        \put(26,35){\Huge{$3$}}
        \put(35,35){\Huge{$4$}}
        \put(44,35){\Huge{$5$}}
        \put(0,25){\Huge{$1$}}
        \put(0,15){\Huge{$2$}}
        \put(0,5){\Huge{$3$}}
        \put(0,-4){\Huge{$4$}}
% arrows
         \put(17,7){\Huge{$c^1$}}
         \put(8,-2){\Huge{$d^1$}}
        \end{picture}}
        \hspace{0.1cm}
\resizebox{0.12\textwidth}{!}{\begin{picture}(50,35)
% Young Diagram
% horizontal lines
  \put(32,32){\line(1,0){18}}
  \put(14,23){\line(1,0){36}}
  \put(5,14){\line(1,0){27}}
  \put(5,5){\line(1,0){27}}
  \put(5,-4){\line(1,0){27}}
% vertical lines
  \put(5,-4){\line(0,1){18}}
  \put(14,-4){\line(0,1){27}}
  \put(23,-4){\line(0,1){27}}
  \put(32,-4){\line(0,1){36}}
  \put(41,23){\line(0,1){9}}
  \put(50,23){\line(0,1){9}}
% first row
        \put(8,35){\Huge{$1$}}
        \put(17,35){\Huge{$2$}}
        \put(26,35){\Huge{$3$}}
        \put(35,35){\Huge{$4$}}
        \put(44,35){\Huge{$5$}}
        \put(0,25){\Huge{$1$}}
        \put(0,15){\Huge{$2$}}
        \put(0,5){\Huge{$3$}}
        \put(0,-4){\Huge{$4$}}
% arrows
         \put(26,16){\Huge{$c^2$}}
         \put(17,-2){\Huge{$d^2$}}
        \end{picture}}
        \hspace{0.1cm}
\resizebox{0.12\textwidth}{!}{\begin{picture}(50,35)
% Young Diagram
% horizontal lines
  \put(32,32){\line(1,0){18}}
  \put(14,23){\line(1,0){36}}
  \put(5,14){\line(1,0){27}}
  \put(5,5){\line(1,0){27}}
  \put(5,-4){\line(1,0){27}}
% vertical lines
  \put(5,-4){\line(0,1){18}}
  \put(14,-4){\line(0,1){27}}
  \put(23,-4){\line(0,1){27}}
  \put(32,-4){\line(0,1){36}}
  \put(41,23){\line(0,1){9}}
  \put(50,23){\line(0,1){9}}
% first row
        \put(8,35){\Huge{$1$}}
        \put(17,35){\Huge{$2$}}
        \put(26,35){\Huge{$3$}}
        \put(35,35){\Huge{$4$}}
        \put(44,35){\Huge{$5$}}
        \put(0,25){\Huge{$1$}}
        \put(0,15){\Huge{$2$}}
        \put(0,5){\Huge{$3$}}
        \put(0,-4){\Huge{$4$}}
% arrows
         \put(44,25){\Huge{$c^4$}}
         \put(35,25){\Huge{$d^4$}}
        \end{picture}}
        \hspace{0.1cm}
\end{center}
        \hspace{0.1cm}
	\caption{When $\lambda=(5,3,3,3)$ and $\nu = (3,1)$,
	we compute the superpotential from the diagram above.
	Note that the $\star$'s in the leftmost figure indicate the 
	southeast corners, and the $\star$'s in the adjacent 
	figure indicate the northwest corners.}
\label{fig:SkewExample}
\end{figure}
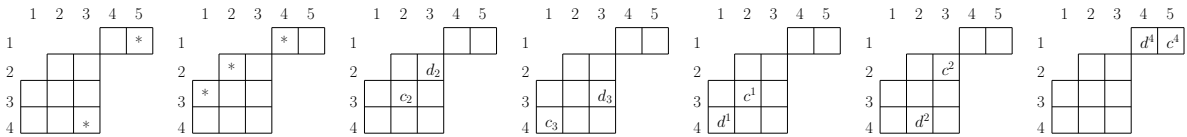

\begin{example}\label{ex:SkewExample}
When $\lambda=(5,3,3,3)$ and $\nu = (3,1)$,
we use the diagrams in \cref{fig:SkewExample} to compute
the superpotential, 
obtaining
$$W^{\lambda/\nu} = 
q_1 \frac{p_{\ydiagram{3,1}}}{p_{\ydiagram{5,1}}} +
q_2 \frac{p_{\ydiagram{3,2,2}}}{p_{\ydiagram{3,3,3,3}}} +
 \frac{p_{\ydiagram{4,1}}}{p_{\ydiagram{3,1}}} +
 \frac{p_{\ydiagram{3,2}}}{p_{\ydiagram{3,1}}} +
 \frac{p_{\ydiagram{3,1,1}}}{p_{\ydiagram{3,1}}} +
 \frac{p_{\ydiagram{3,3,2}}}{p_{\ydiagram{3,3}}} +
 \frac{p_{\ydiagram{3,3,3,1}}}{p_{\ydiagram{3,3,3}}} +
 \frac{p_{\ydiagram{3,2,2,1}}}{p_{\ydiagram{3,1,1,1}}} +
 \frac{p_{\ydiagram{3,3,2,2}}}{p_{\ydiagram{3,2,2,2}}} +
 \frac{p_{\ydiagram{5,1}}}{p_{\ydiagram{4,1}}}.
$$
\end{example}

In the setting of Schubert varieties $X_\lambda$  we had a formula for the canonical superpotential in which the non-quantum-parameter terms were indexed by the $n-1$ boxes along the NW rim of $\lambda$, see \cref{p:3rdpotential}. The number of summands of the canonical superpotential overall was thereby seen to be $n-1+d$ (with $d$ the number of boxes in $\BB^{\SE}_{\out}(\lambda)$). In particular, we had one term for each of the $d$ Schubert divisors, and one for each of the remaining positroid divisors.

We  now extend this description of the canonical superpotential  to 
the skew case.  Let $\BB^{\NW}(\lambda/\nu)$ be the set of boxes in $\lambda/\nu$ along the $\NW$ boundary. We also call this set of boxes the \textit{inner rim} of $\lambda\setminus \nu$. Similarly consider $\BB^{\SE}(\lambda/\nu)$ the set of boxes in $\lambda\setminus\nu$ along the $\SE$ boundary, and call this set the \textit{outer rim}. 

We divide the inner rim into different types of boxes, 
\begin{equation*}
\begin{array}{ccc}
\BB^{\NW}(\lambda/\nu)&=&\BB^{\NW}_{\out}(\lambda/\nu)\sqcup \BB^{\NW}_{\mathrm{hor}}(\lambda/\nu)\sqcup \BB^{\NW}_{\mathrm{vert}}(\lambda/\nu)\sqcup\BB^{\NW}_{\mathrm{in}}(\lambda/\nu),
\end{array}
\end{equation*}
where the segments are defined as follows: 
\begin{itemize}
\item `$\out$' refers to boxes in the inner rim   with no box in $\lambda\setminus\nu$ above or to the left of it, 
\item `$\mathrm{hor}$' refers to boxes along a horizontal segment, meaning with a box to the left and no box above,
\item  `$\mathrm{vert}$' refers to boxes with a box above and no box to the left (along a vertical segment),
\item `$\mathrm{in}$' refers to boxes with both a box above and to the left of it inside $\lambda\setminus\nu$.
\end{itemize}
We also let 
\[\BB^{\NW}_{\neg \mathrm{in}}(\lambda/\nu):=\BB^{\NW}_{\out}(\lambda/\nu)\sqcup \BB^{\NW}_{\mathrm{hor}}(\lambda/\nu)\sqcup \BB^{\NW}_{\mathrm{vert}}(\lambda/\nu),
\text{ so that }
\]
%so that 
\[
\BB^{\NW}(\lambda/\nu)=\BB^{\NW}_{\neg \mathrm{in}}(\lambda/\nu)\sqcup\BB^{\NW}_{\mathrm{in}}(\lambda/\nu).
\]
Note that in the Schubert case, where $\nu=\emptyset$, there are no boxes of type `$\mathrm{in}$', and $\BB^{\NW}_{\neg\mathrm{in}}(\lambda)=\BB^{\NW}(\lambda)$. 

For any box $c$ of the inner rim  in $\BB^{\NW}_{\neg \mathrm{in}}(\lambda\setminus\nu)$, there exists a unique minimal frozen shape $\mu_i$ for which $c$ is an addable box, compare \cref{d:skewfrozens}. We denote this shape by $\mu(c)$. 

Suppose $c\in \BB^{\NW}_{\mathrm{in}}(\lambda/\nu)$. Then  there is a unique box $d_{NE}(c)$ that lies one row above $c$, and for which the rectangle with outer corners $c$ and $d$ is a maximal two-row rectangle in $\lambda\setminus\nu$. Furthermore, there is a unique box $d_{SW}(c)$ that lies one column to the left of $c$, and for which the rectangle with outer corners $c$ and $d$ is a maximal two-column rectangle in $\lambda\setminus\nu$.     

\begin{lemma}\label{l:canonicalskew} Using the notation introduced above we can rewrite the canonical superpotential associated to the skew shaped positroid 
	variety $X_{\lambda/\nu}$ as follows. 
\begin{equation*}
W^{\lambda/\nu} 
        = 
	\sum_{b=b_{\rho_{2\ell-1}} \in \BB^{\SE}_{\out}(\lambda/\nu)}  
        q_\ell \frac{p_{\sh(b)^-}}{p_{\sh(b)}}+
	\sum_{c\in \BB^{\NW}_{\neg\mathrm{in}}(\lambda/\nu)} \frac{p_{\mu(c)\sqcup c}}{p_{\mu(c)}} +
	\sum_{c\in \BB^{\NW}_{\mathrm{in}}(\lambda/\nu)} \left(\frac{p_{\sh(d_{NE}(c))\cup \sh(c)}}{p_{\sh(d_{NE}(c))}}+\frac{p_{\sh(d_{SW}(c))\cup \sh(c)}}{p_{\sh(d_{SW}(c))}}\right).
\end{equation*}

In particular the number of summands of the canonical superpotential is given by the formula
\[
 d+ |\BB^{\NW}(\lambda/\nu)| + |\BB^{\NW}_{\mathrm{in}}(\lambda/\nu)|,
\]
where $d=| \BB^{\SE}_{\out}(\lambda/\nu)|$ is the number of quantum parameters. 
\end{lemma}

\begin{proof}
The first summand of the above formula is identical to the one from \eqref{e:anotherSchub}. If $c$ is in $\BB^{\NW}_{\out}(\lambda/\nu)$ then $c$ is an addable box for $\nu$, so we have $\mu(c)=\nu$, and we recover the second summand of  \eqref{e:anotherSchub}. Note that every rectangle as in \cref{d:skewLG1} has a special box labeled $c$ that lies in the inner rim of $\lambda\setminus\mu$. If $c\in \BB^{\NW}_{\mathrm{hor}}(\lambda/\nu)$ then $c$ belongs to a unique maximal two-column rectangle (as upper right hand corner), but is not occurring as corner in any maximal two-row rectangle. If we label the lower left-hand corner by $d$ then $\mu(c)=\sh(d)$, and $\mu(c)\cup c=\sh(d)\cup\sh(c)$. Similarly if $c\in \BB^{\NW}_{\mathrm{vert}}(\lambda/\nu)$ then $c$ belongs only to a maximal two-row rectangle (as lower left-hand corner), and if we label the upper right-hand corner by $d$ then $\mu(c)=\sh(d)$, and $\mu(c)\cup c=\sh(d)\cup\sh(c)$. Thus the summands in \eqref{e:anotherSchub} associated to  rectangles whose boxes $c$ are of horizontal and vertical type precisely give us the remaining summands of the second sum in the new formula. 

We are left needing to consider the boxes $c$ from the inner rim which are of type `$\mathrm {in}$'. Each one of these occurs both in a maximal two-row rectangle 
and a maximal two-column rectangle. These two occurrences lead to the first and second terms in 
\[
\frac{p_{\sh(d_{NE}(c))\cup \sh(c)}}{p_{\sh(d_{NE}(c))}}+\frac{p_{\sh(d_{SW}(c))\cup \sh(c)}}{p_{\sh(d_{SW}(c))}},
\]
respectively. Therefore we obtain exactly the same terms as in \eqref{e:anotherSchub} and the two formulas coincide.

The number of summands now clearly equals to 
$ d+ |\BB^{\NW}_{\neg \mathrm{in}}(\lambda/\nu)| + 2|\BB^{\NW}_{\mathrm{in}}(\lambda/\nu)|$ and the formula follows by combining one of the $|\BB^{\NW}_{\mathrm{in}}(\lambda/\nu)|$ summands with $ |\BB^{\NW}_{\neg \mathrm{in}}(\lambda/\nu)| $.
 \end{proof}
 
 \begin{cor} We have a bijection between the terms of $W^{\lambda\setminus\nu}$ and the positroid divisors 
	 in $X_{\lambda\setminus\nu}$. 
 \end{cor}

\begin{proof}
The summand $q_\ell \frac{p_{\sh(b)^-}}{p_{\sh(b)}}$ corresponds to the `Schubert'-divisor $X_{\lambda'\setminus\nu}$ where $\lambda'=\lambda\setminus b$. 

The summand $\frac{p_{\mu(c)\sqcup c}}{p_{\mu(c)}}$  corresponds to the positroid divisor whose $\Le$-diagram has shape $\lambda$, with $0$'s precisely in the boxes $\nu\cup c$. This positroid divisor is a skew shaped positroid variety precisely if $c$ is in $\BB_{\out}(\lambda\setminus\nu)$.  

The summand
$\frac{p_{\sh(d_{NE}(c))\cup \sh(c)}}{p_{\sh(d_{NE}(c))}}$
corresponds to the positroid divisor whose $\Le$-diagram has shape $\lambda$; 
the boxes in $\lambda\setminus\nu$ contain a $+$, with the exception of $c$ and all boxes to the left of $c$ in the same row; and the boxes in $\nu$ contain a $0$, with the exception of the box $\NW$ of $c$ and all boxes to its left in the same row.   

The summand 
$\frac{p_{\sh(d_{SW}(c))\cup \sh(c)}}{p_{\sh(d_{SW}(c))}}$
corresponds to the positroid divisor whose $\Le$-diagram has
shape $\lambda$; the boxes in $\lambda\setminus\nu$ contain a $+$, with the exception of $c$ and all boxes above $c$ in the same column; and the boxes in $\nu$ contain a $0$, with the exception of the box $\NW$ of $c$ and all boxes above it in the same row.  

	One can show using e.g. \cite[Proposition 7.2]{RietschClosure} (see also 
\cite[Section 5 and Appendix A]{shelling}) 
that the above divisors are in fact all the positroid divisors
 in $X_{\lambda\setminus\nu}$. 
\end{proof}

\subsection{The superpotential in terms  of the rectangles cluster}\label{s:skewrectW}

In this section we give a Laurent polynomial expression for our 
skew shaped positroid variety  superpotential 
in terms of the rectangles cluster, 
generalizing the formula
from  \cref{s:rectW}.
Again, this formula can be expressed
in terms of a diagram, shown in \cref{fig:superskew}.
The general formula is given in
\cref{prop:skewsuper2}.

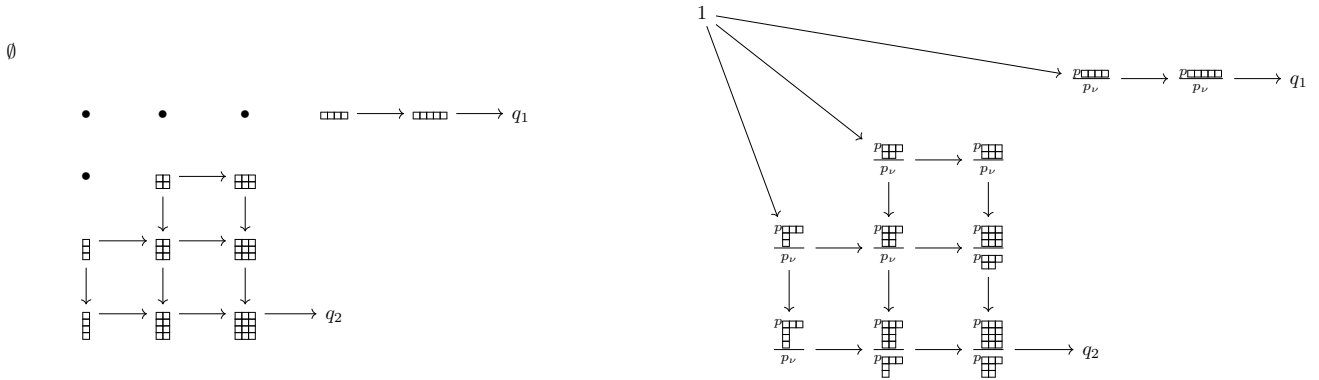
\begin{figure}[h]
	\adjustbox{scale=0.75,center}{%
	\begin{tikzcd}
	\emptyset &&&&&\\
	& \bullet &  \bullet &  \bullet & {\ydiagram{4}} & {\ydiagram{5}} & {q_1} \\
	&  \bullet & {\ydiagram{2,2}} & {\ydiagram{3,3}} &&&& \\
	& {\ydiagram{1,1,1}} & {\ydiagram{2,2,2}} & {\ydiagram{3,3,3}} &&&\\
	& {\ydiagram{1,1,1,1}} & {\ydiagram{2,2,2,2}} & {\ydiagram{3,3,3,3}} & {q_2} &&
	\arrow[from=2-5, to=2-6]
	\arrow[from=2-6, to=2-7]
	\arrow[from=3-3, to=3-4]
	\arrow[from=3-4, to=4-4]
	\arrow[from=3-3, to=4-3]
	\arrow[from=4-2, to=4-3]
	\arrow[from=4-3, to=4-4]
	\arrow[from=4-4, to=5-4]
	\arrow[from=4-3, to=5-3]
	\arrow[from=4-2, to=5-2]
	\arrow[from=5-2, to=5-3]
	\arrow[from=5-3, to=5-4]
	\arrow[from=5-4, to=5-5]
	\end{tikzcd}
%}
	\hspace{-.1cm}
\begin{tikzcd}
	&&1 &&&& & & \\
	 &&&&&& {\frac{p_{\ydiagram{4}}}{p_\nu}} & {\frac{p_{\ydiagram{5}}}{p_\nu}} & {q_1} \\
	 &&&& {\frac{p_{\ydiagram{3,2}}}{p_\nu}} & {\frac{p_{\ydiagram{3,3}}}{p_\nu}} \\
	&&& {\frac{p_{\ydiagram{3,1,1}}}{p_\nu}} & {\frac{p_{\ydiagram{3,2,2}}}{p_\nu}} & {\frac{p_{\ydiagram{3,3,3}}}{p_{\ydiagram{3,2}}}} \\
	&&& {\frac{p_{\ydiagram{3,1,1,1}}}{p_\nu}} & {\frac{p_{\ydiagram{3,2,2,2}}}{p_{\ydiagram{3,1,1}}}} & {\frac{p_{\ydiagram{3,3,3,3}}}{p_{\ydiagram{3,2,2}}}} & {q_2}
	\arrow[from=1-3, to=3-5]
	\arrow[from=1-3, to=2-7]
	\arrow[from=1-3, to=4-4]
	\arrow[from=2-8, to=2-9]
	\arrow[from=2-7, to=2-8]
	\arrow[from=3-5, to=3-6]
	\arrow[from=3-6, to=4-6]
	\arrow[from=3-5, to=4-5]
	\arrow[from=4-6, to=5-6]
	\arrow[from=4-5, to=5-5]
	\arrow[from=4-5, to=4-6]
	\arrow[from=5-5, to=5-6]
	\arrow[from=5-6, to=5-7]
	\arrow[from=5-4, to=5-5]
    \arrow[from=4-4, to=4-5]
	\arrow[from=4-4, to=5-4]
\end{tikzcd}
}
\caption{
	Let $\lambda=(5,3,3,3)$ and $\nu = (3,1)$. The diagram
	at the left shows the rectangles $i\times j$ associated
	to the boxes $(i,j)$ of $\lambda/\nu$.  
	The diagram at the right
	is the quiver $Q_{\lambda/\nu}$ from \cref{def:superquiver2}.}
\label{fig:superskew}
\end{figure}
As before, let $p_{i \times j}$ denote the Pl\"ucker coordinate indexed by the Young diagram
which is an $i \times j$ rectangle.  If $i=0$ or $j=0$ then we
set $p_{i \times j} = p_{\nu} = 1$.

\begin{definition}\label{def:superquiver2}
Let $\lambda/\nu$ be a skew shape.
We label the rows of $\lambda$ from top to bottom, and the columns from left to right.
We refer to the box in row $i$ and column $j$ as $(i,j)$.
Let $i_1<\dots< i_d$ denote the rows containing the outer (southeast) 
corners of $\lambda/\nu$.
We define a labeled quiver $Q_{\lambda/\nu}$, with one vertex 
$v(i,j)$ for every box $(i,j)$ of $\lambda/\nu$,
plus $d+1$ extra vertices
$\{v_0,v_1,\dots,v_d\}$.
The labels and arrows of the quiver are defined as follows.
\begin{itemize}
	\item If $b=(i,j)$ is a box of $\lambda/\nu$, 
                we label  $v(i,j)$  by
		 $\frac{p_{\sh(b)}}{p_{\sh(b)^-}}$.
        \item We label $v_0$ by $p_{\nu}$,
		and we label $v_1,\dots,v_d$ by $q_1,\dots,q_d$.
         \item If $(i,j)$ and $(i,j+1)$ are boxes of $\lambda$,
                 we add an arrow $v(i,j)\to v(i,j+1)$.
         \item If $(i,j)$ and $(i+1,j)$ are boxes of $\lambda$,
                 we add an arrow $v(i,j)\to v(i+1,j)$.
	 \item For every northwest corner $(i,j)$ of $\lambda/\nu$, 
		 we add an arrow 
                  $v_0 \to v(i,j)$.
	  \item For each outer (southeast) 
		  corner in row $i_\ell$, we add 
                  an arrow $v(i_{\ell}, \lambda_{i_{\ell}}) \to v_{\ell}$.
\end{itemize}
        Let $A(Q_{\lambda})$ denote the set of arrows of $Q_{\lambda}$, and for each
        arrow $a:v\to v'$ in $A(Q_{\lambda})$, let $p(a)$ denote the Laurent monomial
        in Pl\"ucker coordinates obtained
        by dividing the label of $v'$ by the label of $v$.

\end{definition}
See  \cref{fig:superskew} for an example of the quiver $Q_{\lambda/\nu}$ associated
to $\lambda = (5,3,3,3)$ and $\nu = (3,1)$.  

\begin{definition}[Rectangles seed]
Given a skew shape $\lambda/\nu$, let $\Rect(\lambda/\nu)$ be the set of all Young 
diagrams of the form $(i \times j) \cup \nu$,
where $(i,j)$ is a box of $\lambda/\nu$.
Let $\mathbb T^{\lambda/\nu}_{\rect}$ 
be the subset of $\check{X}_{\lambda/\nu}$ where
	$p_{\mu} \neq 0$ for all $\mu \in \Rect(\lambda/\nu)$.
\end{definition}

\begin{proposition}[Expansion of the superpotential in the rectangles cluster]\label{prop:skewsuper2}
     Let $\lambda/\nu$ be a skew shape.
       When we restrict $W^{\lambda/\nu}$ to
       $\mathbb T^{\lambda/\nu}_{\rect}$
       (which is a cluster torus for the $\mathcal{A}$-cluster
       structure for the open skew shaped positroid variety, see
       \cref{sec:Acluster})
	we obtain $$W^{\lambda/\nu}_{\rect} = \sum_{a\in A(Q_{\lambda/\nu})} p(a).$$
\end{proposition}

\begin{example}\label{ex:skewsuper}
 When $\lambda = (5,3,3,3)$ and $\nu = (3,1)$,
we obtain
\begin{align*}\label{super}
        W^{\lambda/\nu} =
	&q_1 \frac{p_{\ydiagram{3,1}}}{p_{\ydiagram{5,1}}} + 
	q_2 \frac{ p_{\ydiagram{3,2,2}}}{p_{\ydiagram{3,3,3,3}}} + 
	{	\frac{p_{\ydiagram{4,1}}}{p_{\ydiagram{3,1}}} +} 
	{ \frac{p_{\ydiagram{3,2}}}{p_{\ydiagram{3,1}}} + }
	\frac{p_{\ydiagram{3,1,1}}}{p_{\ydiagram{3,1}}} + 
	{\color{orange} \frac{p_{\ydiagram{3,3,3} } p_{\ydiagram{3,1}}}{p_{\ydiagram{3,2}} p_{\ydiagram{3,2,2}}}} + 
	{\color{orange} \frac{p_{\ydiagram{3,2,2}}}{p_{\ydiagram{3,2}}}} + 
	{\color{green} \frac{p_{\ydiagram{3,2,2}}}{p_{\ydiagram{3,1,1}}} + }
	{\color{green} \frac{p_{\ydiagram{3,2,2,2}} p_{\ydiagram{3,1}}}{p_{\ydiagram{3,1,1}} p_{\ydiagram{3,1,1,1}}}} +  \\
	& {\color{gray} \frac{p_{\ydiagram{5,1}}}{p_{\ydiagram{4,1}}}} + 
	{\color{purple} \frac{p_{\ydiagram{3,3}}}{p_{\ydiagram{3,2}}} + }
	{\color{purple} \frac{p_{\ydiagram{3,3,3,3}} p_{\ydiagram{3,1,1}}}{p_{\ydiagram{3,2,2}} p_{\ydiagram{3,2,2,2}}}} + 
	{\color{purple} \frac{p_{\ydiagram{3,3,3}}p_{\ydiagram{3,1}}}{p_{\ydiagram{3,2}}p_{\ydiagram{3,2,2}}}} + 
	{\color{red}	\frac{p_{\ydiagram{3,1,1,1}}}{p_{\ydiagram{3,1,1}}}} + 
	{\color{red} \frac{p_{\ydiagram{3,2,2,2}} p_{\ydiagram{3,1}}}{p_{\ydiagram{3,1,1}}p_{\ydiagram{3,2,2}}}} + 
	{\color{red}	\frac{p_{\ydiagram{3,3,3,3}}p_{\ydiagram{3,2}}}{p_{\ydiagram{3,2,2}} p_{\ydiagram{3,3,3}}}}.
\end{align*}
\end{example}

\begin{proof}
The proof is a slight generalization of the proof of 
\cref{prop:super2}.  As before, we sum the contributions
of all arrows in a given row, and all arrows in a given column,
and use  Pl\"ucker relations.  
For example, the contributions corresponding to a given row or column
are color coded in \cref{ex:skewsuper}.   One can  verify that when one
sums these contributions, one recovers the formula 
from \cref{ex:SkewExample}.
\end{proof}

\subsection{How results and proofs from previous sections generalize}

We believe that most of our results 
 generalize to the setting of skew shaped positroid varieties.
We now sketch how some  arguments can be extended.

First note that by the results of \cref{sec:appendix}, 
each skew shaped positroid variety
has both an $\mathcal{A}$ and $\mathcal{X}$ cluster algebra structure,
with many (but not all) clusters indexed by corresponding reduced plabic graphs.  There is a \emph{rectangles
cluster} $G_{\lambda/\mu}^{\rect}$
for $X_{\lambda/\mu}^{\circ}$, see \cref{rem:rectangles}.

\begin{theorem}\label{thm:mainskew}
Let $\Sigma_G^{\mathcal{X}}$ be an arbitrary $\mathcal{X}$-cluster
seed for the open skew shaped positroid variety $X_{\lambda/\mu}^{\circ}$.
Then the Newton-Okounkov body $\Delta_G^{\lambda/\mu}$ is a rational
polytope with lattice points $\{\val_G(P_{\nu}) \ \vert \ \mu \subseteq \nu \subseteq \lambda\}$, and it coincides with the superpotential polytope
$\Gamma_{G}^{\lambda/\mu}$.
\end{theorem}

\begin{proof}
To prove \cref{thm:mainskew}, we follow the arguments of \cref{sec:polytopescoincide}.  We start by proving that 
\cref{thm:mainskew} holds in the case of the rectangles cluster
$G = G_{\lambda/\mu}^{\rect}$.
In this case, we can show that 
-- up to a unimodular change of coordinates -- 
the superpotential polytope 
$\Gamma_G^{\lambda/\mu}$
agrees with the order polytope associated to the 
poset $P(\lambda/\mu)$, where $P(\lambda/\mu)$
is the natural generalization of the poset $P(\lambda)$
given by \cref{def:posetlambda}, with one element associated
to each box of the skew diagram $\lambda/\mu$.
Therefore the volume of 
$\Gamma_G^{\lambda/\mu}$ is 
$\frac{e(P(\lambda/\mu))}{|\lambda/\mu|!}$,
where $e(P(\lambda/\mu))$ is the number of linear extensions
of $P(\lambda/\mu)$ and $|\lambda/\mu|$ is the number of boxes
of $\lambda/\mu$.  Note that 
 $e(P(\lambda/\mu))$ is the number of standard Young tableaux
 of shape $\lambda/\mu$.

On the other hand, we can follow the proof of 
\cref{c:ConvEqualsGamma} and 
show that 
if 
$\conv_G^{\lambda/\mu}$ denotes
 the convex hull of the
valuations of the Pl\"ucker coordinates
$\val_G(P_\nu)$ for $\mu\subseteq \nu \subseteq \lambda$, then
the polytopes 
$\conv_G^{\lambda/\mu}$ 
and $\Q^{\lambda/\mu}_{G}$ agree.
We now recall the fact that 
the degree of the skew shaped positroid variety
$X_{\lambda/\mu}$ in its Pl\"ucker 
embedding equals the number of standard tableaux
	of shape $\lambda/\mu$ \cite[Section 4]{StanleySchubert}.
It now follows that 
$\conv_G^{\lambda/\mu}$  -- which so far
we only knew to be contained in the 
 Newton-Okounkov body $\Delta_G^{\lambda/\mu}$  -- 
must be equal to 
  $\Delta_G^{\lambda/\mu}$. 
This completes the proof of 
\cref{thm:mainskew} in the case of the rectangles cluster
$G = G_{\lambda/\mu}^{\rect}$.

To extend the proof to the case of general clusters $G$,
we follow the arguments of \cref{sec:arbitraryseeds}.
We need to verify that each frozen vertex has an optimized 
seed, and then we can apply the arguments of 
\cref{thm:theta} to show that 
there is a theta function basis for 
 $\mathcal{B}(X_{\lambda/\mu}^{\circ})$ for the coordinate ring
$\CC[\widehat{X}_{\lambda/\mu}^{\circ}]$.
This allows us to generalize the proofs of 
	\cref{cor:bijval} and \cref{thm:main}.
\end{proof}

\appendix

\section{Combinatorics of positroid cells 
and positroid varieties}\label{sec:appendix}

%\Comment{All  definitions now use reverse Grassmann necklaces.  
%In theory we could probably 
%delete mention of target labels and usual Grassmann necklaces. Maybe in v2.}

\subsection{The totally nonnegative Grassmannian
and positroid cells}\label{sec:positroid}

Let $Gr_{m,n} = Gr_{m,n}(\mathbb{F}):=Gr_m(\mathbb{F}^n)$
denote the Grassmannian of $m$-planes in $\mathbb{F}^n$, with 
Pl\"ucker coordinates
denoted by $p_{I}$ for $I \in {[n] \choose m}$.

\begin{definition}\cite{Lusztig3, Postnikov}
	We say that $V\in Gr_{m,n}$ is \emph{totally nonnegative}
	if each Pl\"ucker coordinates $p_I(V)\geq 0$ for 
	all $I\in {[n] \choose m}$.
	Similarly, $V$ is \emph{totally positive} if each Pl\"ucker
	coordinate is strictly positive for all $I$.
	We let 
	$Gr_{m,n}^{\geq 0}$
	and 
	$Gr_{m,n}^{>0}$ denote the set of totally nonnegative
	and totally positive elements of $Gr_{m,n}$, respectively.
	$Gr_{m,n}^{\geq 0}$ is called the \emph{totally nonnegative 
	Grassmannian.}
\end{definition}

If we partition 
$Gr_{m,n}^{\geq 0}$ 
into strata based on which Pl\"ucker 
coordinates are strictly positive and which are $0$, 
we get a cell decomposition of 
$Gr_{m,n}^{\geq 0}$  into \emph{positroid cells}
\cite{Postnikov}. Postnikov classified these using  various 
combinatorial objects, among them, 
\emph{equivalence classes of reduced plabic graphs}, and 
\emph{$\Le$-diagrams}, see \cite{Postnikov} and 
\cite{FWZ4} for more background.

\subsection{Plabic graphs}\label{sec:plabic}

In this section we give background on plabic graphs and their moves.

\begin{definition}
A {\it plabic (or planar bicolored) graph\/}
is an undirected graph $G$ drawn inside a disk
(considered modulo homotopy)
with $n$ {\it boundary vertices\/} on the boundary of the disk,
labeled $1,\dots,n$ in clockwise order, as well as some
colored {\it internal vertices\/}.
These internal vertices
are strictly inside the disk and are
colored in black and white. 
An internal vertex of degree one adjacent to a boundary vertex is a \emph{lollipop}.
We will always assume that no vertices of the same color are adjacent, and that 
each boundary vertex $i$ is adjacent to a single internal vertex.
\end{definition}

See \cref{ex:plabic}
for an example of a plabic graph.
\begin{figure}[h]
\centering
\includegraphics[height=1.5in]{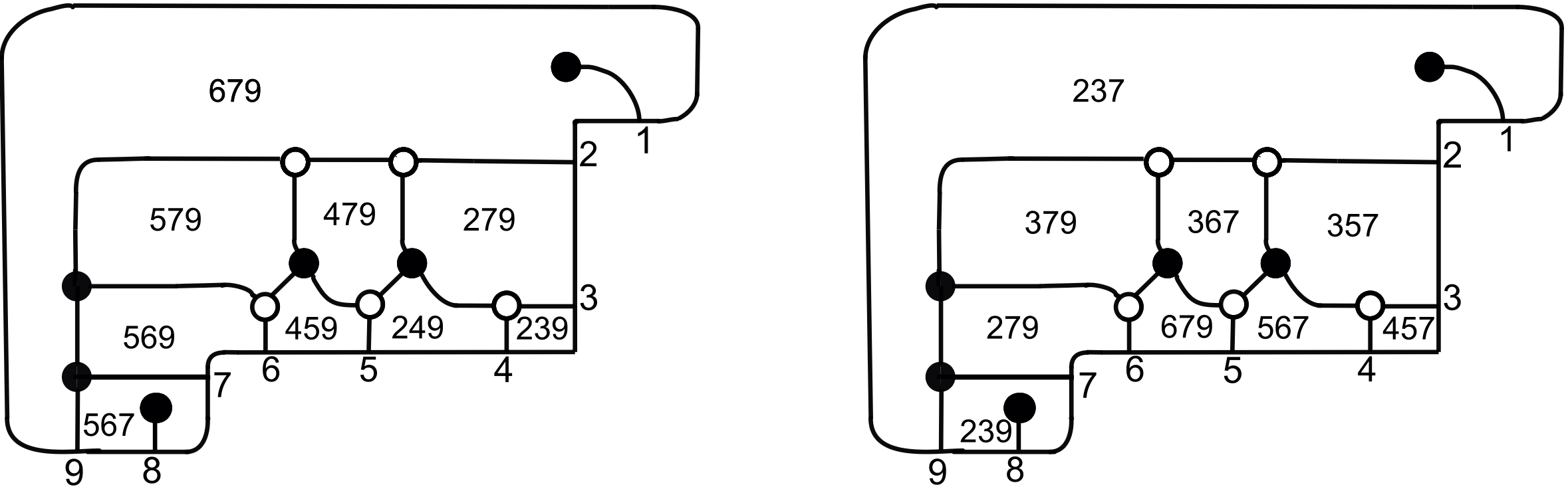}
\caption{A plabic graph $G$ with its source and target face labels.
The trip permutation is 
$\pi_G=(\underline{1}, 5, 4, 6, 9, 2, 3, \underline{8}, 7)$, 
the reverse Grassmann necklace is 
$\protect\overleftarrow{\mathcal{I}}(G)
= (679, 279, 239, 249, 459, 569, 567, 567, 679),$
and 
	the Grassmann necklace is 
$\mathcal{I}(G) = (237, 237, 357, 457, 567, 679, 279, 239, 239)$.
	\label{ex:plabic}}
\end{figure}

There is a natural set of local transformations (moves) of plabic graphs, which we now describe.
We will also assume that $G$ is \emph{leafless}, 
i.e.\ if $G$ has an 
internal vertex of degree $1$, then that vertex must be adjacent to a boundary
vertex.

(M1) SQUARE MOVE (Urban renewal).  If a plabic graph has a square formed by
four trivalent vertices whose colors alternate,
then we can switch the
colors of these four vertices.

(M2) CONTRACTING/EXPANDING A VERTEX.
Two adjacent internal vertices of the same color can be merged
or unmerged.

(M3) MIDDLE VERTEX INSERTION/REMOVAL.
We can always remove/add degree $2$ vertices.

See \cref{M1} for depictions of these three moves.

\begin{figure}[h]
\centering
\includegraphics[height=.5in]{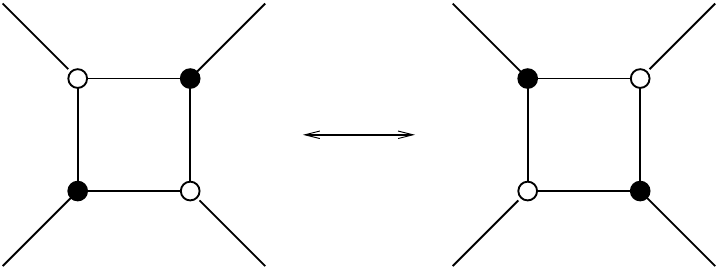}
\hspace{.5in}
\raisebox{6pt}{\includegraphics[height=.4in]{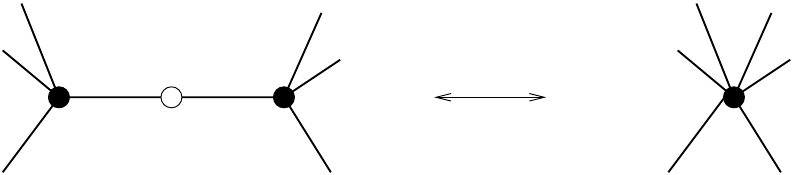}}
\hspace{.5in}
\raisebox{16pt}{\includegraphics[height=.07in]{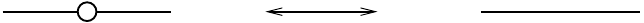}}
\caption{%
	Local moves (M1), (M2), and (M3) on plabic graphs.}
\label{M1}
\end{figure}

\begin{definition}
Two plabic graphs are called \emph{move-equivalent} if they can be obtained
from each other by moves (M1)-(M3).  The \emph{move-equivalence class}
of a given plabic graph $G$ is the set of all plabic graphs which are move-equivalent
to $G$.
A leafless plabic graph without isolated components
is called \emph{reduced} if there is no graph in its move-equivalence
	class in which there is a \emph{bubble}, 
	that is, two adjacent vertices $u$ and $v$ which 
	are connected by more than one edge.
\end{definition}

\begin{definition} A \emph{decorated permutation}  on $[n]$
	is a bijection $\pi: [n]\to [n]$ whose fixed points are each colored either
	black (loop) or white (coloop).  We denote 
	a black fixed point $i$ by $\pi(i)=\underline{i}$ and a white
	fixed point $i$ by $\pi(i)=\overline{i}$.
	An \emph{anti-excedance} of the decorated permutation $\pi$
	is an element $i\in [n]$ such that either
	$\pi^{-1}(i)>i$ or $\pi(i)=\overline{i}$. 
\end{definition}

\begin{definition}\label[definition]{def:rules}
Given a reduced plabic graph $G$,
a  \emph{trip} $T$ is a directed path which starts at
some boundary vertex
$i$, and follows the ``rules of the road": it turns (maximally) right at a
black vertex,  and (maximally) left at a white vertex.
 Note that $T$ will also
end at a boundary vertex $j$; we then refer to this trip as
$T_{i \to j}$. Setting $\pi(i)=j$ for each such trip,
 we associate a (decorated) \emph{trip permutation}
$\pi_G=(\pi(1),\dots,\pi(n))$ to each reduced plabic graph $G$, where a fixed point $\pi(i)=i$ is colored white (black) if there is a white (black) lollipop at boundary vertex $i$.
We say that $G$ has \emph{type $\pi_G$}.
\end{definition}

The plabic graph $G$ in 
\cref{ex:plabic} has trip permutation 
$\pi_G=(\underline{1}, 5, 4, 6, 9, 2, 3, \underline{8}, 7)$.

\begin{remark}\label[remark]{rem:moves}
Note that the trip permutation of a plabic graph is preserved 
by the local moves (M1)-(M3). For reduced plabic graphs the converse holds, namely 
it follows from \cite[Theorem 13.4]{Postnikov}, see also 
	\cite[Theore 7.4.25]{FWZ4},
that any two reduced plabic graphs with the same trip permutation are 
move-equivalent. 
\end{remark}

Now we use the notion of trips to label each face of $G$
by a Pl\"ucker coordinate.
Towards this end, note that every trip
will partition the faces of a plabic graph into
two parts: those on the left of the trip, and those on the right
of a trip.

\begin{definition}\label[definition]{def:faces}
Let $G$ be a reduced plabic graph with $n$ boundary vertices.
For each one-way trip $T_{i\to j}$ with $i \neq j$, we place the label $i$
(respectively, $j$)
 in every face which is to the left of $T_{i\to j}$. If $i=j$ (that is, $i$ is adjacent to a lollipop), we place the label $i$ 
in all faces if the lollipop is white and in no faces if the lollipop is black.
We then obtain a labeling $\mathcal{F}_{\source}(G)$
(respectively, $\mathcal{F}_{\target}(G)$)
of faces of $G$ by subsets of $[b]$ which
we call the \emph{source} (respectively, \emph{target})
\emph{labeling} of $G$.  We identify each $m$-element subset of $[n]$
with the corresponding Pl\"ucker coordinate.

We will often identify the {\bf source labels} of $G$
with the vertical steps of corresponding Young diagrams fitting in an $m \times (n-m)$
rectangle (as in \cref{s:Young}), 
	see \cref{fig:plabic}.
\end{definition}

\subsection{Le-diagrams, Grassmann necklaces, positroid cells, 
and open positroid
varieties}

In this section we define $\Le$-diagrams and 
Grassmann necklaces, as well as the associated positroid cells
and varieties. 
We start by defining $\Le$-diagrams.  Each  $\Le$-diagram gives rise
to an associated reduced plabic graph, which is 
a distinguished representative of its move-equivalence class.

\begin{definition}[{\cite{Postnikov}}]
Fix a partition  $\lambda$ together with
an $m \times (n-m)$ rectangle which contains $\lambda$.
A {\it $\Le$-diagram}
(or Le-diagram) $D$ of shape $\lambda$ 
is a filling by $0$'s and $+$'s of the boxes of the
Young diagram of ${\lambda}$
in such a way that the
{\it $\Le$-property} is satisfied:
there is no $0$ which has a $+$ above it in the same column and a $+$ to its
left in the same row.
See Figure \ref{fig:Le} for
 examples of  $\Le$-diagrams.
\end{definition}
\begin{figure}[ht]
\setlength{\unitlength}{1.3mm}
\begin{center}
\resizebox{0.15\textwidth}{!}{\begin{picture}(42,35)
% horizontal lines
\put(5,32){\line(1,0){54}}
  \put(5,23){\line(1,0){45}}
  \put(5,14){\line(1,0){45}}
  \put(5,5){\line(1,0){54}}
% vertical lines
  \put(5,5){\line(0,1){27}}
  \put(14,5){\line(0,1){27}}
  \put(23,5){\line(0,1){27}}
  \put(59,5){\line(0,1){27}}
  \put(23,14){\line(0,1){18}}
  \put(32,14){\line(0,1){18}}
  \put(41,14){\line(0,1){18}}
  \put(41,23){\line(0,1){9}}
  \put(50,14){\line(0,1){18}}
% first row
   \put(8,26){\Huge{$+$}}
   \put(17,26){\Huge{$0$}}
         \put(26,26){\Huge{$+$}}
         \put(35,26){\Huge{$+$}}
         \put(44,26){\Huge{$0$}}
% second row
         \put(8,17){\Huge{$+$}}
         \put(17,17){\Huge{$0$}}
         \put(26,17){\Huge{$+$}}
         \put(35,17){\Huge{$+$}}
         \put(44,17){\Huge{$+$}}
% third row
         \put(8,8){\Huge{$+$}}
         \put(17,8){\Huge{$0$}}
% arrows
        \end{picture}}
%%%%%%%%%%%%%%%55
\hspace{2cm} \resizebox{0.15\textwidth}{!}{\begin{picture}(42,35)
% horizontal lines
\put(5,32){\line(1,0){36}}
  \put(5,23){\line(1,0){36}}
  \put(5,14){\line(1,0){36}}
  \put(5,5){\line(1,0){36}}
% vertical lines
  \put(5,5){\line(0,1){27}}
  \put(14,5){\line(0,1){27}}
  \put(23,5){\line(0,1){27}}
  \put(41,5){\line(0,1){27}}
  \put(23,14){\line(0,1){18}}
  \put(32,14){\line(0,1){18}}
  \put(41,14){\line(0,1){18}}
  \put(41,23){\line(0,1){9}}
%  \put(50,14){\line(0,1){18}}
% first row
   \put(8,26){\Huge{$+$}}
   \put(17,26){\Huge{$+$}}
         \put(26,26){\Huge{$+$}}
         \put(35,26){\Huge{$+$}}
% second row
         \put(8,17){\Huge{$+$}}
         \put(17,17){\Huge{$+$}}
         \put(26,17){\Huge{$+$}}
         \put(35,17){\Huge{$+$}}
% third row
         \put(8,8){\Huge{$+$}}
         \put(17,8){\Huge{$+$}}
% arrows
        \end{picture}}
%%%%%%%%%%%%%%%55
\hspace{1cm} \resizebox{0.15\textwidth}{!}{\begin{picture}(42,35)
% horizontal lines
\put(5,32){\line(1,0){36}}
  \put(5,23){\line(1,0){36}}
  \put(5,14){\line(1,0){36}}
  \put(5,5){\line(1,0){36}}
% vertical lines
  \put(5,5){\line(0,1){27}}
  \put(14,5){\line(0,1){27}}
  \put(23,5){\line(0,1){27}}
  \put(41,5){\line(0,1){27}}
  \put(23,14){\line(0,1){18}}
  \put(32,14){\line(0,1){18}}
  \put(41,14){\line(0,1){18}}
  \put(41,23){\line(0,1){9}}
%  \put(50,14){\line(0,1){18}}
% first row
   \put(8,26){\Huge{$0$}}
   \put(17,26){\Huge{$0$}}
         \put(26,26){\Huge{$+$}}
         \put(35,26){\Huge{$+$}}
% second row
         \put(8,17){\Huge{$0$}}
         \put(17,17){\Huge{$+$}}
         \put(26,17){\Huge{$+$}}
         \put(35,17){\Huge{$+$}}
% third row
         \put(8,8){\Huge{$+$}}
         \put(17,8){\Huge{$+$}}
% arrows
        \end{picture}}
\end{center}
\caption{A Le-diagram of shape
	$\lambda=(5,5,2)$ 
	contained in a $3 \times 6$ rectangle, 
	and two Le-diagrams of shape $\lambda=(4,4,2)$ contained in a 
	$3 \times 4$ rectangle.}
 \label{fig:Le}
\end{figure}
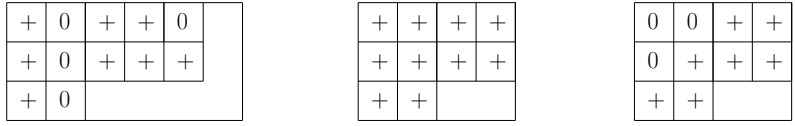

\begin{definition}\label{def:Network}
Let $D$ be a $\Le$-diagram.  Delete the $0$'s and replace
each $+$ with a vertex.  From each vertex we construct a hook
which
goes east and south, to the border of the Young diagram.  
We also place boundary vertices labeled by $1,2,\dots,n$
along the 
edges on the southeast border of the Young diagram.
The resulting
diagram is called the ``hook diagram'' $H(D)$.  
We obtain a network $N(D)$ from $H(D)$ by orienting
	all edges west and south, see \cref{hook}.

We can also get a plabic graph $G(D)$ from the 
hook diagram $H(D)$, by making the local substitutions 
shown at the right of  \cref{hook}.  The plabic graph $G(D)$
associated to $H(D)$ is shown in 
\cref{ex:plabic}.
\end{definition}

\begin{figure}[h]
\centering
\setlength{\unitlength}{1.3mm}
	\hspace{-2cm}
\resizebox{0.15\textwidth}{!}{\begin{picture}(42,35)
% horizontal lines
\put(5,32){\line(1,0){54}}
  \put(9,27){\line(1,0){41}}
  \put(9,18){\line(1,0){41}}
  \put(23,14){\line(1,0){27}}
  \put(9,9){\line(1,0){14}}
  \put(5,5){\line(1,0){18}}
% vertical lines
  \put(5,5){\line(0,1){27}}
  \put(9,5){\line(0,1){23}}
  \put(23,5){\line(0,1){9}}
  %\put(23,5){\line(0,1){27}}
  %\put(23,14){\line(0,1){18}}
  \put(27,14){\line(0,1){13}}
  \put(36,14){\line(0,1){13}}
  %\put(41,14){\line(0,1){18}}
  %\put(41,23){\line(0,1){9}}
  \put(45,14){\line(0,1){4}}
  \put(50,14){\line(0,1){18}}
% border
	\put(55, 32){\circle*{1}}
	\put(55,30){\makebox(0,0){$\mathbf{1}$}}
	\put(50,27){\circle*{1}}
	\put(52,27){\makebox(0,0){$\mathbf{2}$}}
	\put(50,18){\circle*{1}}
	\put(52,18){\makebox(0,0){$\mathbf{3}$}}
	\put(45,14){\circle*{1}}
	\put(45,12){\makebox(0,0){$\mathbf{4}$}}
	\put(36,14){\circle*{1}}
	\put(36,12){\makebox(0,0){$\mathbf{5}$}}
	\put(27,14){\circle*{1}}
	\put(27,12){\makebox(0,0){$\mathbf{6}$}}
	\put(23,9){\circle*{1}}
	\put(25,9){\makebox(0,0){$\mathbf{7}$}}
	\put(18,5){\circle*{1}}
	\put(18,3){\makebox(0,0){$\mathbf{8}$}}
	\put(9,5){\circle*{1}}
	\put(9,3){\makebox(0,0){$\mathbf{9}$}}
% first row
	\put(9,27){\circle*{1}}
	 \put(27,27){\circle*{1}}
	 \put(36,27){\circle*{1}}
         %\put(44,26){\Huge{$0$}}
% second row
	 \put(9,18){\circle*{1}}
         %\put(17,17){\Huge{$0$}}
	 \put(27,18){\circle*{1}}
	 \put(36,18){\circle*{1}}
	 \put(45,18){\circle*{1}}
% third row
	 \put(9,9){\circle*{1}}
         %\put(17,8){\Huge{$0$}}
% arrows
\end{picture}}\hspace{1.5cm}
\resizebox{0.15\textwidth}{!}{\begin{picture}(42,35)
% horizontal lines
\put(5,32){\line(1,0){54}}
  \put(9,27){\line(1,0){41}}
  \put(9,18){\line(1,0){41}}
  \put(23,14){\line(1,0){27}}
  \put(9,9){\line(1,0){14}}
  \put(5,5){\line(1,0){18}}
% vertical lines
  \put(5,5){\line(0,1){27}}
  \put(9,5){\line(0,1){23}}
  \put(23,5){\line(0,1){9}}
  %\put(23,5){\line(0,1){27}}
  %\put(23,14){\line(0,1){18}}
  \put(27,14){\line(0,1){13}}
  \put(36,14){\line(0,1){13}}
  %\put(41,14){\line(0,1){18}}
  %\put(41,23){\line(0,1){9}}
  \put(45,14){\line(0,1){4}}
  \put(50,14){\line(0,1){18}}
% border
	\put(55, 32){\circle*{1}}
	\put(55,30){\makebox(0,0){$\mathbf{1}$}}
	\put(50,27){\circle*{1}}
	\put(52,27){\makebox(0,0){$\mathbf{2}$}}
	\put(43,27){\makebox(0,0){$<$}}
	\put(31,27){\makebox(0,0){$<$}}
	\put(18,27){\makebox(0,0){$<$}}
	\put(40,18){\makebox(0,0){$<$}}
	\put(31,18){\makebox(0,0){$<$}}
	\put(18,18){\makebox(0,0){$<$}}
	\put(18,9){\makebox(0,0){$<$}}
	\put(36,22){\makebox(0,0){$\vee$}}
	\put(27,22){\makebox(0,0){$\vee$}}
	\put(9,22){\makebox(0,0){$\vee$}}
	\put(45,16){\makebox(0,0){$\vee$}}
	\put(36,16){\makebox(0,0){$\vee$}}
	\put(27,16){\makebox(0,0){$\vee$}}
	\put(9,13){\makebox(0,0){$\vee$}}
	\put(9,7){\makebox(0,0){$\vee$}}
	\put(50,18){\circle*{1}}
	\put(52,18){\makebox(0,0){$\mathbf{3}$}}
	\put(45,14){\circle*{1}}
	\put(45,12){\makebox(0,0){$\mathbf{4}$}}
	\put(36,14){\circle*{1}}
	\put(36,12){\makebox(0,0){$\mathbf{5}$}}
	\put(27,14){\circle*{1}}
	\put(27,12){\makebox(0,0){$\mathbf{6}$}}
	\put(23,9){\circle*{1}}
	\put(25,9){\makebox(0,0){$\mathbf{7}$}}
	\put(18,5){\circle*{1}}
	\put(18,3){\makebox(0,0){$\mathbf{8}$}}
	\put(9,5){\circle*{1}}
	\put(9,3){\makebox(0,0){$\mathbf{9}$}}
% first row
	\put(9,27){\circle*{1}}
	 \put(27,27){\circle*{1}}
	 \put(36,27){\circle*{1}}
% second row
	 \put(9,18){\circle*{1}}
	 \put(27,18){\circle*{1}}
	 \put(36,18){\circle*{1}}
	 \put(45,18){\circle*{1}}
% third row
	 \put(9,9){\circle*{1}}
% arrows
\end{picture}}\hspace{1.5cm}
\includegraphics[height=.8in]{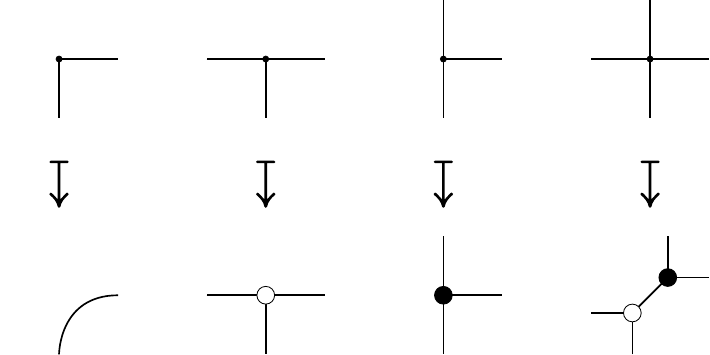}
	\caption{The hook diagram $H(D)$ and  network $N(D)$
	associated to the Le-diagram $D$ at the left of  \cref{fig:Le},
	plus the local substitutions for getting the plabic 
	graph $G(D)$ from  $H(D)$.}
\label{hook}
\end{figure}
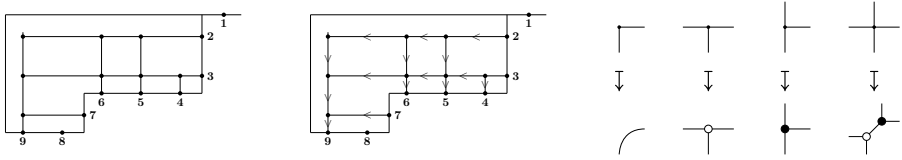

\begin{definition}\label{def:necklace}
Let $m\leq n$ be positive integers.
A \emph{Grassmann necklace} of type $(m,n)$
is a sequence
	$\mathcal{I}= (I_1,I_2,\dots,I_n)$ 
	of subsets $I_{\ell}\in {[n]\choose m}$, with subscripts
	considered modulo $n$, 
such that for any $i\in [n]$,
\begin{itemize}
	\item if $i\in I_i$ then $I_{i+1}=(I_i \setminus \{i\}) 
		\cup \{j\}$ for some $j\in [n]$,
	\item if $i\notin I_i$ then $I_{i+1}=I_i$.
\end{itemize}
And a \emph{reverse Grassmann necklace} of type $(m,n)$
is a sequence
	$\overleftarrow{\mathcal{I}}= (\overleftarrow{I_1},\overleftarrow{I_2},\dots,\overleftarrow{I_n})$ 
	of subsets $\overleftarrow{I_{\ell}}\in {[n]\choose m}$, 
such that for any $i\in [n]$,
\begin{itemize}
	\item if $i\in \overleftarrow{I_i}$ then $\overleftarrow{I_{i-1}}=
		(\overleftarrow{I_i} \setminus \{i\}) 
		\cup \{j\}$ for some $j\in [n]$,
	\item if $i\notin \overleftarrow{I_i}$ then $\overleftarrow{I_{i-1}}=\overleftarrow{I_i}$.
\end{itemize}
\end{definition}

\begin{definition}\label{neckfromG}
We can read off a Grassmann necklace $\mathcal{I}(G)$ 
(respectively, reverse Grassmann necklace)
from each reduced plabic graph $G$
by using the target (resp., source)
face labels and letting $I_i$ be the label of the 
boundary face of $G$ which is incident to the boundary vertices $i-1$
and $i$ (resp., $i$ and $i+1$).
\end{definition}
See \cref{ex:plabic} for an example.

Our next goal is to explain how to read off a positroid cell and variety
from a reduced plabic graph.
\begin{definition}
The $i$-order $<_i$ on the set $[n]$ is the total order
$$i<_i i+1 <_i \dots <_i n <_i 1 <_i \dots <_i i-2 <_i i-1.$$
The \emph{$<_i$-Gale order} on ${[n] \choose m}$
	is the partial order $\leq_i$
defined as follows: 
for any two subsets $S=\{s_1 <_i \dots <_i s_m\}$
and $T = \{t_1 <_i \dots <_i t_m\}$ of $[n]$, we have
$S \leq_i T$ if and only if $s_j \leq_i t_j$
for all $j\in [m]$.
\end{definition}

Given any full rank $m \times n$ matrix $A$, for each $1\leq \ell \leq n$,
 let $\overleftarrow{I_{\ell}}$ be the lexicographically maximal subset with respect to 
$<_{\ell+1}$ such that 
the Pl\"ucker coordinate $p_{I_{\ell}}(A)$
is nonzero.  
The associated sequence
$\overleftarrow{\mathcal{I}}(A):=(\overleftarrow{I_1},\dots,\overleftarrow{I_n})$ is always a reverse Grassmann
necklace  \cite[Lemma 16.3]{Postnikov}.

The following result is a dual version of a result of \cite{Postnikov,oh}\footnote{The original
statement used Grassmann necklaces instead of reverse Grassmann necklaces.}
\begin{theorem}[Positroid cell from Grassmann necklace]\cite{Postnikov, oh}
	Let $\overleftarrow{\mathcal{I}}= (\overleftarrow{I_1},\overleftarrow{I_2},\dots,\overleftarrow{I_n})$ 
be a reverse Grassmann necklace of type $(m,n)$.
Then the collection
	$$\mathcal{B}(\overleftarrow{\mathcal{I}}):=
	\Bigl\{ B\in {[n] \choose m} \vert B \leq_{j+1} I_j
\text{ for all }j\in [n]\Bigr\}$$
is the collection of nonvanishing Pl\"ucker coordinates
	of a positroid cell $S_{\overleftarrow{\mathcal{I}}}$ of $Gr_{m,n}$.
Conversely, every positroid cell arises this way, so we have
	$$Gr_{m,n}^{\geq 0} = \bigsqcup_{\overleftarrow{\mathcal{I}}} S_{\overleftarrow{\mathcal{I}}},$$
	where the union is over all reverse Grassmann necklaces of type
	$(m,n)$.
\end{theorem}

We can use the reverse Grassmann necklace to also define
an associated \emph{open positroid variety}.
Let $\Mat^{\circ}(m,n)$ denote the set of full rank $m \times n$ matrices.
\begin{definition}[Positroid variety from reverse Grassmann necklace]\label{def:posvariety}
	Given a reverse Grassmann necklace $\overleftarrow{\mathcal{I}}=(\overleftarrow{I_1},\dots,\overleftarrow{I_n})$
of type $(m,n)$,
let  
	$$\Mat^{\circ}(\overleftarrow{\mathcal{I}}):=\{A \in \Mat^{\circ}(m,n) \ \vert \ 
	\overleftarrow{\mathcal{I}}(A) = \overleftarrow{\mathcal{I}}.\}$$
The \emph{open positroid variety}  
	$X^{\circ}_{\overleftarrow{\mathcal{I}}}$
is 
	$\Gl_m \setminus \Mat^{\circ}(\overleftarrow{\mathcal{I}})$, i.e. the subvariety
of $Gr_{m,n}$ whose elements can be represented
	as row spans of elements of $\Mat^{\circ}(\overleftarrow{\mathcal{I}})$ \cite{KLS}.
	(We also define the closed positroid variety 
	$X_{\overleftarrow{\mathcal{I}}}$ to be the closure of 
	$X^{\circ}_{\overleftarrow{\mathcal{I}}}$.)
	We have $$Gr_{m,n} = \bigsqcup_{\overleftarrow{\mathcal{I}}} X^{\circ}_{\overleftarrow{\mathcal{I}}}.$$
\end{definition}
\begin{remark}
It is more common to define open positroid varieties  by using 
Grassmann necklaces instead of reverse Grassmann necklaces, but the definitions 
	are equivalent
	\cite[Proposition 2.8]{MullerSpeyer}.
\end{remark}

Since we can read off a (reverse) Grassmann necklace from 
a plabic graph (\cref{neckfromG}), this gives a natural way to 
associate a positroid cell and positroid variety to a plabic graph $G$
or to a $\Le$-diagram $D$.
We will sometimes refer to this cell and variety as 
$S_G$ or $S_D$ and 
$X_G^{\circ}$ or 
$X_D^{\circ}$.

\begin{definition}\label{def:openSchubert}
Let $\lambda$ be a partition.
If $D$ is a $\Le$-diagram of shape $\lambda$ whose boxes contain only $+$'s,
as in the middle diagram in 
 \cref{fig:Le}, we refer to the corresponding positroid variety
as an \emph{open Schubert variety} $X_{\lambda}^{\circ}$.
Let $\lambda/\mu$ be a skew shape.  
If $D$ is a $\Le$-diagram of shape $\lambda$ such that the boxes
in $\lambda/\mu$ are filled with $+$'s and the other boxes are filled 
with $0$'s, as in the right diagram in \cref{fig:Le},
we refer to the corresponding positroid variety
as an \emph{open skew shaped positroid variety} $X_{\lambda/\mu}^{\circ}$.
\end{definition}
\begin{remark}\label{rem:defopenSchubert}
It is not hard to verify that the definition of open Schubert variety
given in \cref{def:openSchubert} agrees with the one 
from \cref{d:openX}.  Indeed, if one computes the 
plabic graph $G(D)$ associated to the $\Le$-diagram $D$ of 
shape $\lambda$ whose boxes contain only $+$'s, and uses the 
source face labels, then these face labels correspond to the 
rectangles contained in $\lambda$, and the components 
	of the reverse Grassmann necklace are exactly the frozen rectangles
	$\mu_1,\dots,\mu_n$ used in \cref{d:openX}.
\end{remark}

\subsection{Quivers from plabic graphs}

We next describe quivers and quiver mutation, 
and how they relate to moves on plabic graphs.
Quiver mutation was first defined by Fomin and Zelevinsky \cite{ca1}
in order to define cluster algebras.  

\begin{definition}[Quiver]\label{quiver}
A \emph{quiver} $Q$ is a directed graph; we will assume that $Q$ has no 
loops or $2$-cycles.
If there are $i$ arrows from vertex $\sigma$ to $\tau$, then 
we will set $b_{\sigma \tau} = i$ and $b_{\tau \sigma} = -i$.
Each vertex is designated either  \emph{mutable} or \emph{frozen}.
The skew-symmetric matrix $B = (b_{\sigma \tau})$ is called the \emph{exchange matrix} of $Q$.
\end{definition}

\begin{definition}[Quiver Mutation]
Let $\sigma$ be a mutable vertex of quiver $Q$.  The quiver mutation 
$\Mut_\sigma$ transforms $Q$ into a new quiver $Q' = \Mut_\sigma(Q)$ via a sequence of three steps:
\begin{enumerate}
\item For each oriented two path $\mu \to \sigma \to \nu$, add a new arrow $\mu \to \nu$
(unless $\mu$ and $\nu$ are both frozen, in which case do nothing).
\item Reverse the direction of all arrows incident to the vertex $\sigma$.
\item Repeatedly remove oriented $2$-cycles until unable to do so.
\end{enumerate}
	If $B$ is the exchange matrix of $Q$, then we let $\Mut_{\sigma}(B)$ denote the 
	exchange matrix of $\Mut_{\sigma}(Q)$.

\end{definition}

We say that two quivers $Q$ and $Q'$ are \emph{mutation equivalent} if $Q$
can be transformed into a quiver isomorphic to $Q'$ by a sequence of mutations.

\begin{definition}
Let $G$ be a reduced plabic graph.  We associate a quiver $Q(G)$ as follows.  The vertices of 
$Q(G)$ are labeled by the faces of $G$.  We say that a vertex of $Q(G)$ is \emph{frozen}
if the 
corresponding face is incident to the boundary of the disk, and is \emph{mutable} otherwise.
For each edge $e$ in $G$ which separates two faces, at least one of which is mutable, 
we introduce an arrow connecting the faces;
 this arrow is oriented so that it ``sees the white endpoint of $e$ to the left and the 
black endpoint to the right'' as it crosses over $e$.  We then remove oriented $2$-cycles
from the resulting quiver to get $Q(G)$. 
\end{definition}

\begin{remark}
Let $D$ be as in  \cref{rem:defopenSchubert}.	Then the quiver $Q(G(D))$ recovers the seed in 
\cref{fig:combconstruct}.

\end{remark}

The following lemma is straightforward, and is implicit in \cite{Scott}.

\begin{lemma}\label{lem:mutG}
If $G$ and $G'$ are related via a square move at a face,
then $Q(G)$ and $Q(G')$ are related via mutation at the corresponding vertex.
\end{lemma}

\subsection{Network charts from plabic graphs}
\label{sec:poschart}
In this section  we will discuss perfect orientations of plabic 
graphs as well as network charts \cite{Postnikov, Talaska}, which allow us to give parameterizations
of positroid cells.

\begin{definition}\label{rem:normalization}
A {\it perfect orientation\/} $\O$ of a plabic graph $G$ is a
choice of orientation of each  edge such that each
black internal vertex $u$ is incident to exactly one edge
directed away from $u$; and each white internal vertex $v$ is incident
to exactly one edge directed towards $v$.
A plabic graph is called {\it perfectly orientable\/} if it admits a perfect orientation.
The {\it source set\/} $I_\O \subset [n]$ of a perfect orientation $\O$ is the set of all $i$ which
are sources of $\O$ (considered as a directed graph). Similarly, if $j \in \overline{I}_{\O} := [n] - I_{\O}$, then $j$ is a sink of $\O$.
\end{definition}
See \cref{perforientation} for an example.

\begin{figure}[h]
\centering
\includegraphics[height=1.3in]{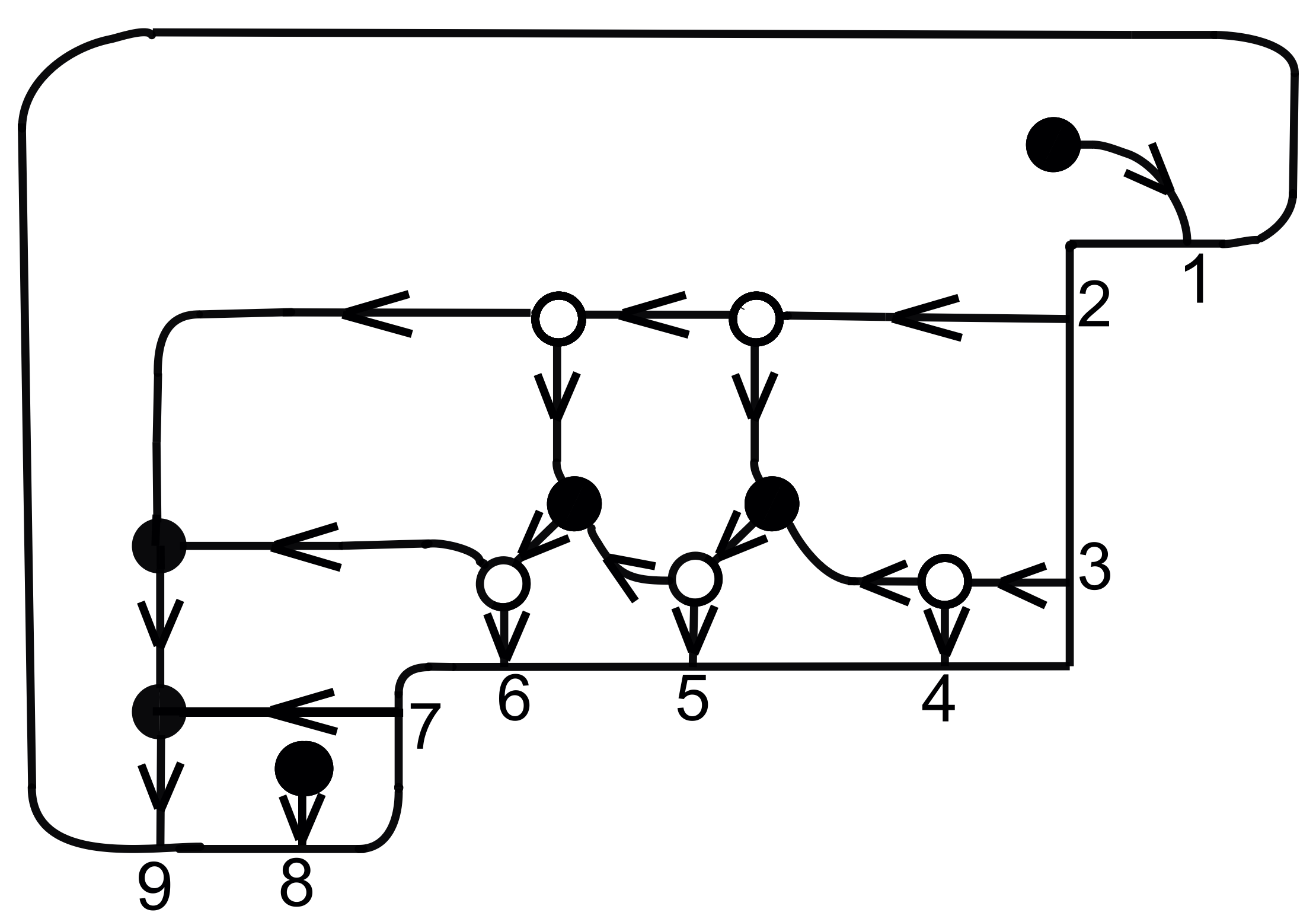}
\caption{An acyclic  perfect orientation $\O$ of 
	the plabic graph from \cref{ex:plabic}.}
\label{perforientation}
\end{figure}

The following lemma appeared in \cite{PSWv1}.\footnote{The
published version of \cite{PSWv1}, namely \cite{PSW}, did not include the lemma, because it turned out
to be unnecessary.}

\begin{lem}[{\cite[Lemma 3.2 and its proof]{PSWv1}}]\label{acyclic} 
Let $G$ be a reduced plabic graph with corresponding 
Grassmann necklace $\mathcal{I}(G)=(I_1,\dots,I_n)$
as in \cref{neckfromG}.  
For each $1\leq j \leq n$, there is an 
 acyclic perfect orientation $\O$
with source set $I_j$.
\end{lem}

Recall from 
\cref{def:faces}
that we can label each face of $G$ by the source face label,
or equivalently, by the corresponding Young diagram contained in an
$m \times (n-m)$ rectangle.
In what follows, we will always choose our perfect orientation to be acyclic.

\begin{definition}\label{def:flow}
A \emph{flow} $F$ from $I_{\O}$ to a set $J$ of boundary vertices
with 
$|J|=|I_{\O}|$ 
is a collection of paths
in $\O$, all pairwise vertex-disjoint, 
such that the sources of these paths are $I_{\O} - (I_{\O} \cap J)$
and the destinations are $J - (I_{\O} \cap J)$.
Note that each path %
$w$ in $\O$ partitions the faces of $G$ into those 
which are on the left and those which are on the right of the 
walk. %
We define the \emph{weight} $\wt(w)$ of 
each such path %
to be the product of 
parameters $x_{\mu}$, where $\mu$ ranges over all face labels 
to the left of the path.  And we define the 
\emph{weight} $\wt(F)$ of a flow $F$ to be the product of the weights of 
all paths in the flow.

Fix a perfect orientation $\O$ of a reduced plabic graph $G$.
Given $J \in {[n] \choose n-k}$, 
we define the {\it flow polynomial}
\begin{equation}\label{eq:Plucker2}
P_J^G = \sum_F \wt(F),
\end{equation}
where $F$ ranges over all flows from $I_{\O}$ to $J$.
\end{definition}

\begin{definition}\label{def:PG}
Let $\tilde{\mathcal{P}}_G$
denote the set of all Young diagrams labeling the faces of $G$.
Let $\minimal$ denote the minimal partition in $\tilde{\mathcal{P}}_G$.
From now on we will always
choose our perfect orientation to be acyclic  with source set $I_{1}$.
We prefer this choice because
then the variable
$x_{\minimal}$ never appears in the expressions for flow polynomials.
Let ${\mathcal P}_G:=
\tilde{\mathcal P}_G\setminus \{\minimal\}$.
Let \begin{equation}\label{e:TB}
{\TB}(G) := \{x_{\mu}\ |\ \mu\in 
	{\mathcal P}_G \}
\end{equation}
be a set of  
parameters which are indexed by the Young diagrams $\mu\in {\mathcal P}_G$.
\end{definition}

\begin{remark}\label{rem:flows}
We can also define flows in the network $N(D)$
(associated to a $\Le$-diagram $D$) as collections
of vertex-disjoint paths, as was done in \cref{sec:Xcluster},
see \cref{SchubertNetwork}.
If we consider the plabic graph $G(D)$
associated to $D$ and direct all edges west, south, or southwest,
then the  flows in $N(D)$ are equivalent
to the flows in this orientation $\O$ of $G(D)$.
If we label the faces of the plabic graph $G(D)$
by source labels, then map the source labels to partitions,
we obtain the labeling by rectangles shown in \cref{SchubertNetwork}.
The flow polynomials coming from the network 
associated to a $\Le$-diagram are denoted by $P_J^{\rect}$.
\end{remark}

We now describe the network chart for 
the open positroid variety 
$X^{\circ}_G$ associated to a reduced
plabic graph $G$.
The statement in \cref{thm:Talaska} concerning the positroid cell
below comes from \cite{Talaska} and \cite[Theorem 12.7]{Postnikov}, while the extension to 
$X^{\circ}_G$ 
 comes from 
 \cite{MullerSpeyer}.

\begin{theorem}[{\cite[Theorem 12.7]{Postnikov}}]\label{thm:Talaska}
Let $G$ be a reduced plabic graph, and choose
an acyclic perfect orientation $\O$. 
Then the map $\Phi_G$ sending 
$(x_{\mu})_{\mu \in \mathcal{P}_G}
\in (\C^*)^{\mathcal{P}_G}$ to the 
collection of flow polynomials 
$\{P_J^G\}_{J\in {[n]\choose m}} \in \mathbb{P}^{{[n] \choose m}-1}$
is an injective map onto a dense open subset of $X_G^{\circ}$ in its
Pl\"ucker embedding.
The restriction of $\Phi_G$ to 
$(\R_{>0})^{\mathcal{P}_G}$ 
gives a parameterization of the positroid cell $S_G$.
	We call the map $\Phi_G$ a \emph{network chart} for $X_{G}^{\circ}$.
\end{theorem}

\begin{definition}[Network torus $\mathbb T_G$]  \label{d:networktorus2}
Define the open dense torus $\mathbb T_G $ in $\openX$ to be the image  of the network chart $\Phi_G$, namely
$\mathbb T_G
:=\Phi_G((\C^*)^{\mathcal{P}_G})$.
We call $\mathbb T_G$ the {\it network torus} in $\openX$ associated to $G$. 
\end{definition}

\begin{example}\label{Ex:Transc}
Since the image of $\Phi_G$ lands in 
$X_G^{\circ}$ (see \cite[Section 1.1]{MullerSpeyer}), %
we can view the parameters $\TBG$ as rational functions on $X_G$ which restrict to coordinates on the open torus $\mathbb T_G$. Therefore we can think of $\TBG$ as a transcendence basis of $\C(X_G)$.
\end{example}

\begin{definition}[Strongly minimal 
 and pointed]\label{def:minimal}
We say that a Laurent monomial 
$\prod_\mu x_\mu^{a_\mu}$ appearing
in a Laurent polynomial $P$ is \emph{strongly minimal}
 in $P$  if 
for every other Laurent monomial $\prod_\mu x_\mu^{b_\mu}$ occurring in $P$,
we have $a_\mu\le b_{\mu}$ for all $\mu$.  
(We can similarly define \emph{strongly maximal} by replacing $\le$
by $\geq$.)
If $P$ has a strongly minimal Laurent
monomial with coefficient $1$, then we say  $P$ is \emph{pointed}.
\end{definition}

\begin{remark}
Flow polynomials $P_J^G$ from plabic graphs 
are always strongly minimal, strongly maximal, and pointed
	\cite[Corollary 12.4]{RW}.
\end{remark}

In general, given a reduced plabic graph $G$, there are many other
plabic graphs in its move-equivalence class, and each one gives rise
to a network chart for the cell $S_G$ and positroid variety $X_G^{\circ}$.
However, this is just a subset of the (generalized) 
network charts we can obtain
by applying \emph{cluster $\mathcal{X}$-mutation}.

\begin{definition} \label{Xseed}
Let $Q$ be a quiver with vertices $V$,  associated exchange matrix $B$
(see \cref{quiver}), and with a 
parameter $x_{\tau}$ associated to each vertex $\tau \in V$.
If $\sigma$ is a mutable vertex of $Q$,
then we define a new set of 
parameters  $\MutVar_{\sigma}^{\mathcal{X}}(\{x_{\tau}\}) := \{x'_{\tau}\}$  where
\begin{equation}\label{e:XclusterMut2}
x'_{\tau} = \begin{cases}
       \ \frac{1}{x_{\sigma}} & %
      \text{if }\tau = \sigma, \\
   \ x_{\tau}(1+x_{\sigma})^{b_{\sigma \tau}}  
  &\text{if there are }b_{\sigma \tau}
\text{ arrows from }\sigma \text{ to } \tau \text{ in }Q, \\
  \ \frac{x_{\tau}}{(1+x_{\sigma}^{-1})^{b_{\tau \sigma}}} &%
 \text{if there are }b_{\tau \sigma}
\text{ arrows from } \tau \text{ to } \sigma \text{ in }Q,\\
  \ x_{\tau} & %
 \text{ otherwise.}
\end{cases}
\end{equation}
We say that 
$(\Mut_{\sigma}(Q), \{x'_{\tau}\})$ is obtained from 
 $(Q, \{x_{\tau}\})$ by \emph{$\mathcal{X}$-seed mutation}
 in direction $\sigma$, and we refer to the ordered pairs 
$(\Mut_{\sigma}(Q), \{x'_{\tau}\})$  and
 $(Q, \{x_{\tau}\})$ as \emph{labeled $\mathcal{X}$-seeds}.
	Note that if we apply the $\mathcal{X}$-seed mutation in direction $\sigma$ to 
$(\Mut_{\sigma}(Q), \{x'_{\tau}\})$, we obtain  
 $(Q, \{x_{\tau}\})$ again.
	
	We say that two labeled $\mathcal{X}$-seeds are 
	\emph{$\mathcal{X}$-mutation equivalent} if one can be obtained
	from the other by a sequence of $\mathcal{X}$-seed mutations.
\end{definition}

Each reduced plabic graph $G$ gives rise to a 
labeled $\mathcal{X}$-seed 
$\Sigma_G^{\mathcal{X}}:=(Q(G), 
{\TB}(G))$. Moreover, the flow polynomial expressions for 
Pl\"ucker coordinates are 
compatible with  $\mathcal{X}$-mutation 
\cite[Lemma 6.15]{RW}: 
whenever two 
plabic graphs are connected by moves, 
the corresponding
$\mathcal{X}$-seeds are $\mathcal{X}$-mutation equivalent.
We can get a larger class of $\mathcal{X}$-seeds by 
mutating at any sequence of mutable vertices of $Q(G)$, obtaining
(generalized) network charts.  We continue to index our 
$\mathcal{X}$-seeds by $G$, even when 
they do not come from a plabic graph.

The following result is from
\cite[Proposition 7.6]{RW} (whose proof holds verbatim for
arbitrary positroids). 
\begin{proposition}\label{thm:Laurent}
Any Pl\"ucker coordinate $P_{\nu}$, when expressed in terms of
a general $\mathcal{X}$-cluster $G$,
is a Laurent polynomial in $\TBG$.
\end{proposition}

\subsection{Cluster charts
from plabic graphs}\label{sec:cluster}

The following result was proved for open Schubert varieties in \cite{SSW}
and skew-shaped positroid varieties in \cite{GKSS}, see also 
\cite{SSW} for the case of skew Schubert varieties and \cite{GalashinLam} for the 
extension to positroids.
For background on cluster algebras, see \cite{ca1, FWZ1, FWZ2}.

\begin{theorem}\label{thm:cluster}
Let $G$ be a reduced plabic graph and
consider the open positroid variety
 $X_{G}^\circ$.
Construct the dual quiver of $G$ and
label its vertices by the Pl\"ucker coordinates
given by the source labeling of $G$; the frozen
vertices are those corresponding to the boundary regions of $G$.
This gives rise to a labeled seed
$\Sigma_G^{\mathcal{A}}$ and a
cluster algebra
$\mathcal{A}
        (\Sigma^{\mathcal{A}}_G)$, which coincides with
 the coordinate ring
        $\CC[\widehat{X}_{G}^\circ]$
of the (affine cone over) $X_{G}^\circ$.
\end{theorem}

\begin{figure}
\includegraphics[height=1.3in]
{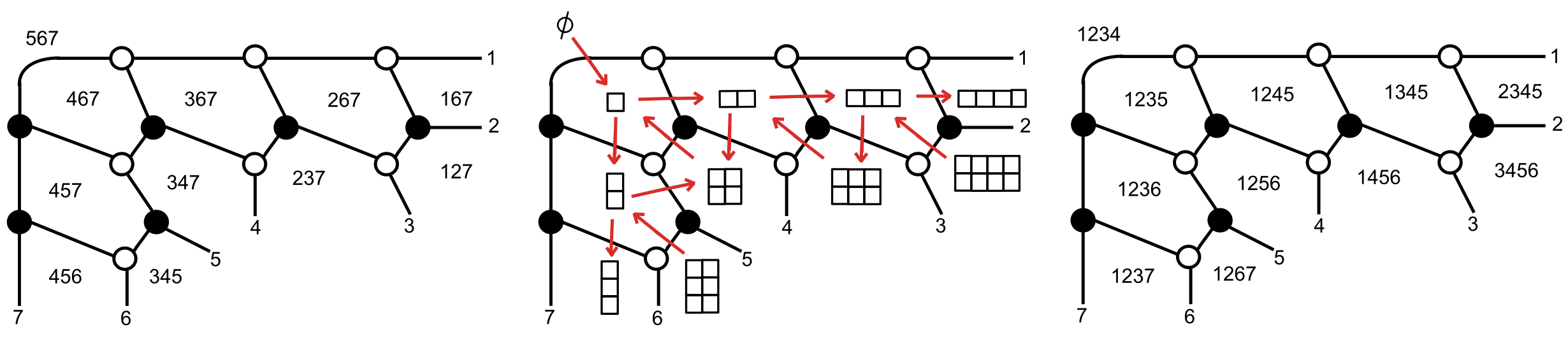}
	\caption{The plabic graph $G=G(D)$ associated to the 
	Le-diagram $D$ which is a Young diagram of shape 
	$\lambda=(4,4,2)$ filled with all $+$'s.
	The left picture has faces labeled using the source face labels; 
	the middle picture has faces labeled by partitions
	(whose vertical steps correspond to the 
	face labels at the left); and the right picture has faces
	labeled by the subset corresponding to the horizontal labels
	of the partitions.  The middle picture also 
	includes the dual quiver $Q(G)$ and recovers the example
	from the left of \cref{fig:restrictedseed}.}
\label{fig:plabic}
\end{figure}

\begin{remark}\label{rem:rectangles}
If $D$ is a $\Le$-diagram, we refer to the cluster obtained from 
$G(D)$ as the \emph{rectangles cluster}.  This is because
when  $D$ is a Young diagram filled with all $+$'s (i.e. when
$X_{G(D)}^{\circ}$ is an open Schubert variety), 
the faces of $G(D)$ are all labeled by rectangular partitions.
See \cref{fig:plabic} for an example.
	In this case we also refer to $G(D)$ as $G_D^{\rect}$.
\end{remark}

\cref{thm:cluster}
implies that the set
$\{p_{\mu} \ \vert \ \mu \in \tilde{\mathcal{P}}_G\}$
of
Pl\"ucker coordinates indexed by the faces of $G$
is a \emph{cluster} for the {cluster algebra}
associated to the
homogeneous coordinate ring of $X_G^{\circ}$.
In particular, these Pl\"ucker coordinates are 
called \emph{cluster variables} and are
algebraically independent; moreover, \emph{any}
Pl\"ucker coordinate for $\checkX$ can be written uniquely as a
positive Laurent polynomial in the variables from 
$\{p_{\mu} \ \vert \ \mu \in \tilde{\mathcal{P}}_G\}$.

Recall from \cref{def:PG} that 
 $\minimal$ denotes the minimal partition in $\tilde{\mathcal{P}}_G$,
and let $\mathcal{P}_G = \tilde{\mathcal{P}}_G \setminus \{\minimal\}.$ Let
\[
	\PCG:=\left\{\frac{p_{\mu}}{p_{\minimal}}\ |\ \mu\in\tilde{P}_G \right\}\subset \C(X_G^{\circ}).
\]
If we choose  the normalization of
Pl\"ucker coordinates on $X_G^{\circ}$ such 
that $p_{\minimal} = 1$,
we get a map
\begin{equation}
\label{eq:clusterchart}
	\Phi_G^{\mathcal{A}}: (\C^*)^{{\mathcal P}_G}\to
	X_G^{\circ}
\end{equation}
which we call a \emph{cluster chart} for $X_G^{\circ}$, 
which satisfies $p_{\nu}(\Phi^\A_G((t_\mu)_{\mu}))=t_\nu$ for $\nu\in{\mathcal P}_G$.
When it is clear that we are setting $p_{\minimal}=1$ we may write
\begin{equation}\label{e:PCG}
\PCG:=\left\{p_{\mu}\ |\ \mu\in{\mathcal P}_G\right\}.
\end{equation}

\begin{definition}[Cluster torus $\mathbb T^\A_G$]  \label{rem2:convention}
Define the open dense torus $\mathbb T^\A_G $ in $\opencheckX$ as the image  of the cluster chart $\Phi^\A_G$,
\[
\mathbb T^\A_G
	:=\Phi_G^\A((\C^*)^{\mathcal{P}_G})=\{x\in X_G^{\circ} \mid \p_{\mu}(x)\ne 0 \text { for all $\mu\in \mathcal P_G$}\}.
\]
We call $\mathbb T^\A_G$ the {\it cluster torus} in $\opencheckX$ associated to $G$.
\end{definition}

We next describe cluster $\mathcal{A}$-mutation, and how it relates to the clusters associated to
plabic graphs $G$.

\begin{definition}\label{Aseed}
Let $Q$ be a quiver with vertices $V$ and associated exchange matrix $B$.
We associate a \emph{cluster variable} $a_{\tau}$ to each vertex $\tau \in V$.
If $\sigma$ is a mutable vertex of $Q$,
then we define a new set of
variables  $\MutVar_{\sigma}^{\mathcal{A}}(\{a_{\tau}\}) := \{a'_{\tau}\}$  where
$a'_{\tau} = a_{\tau}$ if $\tau \neq \sigma$, and otherwise, $a'_{\sigma}$ is determined
by the equation
\begin{equation}\label{e:AclusterMut}
a_{\sigma} a'_{\sigma} =
\prod_{b_{\tau \sigma} > 0} a_{\tau}^{b_{\tau \sigma}} +
\prod_{b_{\tau \sigma} < 0} a_{\tau}^{-b_{\tau \sigma}}.
\end{equation}
We say that
$(\Mut_{\sigma}(Q), \{a'_{\tau}\})$ is obtained from
 $(Q, \{a_{\tau}\})$ by \emph{$\mathcal{A}$-seed mutation}
 in direction $\sigma$, and we refer to the ordered pairs
$(\tau_{\sigma}(Q), \{a'_{\tau}\})$  and
 $(Q, \{a_{\tau}\})$ as \emph{labeled $\mathcal{A}$-seeds}.
	We say that two labeled $\mathcal{A}$-seeds are
	$\mathcal{A}$-mutation equivalent if one can be obtained
	from the other by a sequence of $\mathcal{A}$-seed mutations.
\end{definition}

Using the terminology of \cref{Aseed}, each reduced plabic graph $G$ gives rise to a
labeled $\mathcal{A}$-seed.  Moreover, it is easy to verify that the
square move on a plabic graph corresponds to 
an $\mathcal{A}$-mutation.
However, there are many $\mathcal{A}$-seeds that do not come from plabic
graphs.  We will continue to index $\mathcal{A}$-seeds, cluster charts,
and cluster tori by $G$ even when they do not come from plabic graphs.

\section{Homology classes of positroid divisors}\label{a:positroidhomology}

In this appendix we provide the proof of \cref{p:PositroidHomology}. 
We first review
some standard results about the homology of flag varieties. 
Consider 
the full flag variety  $GL_{n}/B_+$ (over $\C$) and the projection map 
\[
\pi:GL_{n}/B_+\to  GL_n/P_{n-k}=\mathbb X,
\]
and recall from \cref{s:LeDivisors}
that $W$ denotes the symmetric group $S_{n}$, 
whose elements we can identify with permutation matrices. 
\begin{enumerate}
	\item The \textit{Bruhat cells} in the full flag variety, $\Omega_w=B_+w B_+/B_+$ define an algebraic cell decomposition $GL_n/B_+=\bigsqcup_{w\in W}\Omega_w$ called the \emph{Bruhat decomposition}. The dimension of a cell $\Omega_w$ is given by $\ell(w)$. The closure relations are given by the \emph{Bruhat order}, $\Omega_v\subseteq\overline{\Omega_w}$ if and only if $v\le w$.  
	\item The flag variety also has an \emph{opposite Bruhat decomposition} given by the \textit{opposite Bruhat cells} $\Omega^v=B_-v B_+/B_+$.  The opposite Bruhat cell $\Omega^v$ has codimension $\ell(v)$. We have that $w_0\Omega^{v}=\Omega_{w_0v}$, where $w_0$ is the longest element in $W$. The closure relation for the opposite Bruhat cells is  given by $\Omega^v\subseteq\overline{\Omega^w}$ if and only if $v\ge w$.  
  \item The fundamental classes of the closures, the Schubert classes $\overline{\Omega}_w$,  form a basis of the homology $H_*(GL_n/B_+)=\sum_{w\in W}\Z[\overline {\Omega}_w]$.
\item If $w\in W^{P_{n-k}}$, then $w$ is of the form $w_\mu$ for some Young diagram $\mu$ that fits into a $(n-k)\times k$ rectangle, see \cref{r:Young2}. The map on homology 
\[
\pi_*: H_*(GL_{n}/B_+,\Z)\to  H_*(\mathbb X,\Z)
\]
sends the $[\overline{\Omega}_{w_\mu}]$ to a basis of $H_*(\mathbb X,\Z)$ indexed by $\mathcal P_{k,n}$, and it sends $[\overline{\Omega}_{w}]$ to $0$ for all other $w\in W$.
\item The homology class $[\overline{\Omega^{w}}]$ is  identified by Poincar\'e duality with a cohomology class that we call $\sigma^{w}\in H^{2\ell(w)}(GL_n/B_+)$. These classes form the Schubert basis of $H^*(GL_n/B_+)$.
\item The cohomology ring of $H^*(GL_n/B_+)$ is generated by the degree $2$ Schubert classes $\sigma^{s_i}$ associated to the simple reflections $s_i=(i, i+1)$. Let us also write $t_{mr}$ for the transposition $(m,r)$. We have the Pieri formula by which
\begin{equation}\label{e:Pieri}
\sigma^{s_i}\cup \sigma^w=\sum_{\substack{ m\le i<r\\ \ell(wt_{mr})=\ell(w)+1 }}\sigma^{wt_{mr}}.
\end{equation}
\end{enumerate}

We relate our conventions concerning flag varieties with those concerning Grassmannians.  Recall the bijection $\mathcal P_{k,n}\to{{[n]}\choose{n-k}}$ from Section~\ref{s:Young}, and the bijection $\mathcal P_{k,n}\to W_{P_{n-k}}$ from Remark~\ref{r:Young2}.

\begin{lem}\label{l:ProjSchub}
Let $\mu\in P_{k,n}$. The Schubert variety $X_\mu\subset\mathbb X$ is the projection of the closure of the opposite Bruhat cell $\Omega^{w_0w_\mu}$, 
\[
\pi(\overline{\Omega^{w_0w_\mu}}
)=X_\mu.
\]
The opposite Bruhat decomposition has the property that $\pi_*([\overline{\Omega^{w_0w_\mu}}])=[X_\mu]$, and $\pi_*([\Omega^{w}])=0$ if $w$ is not of the form $w_0w_\mu$.  
\end{lem}
\begin{remark}\label{r:cohproj} Note that by (2) above with $v=w_0w$ we have $\Omega^{w_0w}=w_0\Omega_{w}$ and therefore $[\overline{\Omega^{w_0w}}]=[\overline{\Omega}_{w}]$, since translation by an element of the connected group $GL_n(\C)$ will not affect the homology class. The homological statement of the lemma can therefore also be written as $\pi_*([\overline{\Omega}_{w_{\mu}}])=[X_\mu]$. 
\end{remark}

While \cref{l:ProjSchub} is well-known, we include a proof for completeness and because it will
be useful for our subsequent proof of 
 \cref{p:PositroidHomology}.
\begin{proof}[Proof of \cref{l:ProjSchub}]
Consider the SE border of $\mu$ as a path from the SW corner of the rectangle \textit{up} to the NE corner, and number the steps with $\{1,\dotsc, n\}$. Let $m_1<m_2<\dotsc < m_{n-k}$ be the labels of the vertical steps and $m'_1<\dotsc < m_k'$ the labels of the horizontal steps. The permutation $w_\mu$ is known to be the permutation 
	with  a single descent at $n-k$ given by
\begin{equation}
w_\mu=\begin{pmatrix} 
1 &2     &\dotsc & n-k &n-k+1&\dotsc & n\\
m_1& m_2 &\dotsc & m_{n-k}& m'_1&\dotsc &m'_k
\end{pmatrix}.
\end{equation}
Then
\begin{equation}
w_0w_\mu=\begin{pmatrix} 
1 &2     &\dotsc & n-k &n-k+1&\dotsc & n\\
n-m_1+1& n-m_2+1 &\dotsc & n-m_{n-k}+1& n-m'_1+1&\dotsc &n-m'_k+1
\end{pmatrix}.
\end{equation}

In contrast, our conventions for Pl\"ucker coordinates and the definition of $X_\mu$  involved a labelling of steps starting from the NE corner and increasing \textit{down} to the SW corner. See Section~\ref{s:Young}. It follows that the subset of $\{1,\dotsc, n\}$ associated with $\mu$ in that section coincides with the set $\{w_0w_\mu(1),\dotsc, w_0w_\mu(n-k)\}$. 

Consider the action of $B_-$ and of the maximal torus $T$ of $GL_n$ on the flag variety~$GL_n/B_+$ and on the Grassmannian~$\mathbb X$. These actions are compatible in that $\pi$ is an equivariant map. Recall that $\Omega^{w_0w_\mu}$ is by definition the $B_-$-orbit of the torus-fixed point $w_0w_\mu B_+$. 

On the other hand, consider the $n\times (n-k)$ matrix $A=A_\mu$ with columns given by standard basis vectors 
$e_{w_0w_\mu(1)},\dotsc,e_{w_0w_\mu(n-k)}$. From the discussion above it follows that $P_{\nu}(A)\ne 0$ if and only if $\nu=\mu$. We have that $A_\mu$ lies in $\Omega_\mu$, compare \cref{d:SchubertCell}. Moreover $A_\mu$ is a torus-fixed point and the Grassmannian Schubert cell $\Omega_\mu$ is precisely the $B_-$-orbit of  $A_\mu$.

The matrix $A_\mu$ agrees with the first $n-k$ columns of the permutation matrix $w_0w_\mu$. It follows that $\pi(w_0w_\mu)=A_\mu$ and  therefore also $\pi(\Omega^{w_0w_\mu })=\Omega_\mu$ and  $\pi(\overline{\Omega^{w_0w_\mu}})=X_\mu$, proving the lemma.
\end{proof}

Recall the notation for the removable boxes of $\lambda$ as $b_{\rho_{2\ell+1}}$ for $\ell=1,\dotsc d$ from \cref{r:R(lambda)}, and the notation for the NW border boxes of $b'_1,\dotsc, b'_{n-1}$ from \cref{d:Wis}.

 \begin{theorem}[\cref{p:PositroidHomology}] 
The homology class  
of the positroid divisor $D'_i=X_{(s_i,\lambda)}$ is expressed in terms of the Schubert classes $[D_\ell]=[X_{\lambda^-_\ell}]$ by
\[
	[X_{(s_i,\lambda)}]=\sum_{\ell\in SE(b'_i)} [D_\ell], \hspace{.5cm} 
	 \text{ where }
	 \hspace{.5cm}
SE(b'_i):=\{\ell\mid \text{The box $b_{\rho_{2\ell+1}}$ is $SE$ of $b_i'$}\}. 
\]
 \end{theorem}
 \begin{proof}
The positroid divisor $X_{(s_i,w_\lambda)}$ is the projected image of the Richardson variety $\overline {\mathcal R_{s_i,w_\lambda}}$ under $\pi:GL_n/B_+\to\mathbb X$. Moreover we have $\overline {\mathcal R_{s_i,w_\lambda}}=\overline{\Omega^{s_i}}\cap\overline{\Omega}_{w_\lambda}$. The homology class $[\overline{\Omega^{s_i}}]$ is Poincar\'e dual to $\sigma^{s_i}$. Furthermore, by \cref{r:cohproj}, we have that $[\overline{\Omega}_{w_\lambda}]=[\overline{\Omega^{w_0w_\lambda}}]$, so this is the Poincar\'e dual class to $\sigma^{w_0w_\lambda}$. It follows that the homology class $[\overline {\mathcal R_{s_i,w_\lambda}}]$ is the Poincar\'e dual class to the cup product $\sigma^{s_i}\cup \sigma^{w_0w_\lambda}$. The Pieri formula~\eqref{e:Pieri} with $w=w_0w_\lambda$ translates to  
\begin{equation}
\sigma^{s_i}\cup \sigma^{w_0w_\lambda}=\sum_{ \substack {m\le i<r\\ \ell(w_\lambda t_{mr})=\ell(w_\lambda)-1} }\sigma^{w_0w_\lambda t_{mr}}.
\end{equation}
It follows that in homology 
\begin{equation}\label{e:HomPieri}
[\overline {\mathcal R_{s_i,w_\lambda}}]=\sum_{ \substack{ m\le i<r\\ \ell(w_\lambda t_{mr})=\ell(w_\lambda)-1} }[\overline{\Omega^{w_0 w_\lambda t_{mr}}}].
\end{equation}
We have that $\pi_*([\overline{\Omega^{w_0w_\mu}}])=[X_\mu]$ and all other $\pi_*([\overline{\Omega^{w}}])=0$, see \cref{l:ProjSchub}. Applying $\pi_*$ to \eqref{e:HomPieri} we therefore get the identity
\begin{equation}\label{e:HomPieriGrass}
[X_{(s_i,w_\lambda)}]=\sum_{\substack {m\le i<r \\ w_\lambda t_{mr}=w_\mu \text{ with } |\mu|=|\lambda|-1} }[X_\mu]
\end{equation}
in $H_*(\mathbb X,\Z)$ and also in $H_*(X_\lambda,\Z)$, since this is a submodule. The $\mu$ appearing in the sum must be of the form $\lambda^-_\ell$, since $|\mu|=|\lambda|-1$. Therefore the summands are indeed of the form $[X_{\lambda^-_\ell}]=[D_\ell]$. 
It remains to check the following claim. 
\vskip .2cm
\noindent {\it Claim:} The permutation $w_{\lambda^{-}_\ell}$ is of the form $w_\lambda t_{mr}$ for some $m\le i<r$ if and only if $\ell\in SE(b_i')$. 
\vskip .2cm
\noindent {\it Proof of the Claim:}
Let us write $\mu$ for $\lambda_\ell^-$. Note that the statement $\ell\in SE(b'_i)$ is equivalent to saying that the box $b'_i$ is NW of the removed box $b_\mu:=b_{\rho_{2\ell+1}}$. In the NW rim there is a unique box $b'_m$ to the west of $b_\mu$, and a unique box $b'_{r-1}$ to the north of $b_\mu$. Since the boxes along the rim are counted starting from the bottom clockwise we have $1<m\le n-k$ and $n-k<r\le n$.      

Suppose the permutation $w_\lambda$ is given by
\begin{equation}
w_\lambda=\begin{pmatrix} 
1 &2     &\dotsc & n-k &n-k+1&\dotsc & n\\
c_1& c_2 &\dotsc & c_{n-k}& c_{n-k+1}&\dotsc &c_n
\end{pmatrix}.
\end{equation}
Then, as in the proof of \cref{l:ProjSchub}, $c_1,\dotsc, c_{n-k}$ are the vertical steps of the SE border of $\lambda$ counted from the bottom, while  $c_{n-k+1},\dotsc, c_n$ are the horizontal steps that were left out. Removing the box $b_\mu$ from $\lambda$ amounts to swapping $c_m$ and $c_r$. Therefore $w_\mu=w_\lambda t_{mr}$. Now the condition that $m\le i<r$ becomes the condition that $b'_i$ is NW of the removed box $b_\mu$. This completes the proof of the claim and the theorem.  
 \end{proof}

\bibliographystyle{alpha}
\bibliography{bibliography}

\end{document}